\documentclass[11pt]{article}

\usepackage[T1]{fontenc}
\usepackage{lmodern}
\usepackage[margin=1in]{geometry}
\usepackage{amsmath,amssymb,mathtools,amsthm}
\usepackage{graphicx}
\usepackage{xcolor}
\usepackage{enumitem, verbatim}
\usepackage{caption}
\usepackage{tikz}
\usetikzlibrary{arrows.meta,calc,shapes.geometric}
\usepackage{microtype}
\usepackage[colorlinks=true,linkcolor=blue!55!black,
  citecolor=blue!55!black,urlcolor=blue!55!black]{hyperref}
\hypersetup{
  pdftitle={Limit of the octahedron recurrence with GUE boundary data},
  pdfauthor={Hariharan Narayanan}
}

\newtheorem{theorem}{Theorem}
\newtheorem{proposition}[theorem]{Proposition}
\newtheorem{lemma}[theorem]{Lemma}
\newtheorem{corollary}[theorem]{Corollary}
\newtheorem*{supportreciprocity}{Support reciprocity}
\newtheorem*{targetbinarity}{Binarity of the target denominator vector}
\theoremstyle{definition}
\newtheorem{definition}[theorem]{Definition}
\theoremstyle{remark}
\newtheorem{remark}[theorem]{Remark}
\newtheorem*{partiicorollary}{%
  Corollary~\ref{cor:rf-rational-approximants}}
\newtheorem*{partiimain}{%
  Main theorem of Part~II, Theorem~\ref{thm:rf-uniform-continuum}}

\newcommand{\1}{\mathbf 1}
\newcommand{\D}{\mathcal D}
\newcommand{\Z}{\mathbb Z}

\newcommand{\edge}{{\diamond}}

\graphicspath{{./}}
\allowdisplaybreaks

\title{Limit of the octahedron recurrence with GUE boundary data}
\author{Hariharan Narayanan\\
\small School of Technology and Computer Science\\
\small Tata Institute of Fundamental Research, Mumbai}
\date{}

\begin{document}
\maketitle

\begin{abstract}
We study the tropical octahedron recurrence on the lattice tetrahedron
$\{(x_1,x_2,x_3,x_4)\in\mathbb Z_{\geq0}^4:
x_1+x_2+x_3+x_4=n\}$ with an intrinsic double-hive boundary law.
Let $M_1^{(n)},M_2^{(n)},M_3^{(n)}$ be independent standard
$n\times n$ GUE matrices and, for fixed $\ell_1,\ell_2,\ell_3>0$, set
\[
 X_r^{(n)}=\ell_r\sqrt n\,M_r^{(n)},\qquad r=1,2,3.
\]
The boundary law on the two upper panels is the Gibbs density on the
cone of double hives whose four exterior sides correspond, in the
prescribed orientations, to the spectra of
$X_1^{(n)},X_2^{(n)},X_3^{(n)}$, and
$X_1^{(n)}+X_2^{(n)}+X_3^{(n)}$.  We then perform an octahedron-recurrence
sweep to obtain the value at every lattice point of the tetrahedron.
After scaling lattice positions by $n^{-1}$ and recurrence values by
$n^{-2}$, we prove that the resulting random field converges uniformly
in probability to a deterministic Lipschitz function on
$\{(x_1,x_2,x_3,x_4)\in\mathbb R_{\geq0}^4:
x_1+x_2+x_3+x_4=1\}$.
For every strictly interior target lattice point we construct a canonical planar
bipartite graph supported on a union of three triangular panels and
prove an exact perfect-matching formula for the recurrence value.  Its
variable weights are explicit interlacing gaps from the minor processes
of \(X_1^{(n)},X_2^{(n)},X_3^{(n)}\), together with explicit seam and
matching-independent terms.  The asymptotic analysis of this formula
yields a surface-tension variational characterization of the
deterministic limit. This matching
rule supplies, in the present four-simplex setting, the kind of
higher-dimensional combinatorial formula identified as a major
cluster-algebra problem by Henriques and Speyer.
\end{abstract}
\tableofcontents
\section{Introduction}
\label{sec:introduction}

We consider the $d$-dimensional birational octahedron recurrence on
the lattice simplex
\[
 \{(x_1,\ldots,x_{d+1})\in\Z_{\geq0}^{d+1}:
                 x_1+\cdots+x_{d+1}=N\}.
\]
Its initial variables $\{X_q:q\in\mathtt{seed}\}$ are indexed by a
two-dimensional membrane.  The recurrence takes values in Laurent
polynomials in these variables and is given by

\begin{equation*}
 X_{\lambda+e_i+e_k}X_{\lambda+e_j+e_\ell}
 =
 X_{\lambda+e_i+e_j}X_{\lambda+e_k+e_\ell}
 +
 X_{\lambda+e_i+e_\ell}X_{\lambda+e_j+e_k},
 \qquad i<j<k<\ell.
\end{equation*}
The associated tropical octahedron
recurrence takes values in the real numbers and is given by
\begin{equation*}
 h_{\lambda+e_i+e_k}+h_{\lambda+e_j+e_\ell}
 =
 \max\!\left\{
 h_{\lambda+e_i+e_j}+h_{\lambda+e_k+e_\ell},
 h_{\lambda+e_i+e_\ell}+h_{\lambda+e_j+e_k}
 \right\}.
\end{equation*}

Our main result proves a uniform scaling limit for a random field on a
three-dimensional tetrahedron.  The field is obtained by sampling a GUE
double hive on two initial faces and applying the tropical octahedron
recurrence; the boundary law is defined precisely below.  This extends
the two-source limit theorem of \cite{NarLimit} from a triangular hive
to a tetrahedral field.

In the two-source setting, two independent GUE minor processes provide
hives on a two-panel membrane in a tetrahedron.  Speyer's matching
formula \cite{Speyer} expresses each value on an output triangle as a
maximum of weights determined by those two input hives.  Here a third
independent GUE minor process provides a third input panel.  The initial
membrane is still two-dimensional, but it lies in a four-dimensional
simplex, and the output occupies the tetrahedron opposite the common
vertex of the three input panels.  To analyze this output, we construct
for each interior target a single planar matching graph directly on
the initial fan, with explicit weights and boundary corrections.  This
exact formula is the starting point of the probabilistic limit proof.

Henriques and Speyer highlighted the problem of finding combinatorial
formulas for Laurent polynomials of the multidimensional octahedron
recurrence \cite[Introduction]{HenriquesSpeyer}.  The case treated here
is the three-panel fan in a four-dimensional simplex; our matching
formula and its tropical version are stated in
Theorem~\ref{thm:partI-main}.

Throughout, \(u_0,\ldots,u_4\) denote the five vertices of the ambient
simplex.  Compact index strings abbreviate the corresponding indexed
faces: for example, \(02\), \(012\), and \(1234\) denote the edge,
panel, and tetrahedron spanned by the indicated \(u_i\)'s.  In figures,
the label \(u_i\) is abbreviated to \(i\).

The probabilistic object in this paper is a random field on the lattice
tetrahedron \(1234\).  Its law is defined intrinsically: on the top
membrane \(123\cup_{13}134\), we sample a double hive obtained by restricting a quadruply augmented double
hive from an explicit maximum-entropy Gibbs density and then evolve it
through \(1234\) by the tropical octahedron recurrence.  The precise
definition is given below, after the requisite double-hive notation has
been introduced.

Theorem~\ref{thm:rf-top-bottom-double-hive} gives a concrete random-matrix
realization of this intrinsic law.  Start with three hives (derived using the "large-gap" construction) from independent scaled
GUE minor processes on the three-panel fan
\[
   \Sigma_0=012\cup023\cup034
\]
and perform the ordinary octahedron sweeps through \(0123\), \(0134\),
and \(1234\).  The field thereby produced on \(1234\) has exactly the
maximum-entropy law above.  In this realization the large-gap
construction serves only to transform the three Gelfand--Tsetlin patterns
into hives.  The top pair \(123\cup_{13}134\) and the bottom pair
\(124\cup_{24}234\) are the double-hive marginals of the Gibbs
law in its two triangulations and respective boundary conditions.  Consequently the two top faces are
conditionally independent given their \(13\)-diagonal, and the two
bottom faces are conditionally independent given their
\(24\)-diagonal.

The deterministic part of the argument applies to arbitrary admissible
initial data on \(\Sigma_0\).  Its target is a  interior lattice
point in a $3$-dimensional face (of the full $4$-simplex)
\[
   p=(0,a,b,c,d)\in1234,
   \qquad a,b,c,d\geq1,
   \qquad a+b+c+d=N.
\]

A target $p'=(\theta,a,b,c,d)$ with $\theta>0$ is treated by
restricting the recurrence to the translated $4$-simplex
\[
 \{(\theta',a',b',c',d'):\theta'\geq\theta,\quad
                    \theta'+a'+b'+c'+d'=N\}.
\]
The point $p'$ lies on a three-dimensional face of this simplex, so
the same construction applies, and so there is no essential loss of generality even if $\theta > 0$..

\begin{figure}[htbp]
\centering
\includegraphics[width=0.64\textwidth]
  {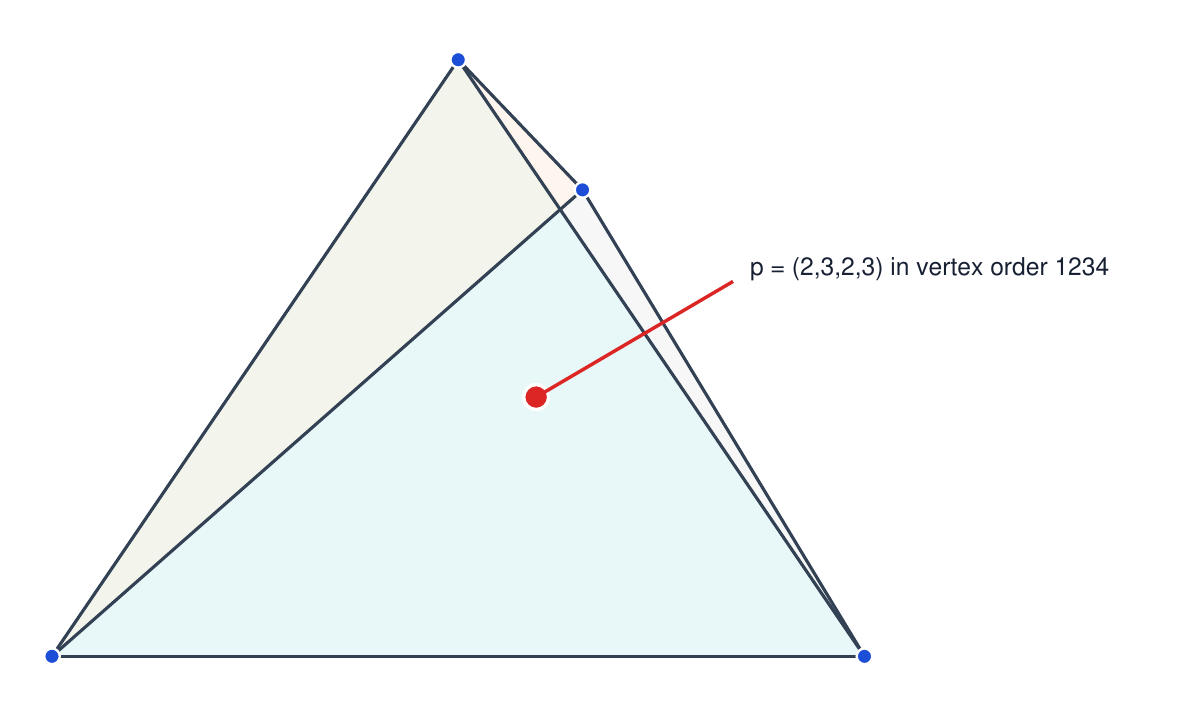}
\caption{A concrete interior target in the tetrahedron \(1234\).
The displayed point has barycentric coordinates
\((a,b,c,d)=(2,3,2,3)\); it illustrates the location of the target
exposed by the final sweep through \(1234\).}
\label{fig:interior-target-1234}
\end{figure}

Applying a single-tetrahedron
Speyer formula successively does not by itself produce a canonical ordinary
tiling on a flattened subset of the initial fan.  In addition, the
frozen boundary monomial is not common to all matching terms.  Each matching
has a further boundary-incidence contribution.

\subsection*{The fixed graph}

Our first main result gives a single graph on the initial fan.  A
reverse Boolean calculation identifies the positive denominator support
as the explicit set \(S_p\), which is connected and has no hole
(Theorem~\ref{thm:light-cone} and
Corollary~\ref{cor:support-disk}).  The matching formula later shows
that every positive denominator coordinate is one.  The boundaries of the selected dual
faces form a connected planar bipartite graph \(\overline G_p\).
Cutting open the two visible seams and inserting canonical alternating
zigzags produces \(G_p\) without changing its matching lattice.

For the running example \(p=(2,3,2,3)_{1234}\),
Figure~\ref{fig:intro-three-panel-tiling} shows the resulting fixed
three-panel region together with one legal lozenge--triangle tiling.
The corresponding lozenge-height values are introduced later, in
Section~\ref{sec:aligned-heights-intro}.

\begin{figure}[htbp]
\centering
\includegraphics[width=0.78\textwidth]{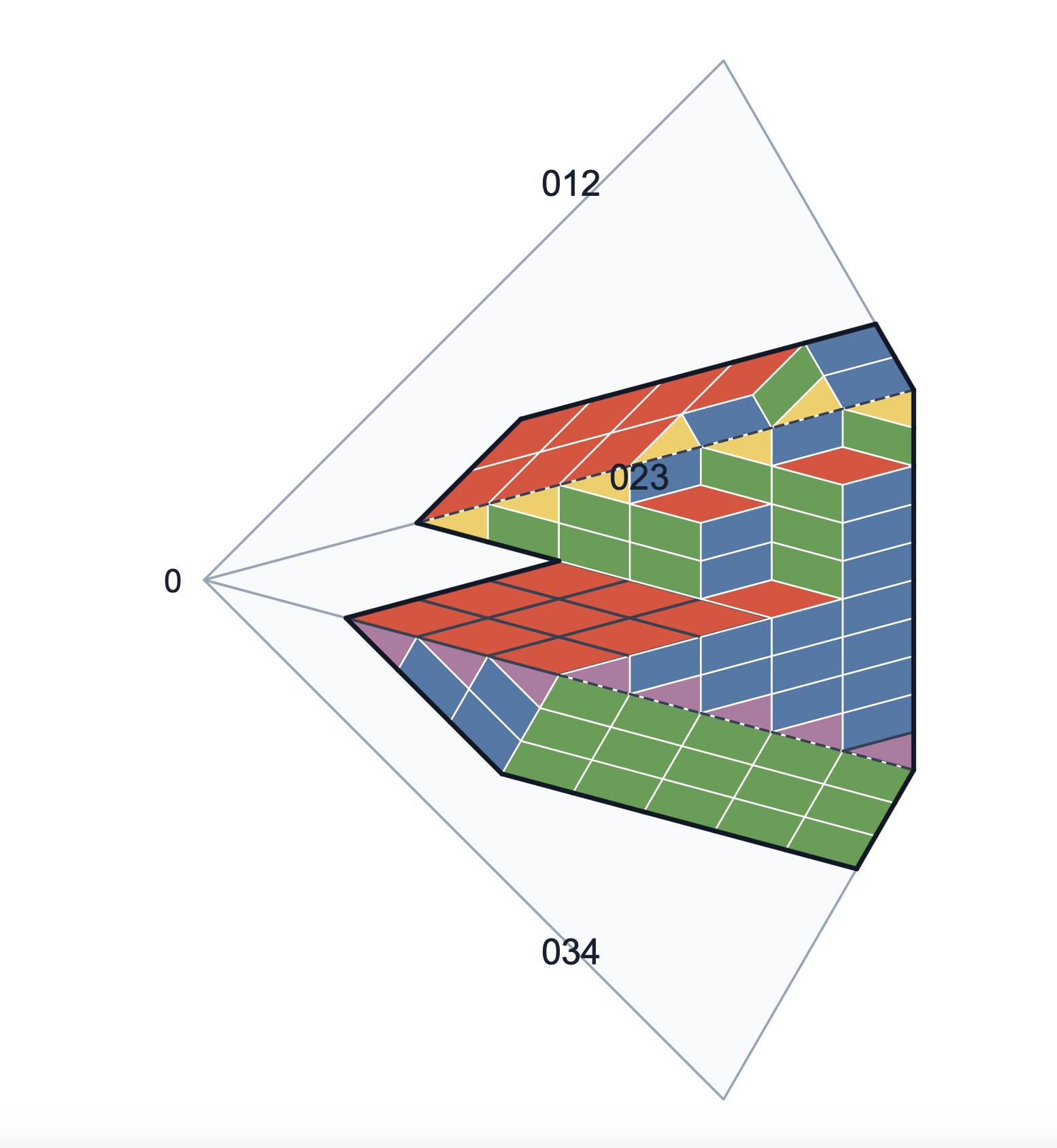}
\caption{A legal lozenge--triangle tiling of the fixed three-panel
region associated with \(p=(2,3,2,3)_{1234}\) on
\(\Sigma_0=012\cup023\cup034\).  The three principal colors
distinguish the ordinary lozenge orientations, while the yellow and
purple triangles are singleton seam triangles created by the cut-open
seams.}
\label{fig:intro-three-panel-tiling}
\end{figure}

The resulting perfect-matching formula contains the full frozen
correction.  For a matching \(M\), let \(m_v(M)\) count the matching
edges crossing primal edges incident to \(v\).
Theorem~\ref{thm:partI-main} proves
\[
       h(p)=\max_{M\in\mathcal D(G_p)}
                  \Lambda_M^{\mathrm{corr}}(h),
\]
where \(\Lambda_M^{\mathrm{corr}}\) contains the selected-face
exponents, the unselected mutable incidences, and the
matching-dependent frozen incidences.

\subsection*{The red-lozenge-free tiling formula}

For the probabilistic problem, the full matching formula is converted
by an explicit dimer gauge into a flip-symmetric red-free normalization.
Red lozenges have weight zero, while green and blue lozenges both have
weight \(3/2\) times their ordinary rhombus weight.  The same gauge gives
the two facing seam triangles the symmetric weights
\(-\lambda_{\min}/2\) and \(\lambda_{\max}/2\).  These seam weights are part of the gauge and
cannot be chosen independently of the bulk weights.  Theorem~\ref{thm:rf-exact-tiling} gives
\[
  h(p)=F_p^{\mathrm{rf}}(h)
       +\max_{\Xi\in\mathcal T(G_p)}W_p^{\mathrm{rf}}(\Xi;h).
\]
A complete local calculation shows that the matching-independent term collapses to
only two boundary values:
\[
   F_p^{\mathrm{rf}}(h)
       =\frac12h(v_0)+\frac12h(v_9)
       \qquad\text{(Theorem~\ref{thm:rf-explicit-fixed-term}).}
\]
This gives a tiling formula of exactly the form needed for the surface-tension
analysis: a blue--green bulk term, two one-dimensional seam terms, and a
fixed boundary statistic.

\subsection*{The intrinsic field and its GUE realization}

We now make the preceding intrinsic construction precise.

\begin{definition}[The maximum-entropy field on \(1234\)]
\label{def:rf-maximum-entropy-field}
Fix \(n\geq1\) and \(\ell_1,\ell_2,\ell_3>0\).  Write
\[
 \mathsf W_n=\{x\in\mathbb R^n:x_1\geq\cdots\geq x_n\},
 \qquad
 x^\vee=(-x_n,\ldots,-x_1).
\]
For \(\rho\in\mathsf W_n\), let
\(\operatorname{GT}_n(\rho)\) be the real Gelfand--Tsetlin polytope
with top row \(\rho\), and let
\(H_n(\rho,\sigma;\tau)\) be the real hive polytope with boundary
increment lists \(\rho,\sigma,\tau\), in the indicated orientations
and with one corner value normalized to zero.
The quadruply augmented double-hive cone in the \(13\)-triangulation is
\begin{align}
 \mathbb D_{n,13}^{4}
 :=\bigsqcup_{\alpha,\beta,\eta,\phi,\nu}
 &\operatorname{GT}_n(\alpha)\times
  \operatorname{GT}_n(\beta)\times
  \operatorname{GT}_n(\eta)\times
  \operatorname{GT}_n(\phi) \notag\\[-2mm]
 &{}\times H_n(\alpha,\beta;\nu)
  \times H_n(\eta,\phi^\vee;\nu^\vee),
\label{eq:rf-double-hive-cone}
\end{align}
where \(\alpha,\beta,\eta,\phi,\nu\in\mathsf W_n\), the two hives are
glued along their oppositely oriented copies of \(\nu\), the second
hive factor is read as the face \(134\) after cyclic relabelling, and
the two trace constraints are
\[
 \sum_i\nu_i=\sum_i\alpha_i+\sum_i\beta_i,
 \qquad
 \sum_i\phi_i=\sum_i\alpha_i+\sum_i\beta_i+\sum_i\eta_i.
\]
The first equality is the trace relation for the first hive; together
with the second equality it gives exactly the oppositely oriented trace
relation required by the second hive.  We define the reference
measure \(d\operatorname{Leb}_{13}\) in explicit independent
coordinates: use all coordinates of \(\alpha,\beta,\eta\), the first
\(n-1\) coordinates of each of \(\phi\) and \(\nu\), and set
\[
 \phi_n=\sum_i\alpha_i+\sum_i\beta_i+\sum_i\eta_i
        -\sum_{i<n}\phi_i,
 \qquad
 \nu_n=\sum_i\alpha_i+\sum_i\beta_i-\sum_{i<n}\nu_i.
\]
Take Euclidean Lebesgue measure in these boundary coordinates and, on
each resulting fiber, the product of the coordinate Lebesgue measures
in the free Gelfand--Tsetlin and hive entries.  Omitting a different
trace coordinate changes this reference measure only by a constant
Jacobian, which is absorbed into \(Z_n\).  Put
\begin{equation}
 Z_n(\ell_1,\ell_2,\ell_3)
 :=
 \int_{\mathbb D_{n,13}^{4}}
 \exp\!\left\{
 -\frac1{2n}\left(
   \frac{|\alpha(x)|^2}{\ell_1^2}
  +\frac{|\beta(x)|^2}{\ell_2^2}
  +\frac{|\eta(x)|^2}{\ell_3^2}
 \right)\right\}
 \operatorname{Leb}_{13}(dx).
\label{eq:rf-double-hive-partition-function}
\end{equation}

Define \(\mathbb Q_{n,13}^{\ell_1,\ell_2,\ell_3}\) on
\(\mathbb D_{n,13}^{4}\) by
\begin{equation}
 \frac{d\mathbb Q_{n,13}^{\ell_1,\ell_2,\ell_3}}
      {d\operatorname{Leb}_{13}}(x)
 =
 \frac1{Z_n(\ell_1,\ell_2,\ell_3)}
 \exp\!\left\{
 -\frac1{2n}\left(
   \frac{|\alpha(x)|^2}{\ell_1^2}
  +\frac{|\beta(x)|^2}{\ell_2^2}
  +\frac{|\eta(x)|^2}{\ell_3^2}
 \right)\right\}.
\label{eq:rf-double-hive-density}
\end{equation}
Thus the fourth exterior side \(\phi\) carries no quadratic penalty.
Proposition~\ref{prop:rf-unpenalized-fourth-side} below proves directly
that the normalizing constant is finite and positive.  Formula
\eqref{eq:rf-double-hive-density} is the definition of the law used
here.  The notation \(\bar f=\infty\) means the limit in which the
quadratic penalty on the fourth exterior side is removed; it is not a
literal substitution into the finite-parameter density.

Sample an augmented double hive from
\(\mathbb Q_{n,13}^{\ell_1,\ell_2,\ell_3}\), forget its four exterior
Gelfand--Tsetlin augmentations, and denote its two hive faces by
\[
 \mathsf H_n^{\rm top}
   =h_{123}\underset{13}{\cup}h_{134}.
\]
Starting from this top membrane, perform a complete legal tropical
octahedron-recurrence sweep through the tetrahedron \(1234\).  The
values exposed at all lattice points of \(1234\), including the
resulting bottom membrane
\[
 \mathsf H_n^{\rm bot}
   =h_{124}\underset{24}{\cup}h_{234},
\]
form the random field
\(\mathcal H_n^{1234}(\ell_1,\ell_2,\ell_3)\).  This is our definition of the
random field on \(1234\).
\end{definition}

\begin{proposition}[Removal of the fourth-side penalty]
\label{prop:rf-unpenalized-fourth-side}
Put
\[
 A=n\ell_1^2,\qquad B=n\ell_2^2,\qquad D=n\ell_3^2.
\]
For \(F>0\), let \(Z_{n,F}\) and \(\mathbb Q_{n,13;F}\) be obtained
from \eqref{eq:rf-double-hive-partition-function} and
\eqref{eq:rf-double-hive-density} by inserting the additional factor
\[
                         \exp\{-|\phi|^2/(2F)\}.
\]
Then, as \(F\to\infty\),
\[
 Z_{n,F}\uparrow Z_n(\ell_1,\ell_2,\ell_3)<\infty,
 \qquad
 \mathbb Q_{n,13;F}\longrightarrow
 \mathbb Q_{n,13}^{\ell_1,\ell_2,\ell_3}
 \quad\text{in total variation}.
\]
More precisely, for a constant \(C_n>0\) depending only on \(n\) and
the Lebesgue normalization,
\begin{equation}
 Z_{n,F}
  =C_n\left(\frac{ABDF}{A+B+D+F}\right)^{n^2/2},
 \qquad
 Z_n(\ell_1,\ell_2,\ell_3)=C_n(ABD)^{n^2/2}.
\label{eq:rf-fourth-side-partition-limit}
\end{equation}
\end{proposition}

\begin{proof}
The finite-\(F\) partition-function calculation is
\cite[Section~3, Lemma~8 and its proof]{NarDeformed}; with the present
variance convention its four raw parameters are \(A,B,D,F\).  This
gives the first formula in
\eqref{eq:rf-fourth-side-partition-limit}, and
\[
             \frac{ABDF}{A+B+D+F}\longrightarrow ABD.
\]
Independently of the formula, the unnormalized finite-\(F\) densities
increase pointwise to the unpenalized integrand in
\eqref{eq:rf-double-hive-partition-function}.  Monotone convergence
therefore identifies its integral with the finite limit in
\eqref{eq:rf-fourth-side-partition-limit}.  The normalized densities
converge pointwise, and both their integrals and the integral of the
limit equal one. We have convergence in \(L^1\) by Scheff\'e's lemma. This
is equivalent to total-variation convergence.
\end{proof}

The terminology ``maximum entropy'' is intrinsic to this definition.
Proposition~\ref{prop:rf-unpenalized-fourth-side} makes precise the
unpenalized-fourth-side, or \(\bar f=\infty\), limit of the
finite-parameter maximum-entropy law in
\cite[Section~3, Lemma~6]{NarDeformed}.  The exponential-family/KL
argument is recorded in \cite[Remark~1]{NarDeformed}; for completeness
we restate it here.  If \(q\) is a
probability density with respect to this measure, has the same three
boundary second moments as
\(\mathbb Q_{n,13}^{\ell_1,\ell_2,\ell_3}\), and the displayed entropy and
relative entropy are finite, then
\[
 \operatorname{Ent}(q)
 -\operatorname{Ent}(\mathbb Q_{n,13}^{\ell_1,\ell_2,\ell_3})
 =
 -D_{\rm KL}\!\left(
 q\,\middle\|\,\mathbb Q_{n,13}^{\ell_1,\ell_2,\ell_3}\right)\leq0.
\]
Thus \(\mathbb Q_{n,13}^{\ell_1,\ell_2,\ell_3}\) is the unique entropy
maximizer among feasible densities for which the displayed quantities
are finite; this is the entropy class used below.

\begin{definition}[The centered large-gap construction]
\label{def:rf-centered-large-gap}
Let \(X_1,X_2,X_3\) be \(n\times n\) Hermitian matrices.  For
\(\varnothing\ne I\subseteq\{1,2,3\}\), put
\[
 X_I=\sum_{r\in I}X_r,
 \qquad
 \operatorname{width}(X_I)
   =\lambda_1(X_I)-\lambda_n(X_I).
\]
A number \(K_{\rm gap}>0\) is \emph{admissible} for this triple if
\begin{equation}
 nK_{\rm gap}>
 2\max_{\varnothing\ne I\subseteq\{1,2,3\}}
       \operatorname{width}(X_I).
\label{eq:rf-intro-admissible-gap}
\end{equation}
Define the centered gap vector
\begin{equation}
 \widehat\Lambda_q
  =nK_{\rm gap}\left(\frac{n+1}{2}-q\right),
 \qquad 1\leq q\leq n.
\label{eq:rf-intro-centered-gap}
\end{equation}
Its consecutive gaps are \(nK_{\rm gap}\), and
\(\sum_q\widehat\Lambda_q=0\).  Regard
\(\operatorname{diag}(X_r)\) as the vector of diagonal entries and put
\[
 B_0=\widehat\Lambda,
 \qquad
 B_r=B_{r-1}+\operatorname{diag}(X_r),
 \qquad r=1,2,3.
\]
On the panel \(P_r=\Delta_{0,r,r+1}(n)\), use the coordinates
\[
 [i,j]_r=(n-j)e_0+(j-i)e_r+i e_{r+1},
 \qquad 0\leq i\leq j\leq n,
\]
and define
\begin{equation}
 H_r(i,j)
  =\sum_{q=1}^{j}(B_{r-1})_q
   +\sum_{q=1}^{i}\lambda_q(X_r^{[j]}),
 \qquad
 h([i,j]_r)=H_r(i,j),
\label{eq:rf-intro-large-gap-hive}
\end{equation}
where \(X_r^{[j]}\) is the leading \(j\times j\) principal minor and
its eigenvalues are in decreasing order.  The three formulas agree on
the common seams and hence define an array on
\(012\cup023\cup034\).  We call this the centered large-gap realization
of the three minor processes.
\end{definition}

\begin{lemma}[Validity of the centered large-gap construction]
\label{lem:rf-centered-large-gap-valid}
If \(K_{\rm gap}\) is admissible in the sense of
\eqref{eq:rf-intro-admissible-gap}, then each array \(H_r\) in
\eqref{eq:rf-intro-large-gap-hive} is a genuine hive.  Consequently the
compatible initial fan is a genuine hive fan, and every face array
obtained from it by the hive associator is a genuine hive.
\end{lemma}

\begin{proof}
Set
\[
 D=nK_{\rm gap},\qquad
 W=\max_{\varnothing\ne I\subseteq\{1,2,3\}}
       \operatorname{width}(X_I),
 \qquad
 Y_{r-1}=\sum_{s<r}X_s,
\]
where \(Y_0=0\), and write
\[
 S_i^{(j)}=\sum_{q=1}^i\lambda_q(X_r^{[j]}).
\]
The two non-red rhombus families are precisely the two Cauchy
interlacing inequalities for consecutive principal minors of \(X_r\).
For the remaining, red family, direct substitution in
\eqref{eq:rf-intro-large-gap-hive} gives the slack
\begin{align*}
 &\bigl[(B_{r-1})_j-(B_{r-1})_{j+1}\bigr]
 +\bigl[S_i^{(j)}-S_i^{(j-1)}\bigr]\\
 &\hspace{38mm}
 -\bigl[S_{i+1}^{(j+1)}-S_{i+1}^{(j)}\bigr].
\end{align*}
Cauchy interlacing makes the first bracketed \(S\)-difference
nonnegative and places the second in
\([0,\operatorname{width}(X_r)]\).  Moreover,
\[
 (B_{r-1})_j-(B_{r-1})_{j+1}
 =D+(Y_{r-1})_{jj}-(Y_{r-1})_{j+1,j+1}
 \geq D-\operatorname{width}(Y_{r-1}).
\]
For \(r>1\), both \(Y_{r-1}\) and \(X_r\) are nonempty subset sums
among those defining \(W\); for \(r=1\), the first width is zero.
The red slack is therefore at least
\[
 D-\operatorname{width}(Y_{r-1})-\operatorname{width}(X_r)
 \geq D-2W>0.
\]
Thus all three rhombus families hold.

For seam compatibility, \(\operatorname{Tr}(X_r^{[j]})\) is the sum
of the first \(j\) diagonal entries of \(X_r\), and therefore
\[
 H_r(j,j)=\sum_{q=1}^j(B_r)_q=H_{r+1}(0,j).
\]
The three hives consequently form a compatible initial fan.  Finally,
the hive associator sends compatible genuine hives to compatible
genuine hives \cite{NSTGUE}, which proves the last assertion.
\end{proof}

\paragraph{GUE normalization.}
Throughout the paper, a \emph{standard GUE matrix} \(M\) has density
proportional to
\(\exp\{-\tfrac12\operatorname{Tr}(M^2)\}\) with respect to Lebesgue
measure on the independent real coordinates of the Hermitian matrix.
Equivalently, the diagonal entries are real \(N(0,1)\) variables and
the real and imaginary parts of the strict upper-triangular entries are
independent \(N(0,\tfrac12)\) variables.  With this convention,
\(s\sqrt n\,M\) has density proportional to
\(\exp\{-\operatorname{Tr}(X^2)/(2ns^2)\}\).

\begin{theorem}[GUE-minor realization of the field on \(1234\)]
\label{thm:rf-top-bottom-double-hive}
Let \(M_1^{(n)},M_2^{(n)},M_3^{(n)}\) be independent GUE
matrices in the normalization fixed above and put
\[
 \widehat X_r^{(n)}=\ell_r\sqrt n\,M_r^{(n)},
 \qquad r=1,2,3.
\]
Put the three corresponding GUE minor processes on
\(012,023,034\), choose any \(K_{\rm gap}\) admissible for the realized
matrix triple, apply
Definition~\ref{def:rf-centered-large-gap}, and then perform the
associator/octahedron-recurrence sweeps.  Let
\(\pi_{13}^{\rm dh}\) forget the four exterior Gelfand--Tsetlin
augmentations in \(\mathbb D_{n,13}^{4}\) and retain the two pasted
hives.  Then the top double hive produced by the three minor processes
satisfies
\begin{equation}
 \mathcal L(\mathsf H_n^{\rm top})
  =(\pi_{13}^{\rm dh})_\#
    \mathbb Q_{n,13}^{\ell_1,\ell_2,\ell_3}.
\label{eq:rf-top-double-hive-law}
\end{equation}
Consequently the resulting field on \(1234\) has law
\(\mathcal H_n^{1234}(\ell_1,\ell_2,\ell_3)\) from
Definition~\ref{def:rf-maximum-entropy-field}, and is exactly
independent of the chosen admissible value of \(K_{\rm gap}\).

Let \(\mathbb Q_{n,24}^{\ell_1,\ell_2,\ell_3}\) denote the pushforward of the
\emph{abstract} law \(\mathbb Q_{n,13}^{\ell_1,\ell_2,\ell_3}\) under the
octahedron recurrence from diagonal \(13\) to diagonal \(24\), acting
on the two hive factors and retaining the four exterior augmentations.
This pushforward has density \eqref{eq:rf-double-hive-density} with
respect to the natural Lebesgue measure in the \(24\)-triangulation.
If \(\pi_{24}^{\rm dh}\) again forgets the four exterior augmentations,
then
\begin{equation}
 \mathcal L(\mathsf H_n^{\rm bot})
  =(\pi_{24}^{\rm dh})_\#
    \mathbb Q_{n,24}^{\ell_1,\ell_2,\ell_3}.
\label{eq:rf-top-bottom-double-hive-laws}
\end{equation}
In particular, if
\(\nu_{13}\) and \(\nu_{24}\) are the respective middle diagonals,
\begin{equation}
 h_{123}\ \perp\!\!\!\perp\ h_{134}\mid\nu_{13},
 \qquad
 h_{124}\ \perp\!\!\!\perp\ h_{234}\mid\nu_{24}.
\label{eq:rf-top-bottom-conditional-independence}
\end{equation}
\end{theorem}

The proof of this theorem appears in the beginning of Part II.
\subsection*{The uniform variational limit}

Having identified the GUE-minor sampler with the intrinsic field, we
use that realization to analyze its large-\(n\) limit.  For a
macroscopic interior point
\[
   p=(\alpha,\beta,\gamma,\delta),
   \qquad \alpha+\beta+\gamma+\delta=1,
\]
the limiting canonical region is a polygonal complex
$\mathsf R_p$.  Its admissible continuum heights form the compact space
$\operatorname{AHT}(\mathsf R_p)$.  The variational functional is
\[
   \mathcal S_p(f)
      =S_{p,\mathrm{edge}}(f)
       +S_{p,\Delta}^{\mathrm{rf}}(f)
       +S_{p,0}.
\]
Here $S_{p,\mathrm{edge}}$ is the panelwise integral of the
flip-symmetric surface tension $\sigma_\edge^{\mathrm{sym}}$ defined by
the \(0,3/2,3/2\) red--green--blue weights;
$S_{p,\Delta}^{\mathrm{rf}}$ is an explicit affine line energy on the two
seams; and $S_{p,0}$ is a fixed term independent of $f$.
The finite-panel localization theorem
(Theorem~\ref{thm:rf-finite-panel-localization}) proves convergence of the
normalized expected maximum and an $o(n^4)$ variance bound.  Consequently,
for lattice approximants $p_n/n\to p$,
\[
   \frac{\widehat O_n(p_n)}{n^2}
       \xrightarrow{\mathbb P}
       \Phi_{\ell_1,\ell_2,\ell_3}(p)
       :=\sup_{f\in\operatorname{AHT}(\mathsf R_p)}\mathcal S_p(f).
\]
The constant-coordinate hive slices give a uniform lattice Lipschitz bound.
Using this fact, Theorem~\ref{thm:rf-uniform-continuum} upgrades the pointwise law of
large numbers to uniform convergence on the closed tetrahedron $1234$.

The proof follows the architecture of the two-source argument but requires
three new ingredients: a matching-dependent correction, an exact
red-free seam gauge, and a seam-compatible global patching theorem.  The
global variational upper bound uses a quantitative-differentiation theorem
communicated by C. Fefferman \cite{FeffermanQD}, whose proof is
recorded in \cite[Theorem~7.6 and Appendix~B]{NarLimit}.  Applied
separately on the three panel cores, it produces a function-dependent
GOOD/BAD decomposition on which the local
surface-tension estimates can be summed uniformly over all candidate
heights.  The lower bound uses a different, fixed-target two-scale
construction.  At a large lattice scale \(C_0\), the prescribed continuum
height is interpolated on a deterministic triangular mesh.  A differentiation
lemma retains the triangles on which the target is relatively close to its
affine interpolant; the discarded triangles have vanishing area.  Local
optimizers are inserted only in the retained triangles, and the
seam-compatible global patching theorem fills the collars.
The bulk random-matrix input ultimately comes from the fixed-index
theory of the GUE minor process developed by T. Tao
\cite[Theorems~1.2--1.4]{TaoGUEFixedIndex}.  In particular, the
fixed-index universality of the minor process and the logarithmically
improved variance estimate for linear statistics of bulk interlacing
gaps are principal inputs to the GUE-hive localization and
concentration results of \cite{NarLimit}, within the framework
developed in \cite{NSTGUE}.  The fixed and seam terms are controlled
separately by extreme-eigenvalue estimates.

\subsection*{Dependence on earlier work}
Part~I uses Lam--Postnikov's identification of membrane
moves with cluster mutations \cite[Theorem~28.5]{LamPostnikov}; Propp's
distributive lattice of perfect matchings \cite[Theorem~2]{ProppLattice},
in the form recorded in \cite[Theorem~2.6]{MMSBV}; the dimer
\(F\)-polynomial theorem and full-quiver lifting result
\cite[Theorem~6.8 and Proposition~5.11]{MMSBV}; the separation formula
\cite[Theorem~3.7 and Corollary~6.3]{FZIV}; the constant-term theorem
for \(F\)-polynomials of skew-symmetric cluster algebras
\cite[Theorem~1.7]{DWZII}; the \(g\)-vector
injectivity of Cao--Li
\cite[Corollary~5.7(i) and Remark~5.8(i)]{CaoLi}; and their
denominator-vector transport and compatibility criterion
\cite[Proposition~2.8 and Theorem~6.3(i),(ii)]{CaoLiDenominators}.  The
support geometry, the proof that the graph is well-matched, the polygonal
extremal matchings, and the frozen correction are proved in this paper.

Part~II imports the GUE-hive associator and the bulk covariance and
spectral-edge estimates of \cite{NSTGUE}, together with the local
surface-tension, concentration, and two-source localization theory of
\cite{NarLimit}.  The latter ultimately uses Tao's fixed-index GUE-minor
theory \cite[Theorems~1.2--1.4]{TaoGUEFixedIndex}.  The upper bound uses
Fefferman's theorem \cite{FeffermanQD}, the lower bound uses the
fixed-target interpolation argument of Cohn--Kenyon--Propp
\cite[Lemma~2.2]{CKP}, the variance of the maximum uses Klartag's KLS
estimate \cite{KlartagKLS} through \cite[Lemma~3]{NSTGUE}, and the
uniform extension to the closed tetrahedron uses the local hive property
of Knutson--Tao--Woodward \cite[\S5]{KTWOR}.  The new ingredients here are
the flip-symmetric red-free gauge, the fixed and seam terms, the global
cut-compatible admissible space, and the seam-compatible scaffold
patching theorem.

\subsection*{Organization}

The introduction defines the maximum-entropy field and states its
GUE-minor realization; the proof of that realization is given at the
beginning of Part~II.  Part~I develops the deterministic theory: it
gives the cluster-algebra dictionary, constructs the denominator
support and fixed graph, proves well-matchedness by height forcing, and
derives the perfect-matching formula.  Part~II then converts
that formula to red-lozenge-free tiling weights, evaluates the fixed
and seam contributions, and proves the finite-panel and uniform
continuum limit theorems.

\part{The exact matching formula}

\section{Setup, terminology, and the main theorem}
\label{sec:partI-roadmap}

At lattice scale \(N\), the ambient vertex \(u_i\) is the corner
\(Ne_i\).  For \(I\subseteq\{0,1,2,3,4\}\), put
\[
 \Delta_I(N):=
 \left\{q=(q_0,\ldots,q_4)\in\Z_{\geq0}^{5}:
       \sum_{i=0}^{4}q_i=N,\quad q_j=0\text{ for }j\notin I\right\}.
\]
Thus \(\Delta_I(N)\) is the face spanned by the vertices \(u_i\),
\(i\in I\).  We abbreviate \(\Delta_{\{i,j,k\}}(N)\) to \(ijk\).  Our initial
membrane is the three-panel fan
\[
 \Sigma _0=\Delta_{012}(N)\cup\Delta_{023}(N)\cup\Delta_{034}(N).
\]
The octahedron recurrence is
\begin{equation}
 X_{\lambda+e_i+e_k}X_{\lambda+e_j+e_\ell}
 =
 X_{\lambda+e_i+e_j}X_{\lambda+e_k+e_\ell}
 +
 X_{\lambda+e_i+e_\ell}X_{\lambda+e_j+e_k},
 \qquad i<j<k<\ell.
 \label{eq:higher-OR}
\end{equation}
Its max-plus, or \emph{tropical}, form is obtained by writing
\(X_q=\exp(t h_q)\), taking \(t^{-1}\log\), and letting
\(t\to\infty\).  Formally, on every elementary octahedron indexed by
\(\lambda\) and \(i<j<k<\ell\), the \emph{tropical octahedron
recurrence} is
\begin{equation}
 h_{\lambda+e_i+e_k}+h_{\lambda+e_j+e_\ell}
 =
 \max\!\left\{
 h_{\lambda+e_i+e_j}+h_{\lambda+e_k+e_\ell},
 h_{\lambda+e_i+e_\ell}+h_{\lambda+e_j+e_k}
 \right\}.
\label{eq:tropical-higher-OR}
\end{equation}
Thus, if \(h_{\lambda+e_i+e_k}\) is the value on the old membrane, the
membrane move replaces it by the opposite value
\begin{equation}
 h_{\lambda+e_j+e_\ell}
 =
 \max\!\left\{
 h_{\lambda+e_i+e_j}+h_{\lambda+e_k+e_\ell},
 h_{\lambda+e_i+e_\ell}+h_{\lambda+e_j+e_k}
 \right\}
 -h_{\lambda+e_i+e_k},
\label{eq:tropical-higher-OR-solved}
\end{equation}
while the four equatorial values remain fixed.  We use the max-plus
convention throughout; the min-plus convention found elsewhere is
obtained by replacing \(h\) with \(-h\).

\begin{figure}[htbp]
\centering
\includegraphics[width=0.88\textwidth]
  {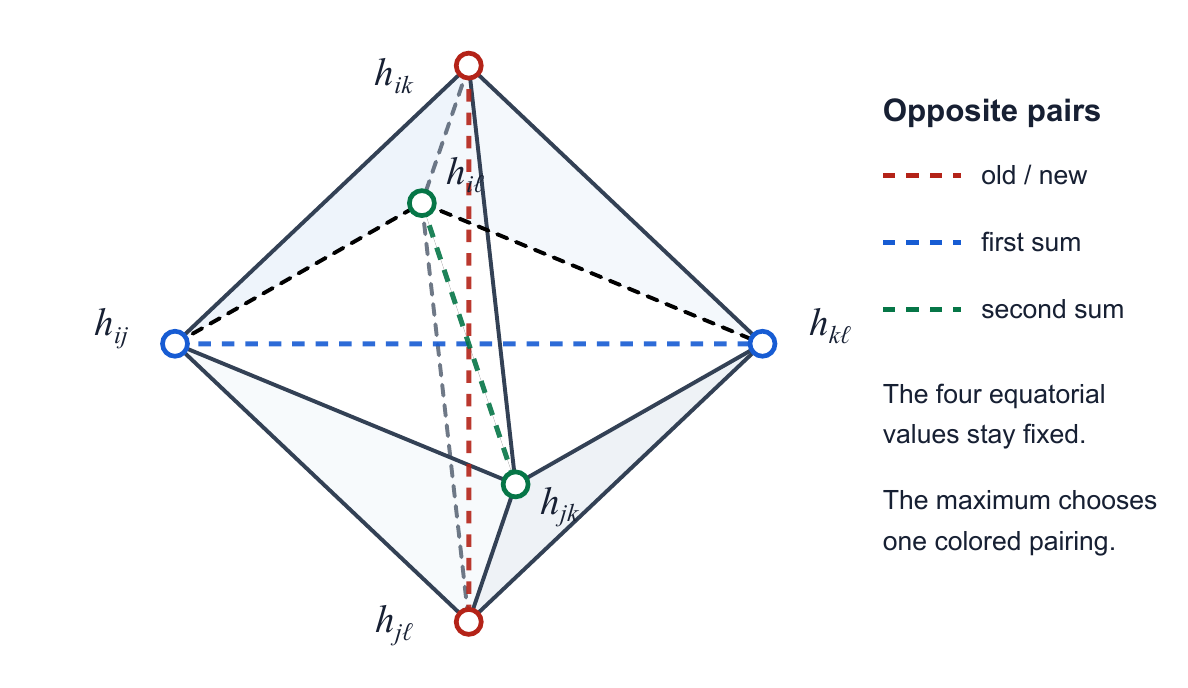}
\caption{The six arguments of the tropical octahedron recurrence.
The red dashed segment joins the old and new membrane values.  The
blue (EW) and green (NS) dashed segments join the two pairs of opposite
equatorial vertices whose sums compete in
\eqref{eq:tropical-higher-OR-solved}.}
\label{fig:tropical-octahedron-recurrence}
\end{figure}

Lam and Postnikov identify the corresponding membrane moves with
cluster mutations \cite[Theorem~28.5]{LamPostnikov}.  We use the short
path
\[
 (012,023,034)
 \longrightarrow(013,123,034)
 \longrightarrow(014,134,123)
 \longrightarrow(014,124,234),
\]
that is, complete sweeps through \(0123\) and \(0134\), followed by the
sweep through \(1234\) that exposes the target.  Strong pentagon
compatibility guarantees that the resulting cluster variable does not
depend on the chosen legal flip path
\cite[Theorem~28.5]{LamPostnikov}.

\begin{figure}[htbp]
\centering
\begin{tikzpicture}[
  scale=0.82,
  every node/.style={font=\small},
  outer/.style={draw=black!75,line width=0.9pt},
  membrane/.style={draw=blue!70!black,line width=1.2pt},
  vertex/.style={circle,fill=black,inner sep=1.25pt},
  move/.style={-{Latex[length=2.4mm]},line width=0.9pt}
]
\newcommand{\membranevertices}[1]{%
  \coordinate (#1-0) at (-2.0,0);
  \coordinate (#1-1) at (0,1.75);
  \coordinate (#1-2) at (1.35,0.75);
  \coordinate (#1-3) at (1.35,-0.75);
  \coordinate (#1-4) at (0,-1.75);
  \draw[outer] (#1-0)--(#1-1)--(#1-2)--(#1-3)--(#1-4)--cycle;
  \foreach \i/\pos in {0/left,1/above,2/right,3/right,4/below}{
    \node[vertex] at (#1-\i) {};
    \node[\pos=2pt] at (#1-\i) {\(\i\)};
  }
}
\begin{scope}[xshift=-6.0cm]
  \membranevertices{A}
  \draw[membrane] (A-0)--(A-2) (A-0)--(A-3);
  \node at (-0.55,0.80) {\(012\)};
  \node at (-0.35,0.00) {\(023\)};
  \node at (-0.50,-0.82) {\(034\)};
  \node[font=\small\bfseries] at (0,2.35) {\(012,023,034\)};
\end{scope}
\begin{scope}
  \membranevertices{B}
  \draw[membrane] (B-0)--(B-3) (B-1)--(B-3);
  \node at (-0.55,0.55) {\(013\)};
  \node at (0.72,0.57) {\(123\)};
  \node at (-0.15,-0.75) {\(034\)};
  \node[font=\small\bfseries] at (0,2.35) {\(013,123,034\)};
\end{scope}
\begin{scope}[xshift=6.0cm]
  \membranevertices{C}
  \draw[membrane] (C-1)--(C-4) (C-1)--(C-3);
  \node at (-0.65,0.00) {\(014\)};
  \node at (0.78,0.70) {\(123\)};
  \node at (0.76,-0.52) {\(134\)};
  \node[font=\small\bfseries] at (0,2.35) {\(014,134,123\)};
\end{scope}
\draw[move] (-3.45,0)--(-2.55,0);
\draw[move] (2.55,0)--(3.45,0);
\end{tikzpicture}
\caption{The first three membranes in the chosen legal flip path.
Starting from the fan \(012\cup023\cup034\), the first sweep produces
\(013\cup123\cup034\), and the second produces
\(014\cup134\cup123\).}
\label{fig:three-triangulations}
\end{figure}
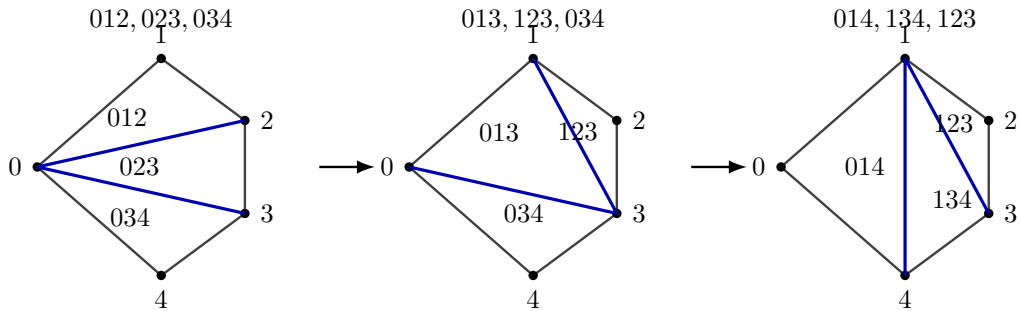

\subsection*{Cluster-algebra dictionary}
\label{sec:cluster-dictionary}

We recall the cluster-algebra terminology used in Part~I; no other
cluster-algebra formalism is needed.  A \emph{seed} consists of a
family of variables and a directed graph, or \emph{quiver}, describing
how those variables are exchanged.  The vertices of the quiver are
divided into \emph{mutable} vertices, at which exchanges may be
performed, and \emph{frozen} vertices, which are never mutated and
serve as coefficients.  For the initial fan seed, the variables are
\(X_v\), \(v\in\Sigma _0\).  If
\(\partial\Sigma _0\) denotes the frozen outer boundary, then
\[
 I_{\rm mut}:=\Sigma _0\setminus\partial\Sigma _0,
 \qquad I_{\rm fr}:=\partial\Sigma _0
\]
are the mutable and frozen index sets.

We shall use two globally opposite orientations of the initial fan
quiver, and distinguish them whenever signs matter.  The quiver drawn
in the figures and obtained directly from the colored membrane is
denoted by \(Q_0^{\rm LP}\).  It uses the convention of
Lam--Postnikov \cite[Definition~28.4]{LamPostnikov}: an arrow has the
adjacent black elementary triangle on its right and the adjacent white
triangle on its left.  Equivalently, when the membrane triangles are
viewed as vertices of the plane bipartite dual graph, its white endpoint
lies on the left.  Mészáros--Musiker--Sherman-Bennett--Vidinas use
the opposite convention \cite[Definition~6.1]{MMSBV}, with the white
endpoint on the right.  We therefore put
\begin{equation}
                 Q_0^{\rm R}:=(Q_0^{\rm LP})^{\rm op}.
\label{eq:two-quiver-conventions}
\end{equation}
All principal-coefficient, \(F\)-polynomial, and \(g\)-vector
calculations below use \(Q_0^{\rm R}\), so that the results of
\cite{MMSBV} apply with their stated signs.  The graph figures and the
membrane-dual assertions use \(Q_0^{\rm LP}\).  Coefficient-free
octahedron variables are the same for the two orientations, as recorded
below.

The \emph{extended exchange matrix}
\(B=(b_{vu})\) has one row for every
\(v\in I_{\rm mut}\cup I_{\rm fr}\) and one column for every mutable
\(u\in I_{\rm mut}\).  Our sign convention is
\(b_{vu}>0\) when the quiver has \(b_{vu}\) arrows from \(v\) to
\(u\), after cancellation of oriented \(2\)-cycles.  Unless explicitly
marked \({\rm LP}\), \(B\) denotes the matrix of \(Q_0^{\rm R}\).
Its square mutable part is
\[
             B_{\rm mut}:=(b_{vu})_{v,u\in I_{\rm mut}}.
\]
Mutation at a
mutable index \(k\) replaces \(X_k\) by \(X_k'\), where
\begin{equation}
 X_kX_k'
  =\prod_{b_{vk}>0}X_v^{\,b_{vk}}
   +\prod_{b_{vk}<0}X_v^{-b_{vk}},
\label{eq:cluster-exchange-dictionary}
\end{equation}
The quiver is mutated at the same time as follows.  For every directed
two-step path \(i\longrightarrow k\longrightarrow j\), add an arrow
\(i\longrightarrow j\); then reverse every arrow incident to \(k\);
finally, cancel all oriented \(2\)-cycles.  Parallel arrows are retained
with their multiplicities.  In the
fan seeds used here, \eqref{eq:cluster-exchange-dictionary} is exactly
an octahedron-recurrence relation.  A variable obtained by a sequence
of such exchanges is a \emph{cluster variable}; a product of variables
belonging to one seed is a \emph{cluster monomial}.

\begin{lemma}[Global arrow reversal]
\label{lem:global-arrow-reversal}
For every mutation sequence, the coefficient-free cluster variables
obtained from \(Q_0^{\rm LP}\) and \(Q_0^{\rm R}\) are identical Laurent
polynomials in the correspondingly labelled initial variables.
\end{lemma}

\begin{proof}
If \(B^{\rm LP}=-B\), then matrix mutation satisfies
\(\mu_k(B^{\rm LP})=-\mu_k(B)\).  In
\eqref{eq:cluster-exchange-dictionary}, replacing \(B\) by \(-B\)
only interchanges the two product summands.  The first mutated variable
is therefore unchanged, and induction along the mutation sequence gives
the assertion.
\end{proof}

The Laurent phenomenon \cite[Theorem~3.1]{FZI} says that every cluster variable \(Z\) is a
Laurent polynomial in the initial variables.  We write \(D_Z(q)\) for
the exponent of \(X_q\) in its reduced denominator, allowing negative
coordinates and using
\[
                      D_{X_r}(q)=-\delta_{rq}
\]
for an initial variable.  The resulting fan-indexed field
\((D_Z(q))_{q\in\Sigma _0}\) will be called the
\emph{denominator vector}; the  cluster \(d\)-vector is its
mutable restriction.  Thus a positive coordinate records an actual
denominator factor, while the convention above remembers an initial
numerator variable.

Two further invariants organize the numerator.  The cluster algebra
with \emph{principal coefficients} has extended exchange matrix
\[
                  \begin{pmatrix}B_{\rm mut}\\ I\end{pmatrix}.
\]
Its initial mutable variables are \(X_v\), \(v\in I_{\rm mut}\),
and its frozen variables are the independent coefficients \(y_u\),
\(u\in I_{\rm mut}\).  This entire matrix is mutated by the same
matrix-mutation rule.  Thus the initial exchange at \(k\) is
\[
 X_kX_k'
   =y_k\prod_{\substack{v\in I_{\rm mut}\\b_{vk}>0}}X_v^{\,b_{vk}}
      +\prod_{\substack{v\in I_{\rm mut}\\b_{vk}<0}}X_v^{-b_{vk}},
\]
and subsequent exchanges use the coefficient monomials recorded by
the mutated lower block.  Given a cluster variable \(Z\) obtained by
a mutation sequence in the fan, use the same sequence in this
principal-coefficient algebra to define its principal-coefficient lift.
The \emph{\(F\)-polynomial} \(F_Z(y)\) is obtained from that lift by
setting all initial mutable variables equal to \(1\).  Since
\(B_{\rm mut}\) is skew-symmetric, every such \(F\)-polynomial has constant term \(1\)
\cite[Theorem~1.7]{DWZII}.  For the dimer variable constructed in
Step~IV, this term is represented concretely by the bottom matching,
whose face-height vector is zero.  The original frozen boundary pattern
is then introduced by the normalized coefficient specialization
\[
             y_u\longmapsto
             \prod_{v\in I_{\rm fr}}X_v^{b_{vu}}.
\]
For the resulting fan variable \(Z\), the separation formula of
\cite[Theorem~3.7 and Corollary~6.3]{FZIV} takes the form
\begin{equation}
 Z=X_\partial^{\beta_Z}
   \left(\prod_{v\in I_{\rm mut}}X_v^{g_Z(v)}\right)
   F_Z(\widehat y),
 \qquad
 \widehat y_u=\prod_{v\in\Sigma _0}X_v^{b_{vu}}.
\label{eq:cluster-separation-dictionary}
\end{equation}
Here
\[
                     g_Z\in\mathbb Z^{I_{\rm mut}}
\]
is the ordinary \emph{\(g\)-vector}.  It is indexed only by mutable
variables: equivalently, it is the mutable exponent of the
distinguished monomial multiplying the constant term \(1\) of
\(F_Z\).  The  \(g\)-vector mutation rule and the injectivity
theorem for \(g\)-vectors concern this mutable vector.

The exponent
\(\beta_Z\in\mathbb Z^{I_{\rm fr}}\) in
\eqref{eq:cluster-separation-dictionary} is different.  It is the
\emph{frozen normalization} produced by the inverse tropical
\(F\)-polynomial denominator after coefficient specialization.  It is
not a frozen coordinate of the  \(g\)-vector.  If \(g_Z\) is
padded by zeros on \(I_{\rm fr}\), we call
\begin{equation}
                    \widetilde g_Z:=\beta_Z+g_Z
\label{eq:extended-bottom-exponent-dictionary}
\end{equation}
the \emph{extended bottom exponent}.  This is convenient terminology
for the exponent of the distinguished constant-term monomial, but it
is deliberately not called an extended \(g\)-vector: \(\beta_Z\)
comes from a separate coefficient normalization, and the usual
\(g\)-vector theory applies only to \(g_Z\).

More generally, if \(y^r\) is a monomial of \(F_Z\), substitution of
\(\widehat y\) gives the full initial-variable exponent
\begin{equation}
                       \epsilon_Z(r)
                         =\beta_Z+g_Z+Br,
\label{eq:matching-exponent-dictionary}
\end{equation}
where \(g_Z\) is again padded by zero frozen coordinates.  Thus neither
\(\widetilde g_Z\) nor \(g_Z\) is intended to be the exponent of every
Laurent monomial.  In the dimer model, \(r\) is the face-height vector
of a matching, \(r=0\) is the bottom matching, and
\eqref{eq:matching-exponent-dictionary} is converted into the visible
incidence formula in Step~IV.

Finally, \emph{tropicalization} sends multiplication to addition,
division to subtraction, and a positive sum to a maximum.  Equivalently,
one may substitute \(X_v=\exp(t h(v))\), divide the logarithm by
\(t\), and let \(t\to\infty\).  The Laurent sums tropicalized below are
the explicit matching sums obtained in Step~IV; their coefficients are
nonnegative, so there is no cancellation and tropicalization gives the
maximum of the linear forms defined by their monomial exponent vectors.

\subsection*{The fixed graph and its matching statistics}

We now introduce the graph-theoretic objects needed to state the
conclusion of Part~I.

For purposes of the theorem statement, define \(S_p\subseteq\Sigma _0\)
by
\begin{equation}
\boxed{
\begin{aligned}
S_p\cap\Delta_{012}
 &=\{q:q_0\geq1,\ q_1\leq a-1,\ q_2\geq b+1\},\\
S_p\cap\Delta_{023}
 &=\{q:q_0\geq1,\ (q_2\geq b+1\ \text{or}\ q_3\geq c+1)\},\\
S_p\cap\Delta_{034}
 &=\{q:q_0\geq1,\ q_3\geq c+1,\ q_4\leq d-1\}.
\end{aligned}}
\label{eq:Sp-formula}
\end{equation}
Step~I proves that this explicitly defined set is exactly the positive
denominator support of \(X_p\).

Let \(\mathcal T_0\) be the subdivision of \(\Sigma _0\) into
elementary lattice triangles.  Within each panel, color an elementary
triangle black if it has the same orientation as the ambient triangular
panel, and white if it has the opposite orientation.  Put a dual vertex in every
elementary triangle and a dual edge across every elementary lattice
edge whose incident triangles have opposite colors.  Along either
seam \(02\) or \(03\), identify same-colored dual vertices belonging
to incident panels.  Denote the resulting plane bipartite graph by
\(\Gamma _0\).  Every mutable lattice vertex \(v\) labels a bounded
face \(F_v\) of \(\Gamma _0\).  Directing a crossing arrow so that its
black endpoint is on the right and its white endpoint is on the left
gives the Lam--Postnikov quiver \(Q_0^{\rm LP}\); this is the
orientation displayed below.  Reversing every arrow gives the
white-on-the-right quiver \(Q_0^{\rm R}\) used for the dimer
\(F\)-polynomial calculation.  We write \(Q_0^{\rm LP}[S_p]\) and
\(Q_0^{\rm R}[S_p]\) for the corresponding induced quivers.

\begin{definition}[Unreduced graph]
\label{def:unreduced}
Define
\begin{equation}
\boxed{\qquad
 \overline G_p:=\bigcup_{v\in S_p}\partial F_v ,
\qquad}
\label{eq:explicit-Gp}
\end{equation}
deleting isolated vertices.  Thus a dual edge is retained exactly
when it bounds at least one selected face.
\end{definition}

\begin{figure}[htbp]
\centering
\begin{minipage}[t]{0.48\textwidth}
\vspace{0pt}\centering
\textbf{(a)}~The induced quiver \(Q_0^{\rm LP}[S_p]\).\par\smallskip
\begin{tikzpicture}[
  x=0.58cm,y=0.58cm,
  quiver arrow/.style={-{Stealth[length=2.15pt,width=2.35pt]},
    draw=orange!85!black,line width=0.42pt,
    shorten <=1.35pt,shorten >=1.35pt},
  quiver vertex/.style={circle,draw=white,line width=0.28pt,
    fill=blue!55!black,inner sep=0pt,minimum size=2.45pt},
  panel edge/.style={draw=black!38,line width=0.6pt},
  grid line/.style={draw=black!10,line width=0.22pt},
  polygon corner/.style={circle,draw=black!62,line width=0.65pt,
    fill=white,inner sep=0pt,minimum size=4.2pt}]

\begin{scope}[rotate=-90]

\coordinate (quiverambientzero) at (0,0);
\coordinate (quiverambientone) at (-10,0);
\coordinate (quiverambienttwo) at (-5,8.660254);
\coordinate (quiverambientthree) at (5,8.660254);
\coordinate (quiverambientfour) at (10,0);

\def\quiverpanelgrid#1#2#3{%
  \foreach \t in {1,...,9}{%
    \pgfmathsetmacro{\f}{\t/10}%
    \draw[grid line] ($(#1)!\f!(#2)$) -- ($(#3)!\f!(#2)$);%
    \draw[grid line] ($(#1)!\f!(#3)$) -- ($(#2)!\f!(#3)$);%
    \draw[grid line] ($(#2)!\f!(#1)$) -- ($(#3)!\f!(#1)$);%
  }%
}
\quiverpanelgrid{quiverambientzero}{quiverambientone}{quiverambienttwo}
\quiverpanelgrid{quiverambientzero}{quiverambienttwo}{quiverambientthree}
\quiverpanelgrid{quiverambientzero}{quiverambientthree}{quiverambientfour}

\draw[panel edge]
  (quiverambientzero)--(quiverambientone)--(quiverambienttwo)--cycle;
\draw[panel edge]
  (quiverambientzero)--(quiverambienttwo)--(quiverambientthree)--cycle;
\draw[panel edge]
  (quiverambientzero)--(quiverambientthree)--(quiverambientfour)--cycle;

\foreach \i in {0,...,10}{\foreach \j in {0,...,10}{
  \pgfmathtruncatemacro{\k}{10-\i-\j}
  \pgfmathtruncatemacro{\selected}
    {(\k>=0)&&(\i>=1)&&(\j<=1)&&(\k>=4)}
  \ifnum\selected=1
    \pgfmathtruncatemacro{\czerotarget}{(\j>=1)}
    \ifnum\czerotarget=1
      \draw[quiver arrow]
        ({-\j-.5*\k},{.8660254*\k}) --
        ({-(\j-1)-.5*(\k+1)},{.8660254*(\k+1)});
    \fi
    \pgfmathtruncatemacro{\conetarget}{(\j>=1)&&(\k>=5)}
    \ifnum\conetarget=1
      \draw[quiver arrow]
        ({-\j-.5*\k},{.8660254*\k}) --
        ({-\j-.5*(\k-1)},{.8660254*(\k-1)});
    \fi
    \pgfmathtruncatemacro{\ctwotarget}{(\i>=2)&&(\j==0)}
    \ifnum\ctwotarget=1
      \draw[quiver arrow]
        ({-\j-.5*\k},{.8660254*\k}) --
        ({-(\j+1)-.5*\k},{.8660254*\k});
    \fi
  \fi
}}

\foreach \i in {0,...,10}{\foreach \j in {0,...,10}{
  \pgfmathtruncatemacro{\k}{10-\i-\j}
  \pgfmathtruncatemacro{\selected}
    {(\k>=0)&&(\i>=1)&&((\j>=4)||(\k>=3))}
  \ifnum\selected=1
    \pgfmathtruncatemacro{\czerotarget}
      {(\j>=1)&&(((\j-1)>=4)||((\k+1)>=3))}
    \ifnum\czerotarget=1
      \draw[quiver arrow]
        ({-.5*\j+.5*\k},{.8660254*(\j+\k)}) --
        ({-.5*(\j-1)+.5*(\k+1)},{.8660254*(\j+\k)});
    \fi
    \pgfmathtruncatemacro{\conetarget}
      {(\j>=1)&&(\k>=1)&&((\j>=4)||((\k-1)>=3))}
    \ifnum\conetarget=1
      \draw[quiver arrow]
        ({-.5*\j+.5*\k},{.8660254*(\j+\k)}) --
        ({-.5*\j+.5*(\k-1)},{.8660254*(\j+\k-1)});
    \fi
    \pgfmathtruncatemacro{\ctwotarget}
      {(\i>=2)&&(\k>=1)&&(((\j+1)>=4)||(\k>=3))}
    \ifnum\ctwotarget=1
      \draw[quiver arrow]
        ({-.5*\j+.5*\k},{.8660254*(\j+\k)}) --
        ({-.5*(\j+1)+.5*\k},{.8660254*(\j+\k+1)});
    \fi
  \fi
}}

\foreach \i in {0,...,10}{\foreach \j in {0,...,10}{
  \pgfmathtruncatemacro{\k}{10-\i-\j}
  \pgfmathtruncatemacro{\selected}
    {(\k>=0)&&(\i>=1)&&(\j>=3)&&(\k<=2)}
  \ifnum\selected=1
    \pgfmathtruncatemacro{\czerotarget}{(\j>=4)&&((\k+1)<=2)}
    \ifnum\czerotarget=1
      \draw[quiver arrow]
        ({.5*\j+\k},{.8660254*\j}) --
        ({.5*(\j-1)+(\k+1)},{.8660254*(\j-1)});
    \fi
    \pgfmathtruncatemacro{\conetarget}{(\j>=1)&&(\k>=1)}
    \ifnum\conetarget=1
      \draw[quiver arrow]
        ({.5*\j+\k},{.8660254*\j}) --
        ({.5*\j+(\k-1)},{.8660254*\j});
    \fi
    \pgfmathtruncatemacro{\ctwotarget}{(\i>=2)&&(\k>=1)}
    \ifnum\ctwotarget=1
      \draw[quiver arrow]
        ({.5*\j+\k},{.8660254*\j}) --
        ({.5*(\j+1)+\k},{.8660254*(\j+1)});
    \fi
  \fi
}}

\foreach \i in {0,...,10}{\foreach \j in {0,...,10}{
  \pgfmathtruncatemacro{\k}{10-\i-\j}
  \pgfmathtruncatemacro{\selected}
    {(\k>=0)&&(\i>=1)&&(\j<=1)&&(\k>=4)}
  \ifnum\selected=1
    \node[quiver vertex] at
      ({-\j-.5*\k},{.8660254*\k}) {};
  \fi
  \pgfmathtruncatemacro{\selected}
    {(\k>=0)&&(\i>=1)&&((\j>=4)||(\k>=3))}
  \ifnum\selected=1
    \node[quiver vertex] at
      ({-.5*\j+.5*\k},{.8660254*(\j+\k)}) {};
  \fi
  \pgfmathtruncatemacro{\selected}
    {(\k>=0)&&(\i>=1)&&(\j>=3)&&(\k<=2)}
  \ifnum\selected=1
    \node[quiver vertex] at
      ({.5*\j+\k},{.8660254*\j}) {};
  \fi
}}

\end{scope}

\node[left=4pt] at (quiverambientzero) {$0$};
\node[above=3pt] at (quiverambientone) {$1$};
\node[right=3pt] at (quiverambienttwo) {$2$};
\node[right=3pt] at (quiverambientthree) {$3$};
\node[below=3pt] at (quiverambientfour) {$4$};
\node[black!45] at (2.15,5.2) {$012$};
\node[black!45] at (1.55,0) {$023$};
\node[black!45] at (2.15,-5.2) {$034$};

\coordinate (quiverpolyvfour) at (2.598076,3.5);
\coordinate (quiverpolyveight) at (2.598076,1.5);
\coordinate (quiverpolyvsix) at (1.732051,-1);
\coordinate (quiverpolyvfive) at (1.732051,-4);
\node[polygon corner] at (quiverpolyvfour) {};
\node[polygon corner] at (quiverpolyveight) {};
\node[polygon corner] at (quiverpolyvsix) {};
\node[polygon corner] at (quiverpolyvfive) {};
\node[left=3pt,font=\scriptsize] at (quiverpolyvfour) {$v_4$};
\node[left=3pt,font=\scriptsize] at (quiverpolyveight) {$v_8$};
\node[left=3pt,font=\scriptsize] at (quiverpolyvsix) {$v_6$};
\node[left=3pt,font=\scriptsize] at (quiverpolyvfive) {$v_5$};

\end{tikzpicture}
\end{minipage}\hfill
\begin{minipage}[t]{0.48\textwidth}
\vspace{0pt}\centering
\textbf{(b)}~The unreduced graph \(\overline G_p\).\par\smallskip
\pgfdeclarelayer{gbar grid}
\pgfdeclarelayer{gbar edges}
\pgfsetlayers{gbar grid,gbar edges,main}
\begin{tikzpicture}[
  x=0.58cm,y=0.58cm,
  gbar edge/.style={draw=blue!42!black,line width=0.46pt},
  black dual/.style={circle,draw=black,fill=black,
    inner sep=0pt,minimum size=2.8pt},
  white dual/.style={circle,draw=black,line width=0.48pt,fill=white,
    inner sep=0pt,minimum size=3.0pt},
  panel edge/.style={draw=black!38,line width=0.6pt},
  grid line/.style={draw=black!9,line width=0.20pt}]

\begin{scope}[rotate=-90]

\coordinate (gbarzero) at (0,0);
\coordinate (gbarone) at (-10,0);
\coordinate (gbartwo) at (-5,8.660254);
\coordinate (gbarthree) at (5,8.660254);
\coordinate (gbarfour) at (10,0);

\begin{pgfonlayer}{gbar grid}
  \def\gbarpanelgrid#1#2#3{%
    \foreach \t in {1,...,9}{%
      \pgfmathsetmacro{\f}{\t/10}%
      \draw[grid line] ($(#1)!\f!(#2)$) -- ($(#3)!\f!(#2)$);%
      \draw[grid line] ($(#1)!\f!(#3)$) -- ($(#2)!\f!(#3)$);%
      \draw[grid line] ($(#2)!\f!(#1)$) -- ($(#3)!\f!(#1)$);%
    }%
  }
  \gbarpanelgrid{gbarzero}{gbarone}{gbartwo}
  \gbarpanelgrid{gbarzero}{gbartwo}{gbarthree}
  \gbarpanelgrid{gbarzero}{gbarthree}{gbarfour}
  \draw[panel edge] (gbarzero)--(gbarone)--(gbartwo)--cycle;
  \draw[panel edge] (gbarzero)--(gbartwo)--(gbarthree)--cycle;
  \draw[panel edge] (gbarzero)--(gbarthree)--(gbarfour)--cycle;
\end{pgfonlayer}

\foreach \panel in {1,2,3}{
  \foreach \i in {0,...,8}{
    \foreach \j in {0,...,8}{
      \pgfmathtruncatemacro{\k}{8-\i-\j}
      \foreach \c in {0,1,2}{

        \ifcase\c
          \pgfmathtruncatemacro{\ai}{\i+1}
          \pgfmathtruncatemacro{\aj}{\j}
          \pgfmathtruncatemacro{\ak}{\k}
          \pgfmathtruncatemacro{\qai}{\i+1}
          \pgfmathtruncatemacro{\qaj}{\j+1}
          \pgfmathtruncatemacro{\qak}{\k}
          \pgfmathtruncatemacro{\qbi}{\i+1}
          \pgfmathtruncatemacro{\qbj}{\j}
          \pgfmathtruncatemacro{\qbk}{\k+1}
        \or
          \pgfmathtruncatemacro{\ai}{\i}
          \pgfmathtruncatemacro{\aj}{\j+1}
          \pgfmathtruncatemacro{\ak}{\k}
          \pgfmathtruncatemacro{\qai}{\i}
          \pgfmathtruncatemacro{\qaj}{\j+1}
          \pgfmathtruncatemacro{\qak}{\k+1}
          \pgfmathtruncatemacro{\qbi}{\i+1}
          \pgfmathtruncatemacro{\qbj}{\j+1}
          \pgfmathtruncatemacro{\qbk}{\k}
        \or
          \pgfmathtruncatemacro{\ai}{\i}
          \pgfmathtruncatemacro{\aj}{\j}
          \pgfmathtruncatemacro{\ak}{\k+1}
          \pgfmathtruncatemacro{\qai}{\i+1}
          \pgfmathtruncatemacro{\qaj}{\j}
          \pgfmathtruncatemacro{\qak}{\k+1}
          \pgfmathtruncatemacro{\qbi}{\i}
          \pgfmathtruncatemacro{\qbj}{\j+1}
          \pgfmathtruncatemacro{\qbk}{\k+1}
        \fi

        \ifcase\panel
        \or
          \pgfmathtruncatemacro{\sela}
            {(\qai>=1)&&(\qaj<=1)&&(\qak>=4)}
          \pgfmathtruncatemacro{\selb}
            {(\qbi>=1)&&(\qbj<=1)&&(\qbk>=4)}
        \or
          \pgfmathtruncatemacro{\sela}
            {(\qai>=1)&&((\qaj>=4)||(\qak>=3))}
          \pgfmathtruncatemacro{\selb}
            {(\qbi>=1)&&((\qbj>=4)||(\qbk>=3))}
        \or
          \pgfmathtruncatemacro{\sela}
            {(\qai>=1)&&(\qaj>=3)&&(\qak<=2)}
          \pgfmathtruncatemacro{\selb}
            {(\qbi>=1)&&(\qbj>=3)&&(\qbk<=2)}
        \fi
        \pgfmathtruncatemacro{\retained}
          {(\k>=0)&&((\sela==1)||(\selb==1))}

        \ifnum\retained=1
          \pgfmathsetmacro{\wcj}{\j+0.6666667}
          \pgfmathsetmacro{\wck}{\k+0.6666667}

          \ifcase\panel
          \or
            \pgfmathsetmacro{\wx}{-\wcj-.5*\wck}
            \pgfmathsetmacro{\wy}{.8660254*\wck}
            \ifnum\aj=0
              \pgfmathsetmacro{\bx}{-.5*(\ak+.5)}
              \pgfmathsetmacro{\by}{.8660254*(\ak+.5)}
            \else
              \pgfmathsetmacro{\bx}
                {-(\aj+.3333333)-.5*(\ak+.3333333)}
              \pgfmathsetmacro{\by}{.8660254*(\ak+.3333333)}
            \fi
          \or
            \pgfmathsetmacro{\wx}{-.5*\wcj+.5*\wck}
            \pgfmathsetmacro{\wy}{.8660254*(\wcj+\wck)}
            \ifnum\ak=0
              \pgfmathsetmacro{\bx}{-.5*(\aj+.5)}
              \pgfmathsetmacro{\by}{.8660254*(\aj+.5)}
            \else
              \ifnum\aj=0
                \pgfmathsetmacro{\bx}{.5*(\ak+.5)}
                \pgfmathsetmacro{\by}{.8660254*(\ak+.5)}
              \else
                \pgfmathsetmacro{\bx}
                  {-.5*(\aj+.3333333)+.5*(\ak+.3333333)}
                \pgfmathsetmacro{\by}
                  {.8660254*(\aj+\ak+.6666667)}
              \fi
            \fi
          \or
            \pgfmathsetmacro{\wx}{.5*\wcj+\wck}
            \pgfmathsetmacro{\wy}{.8660254*\wcj}
            \ifnum\ak=0
              \pgfmathsetmacro{\bx}{.5*(\aj+.5)}
              \pgfmathsetmacro{\by}{.8660254*(\aj+.5)}
            \else
              \pgfmathsetmacro{\bx}
                {.5*(\aj+.3333333)+(\ak+.3333333)}
              \pgfmathsetmacro{\by}{.8660254*(\aj+.3333333)}
            \fi
          \fi

          \begin{pgfonlayer}{gbar edges}
            \draw[gbar edge] (\wx,\wy)--(\bx,\by);
          \end{pgfonlayer}
          \node[white dual] at (\wx,\wy) {};
          \node[black dual] at (\bx,\by) {};
        \fi
      }
    }
  }
}

\end{scope}

\node[left=4pt] at (gbarzero) {$0$};
\node[above=3pt] at (gbarone) {$1$};
\node[right=3pt] at (gbartwo) {$2$};
\node[right=3pt] at (gbarthree) {$3$};
\node[below=3pt] at (gbarfour) {$4$};
\node[black!45] at (2.15,5.2) {$012$};
\node[black!45] at (1.55,0) {$023$};
\node[black!45] at (2.15,-5.2) {$034$};

\end{tikzpicture}
\end{minipage}
\caption{The quiver--graph correspondence for
\(p=(2,3,2,3)_{1234}\), in the common three-panel orientation.
\textup{(a)} The induced Lam--Postnikov quiver
\(Q_0^{\rm LP}[S_p]\), with \(59\) blue
selected vertices and \(134\) orange arrows; the hollow points
\(v_4,v_8,v_6,v_5\) are unselected comparison points.
\textup{(b)} The unreduced bipartite graph \(\overline G_p\), with
\(68\) vertices of each color and \(194\) edges: \(134\) internal
edges dual to the arrows in \textup{(a)}, and \(60\) exterior-boundary
edges.  Each orange arrow is directed with the adjacent black dual
vertex on its right and the adjacent white dual vertex on its left,
which is the convention of \cite[Definition~28.4]{LamPostnikov}.
Retained black seam pairs are identified at their common
midpoint, and the \(59\) bounded faces are indexed by \(S_p\).}
\label{fig:fixed-graph-running-example}
\label{fig:induced-quiver-running-example}
\label{fig:unreduced-graph-running-example}
\end{figure}

\begin{definition}[Cut-open seam rule]
\label{def:cut-open-seam}
The only possible three-sector identification occurs at the
\(u_0\)-corner.  Every primal vertex incident to that three-sector dual
vertex has \(q_2\leq1\) and \(q_3\leq1\), whereas
\(b+1\geq2\) and \(c+1\geq2\).  Hence none of these incident vertices
lies in \(S_p\), and the three-sector dual vertex is deleted.
Consequently every retained seam vertex has at most two sectors, so
the number \(r\) of sectors at such a vertex satisfies \(r\leq2\).  If \(r=1\), no modification is needed.  If a vertex
\(x\) of \(\overline G_p\) identifies two same-colored sector vertices,
replace it by \(x_{\rm L},x_{\rm R}\), reattach the two sector edge sets
to their corresponding copies, and insert the alternating path
\[
                         x_{\rm L}-z-x_{\rm R} ,
\]
where \(z\) has the color opposite to \(x_{\rm L},x_{\rm R}\); see
Figure~\ref{fig:canonical-seam-zigzag}.  Apply this at every
two-sector seam identification.  The resulting plane bipartite graph,
embedded in the cut-open fan, is \(G_p\).  Its inserted seam edges have
weight one.
\end{definition}

\begin{figure}[htbp]
\centering
\includegraphics[width=0.90\textwidth]
  {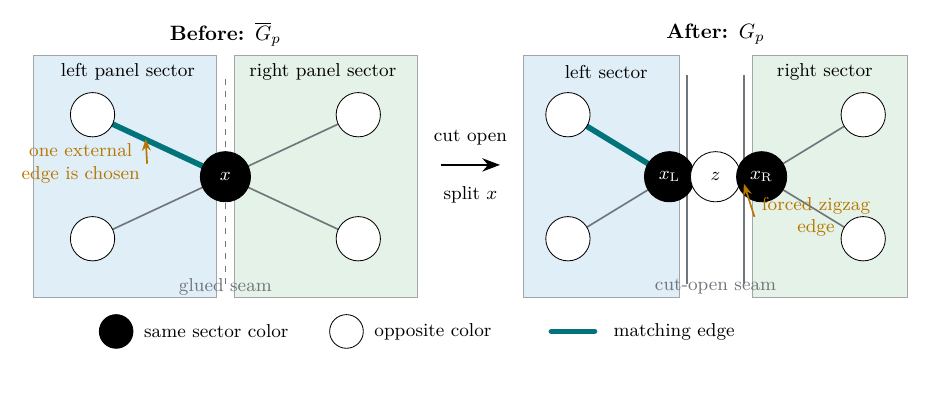}
\caption{The cut-open seam rule in the only nontrivial case \(r=2\).
Before cutting, the two panel sectors share the same-colored vertex
\(x\).  After cutting, \(x\) is replaced by \(x_{\rm L},x_{\rm R}\), and
the opposite-colored vertex \(z\) is inserted to form
\(x_{\rm L}-z-x_{\rm R}\).  The thick edges show one local matching state:
if \(x_{\rm L}\) is matched externally in the left sector, then
\(zx_{\rm R}\) is forced.  The reflected state, with \(x_{\rm R}\)
matched externally, forces \(x_{\rm L}z\).}
\label{fig:canonical-seam-zigzag}
\end{figure}

Write \(\D(G_p)\) for the set of perfect matchings of \(G_p\).
Let \(E_v\) be the set of edges of \(\overline G_p\) crossing
elementary lattice edges incident to \(v\in\Sigma _0\).  For
\(M\in\D(G_p)\), contract the inserted seam zigzags and put
\[
 m_v(M):=|\overline M\cap E_v|,\qquad
 w_v:=\frac{|\partial_{\overline G_p}F_v|}{2}\quad(v\in S_p),
\]
where \(\overline M\) denotes the contracted edge set.
The theorem below proves that \(\D(G_p)\) is nonempty, that the seam
contraction is a bijection, and that its matching lattice has a unique
bottom element.  It then makes the maximum incidence statistics used
in the formula well defined.

Finally, introduce the five initial fan vertices
\begin{equation}
\begin{aligned}
v_0&=(0,0,0,N-d,d),&
v_1&=(d,0,0,N-d,0),\\
v_2&=(d,0,a+b,c,0),&
v_3&=(c+d,0,a+b,0,0),\\
v_4&=(c+d,a,b,0,0).
\end{aligned}
\label{eq:five-g-points}
\end{equation}
Here \(v_0\in I_{\rm fr}\) and
\(v_1,v_2,v_3,v_4\in I_{\rm mut}\).

\begin{theorem}[Main theorem of Part~I: the corrected fixed-graph formula]
\label{thm:partI-main}
Let \(N\geq4\) and let
\[
 p=(0,a,b,c,d)\in\Delta_{1234}(N),
 \qquad a,b,c,d\geq1.
\]
For the set \(S_p\) and the graphs \(\overline G_p,G_p\) defined
above, the following statements hold.
\begin{enumerate}[label=\textup{(\roman*)}]
\item The denominator vector of \(X_p\), relative to the initial fan
seed, is the indicator of \(S_p\).  In particular it is binary.

\item The graph \(\overline G_p\) is connected and plane bipartite,
its bounded faces are precisely \(F_v\), \(v\in S_p\), and its
Lam--Postnikov oriented dual is \(Q_0^{\rm LP}[S_p]\); its opposite,
white-on-the-right dual is \(Q_0^{\rm R}[S_p]\).  Cutting open the
seams gives a canonical
weight-preserving bijection
\[
          \D(\overline G_p)\longleftrightarrow\D(G_p).
\]
Every edge of either graph belongs to a perfect matching.  The
matching set is ordered as follows.  Fix any reference matching
\(M_\ast\) of \(\overline G_p\), orient every primal edge from white to
black, and let \(H_M\) be the dual potential of the divergence-free
oriented flow \(\mathbf 1_M-\mathbf 1_{M_\ast}\), normalized to vanish
on the unbounded face.  Thus crossing an oriented primal edge changes
\(H_M\) by the signed multiplicity of that edge in
\(M-M_\ast\).  Define
\[
 M\leq M'
 \quad\Longleftrightarrow\quad
 H_M(F)\leq H_{M'}(F)
 \quad\hbox{for every bounded face \(F\)}.
\]
This order is independent of \(M_\ast\), and it is transported to
\(\D(G_p)\) by the seam-cut bijection above.  Its cover relations are
the positive face rotations, which increase the height at one face by
\(1\).  With this order the matching set is a finite distributive
lattice and has a unique bottom matching \(M_{\min}\).  Put
\[
             m_v^{\max}:=\max_{M\in\D(G_p)}m_v(M).
\]

\item The mutable \(g\)-vector and the separate frozen normalization
are
\begin{equation}
\boxed{\qquad
 \beta_p=\mathbf e_{v_0},\qquad
 g_p=g^{G_p}
      =-\mathbf e_{v_1}+\mathbf e_{v_2}
      -\mathbf e_{v_3}+\mathbf e_{v_4}.
\qquad}
\label{eq:five-point-g}
\end{equation}
Thus the  \(g\)-vector is the four-term mutable vector, while
\(\beta_p\) is the separately computed frozen normalization.  Their
sum
\[
 \widetilde g_p:=\beta_p+g_p
 =\mathbf e_{v_0}-\mathbf e_{v_1}+\mathbf e_{v_2}
  -\mathbf e_{v_3}+\mathbf e_{v_4}
\]
is the extended bottom exponent in the terminology of
\eqref{eq:extended-bottom-exponent-dictionary}.  It is not an
extended \(g\)-vector.

\item Put \(X_\partial^{\beta_p}=X_{v_0}\) and, for initial tropical
heights \(h=(h(v))_{v\in\Sigma _0}\), put
\(\mathsf b_p(h)=h(v_0)\) and
\begin{align}
\Lambda_M^{\rm corr}(h):={}&\mathsf b_p(h)
 +\sum_{v\in S_p}\bigl(w_v-m_v(M)-1\bigr)h(v)
\notag\\
&+\sum_{\substack{v\in I_{\rm mut}\setminus S_p\\E_v\ne\varnothing}}
       \bigl(m_v^{\max}-m_v(M)\bigr)h(v)
\notag\\
&+\sum_{v\in I_{\rm fr}}
       \bigl(m_v(M_{\min})-m_v(M)\bigr)h(v).
\label{eq:linear-form}
\end{align}
Then the octahedron variable and its tropicalization are
\begin{equation}
\boxed{
\begin{aligned}
X_p=X_\partial^{\beta_p}\sum_{M\in\D(G_p)}
&\prod_{v\in S_p}X_v^{\,w_v-m_v(M)-1}\\
{}\times{}&
\prod_{\substack{v\in I_{\rm mut}\setminus S_p\\E_v\ne\varnothing}}
 X_v^{\,m_v^{\max}-m_v(M)}\\
{}\times{}&
\prod_{v\in I_{\rm fr}}
 X_v^{\,m_v(M_{\min})-m_v(M)} ,
\end{aligned}}
\label{eq:laurent-matching-formula}
\end{equation}
and
\begin{equation}
\boxed{\qquad
 h(p)=\max_{M\in\D(G_p)}\Lambda_M^{\rm corr}(h).
\qquad}
\label{eq:tropical-matching-formula}
\end{equation}
In particular, the frozen contribution is
matching-dependent and cannot be replaced by a common boundary
monomial.
\end{enumerate}
\end{theorem}

The proof is organized into Steps~I--V; their logical dependencies are
summarized next.

\subsection*{Proof strategy}
The proof of Theorem~\ref{thm:partI-main} has three stages.  Steps~I
and~II determine the positive denominator support and turn it into one
fixed cut-open planar graph.  Step~III constructs the extremal heights
as directed boundary envelopes and uses a raised-endpoint partial datum
to prove directly that every edge of this graph occurs in a perfect
matching.  It then invokes flip-connectivity and identifies the
polygonal descriptions with the two extremal matchings.  Steps~IV
and~V apply the dimer
\(F\)-polynomial theorem, compute the complete incidence correction,
and identify the resulting cluster variable with \(X_p\) by equality
of the four mutable \(g\)-coordinates.  The matching exponents then
complete the proof that the positive denominator coordinates are all
one.  The frozen normalization is
computed separately from the recurrence and is inserted only after
this principal-coefficient identification.

\section{Step I: The denominator support}
\label{sec:step1-support}
\label{sec:tropical-refraction}

\subsection{Backward dependence, reversal, and support reciprocity}

For cluster variables \(X,Y\), write \(d(X,Y)\) for the
\(X\)-coordinate of the denominator vector of \(Y\).  Thus the
denominator coordinate required in Theorem~\ref{thm:partI-main} is
\[
                         D_p(q):=d(X_q,X_p),
                         \qquad q\in\Sigma _0.
\]
We will find it more convenient to keep a terminal seed containing \(X_p\) fixed and
reverse the three sweeps.  On every membrane encountered in this
reverse circuit, put
\begin{equation}
                         \delta_p(r):=d(X_p,X_r).
 \label{eq:reverse-d-field}
\end{equation}
At the terminal membrane this scalar field is \(-1\) at \(p\) and
zero at every other seed vertex.  The octahedron exchange relation is
time-reversible.  Consequently, reversing a mutation replaces the old
opposite variable by the new one while the terminal seed remains fixed,
and its \(X_p\)-denominator coordinate obeys
\begin{equation}
 \delta_{\mathrm{new}}
   =\max\{\delta_1+\delta_3,\delta_2+\delta_4\}
      -\delta_{\mathrm{old}}.
 \label{eq:tropical-d}
\end{equation}
This is an ordinary fixed-seed denominator calculation, not a rule for
changing the seed relative to which denominators are measured.

Indeed, relative to the fixed terminal seed, regard all variables other
than \(X_p\) as nonzero formal coefficients and take the least
\(X_p\)-exponent.  Every exchange expression is subtraction-free, so
least exponents add under products and take their minimum under sums.
Negating them gives exactly \eqref{eq:tropical-d}.  Time reversibility
therefore identifies the successive scalar fields in Figure~7 with the
numerical coordinates \(d(X_p,X_r)\).

The following observation transfers their positive support to the
denominator vector in the required order.

\begin{supportreciprocity}
For every initial fan variable \(X_q\ne X_p\), including a frozen
boundary variable,
\begin{equation}
                         d(X_p,X_q)>0
                 \quad\Longleftrightarrow\quad
                         d(X_q,X_p)>0.
 \label{eq:positive-d-support-reciprocity}
\end{equation}
\end{supportreciprocity}

\begin{proof}
First temporarily unfreeze the boundary vertices.  If the ice exchange
matrix is \(\widetilde B=\binom{B}{C}\), replace it by the square
skew-symmetric matrix
\[
                 \overline B=
                 \begin{pmatrix}B&-C^{\mathsf T}\\ C&0\end{pmatrix}.
\]
Mutate only in the original mutable directions.  Induction under matrix
mutation shows that the columns in those directions are exactly the
columns of the mutated ice matrix.  Hence every exchange relation and
every cluster variable used in the three sweeps is unchanged, and the
ordinary denominator coordinates in the enlarged algebra are the
extended denominator coordinates used here.

By \cite[Theorem~6.3(i),(ii)]{CaoLiDenominators}, for two distinct
cluster variables the corresponding denominator coordinate is zero
exactly when the variables occur together in a cluster, and is positive
otherwise.  Compatibility is symmetric, so positivity is symmetric as
well.  Applying this in the enlarged algebra proves
\eqref{eq:positive-d-support-reciprocity}, including for the formerly
frozen vertices.
\end{proof}

\begin{definition}[Backward light cone and seam-compatible extension]
\label{def:tropical-refraction}
For a target variable \(X_p\), the positive support of
\(q\mapsto D_p(q)\) on an earlier membrane is its \emph{backward light
cone}.  By support reciprocity, this is also
the positive support of the reverse field \(q\mapsto\delta_p(q)\) that
is propagated in Figure~7.

Suppose two input panels \(ijk\) and \(ik\ell\) meet along the seam
\(ik\).  We say that a binary light-cone contribution
\emph{extends compatibly over \(ijk\cup_{ik}ik\ell\)} if its two panel
pieces have the same trace on the seam, the threshold tests in the
tangential coordinates \(q_i,q_k\) agree, and the side-dependent
transverse cap is written in \(q_j\) on one panel and in \(q_\ell\) on
the other.  This is the piecewise-affine seam-compatibility condition
used below.
\end{definition}

Figure~\ref{fig:step-I-backward-light-cone}
gives a chronological guide to the support masks for the running example
\(p=(2,3,2,3)_{1234}\), read opposite to the forward sweep order.
Neutral markers show the total intermediate support.  The colored sets
in the last two pictures record the two clauses of the final Boolean
support, not independently propagated objects and not multiplicities.
To make the coupled transition from \textup{(c)} to \textup{(d)}
visible, Figure~\ref{fig:step-I-c-to-d-excavation-square} develops the
same reverse field in square coordinates and displays its scalar
max--minus--old recurrence.

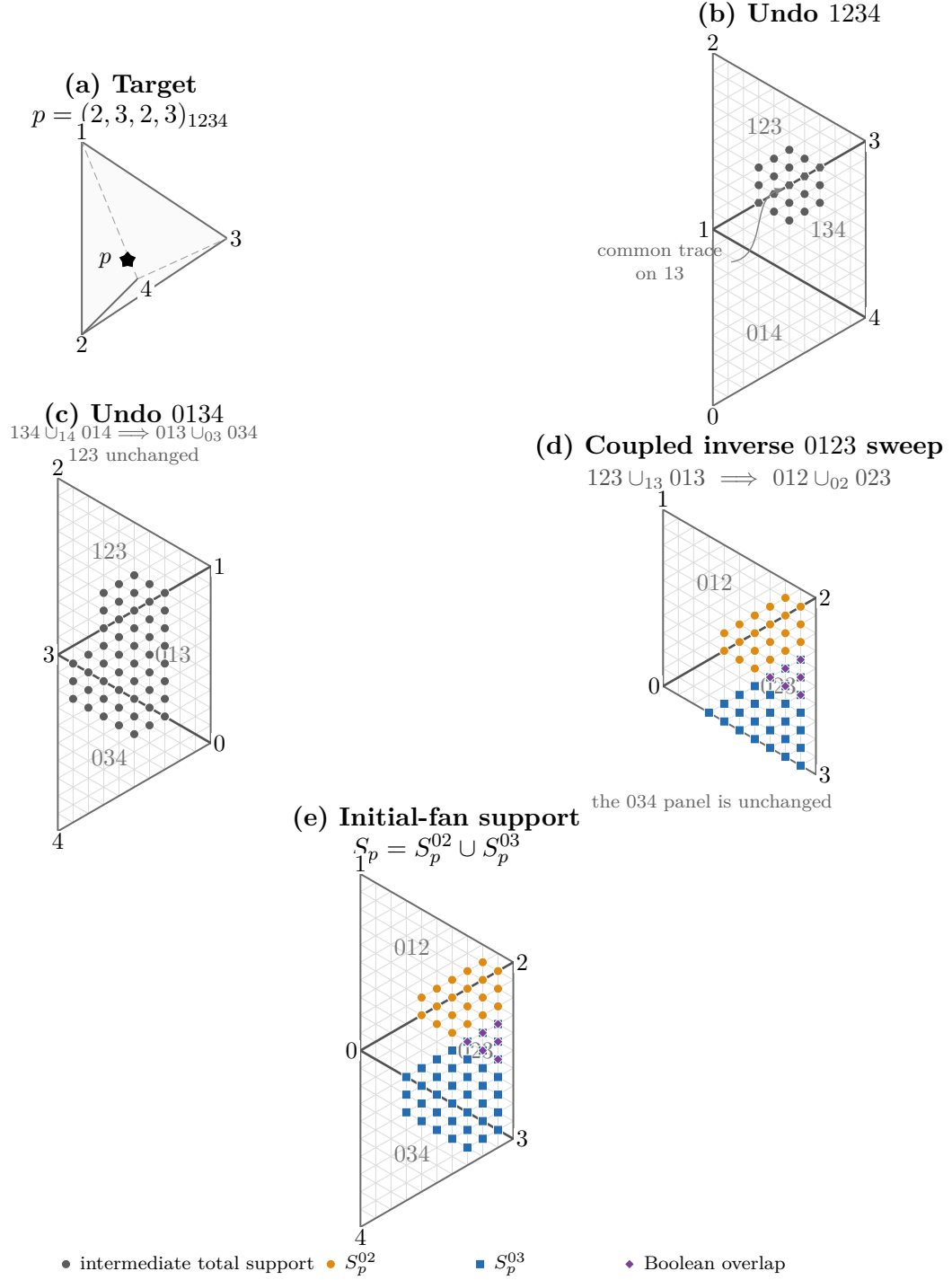
\begin{figure}[p!]
\centering
\definecolor{frontamber}{RGB}{222,139,20}
\definecolor{frontblue}{RGB}{38,108,176}
\definecolor{frontpurple}{RGB}{123,65,148}

\tikzset{
  lc grid/.style={draw=black!14,line width=.28pt},
  lc boundary/.style={draw=black!58,line width=.75pt},
  lc seam/.style={draw=black!68,line width=1pt},
  lc interim/.style={circle,fill=black!64,draw=white,line width=.28pt,
                     inner sep=0pt,minimum size=3.8pt},
  lc front02/.style={circle,fill=frontamber,draw=white,line width=.28pt,
                     inner sep=0pt,minimum size=3.8pt},
  lc front03/.style={rectangle,fill=frontblue,draw=white,line width=.28pt,
                     inner sep=0pt,minimum size=3.8pt},
  lc overlap/.style={diamond,aspect=1,fill=frontpurple,draw=white,
                     line width=.28pt,inner sep=0pt,minimum size=4.2pt},
  lc commontrace/.style={regular polygon,regular polygon sides=6,
                     fill=black!62,draw=white,line width=.28pt,
                     inner sep=0pt,minimum size=4.2pt},
  lc panelname/.style={font=\small,text=black!54},
  lc vertexname/.style={font=\small,fill=white,inner sep=1pt},
  lc state/.style={font=\normalsize\bfseries,align=center},
  lc transmit/.style={-{Latex[length=1.9mm]},line width=.95pt}
}

\newcommand{\lctrigrid}[3]{%
  \foreach \t in {1,...,9}{%
    \pgfmathsetmacro{\f}{\t/10}%
    \draw[lc grid] ($(#1)!\f!(#2)$)--($(#3)!\f!(#2)$);
    \draw[lc grid] ($(#1)!\f!(#3)$)--($(#2)!\f!(#3)$);
    \draw[lc grid] ($(#2)!\f!(#1)$)--($(#3)!\f!(#1)$);
  }%
}

\newcommand{\lcfanbase}[5]{%
  \coordinate (C)  at (0,0);
  \coordinate (O1) at (0,2.6);
  \coordinate (O2) at (2.251666,1.3);
  \coordinate (O3) at (2.251666,-1.3);
  \coordinate (O4) at (0,-2.6);
  \lctrigrid{C}{O1}{O2}\lctrigrid{C}{O2}{O3}\lctrigrid{C}{O3}{O4}
  \draw[lc boundary] (C)--(O1)--(O2)--(O3)--(O4)--cycle;
  \draw[lc seam] (C)--(O2) (C)--(O3);
  \node[lc vertexname,left]  at (C)  {$#1$};
  \node[lc vertexname,above] at (O1) {$#2$};
  \node[lc vertexname,right] at (O2) {$#3$};
  \node[lc vertexname,right] at (O3) {$#4$};
  \node[lc vertexname,below] at (O4) {$#5$};
}

\newcommand{\lcfirstdot}[4]{%
  \node[#4] at ({.2251666*(#3)},{.26*(#2)+.13*(#3)}) {};
}
\newcommand{\lcmiddledot}[4]{%
  \node[#4] at ({.2251666*((#2)+(#3))},{.13*(#2)-.13*(#3)}) {};
}
\newcommand{\lcthirddot}[4]{%
  \node[#4] at ({.2251666*(#2)},{-.13*(#2)-.26*(#3)}) {};
}

\centering

\begin{minipage}[c][5.85cm][c]{.46\textwidth}
\centering
\begin{tikzpicture}[line cap=round,line join=round]
  \node[lc state] at (0,2.02) {(a) Target\\[-1pt]
    $p=(2,3,2,3)_{1234}$};
  \coordinate (T1) at (-.72,1.42);
  \coordinate (T2) at (-.72,-1.42);
  \coordinate (T3) at (1.42,0);
  \coordinate (T4) at (.10,-.60);
  \fill[black!2] (T1)--(T2)--(T3)--cycle;
  \draw[lc boundary] (T1)--(T2)--(T3)--cycle;
  \draw[black!38,densely dashed] (T4)--(T1) (T4)--(T3);
  \draw[lc boundary] (T4)--(T2);
  \node[lc vertexname,above] at (T1) {$1$};
  \node[lc vertexname,below] at (T2) {$2$};
  \node[lc vertexname,right] at (T3) {$3$};
  \node[lc vertexname,below right] at (T4) {$4$};
  \coordinate (p) at ($.2*(T1)+.3*(T2)+.2*(T3)+.3*(T4)$);
  \node[star,star points=5,fill=black,draw=white,line width=.4pt,
        minimum size=8pt,inner sep=0pt] at (p) {};
  \node[font=\small,left=3pt] at (p) {$p$};
\end{tikzpicture}
\end{minipage}
\hfill
\begin{minipage}[c][5.85cm][c]{.46\textwidth}
\centering
\begin{tikzpicture}[line cap=round,line join=round]
  \node[lc state] at (1.12,3.16) {(b) Undo $1234$};
  \lcfanbase{1}{2}{3}{4}{0}
  \node[lc panelname] at (.76,1.52) {$123$};
  \node[lc panelname] at (1.70,0) {$134$};
  \node[lc panelname] at (.76,-1.52) {$014$};
  \foreach \qone in {0,...,10}{\foreach \qtwo in {0,...,10}{%
    \pgfmathtruncatemacro{\qthree}{10-\qone-\qtwo}
    \pgfmathtruncatemacro{\hit}{(\qthree>=0)&&(\qone>=3)&&(\qtwo<=2)&&(\qthree>=3)}
    \ifnum\hit=1 \lcfirstdot{\qone}{\qtwo}{\qthree}{lc interim}\fi
  }}
  \foreach \qone in {0,...,10}{\foreach \qthree in {0,...,10}{%
    \pgfmathtruncatemacro{\qfour}{10-\qone-\qthree}
    \pgfmathtruncatemacro{\hit}{(\qfour>=0)&&(\qone>=3)&&(\qthree>=3)&&(\qfour<=2)}
    \ifnum\hit=1 \lcmiddledot{\qone}{\qthree}{\qfour}{lc interim}\fi
  }}
  \foreach \qone in {3,...,7}{%
    \pgfmathtruncatemacro{\qthree}{10-\qone}
    \lcmiddledot{\qone}{\qthree}{0}{lc commontrace}
  }
  \node[font=\scriptsize,text=black!62,align=center]
    (tracelabel) at (-.78,-.48) {common trace\\on $13$};
  \draw[draw=black!48,-{Latex[length=1.5mm]},line width=.55pt]
    (tracelabel.east) to[out=8,in=200] (1.06,.61);
\end{tikzpicture}
\end{minipage}

\vspace{2mm}

\begin{minipage}[c][5.85cm][c]{.46\textwidth}
\centering
\begin{tikzpicture}[line cap=round,line join=round]
  \node[lc state] at (1.12,3.52) {(c) Undo $0134$};
  \node[font=\scriptsize,text=black!64,align=center] at (1.12,3.08)
    {$134\cup_{14}014\Longrightarrow013\cup_{03}034$\\
     $123$ unchanged};
  \lcfanbase{3}{2}{1}{0}{4}
  \node[lc panelname] at (.76,1.52) {$123$};
  \node[lc panelname] at (1.70,0) {$013$};
  \node[lc panelname] at (.76,-1.52) {$034$};
  \foreach \qthree in {0,...,10}{\foreach \qtwo in {0,...,10}{%
    \pgfmathtruncatemacro{\qone}{10-\qthree-\qtwo}
    \pgfmathtruncatemacro{\hit}{(\qone>=0)&&(\qone>=3)&&(\qtwo<=2)&&(\qthree>=3)}
    \ifnum\hit=1 \lcfirstdot{\qthree}{\qtwo}{\qone}{lc interim}\fi
  }}
  \foreach \qthree in {0,...,10}{\foreach \qone in {0,...,10}{%
    \pgfmathtruncatemacro{\qzero}{10-\qthree-\qone}
    \pgfmathtruncatemacro{\hit}{(\qzero>=1)&&(\qthree>=3)}
    \ifnum\hit=1 \lcmiddledot{\qthree}{\qone}{\qzero}{lc interim}\fi
  }}
  \foreach \qthree in {0,...,10}{\foreach \qzero in {0,...,10}{%
    \pgfmathtruncatemacro{\qfour}{10-\qthree-\qzero}
    \pgfmathtruncatemacro{\hit}{(\qfour>=0)&&(\qzero>=1)&&(\qthree>=3)&&(\qfour<=2)}
    \ifnum\hit=1 \lcthirddot{\qthree}{\qzero}{\qfour}{lc interim}\fi
  }}
\end{tikzpicture}
\end{minipage}
\hfill
\begin{minipage}[c][5.85cm][c]{.46\textwidth}
\centering
\begin{tikzpicture}[line cap=round,line join=round]
  \node[lc state] at (1.12,3.52) {(d) Coupled inverse $0123$ sweep};
  \node[font=\small,align=center,text=black!68] at (1.12,3.05)
    {$123\cup_{13}013\ \Longrightarrow\ 012\cup_{02}023$};
  \coordinate (C)  at (0,0);
  \coordinate (O1) at (0,2.6);
  \coordinate (O2) at (2.251666,1.3);
  \coordinate (O3) at (2.251666,-1.3);
  \lctrigrid{C}{O1}{O2}\lctrigrid{C}{O2}{O3}
  \draw[lc boundary] (C)--(O1)--(O2)--(O3)--cycle;
  \draw[lc seam] (C)--(O2);
  \node[lc vertexname,left]  at (C)  {$0$};
  \node[lc vertexname,above] at (O1) {$1$};
  \node[lc vertexname,right] at (O2) {$2$};
  \node[lc vertexname,right] at (O3) {$3$};
  \node[lc panelname] at (.76,1.52) {$012$};
  \node[lc panelname] at (1.70,0) {$023$};
  \foreach \qzero in {0,...,10}{\foreach \qone in {0,...,10}{%
    \pgfmathtruncatemacro{\qtwo}{10-\qzero-\qone}
    \pgfmathtruncatemacro{\hit}{(\qtwo>=0)&&(\qzero>=1)&&(\qone<=1)&&(\qtwo>=4)}
    \ifnum\hit=1 \lcfirstdot{\qzero}{\qone}{\qtwo}{lc front02}\fi
  }}
  \foreach \qzero in {0,...,10}{\foreach \qtwo in {0,...,10}{%
    \pgfmathtruncatemacro{\qthree}{10-\qzero-\qtwo}
    \pgfmathtruncatemacro{\ha}{(\qthree>=0)&&(\qzero>=1)&&(\qtwo>=4)}
    \pgfmathtruncatemacro{\hb}{(\qthree>=0)&&(\qzero>=1)&&(\qthree>=3)}
    \ifnum\ha=1 \lcmiddledot{\qzero}{\qtwo}{\qthree}{lc front02}\fi
    \ifnum\hb=1 \lcmiddledot{\qzero}{\qtwo}{\qthree}{lc front03}\fi
    \pgfmathtruncatemacro{\hab}{\ha*\hb}
    \ifnum\hab=1 \lcmiddledot{\qzero}{\qtwo}{\qthree}{lc overlap}\fi
  }}
  \node[font=\scriptsize,text=black!62,align=center] at (.72,-1.72)
    {the $034$ panel is unchanged};
\end{tikzpicture}
\end{minipage}

\vspace{0mm}

\begin{minipage}[c][5.85cm][c]{.58\textwidth}
\centering
\begin{tikzpicture}[line cap=round,line join=round]
  \node[lc state] at (1.12,3.16) {(e) Initial-fan support\\[-1pt]
    $S_p=S_p^{02}\cup S_p^{03}$};
  \lcfanbase{0}{1}{2}{3}{4}
  \node[lc panelname] at (.76,1.52) {$012$};
  \node[lc panelname] at (1.70,0) {$023$};
  \node[lc panelname] at (.76,-1.52) {$034$};
  \foreach \qzero in {0,...,10}{\foreach \qone in {0,...,10}{%
    \pgfmathtruncatemacro{\qtwo}{10-\qzero-\qone}
    \pgfmathtruncatemacro{\hit}{(\qtwo>=0)&&(\qzero>=1)&&(\qone<=1)&&(\qtwo>=4)}
    \ifnum\hit=1 \lcfirstdot{\qzero}{\qone}{\qtwo}{lc front02}\fi
  }}
  \foreach \qzero in {0,...,10}{\foreach \qtwo in {0,...,10}{%
    \pgfmathtruncatemacro{\qthree}{10-\qzero-\qtwo}
    \pgfmathtruncatemacro{\ha}{(\qthree>=0)&&(\qzero>=1)&&(\qtwo>=4)}
    \pgfmathtruncatemacro{\hb}{(\qthree>=0)&&(\qzero>=1)&&(\qthree>=3)}
    \ifnum\ha=1 \lcmiddledot{\qzero}{\qtwo}{\qthree}{lc front02}\fi
    \ifnum\hb=1 \lcmiddledot{\qzero}{\qtwo}{\qthree}{lc front03}\fi
    \pgfmathtruncatemacro{\hab}{\ha*\hb}
    \ifnum\hab=1 \lcmiddledot{\qzero}{\qtwo}{\qthree}{lc overlap}\fi
  }}
  \foreach \qzero in {0,...,10}{\foreach \qthree in {0,...,10}{%
    \pgfmathtruncatemacro{\qfour}{10-\qzero-\qthree}
    \pgfmathtruncatemacro{\hit}{(\qfour>=0)&&(\qzero>=1)&&(\qthree>=3)&&(\qfour<=2)}
    \ifnum\hit=1 \lcthirddot{\qzero}{\qthree}{\qfour}{lc front03}\fi
  }}
\end{tikzpicture}
\end{minipage}

\vspace{3mm}

\begin{tikzpicture}[baseline=-.5ex]
  \node[lc interim,label={[font=\scriptsize]right:intermediate total support}]
    at (0,0) {};
  \node[lc front02,label={[font=\scriptsize]right:$S_p^{02}$}]
    at (3.9,0) {};
  \node[lc front03,label={[font=\scriptsize]right:$S_p^{03}$}]
    at (6.1,0) {};
  \node[lc overlap,label={[font=\scriptsize]right:Boolean overlap}]
    at (8.3,0) {};
\end{tikzpicture}
\vspace{-3pt}
\caption{Backward support masks for \(N=10\) and
\(p=(2,3,2,3)_{1234}\), read in reverse sweep order.  Neutral markers
in \textup{(b)}--\textup{(c)} show total intermediate support.
The transition \textup{(b)}--\textup{(c)} is the inverse \(0134\)
sweep whose Boolean verification is given in the middle reverse-sweep
part of the proof of Theorem~\ref{thm:light-cone}.
The moving parts in panels \textup{(c)} and \textup{(d)} lie on the
endpoint membranes \(123\cup_{13}013\) and
\(012\cup_{02}023\); \(034\) is unchanged.
Figure~\ref{fig:step-I-c-to-d-excavation-square} gives the
max--minus--old calculation connecting them.  Amber and blue in
\textup{(d)}--\textup{(e)} are the two clauses of one Boolean support;
dark hexagons mark a common seam trace, and purple diamonds an overlap
of value \(1\), not \(2\).  The exact final inequalities are
\eqref{eq:Sp-formula}, equivalently
Proposition~\ref{prop:refracted-lobes}.}
\label{fig:step-I-backward-light-cone}
\end{figure}

\begin{figure}[p!]
\centering
\input{backward-light-cone-c-to-d-detail-tikz-input.tex}
\caption{Excavation connecting the exact reverse-field masks in
Figure~\ref{fig:step-I-backward-light-cone}\textup{(c)}--\textup{(d)};
the unchanged \(034\) panel is suppressed.  \textup{(a)} The incoming
mask on \(123\cup_{13}013\).  \textup{(b)} Its affine development in the
NST square, with no change of values.  The scalar circuit applies
\(T=\max\{N+S,E+W\}-O\) in excavation order; the central cards show
deletion of a weak arrival and reduction of two arrivals to one.
\textup{(c)} The bottom mask, cut along \(I+J=10\).  The last
reverse-sweep part of the proof of Theorem~\ref{thm:light-cone} proves
the corresponding general-\(N\) formula on the tetrahedron \(0123\).
\textup{(d)} Inverse development gives Figure~\ref{fig:step-I-backward-light-cone}
\textup{(d)}.  The calculation acts on the whole mask, not pointwise on
its markers.}
\label{fig:step-I-c-to-d-excavation-square}
\end{figure}
\clearpage

Figure~\ref{fig:step-I-c-to-d-excavation-square} isolates the only
coupled operation hidden between
Figure~\ref{fig:step-I-backward-light-cone}\textup{(c)} and
Figure~\ref{fig:step-I-backward-light-cone}\textup{(d)}.
The starting object in panel~\textup{(a)} is the single binary reverse
field \(\delta_p^{(c)}\) on \(123\cup_{13}013\), not a collection of
independent sources.  A neutral marker means \(\delta_p^{(c)}(r)=1\),
and an unmarked lattice point means \(\delta_p^{(c)}(r)=0\).  Its
general-\(N\) formula is
\eqref{eq:incoming-square-field} below.  Panel~\textup{(b)} merely
develops this marked rhombus affinely into a square.  In its coordinates
the incoming panels are
\[
 \begin{array}{ll}
  123:\ I\leq J,
    &(q_0,q_1,q_2,q_3)=(0,N-J,J-I,I),\\
  013:\ J\leq I,
    &(q_0,q_1,q_2,q_3)=(I-J,N-I,0,J),
 \end{array}
\]
and their single identified seam is \(I=J\).  The outgoing development is
\[
 \begin{array}{ll}
  012:\ I+J\leq N,
    &(q_0,q_1,q_2,q_3)=(I,N-I-J,J,0),\\
  023:\ I+J\geq N,
    &(q_0,q_1,q_2,q_3)=(N-J,0,N-I,I+J-N),
 \end{array}
\]
with output seam \(I+J=N\).  Seam markers drawn only once in
Figure~\ref{fig:step-I-backward-light-cone}\textup{(c)} must not be
replaced by zero on a second copy when the two incoming panels are
developed into the square.

\paragraph{The general-\(N\) incoming field.}
Write \(Q_N:=\{0,1,\ldots,N\}^2\).  For
\(p=(0,a,b,c,d)\), the incoming membrane
\(123\cup_{13}013\) is cut along \(I=J\), and its binary reverse field
is given exactly by
\begin{equation}
\delta_p^{(c)}(I,J)=1
\quad\Longleftrightarrow\quad
\begin{cases}
I<J,\quad I\geq c+1,\quad J\leq N-a-1,\quad J-I\leq b-1,\\
I=J,\quad c+1\leq I\leq N-a-1,\\
I>J,\quad J\geq c+1.
\end{cases}
\label{eq:incoming-square-field}
\end{equation}
At every other site of \(Q_N\), the field has value \(0\).  Equivalently,
in panel coordinates its support is
\[
\begin{array}{ll}
123:&q_1\geq a+1,\quad q_2\leq b-1,\quad q_3\geq c+1,\\
013:&q_3\geq c+1,\quad
      (q_0\geq1\ \text{or}\ q_1\geq a+1).
\end{array}
\]
The middle reverse-sweep part of the proof of
Theorem~\ref{thm:light-cone} proves the required Boolean formula
pointwise on the tetrahedron \(0134\).

\paragraph{The column hidden by the square development.}
The precise relation between panels~\textup{(b)} and~\textup{(c)} of
Figure~\ref{fig:step-I-c-to-d-excavation-square} is most transparent
if one restores the coordinate suppressed by the planar development.
For each square site \((I,J)\), put
\begin{equation}
 \xi_t(I,J):=(I-t,\,N-I-J+t,\,J-t,\,t),
 \qquad
 t_-(I,J)\leq t\leq t_+(I,J),
 \label{eq:excavation-column}
\end{equation}
where
\[
 t_-(I,J):=\max\{0,I+J-N\},
 \qquad
 t_+(I,J):=\min\{I,J\}.
\]
Thus a square site represents an entire finite octahedral column, not
a single tetrahedral vertex.  Substitution of \(t=t_+\) gives exactly
the two incoming coordinate formulas displayed above, on
\(123\cup_{13}013\); substitution of \(t=t_-\) gives the two output
formulas, on \(012\cup_{02}023\).  Thus the sites displayed in
panel~\textup{(b)} occupy the top endpoints \(\xi_{t_+}(I,J)\), whereas
the sites displayed in panel~\textup{(c)} occupy the bottom endpoints
\(\xi_{t_-}(I,J)\).

Write \(\delta_t(I,J):=\delta_p(\xi_t(I,J))\) when the corresponding
vertex lies on the current reverse membrane.  Its top values are the
binary field in
panel~\textup{(b)}, and it is filled from \(t_+\) down to \(t_-\).  At
every interior exchange impose the tropical octahedron recurrence
\begin{equation}
 \begin{split}
 \delta_{t-1}(I,J)=\max\bigl\{&
 \delta_{t-1}(I,J-1)+\delta_t(I,J+1),\\
 &\delta_{t-1}(I-1,J)+\delta_t(I+1,J)\bigr\}
 -\delta_t(I,J).
 \end{split}
 \label{eq:excavation-column-recurrence}
\end{equation}
For a fixed new layer \(t-1\), the sites are exposed in increasing
\(I+J\), so the two terms already carrying the subscript \(t-1\) are
known when the displayed value is computed; the remaining two side
values and the old opposite value come from layer \(t\).  Columns with
\(t_+=t_-\) lie on the boundary of the square and remain unchanged.
Consequently, panel~\textup{(c)} is not obtained by moving or recoloring
the markers of panel~\textup{(b)}: the max--minus--old operation acts on
the entire reverse field.  Equation~\eqref{eq:excavation-column-recurrence}
is legitimate because the terminal seed containing \(X_p\) remains
fixed throughout the reverse circuit.  The two Boolean formulas in the
proof below solve it pointwise on \(0134\) and \(0123\).  The formulas
below specialize those fields to \(N=10\).

For the running example \(N=10\), panel~\textup{(b)} is therefore the
indicator of
\[
 \begin{split}
 &\{I<J:\ I\geq3,\ J\leq7,\ J-I\leq2\}
 \ \cup\ \{I=J:\ 3\leq I\leq7\}\\
 &\hspace{38mm}\cup\ \{I>J:\ J\geq3\}.
 \end{split}
\]
In particular, the top of a column does \emph{not} always have value
\(1\): it starts with \(1\) exactly at a neutral marker in
panel~\textup{(b)}, and with \(0\) at every unmarked square site.
After the coupled sweep, panel~\textup{(c)} is the indicator of
\[
 \begin{array}{ll}
  012:\ I+J\leq10,
   &I\geq1,\quad J\geq4,\quad I+J\geq9,\\
  023:\ I+J\geq10,
   &J\leq9,\quad\bigl(I\leq6\ \text{or}\ I+J\geq13\bigr).
 \end{array}
\]
The amber and blue conditions on \(023\) are the two alternatives in
the last parenthesis.  Their overlap still has reverse-field value one,
which is why it is marked by a single purple diamond.
For example, three complete columns of the running example have the
value strings
\[
 \begin{aligned}
  (I,J)=(9,9):&\quad
      (\delta_9,\delta_8)=(0,1),\\
  (I,J)=(7,5):&\quad
      (\delta_5,\delta_4,\delta_3,\delta_2)
       =(1,1,1,0),\\
  (I,J)=(6,8):&\quad
      (\delta_6,\delta_5,\delta_4)=(0,1,1).
 \end{aligned}
\]
They exhibit, respectively, creation of a new value, removal by the
subtracted old value, and two arrivals reduced by that subtraction to
the binary value \(1\).

For an interior output point \(q=(q_0,q_1,q_2,q_3)\), and an input
square site \((I,J)\), put
\[
 \widehat I=I-q_3,\qquad \widehat J=N-J-q_1,\qquad
 n_q=N-q_1-q_3=q_0+q_2,\qquad s_q=q_0.
\]
Its closed input dependency region is the NST excavation hexagon
\[
 \mathsf H_q
 =\{(I,J):0\leq\widehat I,\widehat J\leq n_q,\quad
              s_q\leq\widehat I+\widehat J\leq n_q+s_q\}.
\]
This hexagon is a dependency region, not a region whose marked sites are
simply unioned.  Its octahedra are processed in excavation order, and at
every local step
\[
 \delta_{\rm new}
 =\max\{\delta_{\rm N}+\delta_{\rm S},
         \delta_{\rm E}+\delta_{\rm W}\}
   -\delta_{\rm O}.
\]
The three examples in the figure separate the possible effects.  At
\(q_C=(5,0,3,2)_{023}\), the four side values are \(0,1,1,0\) and the
old value is \(1\), so the winning pre-division value \(1\) is removed:
\[
                         \max\{0+1,1+0\}-1=0.
\]
At the blue-only point \(q_B=(1,0,1,8)_{023}\), the same two pair sums
occur but the old value is \(0\), and hence the output is \(1\).  At the
overlap point \(q_{\cap}=(2,0,4,4)_{023}\), all four side values and the
old value are \(1\), whence
\[
                         \max\{1+1,1+1\}-1=1,
\]
not \(2\).  These are the reverse-field values present when the indicated
site is exposed; a later stencil may use outputs computed in earlier
excavation layers.  Applying this operation at every output site fills
the reverse field.  Its bottom mask is the mask of
\(\delta_p^{(d)}\) on \(012\cup_{02}023\), while \(034\) is
unchanged.  This is the precise content of the optical metaphor:
subtracting the old opposite value removes duplicated tropical
contributions.  It is neither pointwise motion of the marked sites nor
cancellation between Laurent monomials.

The elementary building block is a single tetrahedral sweep.  Its
closed dependency region is the excavation hexagon of
Narayanan--Sheffield--Tao
\cite[Definition~3, Theorem~4, and Appendix~A]{NSTGUE}, and Speyer's
matching formula expresses the corresponding recurrence variable as a
positive sum over the tilings of that hexagon
\cite[Main Theorem]{Speyer}.  To pass from this closed dependency
region to the exact denominator support, we shall use the following
elementary completion observation.

\begin{lemma}[Alternating-star completion]
\label{lem:excavation-star-completion}
Let \(0<s<n\), and let \(\mathsf H_{(s,s)}^n\) be the NST excavation
hexagon centred at the diagonal point \((s,s)\), with its two halves
glued along the equator.  At every lattice vertex \(q\) in the
lattice-strict interior of \(\mathsf H_{(s,s)}^n\), there is a
lozenge--border triangle tiling whose dual perfect matching uses a
maximal alternating set of edges around the face corresponding to
\(q\): three edges when \(q\) is off the equator, and two edges when
\(q\) lies on the equator.
\end{lemma}

\begin{figure}[htbp]
\centering
\includegraphics[width=0.50\textwidth]{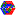}
\caption{An NST excavation-region tiling of \(\mathsf H_{(5,5)}^{10}\) for \(N=10\), in the coloring
convention of \cite[Figure~14, p.~1144]{NSTGUE}.  The horizontal
equator glues the two half-regions.  Reading the positions of the green
lozenges row by row in the two halves gives the pair of coupled
Gelfand--Tsetlin patterns used below; the white singleton triangles
along the equator are the border triangles.}
\label{fig:nst-excavation-n10}
\end{figure}
\newpage

\begin{proof}
Record, row by row in each half of \(\mathsf H_{(s,s)}^n\), the positions of
the green lozenges.  Figure~\ref{fig:nst-excavation-n10} displays the
resulting two-half geometry for \(N=10\); the corresponding published
NST picture is \cite[Figure~14, p.~1144]{NSTGUE}.  These positions form two integral
Gelfand--Tsetlin patterns; this is the usual particle encoding of
lozenge tilings by interlacing arrays, as reviewed in
\cite[Lecture~20]{GorinLozengeTilings}.  Measure the positions in the second half
from the reflected endpoint of the equator and replace particles by
holes.  With this convention both arrays obey the usual interlacing
inequalities, and the rule that exactly one of the two triangles facing
each equatorial edge is a border triangle says precisely that their
equatorial rows agree.  Conversely, such a coupled pair determines the
tiling: between two consecutive green rows, the remaining elementary
triangles occur in alternating intervals and are therefore paired
uniquely in their red and blue directions; the common row determines
the complementary border triangles.

There is a useful order-ideal formulation of this standard particle
encoding.  Replace every bounded integral GT entry \(z\) by its unit
thresholds \(\mathbf 1_{\{z\geq t\}}\).  The interlacing inequalities
become implications among these thresholds, and equality of the two
equatorial rows identifies the corresponding thresholds.  Hence the
coupled GT patterns are precisely the order ideals of a finite poset.
The two extremal, all-left and all-right patterns are its empty and full
ideals, after thresholds forced by the boundary have been omitted.

Here is the coordinate check that no interior face is missed.  In the
square coordinates \((I,J)\) of \cite[Definition~3]{NSTGUE},
\[
 \mathsf H_{(s,s)}^n
 =\{0\leq I,J\leq n:\ s\leq I+J\leq n+s\},
 \qquad \text{equator } I+J=n.
\]
The non-forced threshold elements are the unit cubes between the
all-left and all-right stepped surfaces.  In each row
\(I=m\), \(1\leq m\leq n-1\), their cube-diagonal projections occupy,
possibly with multiplicity, exactly the consecutive positions
\[
 \max\{1,s+1-m\}\leq J\leq
 \min\{n-1,n+s-1-m\}.
\]
This is the usual row-by-row GT construction: the two endpoints are
the lower and upper boundary values of the corresponding particle,
and each integral position strictly between them is one of its unit
thresholds.  As \(m\) varies, the displayed inequalities say precisely
\[
 1\leq I,J\leq n-1,
 \qquad s+1\leq I+J\leq n+s-1,
\]
which is the lattice-strict interior of the hexagon.  On the row
\(I+J=n\), the identified thresholds of the two patterns give the
same conclusion for the equatorial rhombi.  Projection need not be
one-to-one, but every strict-interior vertex has at least one threshold
above it.

Adding an off-equator threshold is the elementary hexagon flip at its
projected vertex; adding an identified equatorial threshold is the
two-edge flip at the corresponding rhombus.  Fix a threshold above
\(q\), choose a linear extension of the poset, and follow the associated
saturated chain of order ideals from the all-left to the all-right
pattern.  When this threshold is inserted, the tilings immediately
before and after that step display the two maximal alternating
configurations at \(q\).  The same row calculation includes the
width-one cases: if the displayed interval is empty, that row simply
contains no strict-interior site.
\end{proof}

\begin{lemma}[Single-sweep compatible-extension formula]
\label{lem:single-sweep-refraction}
Let \(i<j<k<\ell\), and sweep the tetrahedron \(ijk\ell\) from the
input panels \(ijk\cup ik\ell\).  Let
\(r\in\Delta_{ijk\ell}(N)\) satisfy \(r_j,r_\ell\geq1\).  Relative to
this input membrane, the backward denominator field of \(r\) is binary
and has support
\begin{align}
 \mathcal C_r\cap\Delta_{ijk}
 &=\{q:q_i\geq r_i+1,\ q_j\leq r_j-1,\ q_k\geq r_k+1\},
 \notag\\
 \mathcal C_r\cap\Delta_{ik\ell}
 &=\{q:q_i\geq r_i+1,\ q_k\geq r_k+1,\ q_\ell\leq r_\ell-1\}.
 \label{eq:single-sweep-refraction}
\end{align}
Both pieces restrict on the seam \(ik\) to
\[
                  q_i\geq r_i+1,\qquad q_k\geq r_k+1.
\]
In other words, the two panel formulas have the same tangential
thresholds and the same seam trace; only the side-dependent transverse
cap changes from \(q_j\) to \(q_\ell\).
\end{lemma}

\begin{proof}
Translate the backward light cone by putting
\[
 x=q_i-r_i,\qquad y=q_k-r_k,
 \qquad n'=N-r_i-r_k=r_j+r_\ell.
\]
In the tetrahedral coordinates of Appendix~A of~\cite{NSTGUE}, this is
the excavation of the smaller tetrahedron of size \(n'\) from the
target
\[
                         [0,0,r_j,r_\ell].
\]
The associated point in the developed square is
\(v=(r_\ell,r_\ell)\).  Since \(r_j,r_\ell\geq1\), it is an interior
point, including when either transverse width is one.  Its excavation
hexagon has vertices
\[
 (0,n'),\ (0,r_\ell),\ (r_\ell,0),\ (n',0),\
 (n',r_\ell),\ (r_\ell,n'),
\]
and its equator is the common copy of the seam \(ik\).

Speyer's matching formula, in the lozenge--border triangle form of
\cite[Appendix~A]{NSTGUE}, is a positive Laurent sum over tilings of
this hexagon.  If \(o_q(M)\) denotes the number of matching edges
incident to the face indexed by an input site \(q\), the exponent of
that input variable in the term indexed by \(M\) is
\begin{align*}
 c_q(M)&=2-o_q(M)
   &&\text{for a non-equatorial interior hexagon},\\
 c_q(M)&=1-o_q(M)
   &&\text{for an equatorial interior rhombus}.
\end{align*}
For a non-equatorial boundary face it is \(1-o_q(M)\), and for either
exterior equatorial rhombus it is zero.  A matching uses at most three,
two, and one incident edges in these respective cases.  Consequently
no exponent is below \(-1\), and a boundary variable never has a
negative exponent.

By Lemma~\ref{lem:excavation-star-completion}, at every
lattice-strict interior site some tiling uses three alternating
incident edges off the equator, or two on it.  Its exponent is therefore
\(-1\).  Positivity of Speyer's sum prevents cancellation, so the
denominator coordinate is exactly \(1\) at every such site and \(0\)
elsewhere.  In the original coordinates, the lattice-strict interior
on the two input panels is precisely
\[
\begin{array}{ll}
q_i\geq r_i+1,\ q_k\geq r_k+1,\ q_j\leq r_j-1,
   &q\in\Delta_{ijk},\\[2pt]
q_i\geq r_i+1,\ q_k\geq r_k+1,\ q_\ell\leq r_\ell-1,
   &q\in\Delta_{ik\ell}.
\end{array}
\]
This proves \eqref{eq:single-sweep-refraction}.  On the seam
\(q_j=q_\ell=0\), the transverse inequalities are automatic on both
sides, leaving the common trace
\(q_i\geq r_i+1\), \(q_k\geq r_k+1\).
\end{proof}

\subsection{The composite support}

We now follow Figure~\ref{fig:step-I-backward-light-cone} in its
displayed order.  The first inverse sweep is the NST excavation
hexagon.  The next two inverse sweeps are encoded by one Boolean
formula on \(0134\) and one on \(0123\), respectively.

\begin{theorem}[Explicit positive support]
\label{thm:light-cone}
Let
\[
 p=(0,a,b,c,d)\in\Delta_{1234}(N),
 \qquad a,b,c,d\geq1,\qquad a+b+c+d=N.
\]
Then
\begin{equation}
                         D_p(q)>0
                 \quad\Longleftrightarrow\quad q\in S_p,
 \label{eq:target-positive-support}
\end{equation}
where \(S_p\) is the three-panel set in \eqref{eq:Sp-formula}.  The
reverse field \(\delta_p\) used to prove this assertion is binary on
each preceding membrane displayed after panel~\textup{(a)} of
Figure~\ref{fig:step-I-backward-light-cone}, and its two panel formulas
agree on every seam.  The assertion here concerns
the positive support of \(D_p\); its numerical binarity is proved after
the matching variable is identified with \(X_p\) in Step~V.
\end{theorem}

\begin{proof}
All scalar recurrences below are evaluated relative to the one fixed
terminal seed containing \(X_p\), as in
\eqref{eq:reverse-d-field}.  Hence time reversal gives the genuine
denominator coordinate \(\delta_p(r)=d(X_p,X_r)\) at every exposed
vertex.

\paragraph{Figure~\ref{fig:step-I-backward-light-cone}\textup{(a)} to
\textup{(b)}: undo \(1234\).}
The time-reversed form of
Lemma~\ref{lem:single-sweep-refraction} is the same NST excavation
hexagon.  It gives the binary field
\begin{equation}
\begin{array}{ll}
123:&\1[q_1\geq a+1,\ q_2\leq b-1,\ q_3\geq c+1],\\
134:&\1[q_1\geq a+1,\ q_3\geq c+1,\ q_4\leq d-1],\\
014:&0.
\end{array}
\label{eq:reverse-stage-b}
\end{equation}
The two nonzero lines have the same trace on \(13\).  Equivalently,
one may apply the proof of Lemma~\ref{lem:single-sweep-refraction} to
the reversed stepped surfaces: the recurrence and the strict interior
of the excavation hexagon are unchanged.

\paragraph{Figure~\ref{fig:step-I-backward-light-cone}\textup{(b)} to
\textup{(c)}: one Boolean formula on \(0134\).}
For \(r=(r_0,r_1,0,r_3,r_4)\in\Delta_{0134}(N)\), put
\[
 A=\1[r_0\geq1],\qquad H=\1[r_1\geq a+1],\qquad
 E=\1[r_3\geq c+1],\qquad L=\1[r_4\leq d-1].
\]
The whole inverse sweep is described by the single Boolean function
\begin{equation}
                         \mathsf B_p(r):=EL(A\vee H).
\label{eq:reverse-0134-boolean-field}
\end{equation}
On the input faces \(134\) and \(014\), it restricts to \(ELH\) and
zero, exactly as in \eqref{eq:reverse-stage-b}; on the output faces
\(013\) and \(034\), it restricts to \(E(A\vee H)\) and \(ELA\).

We verify \eqref{eq:reverse-0134-boolean-field} pointwise.  Let
\(r=(x,y,z,t)\) be the old \(14\)-vertex, in coordinate order
\(0,1,3,4\), and let
\(r'=r+e_0-e_1+e_3-e_4\) be the new \(03\)-vertex.  Write
\[
 E_+=\1[z\geq c],\qquad L_+=\1[t\leq d],\qquad
 H_-=\1[y\geq a+2].
\]
The old value, proposed new value, and the two side-pair sums are
\begin{align*}
 O&=EL(A\vee H),& T&=E_+L_+,\\
 P_1&=EL_+ + E_+L(A\vee H_-),&
 P_2&=E_+L_+(A\vee H)+EL.
\end{align*}
If \(A=1\), then
\(P_2-P_1=(E_+-E)(L_+-L)\geq0\) and \(P_2=O+T\).
The same calculation applies when \(A=0\) and \(y\geq a+2\).  If
\(A=0,y=a+1\), direct substitution again gives
\(P_2=O+T\) and \(P_2\geq P_1\).  Finally, if
\(A=0,y\leq a\), then \(O=0\) and
\(\max(P_1,P_2)=EL_+\).  This can differ from \(T\) only when
\(x=0,y\leq a,z=c,t\leq d\), which would imply
\(x+y+z+t\leq a+c+d=N-b<N\).  Since every tetrahedral vertex has
coordinate sum \(N\) and \(b\geq1\), this case is impossible.  Hence
\(\max(P_1,P_2)-O=T\) at every octahedron.  Induction in the inverse
excavation order (or on \(r_0+r_3\)) now proves
\eqref{eq:reverse-0134-boolean-field} throughout \(0134\).

Thus panel \(123\) remains unchanged and the field in
Figure~\ref{fig:step-I-backward-light-cone}\textup{(c)} is
\begin{equation}
\begin{array}{ll}
123:&\1[q_1\geq a+1,\ q_2\leq b-1,\ q_3\geq c+1],\\
013:&\1[q_3\geq c+1]\,
       \1[q_0\geq1\ \text{or}\ q_1\geq a+1],\\
034:&\1[q_0\geq1,\ q_3\geq c+1,\ q_4\leq d-1].
\end{array}
\label{eq:reverse-stage-c}
\end{equation}
The traces agree on both \(13\) and \(03\).  Substitution of the square
coordinates gives exactly \eqref{eq:incoming-square-field}.

\paragraph{Figure~\ref{fig:step-I-backward-light-cone}\textup{(c)} to
\textup{(d)}: one Boolean formula on \(0123\).}
For \(r=(r_0,r_1,r_2,r_3,0)\in\Delta_{0123}(N)\), put
\[
\begin{gathered}
 A=\1[r_0\geq1],\qquad
 B=\1[r_1\leq a-1],\qquad H=\1[r_1\geq a+1],\\
 C=\1[r_2\geq b+1],\qquad L=\1[r_2\leq b-1],\qquad
 E=\1[r_3\geq c+1].
\end{gathered}
\]
The full coupled inverse sweep is the single Boolean function
\begin{equation}
 \mathsf C_p(r):=AB(C\vee E)\ \vee\ EL(A\vee H).
\label{eq:reverse-0123-boolean-field}
\end{equation}
On the two input faces \(123\) and \(013\), this becomes \(ELH\) and
\(E(A\vee H)\), the first two lines of
\eqref{eq:reverse-stage-c}.  On the output faces \(012\) and \(023\),
it becomes
\begin{equation}
              ABC,\qquad A(C\vee E),
\label{eq:reverse-0123-output-faces}
\end{equation}
respectively.

It remains to verify \eqref{eq:reverse-0123-boolean-field} locally.
Let \(r=(x,y,z,t)\) be the old \(13\)-vertex and
\(r'=r+e_0-e_1+e_2-e_3\) the new \(02\)-vertex.  Put
\[
 B_+=\1[y\leq a],\qquad C_+=\1[z\geq b],\qquad
 H_-=\1[y\geq a+2],\qquad L_-=\1[z\leq b-2].
\]
There are three possible positions relative to the wall \(t=c+1\).

If \(t\leq c\), the old and new \(E\)-states vanish.  The old value,
new value, and pair sums are
\[
 O=ABC,\quad T=B_+C_+,\quad
 P_1=BC+AB_+C_+,\quad P_2=ABC_++B_+C.
\]
A direct Boolean check gives \(\max(P_1,P_2)=O+T\), except nominally
when \(A=0,B_+=1,C=0,C_+=1\).  That assignment says
\(x=0,y\leq a,z=b,t\leq c\), and hence has coordinate sum at most
\(a+b+c<N\), so it is inadmissible.

If \(t\geq c+2\), both \(E\)-states equal one and
\[
\begin{aligned}
 O&=A(B\vee L)\vee HL,&T&=B_+\vee L_-,\\
 P_1&=(B\vee L)+[A(B_+\vee L_-)\vee H_-L_-],\\
 P_2&=[A(B\vee L_-)\vee HL_-]+(B_+\vee L).
\end{aligned}
\]
For \(y\leq a-1\), both pair sums are \(1+A=O+T\).
For \(y=a\), the three possibilities
\(z\leq b-2,z=b-1,z\geq b\) give
\(\max(P_1,P_2)=1+AL=O+T\).  For \(y=a+1\), one has
\(P_2=L+L_-=O+T\) and \(P_2\geq P_1\); for \(y\geq a+2\), both
pair sums equal \(L+L_-=O+T\).

Finally let \(t=c+1\), so the old \(E\)-state is one and the new one
is zero.  Then
\[
\begin{aligned}
 O&=A(B\vee L)\vee HL,&T&=B_+C_+,\\
 P_1&=BC+[A(B_+\vee L_-)\vee H_-L_-],\\
 P_2&=ABC_+ +(B_+\vee L).
\end{aligned}
\]
For \(y\leq a-1\), one has
\(P_1=C+A\) and \(P_2=1+AC_+\); for \(y=a\), one has
\(P_1=A\) and \(P_2=1\).  In either regime the only nominal exception
is \(A=0,C_+=0\).  For \(y\geq a+1\), one has
\(P_2=L=O+T\) and \(P_2\geq P_1\).  The nominal exception says
\(x=0,y\leq a,z\leq b-1,t=c+1\).  Its coordinate sum is at most
\(a+b+c<N\), again impossible because \(d\geq1\).  This proves the
octahedron identity everywhere.  The inverse sweep is triangular in
excavation order---equivalently, one inducts on \(r_0+r_2\)---so its
two input faces determine it uniquely.  Hence the reverse field is
\eqref{eq:reverse-0123-boolean-field} throughout \(0123\).

Combining \eqref{eq:reverse-0123-output-faces} with the unchanged
\(034\)-line of \eqref{eq:reverse-stage-c}, the bottom reverse field is
\[
\begin{array}{ll}
012:&\1[q_0\geq1,\ q_1\leq a-1,\ q_2\geq b+1],\\
023:&\1[q_0\geq1]\,
       \1[q_2\geq b+1\ \text{or}\ q_3\geq c+1],\\
034:&\1[q_0\geq1,\ q_3\geq c+1,\ q_4\leq d-1].
\end{array}
\]
This is the indicator of \(S_p\), with common traces
\(q_0\geq1,q_2\geq b+1\) on \(02\) and
\(q_0\geq1,q_3\geq c+1\) on \(03\).  Thus
\(d(X_p,X_q)>0\) exactly for \(q\in S_p\).  Support reciprocity now gives
\(D_p(q)=d(X_q,X_p)>0\) exactly on the same set.  Indeed, the computed
reverse coordinate is always \(0\) or \(1\), never \(-1\); by
\cite[Theorem~6.3(ii)]{CaoLiDenominators}, this also verifies that no
initial variable \(X_q\) coincides with \(X_p\), so the distinctness
hypothesis in support reciprocity applies at every initial site.  This proves
\eqref{eq:target-positive-support}.
\end{proof}

\subsection{Geometry of the final support}

A geometric restatement will be useful.  It separates
the contributions extending compatibly over
\(012\cup_{02}023\) and \(023\cup_{03}034\). 

\begin{proposition}[Two seam-compatible light-cone pieces]
\label{prop:refracted-lobes}
Define \(S_p^{02},S_p^{03}\subseteq\Sigma_0\) by
\begin{align*}
S_p^{02}\cap012
 &=\{q:q_0\ge1,\ q_1\le a-1,\ q_2\ge b+1\},\\
S_p^{02}\cap023
 &=\{q:q_0\ge1,\ q_2\ge b+1\},
 &S_p^{02}\cap034&=\varnothing,\\[2mm]
S_p^{03}\cap012&=\varnothing,
 &S_p^{03}\cap023
 &=\{q:q_0\ge1,\ q_3\ge c+1\},\\
S_p^{03}\cap034
 &=\{q:q_0\ge1,\ q_3\ge c+1,\ q_4\le d-1\}.
\end{align*}
Then
\[
                     S_p=S_p^{02}\cup S_p^{03}.
\]
The first piece extends compatibly over \(012\cup_{02}023\), the second
extends compatibly over \(023\cup_{03}034\), and on the middle panel they
overlap in
\[
 S_p^{02}\cap S_p^{03}
 =\{q\in023:q_0\ge1,\ q_2\ge b+1,\ q_3\ge c+1\}.
\]
\end{proposition}

\begin{proof}
This is a regrouping of the three lines of
\eqref{eq:Sp-formula}.  On each seam the middle-panel trace is exactly
the trace of the adjacent outer-panel piece, as in
Lemma~\ref{lem:single-sweep-refraction}.
\end{proof}

Panels~\textup{(d)}--\textup{(e)} of
Figure~\ref{fig:step-I-backward-light-cone} visualize precisely this
regrouping and its Boolean overlap.

\begin{corollary}[Topology of the support]
\label{cor:support-disk}
The selected faces indexed by \(S_p\) form a connected region of the
three-panel fan, and their complement has no bounded component.
\end{corollary}

\begin{proof}
By Proposition~\ref{prop:refracted-lobes}, the support is the union of
two connected light-cone pieces.  On \(023\), the union of the two wedges
\[
 q_2\ge b+1,\qquad q_3\ge c+1
\]
contains the entire row \(q_0=1\).  Indeed, if neither inequality held
there, then
\[
 N-1=q_2+q_3\le b+c=N-a-d\le N-2,
\]
a contradiction.  Each outer piece meets its middle-panel wedge along
the corresponding seam, so \(S_p\) is connected.

For the absence of holes, the complements are
\[
\begin{array}{ll}
012:&q_0=0\ \text{or}\ q_1\ge a\ \text{or}\ q_2\le b,\\
023:&q_0=0\ \text{or}\ (q_2\le b\ \text{and}\ q_3\le c),\\
034:&q_0=0\ \text{or}\ q_3\le c\ \text{or}\ q_4\ge d.
\end{array}
\]
Every point in the complement can be moved to the outer boundary while
preserving one of these alternatives.  On \(012\), move toward the
\(u_1\)-corner when \(q_1\ge a\), or toward the edge \(01\) when
\(q_2\le b\).  The reflected moves work on \(034\).  On \(023\), if
\(q_2\le b\) and \(q_3\le c\), replace units of \(e_2,e_3\) by
\(e_0\) until the \(u_0\)-corner is reached.  Hence every complementary
component meets the outer face.
\end{proof}

\paragraph{Output of Step~I.}
The target has positive denominator support \(S_p\), explicitly given
by \eqref{eq:Sp-formula}.  It is the union of two seam-compatible
light-cone pieces, is connected, and has no hole.  For the running
example, the backward support propagation and final support are shown in
Figure~\ref{fig:step-I-backward-light-cone}; the associated
lozenge--triangle tiling appears in
Figure~\ref{fig:intro-three-panel-tiling}.  These are exactly the properties needed in
Step~II: retaining the boundaries of the dual faces indexed by \(S_p\)
will produce one connected planar graph whose bounded faces are
precisely the selected faces.

\section{Step II: The canonical fixed graph}
\label{sec:step2-fixed-graph}

\paragraph{Goal of Step~II.}
Step~I tells us exactly which primal vertices are selected.  We now
turn that selected region into a planar bipartite graph whose bounded
faces are precisely the selected dual faces, and then cut open the two
seams without changing the matching lattice.  No matching formula is
used yet; this step is purely the construction of the graph.

The initial dual graph \(\Gamma _0\), the unreduced graph
\(\overline G_p\), and the cut-open graph \(G_p\) were defined before
Theorem~\ref{thm:partI-main}.  We now prove the graph assertions made
there from the support properties established in Step~I.

\begin{theorem}[Fixed support graph]
\label{thm:fixed-graph}
For every strictly interior \(p=(0,a,b,c,d)\):
\begin{enumerate}
 \item \(\overline G_p\) is a connected plane bipartite graph whose
 bounded faces are exactly \(F_v\), \(v\in S_p\).
 \item With the Lam--Postnikov orientation, its dual quiver is the
 induced subquiver
 \[
              Q^{\rm LP}_{\overline G_p}=Q_0^{\rm LP}[S_p].
 \]
 With the white-on-the-right orientation it is the opposite quiver
 \(Q_0^{\rm R}[S_p]\).
 \item Cutting open the seams gives a canonical weight-preserving
 bijection
 \[
          \D(\overline G_p)\longleftrightarrow\D(G_p)
 \]
 between perfect matchings.
\end{enumerate}
\end{theorem}

\begin{proof}
By Theorem~\ref{thm:light-cone} and
Corollary~\ref{cor:support-disk}, the selected faces form a connected
region with no hole.

Every retained edge belongs to the boundary of a selected face.
Face-connectivity therefore makes \(\overline G_p\) connected, and
it is bipartite because it is a subgraph of \(\Gamma _0\).  The
boundary of each selected face is retained in full.  Conversely, an
extra bounded face would be a bounded complementary component of the
selected region, contrary to Corollary~\ref{cor:support-disk}.  This
proves part~1.

An arrow of \(Q^{\rm LP}_{\overline G_p}\) comes exactly from an edge separating
two selected bounded faces \(F_v,F_{v'}\), with the orientation
inherited from \(Q_0^{\rm LP}\).  An edge between a selected and an unselected
face now borders the infinite face and contributes no mutable arrow.
After cancellation of oriented \(2\)-cycles, this proves the first
assertion in part~2; globally reversing its arrows gives the second.

Consider one identified seam vertex \(x\), say black.  A perfect
matching of \(\overline G_p\) uses exactly one edge incident to \(x\);
suppose it lies in sector \(i\).  In the cut-open graph use that
external edge at \(x_i\).  Removing \(x_i\) from
\[
 x_1-z_1-x_2-\cdots-z_{r-1}-x_r
\]
leaves two even paths, each with a unique perfect matching.
Conversely, the zigzag has \(r\) black sector vertices and \(r-1\)
white auxiliary vertices, so exactly one sector vertex must be
matched outside the zigzag.  Contracting the zigzag recovers a unique
matching of \(\overline G_p\).  These local bijections commute, and
all inserted edges have weight one.  This proves part~3.

\end{proof}

\begin{lemma}[Seam-cut invariance]
\label{lem:seam-cut-invariance}
Under the matching bijection of
Theorem~\ref{thm:fixed-graph}, part~3:
\begin{enumerate}
 \item the bounded faces of \(G_p\) are canonically indexed by the
 same set \(S_p\) as those of \(\overline G_p\);
 \item after cancellation of oriented \(2\)-cycles, the
 Lam--Postnikov oriented dual quiver of \(G_p\) is
 \(Q_0^{\rm LP}[S_p]\), and its opposite is \(Q_0^{\rm R}[S_p]\);
 \item face rotations correspond under the matching bijection; and
 \item the face-height vector is unchanged.
\end{enumerate}
\end{lemma}

\begin{proof}
It is enough to check one insertion between two consecutive sectors
at an identified seam vertex.  Splitting the vertex and inserting the
two-edge part of the zigzag does not enclose a new region: it only
replaces a boundary corner of each incident old face by a three-edge
detour.  Thus the old bounded faces persist and no new bounded face is
created.

Locally, the detour contributes two additional boundary edges to the
relevant face.  Its unique alternating completion uses exactly one of
them.  Hence, for every corresponding matching,
\[
 \frac{|\partial F|}{2}\longmapsto
 \frac{|\partial F|}{2}+1,
 \qquad
 |M\cap\partial F|\longmapsto |M\cap\partial F|+1.
\]
The quantity
\[
 h_F^G(M):=\frac{|\partial F|}{2}
            -|M\cap\partial F|-1
\]
is therefore unchanged.  Performing a face rotation before the cut
and then taking the alternating zigzag completion gives the same
matching as first cutting and then rotating the corresponding face.
Thus face rotations and face-height differences are preserved.

Finally, the old dual arrows away from the seam are unchanged.  At
the inserted detour, its two new crossings contribute arrows with
opposite orientations between the same pair of face vertices; they
form an oriented \(2\)-cycle and cancel.  The reduced dual quiver is
therefore \(Q_0^{\rm LP}[S_p]\) in the Lam--Postnikov orientation;
reversing all arrows gives \(Q_0^{\rm R}[S_p]\).
\end{proof}

\section{Step III: Extremal heights and well-matchedness}
\label{sec:step3-well-matched}

\paragraph{Goal of Step~III.}
We prove that every ordinary edge and every inserted seam edge of the
fixed graph is contained in a perfect matching.  The canonical objects
in the proof are the lower and upper directed envelopes of the forced
boundary height.  They give the bottom and top height functions.  A compatible partial datum obtained by raising
one interior vertex by three then gives, directly, a matching containing
any prescribed ordinary lozenge.  After well-matchedness has been
proved do we invoke flip-connectivity and identify the two explicit
polygonal tilings with the bottom and top matchings.
Let $\mathcal A_p$ be the triangular cell complex obtained from the
three panel pieces incident to $S_p$, with the two copies of every seam
vertex identified but with the two incident panel triangles kept as
distinct cells.  By Corollary~\ref{cor:support-disk},
$\mathcal A_p$ is a finite topological disk.  Put
\[
                 B_p:=\partial\mathcal A_p.
\]

\begin{definition}[Legal lozenge--triangle tiling]
\label{def:legal-lozenge-triangle-tiling}
A \emph{legal lozenge--triangle tiling} of \(\mathcal A_p\) is a
partition of its elementary triangles into tiles satisfying the
following two conditions:
\begin{enumerate}
\item every tile is either an ordinary lozenge formed from two adjacent
      triangles in one panel, or a single triangle facing a cut seam;
\item at each retained seam segment, exactly one of its two facing
      panel triangles is used as a singleton seam triangle, while the
      other is paired into an ordinary lozenge.
\end{enumerate}
Write \(\mathcal T_p\) for the set of legal tilings and
\(\mathcal S_p^{\rm seam}\) for the set of retained seam segments.
The complementary-singleton condition is precisely the cut-open seam
rule: taking the dual edge across every ordinary lozenge and completing
the corresponding alternating seam zigzags gives a canonical bijection
\[
                         \mathcal T_p\longleftrightarrow\D(G_p).
\]
We therefore use the same symbol for a tiling and its perfect
matching when no confusion can result.
\end{definition}

\subsection{Aligned heights and edge increments}
\label{sec:aligned-heights-intro}

On the panel $P_r=\Delta_{0,r,r+1}(N)$ use the coordinates
\[
 [i,j]_r=(N-j)e_0+(j-i)e_r+ie_{r+1},
 \qquad 0\le i\le j\le N.
\]
Orient the three families of unit lattice edges by declaring the
following directions to be positive, whenever the displayed endpoints
belong to the panel:
\begin{equation}
 [i,j]_r\longrightarrow[i-1,j]_r,
 \qquad
 [i,j]_r\longrightarrow[i,j-1]_r,
 \qquad
 [i,j]_r\longrightarrow[i+1,j+1]_r.
 \label{eq:positive-panel-directions}
\end{equation}
On the reference triangle \(0r(r+1)\), these are respectively the
directions (see Figure~\ref{fig:three-panel-lozenge-height})
\[
 u_{r+1}\longrightarrow u_r,
 \qquad u_r\longrightarrow u_0,
 \qquad u_0\longrightarrow u_{r+1}.
\]
\newcommand{\positivepaneldiagram}{%
\begin{tikzpicture}[
  panel/.style={draw=gray!70,line width=0.8pt},
  positive/.style={
    draw=blue!65!black,line width=1.2pt,
    -{Stealth[length=2.2mm,width=1.7mm]},
    shorten <=3pt,shorten >=3pt
  },
  vertex/.style={font=\small},
  direction/.style={
    font=\small,fill=white,inner sep=1.2pt,text=blue!55!black
  },
  panelname/.style={font=\small}
]
\begin{scope}
  \coordinate (O) at (0,0);
  \coordinate (R) at (0,2.6);
  \coordinate (S) at (2.25,1.3);
  \draw[panel] (O)--(R)--(S)--cycle;
  \draw[positive] (S)--node[direction,above left] {\(d_1\)} (R);
  \draw[positive] (R)--node[direction,left] {\(d_2\)} (O);
  \draw[positive] (O)--node[direction,below right] {\(d_3\)} (S);
  \node[vertex,below left] at (O) {\(0\)};
  \node[vertex,above left] at (R) {\(1\)};
  \node[vertex,right] at (S) {\(2\)};
  \node[panelname] at (1.125,-0.48) {\(P_1=012\)};
\end{scope}
\begin{scope}[xshift=4.3cm]
  \coordinate (O) at (0,1.3);
  \coordinate (R) at (2.25,2.6);
  \coordinate (S) at (2.25,0);
  \draw[panel] (O)--(R)--(S)--cycle;
  \draw[positive] (S)--node[direction,right] {\(d_1\)} (R);
  \draw[positive] (R)--node[direction,above left] {\(d_2\)} (O);
  \draw[positive] (O)--node[direction,below left] {\(d_3\)} (S);
  \node[vertex,left] at (O) {\(0\)};
  \node[vertex,above right] at (R) {\(2\)};
  \node[vertex,below right] at (S) {\(3\)};
  \node[panelname] at (1.125,-0.48) {\(P_2=023\)};
\end{scope}
\begin{scope}[xshift=8.6cm]
  \coordinate (O) at (0,2.6);
  \coordinate (R) at (2.25,1.3);
  \coordinate (S) at (0,0);
  \draw[panel] (O)--(R)--(S)--cycle;
  \draw[positive] (S)--node[direction,below right] {\(d_1\)} (R);
  \draw[positive] (R)--node[direction,above right] {\(d_2\)} (O);
  \draw[positive] (O)--node[direction,left] {\(d_3\)} (S);
  \node[vertex,above left] at (O) {\(0\)};
  \node[vertex,right] at (R) {\(3\)};
  \node[vertex,below left] at (S) {\(4\)};
  \node[panelname] at (1.125,-0.48) {\(P_3=034\)};
\end{scope}
\end{tikzpicture}%
}

Equivalently, in the equilateral development choose unit vectors
\(\widehat\imath,\widehat\jmath\) so that decreasing the first array
index points along \(\widehat\imath\), while increasing the second
array index points along \(\widehat\jmath\).  The three positive unit
directions are then
\[
 d_1=\widehat\imath,
 \qquad d_2=-\widehat\jmath,
 \qquad d_3=-\widehat\imath+\widehat\jmath;
\]
they are separated by angles of \(2\pi/3\).  The corresponding convex
set of admissible macroscopic gradients is
\begin{equation}
 K=\operatorname{conv}\{-2d_1,-2d_2,-2d_3\}
  =\operatorname{conv}\bigl\{-2\widehat\imath,
       2\widehat\jmath,2(\widehat\imath-\widehat\jmath)\bigr\}.
 \label{eq:panel-slope-triangle}
\end{equation}
Indeed, if \(\theta_k\) is the local proportion of lozenges whose
selected edge is parallel to \(d_k\), then
\[
 \theta_k\ge0,
 \qquad \sum_{k=1}^3\theta_k=1,
 \qquad
 \nabla f=-2\sum_{k=1}^3\theta_kd_k\in K,
 \qquad
 \partial_{d_k}f=1-3\theta_k.
\]
Thus the vertices of \(K\) are the three pure lozenge phases, while
its interior describes mixtures of all three orientations.  The set
\(K\) is used in Part~II for the ordinary sheetwise heights; the
aligned height introduced below differs by a panel-dependent affine
term and consequently has the corresponding translated slope set.

For a legal tiling $M$, let $F^M_r$ be its ordinary lozenge height
on the $r$-th developed panel: along a positive unit edge its
increment is $+1$, except on the selected edge of a lozenge, where it
is $-2$.

At the seam between $P_r$ and $P_{r+1}$, let
$\epsilon_-,\epsilon_+\in\{0,1\}$ record the two possible singleton
triangles.  When the seam is read in increasing level $k$, the two
increments are
\[
 \Delta F_- =1-3\epsilon_-,
 \qquad
 \Delta F_+ =-1+3\epsilon_+.
\]
Hence the complementary-singleton rule
$\epsilon_-+\epsilon_+=1$ is equivalent, after choosing additive
constants, to
\[
 F^M_{r+1}([0,k]_{r+1})=F^M_r([k,k]_r)+k.
\]
Figure~\ref{fig:three-panel-lozenge-height} shows these sheetwise
height functions for the two polygonal matchings.  At a retained
seam level, the two nearby labels belong to the two cut-open copies of
the same seam vertex before affine alignment.

\begin{figure}[htbp]
\centering
\positivepaneldiagram
\par\medskip
\begin{minipage}[t]{0.46\textwidth}
\centering
\includegraphics[width=\linewidth]
  {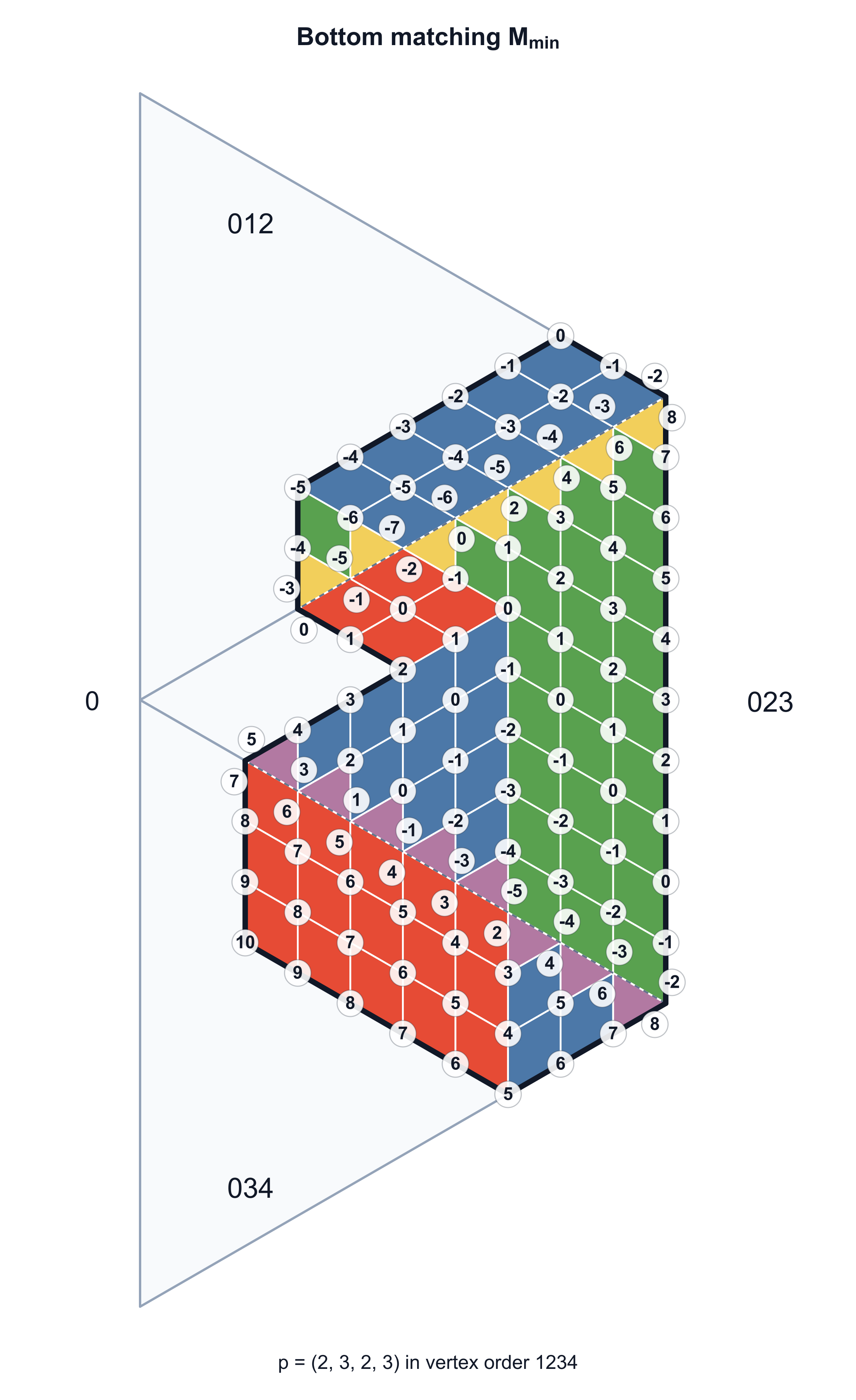}
\end{minipage}\hfill
\begin{minipage}[t]{0.46\textwidth}
\centering
\includegraphics[width=\linewidth]
  {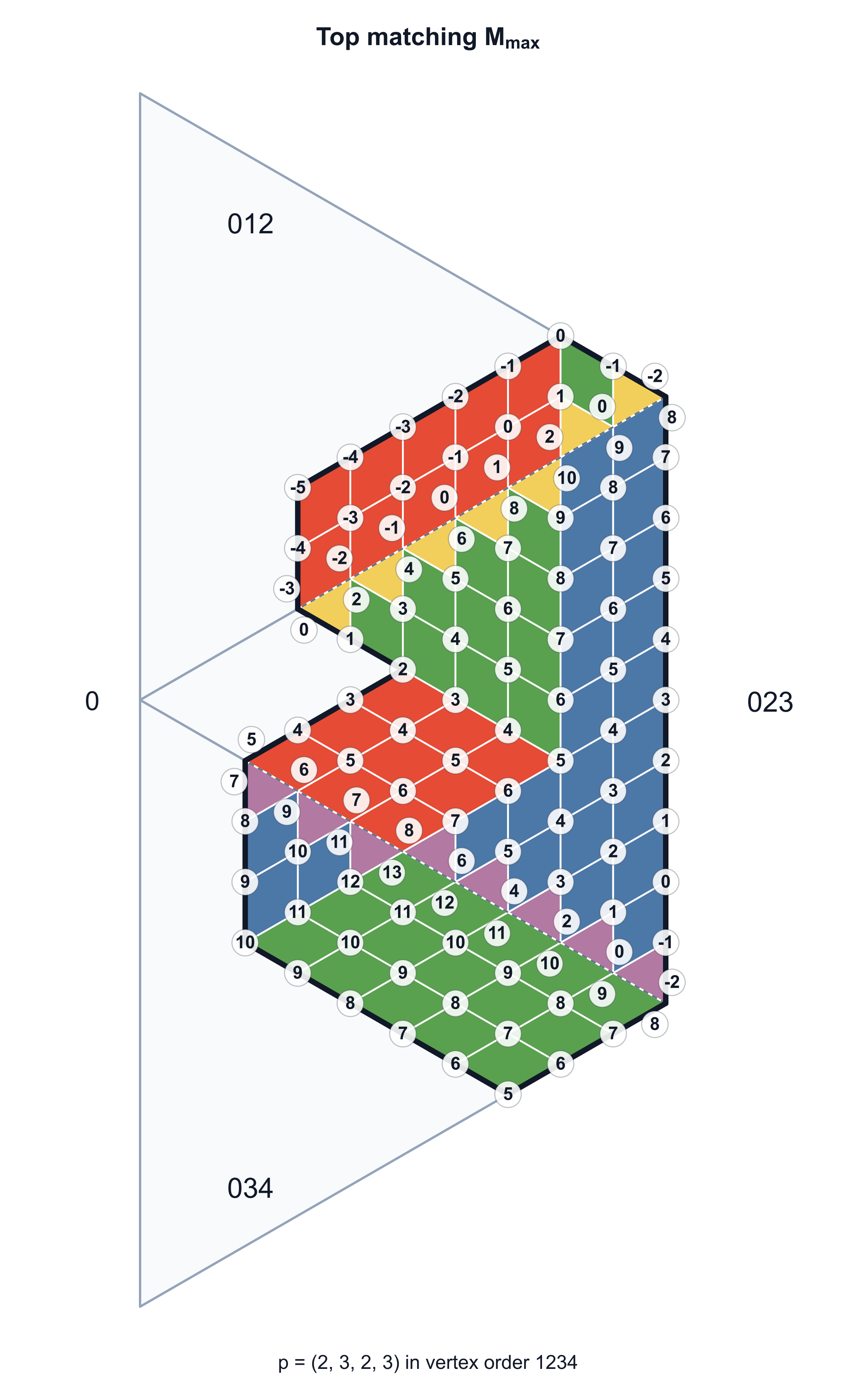}
\end{minipage}
\caption{\emph{Top:} the positive unit-edge directions on the three
developed panel frames, in the same relative orientations as
Figure~\ref{fig:step-v-bottom-matching}.  On \(P_r=0r(r+1)\), the
arrows \(d_1,d_2,d_3\) point respectively from \(u_{r+1}\) to \(u_r\),
from \(u_r\) to \(u_0\), and from \(u_0\) to \(u_{r+1}\); each arrow represents the
parallel orientation of every lattice edge in that family.
\emph{Bottom:} the ordinary lozenge-height functions for the two
polygonal tilings of \(\mathcal A_p\), with
\(p=(2,3,2,3)_{1234}\).  The left panel shows \(M_{\rm poly}\), and
the right panel shows
\(M_{\rm poly}^{+}\); in each picture the three developed
panels are congruent equilateral triangles.  On each developed panel,
a positive edge has increment \(+1\), except when it is the selected
edge of a lozenge, in which case its increment is \(-2\).  At a retained
seam of level \(k\), the two displayed values belong to the two cut-open
copies and satisfy
\(F^M_{r+1}([0,k]_{r+1})=F^M_r([k,k]_r)+k\).  After subtracting the
panel-dependent affine term in~\eqref{eq:aligned-global-height}, the two
pictures give single-valued aligned heights on \(\mathcal A_p\).  They
are identified below with \(h_{\rm bot}\) and \(h_{\rm top}\), after
well-matchedness has been proved.}
\label{fig:three-panel-lozenge-height}
\end{figure}
Consequently
\begin{equation}
 h_M([i,j]_r):=F^M_r([i,j]_r)-(r-1)j
 \label{eq:aligned-global-height}
\end{equation}
is a single-valued integer function on the vertices of
$\mathcal A_p$.

For a positive edge occurrence
\[
 e:x=[i,j]_r\longrightarrow y=[i',j']_r
\]
put
\[
     \omega(e):=1-(r-1)(j'-j).
\]
Then every legal tiling satisfies
\begin{equation}
 h_M(y)-h_M(x)\in\{\omega(e),\omega(e)-3\},
 \label{eq:two-edge-increments}
\end{equation}
and the second value occurs exactly when $e$ is selected by the
tiling.  Thus two aligned heights differ by a multiple of three at
every vertex.  In particular, after a common normalization, a positive
face rotation raises the aligned height at its central lattice vertex
by three and leaves every other vertex unchanged.

\subsection{Upper and lower envelopes}
\label{sec:directed-envelopes}

Let \(\mathcal V_p\) be the vertex set of \(\mathcal A_p\).  A positive
edge path is a composable sequence of positive edge occurrences, with
the two developed occurrences of a retained seam edge kept distinct.
Give such a path \(\gamma\) the weight
\[
                         |\gamma|_\omega
                 :=\sum_{e\in\gamma}\omega(e),
\]
and let
\begin{equation}
 d_p(x,y):=\min_{\gamma:x\rightsquigarrow y}|\gamma|_\omega.
\label{eq:directed-distance}
\end{equation}

On panel \(P_r\), for $r = 1, 2, 3$, the three positive edge directions \(d_1,d_2,d_3\)
have weights
\[
                         1,\qquad r,\qquad 2-r.
\]
Thus the three panels have weights
\[
\begin{array}{c|ccc}
 &d_1&d_2&d_3\\ \hline
P_1&1&1&1\\
P_2&1&2&0\\
P_3&1&3&-1.
\end{array}
\]
So \(d_p\) is not an ordinary nonnegative graph distance: \(P_2\) has
zero-weight edges and \(P_3\) has negative-weight edges.  A priori, one
might worry about zero or negative-weight directed cycles.
Lemma~\ref{lem:positive-cycle-gap} rules those out.

\begin{lemma}[Positive-cycle gap and exact adjacent distances]
\label{lem:positive-cycle-gap}
For a positive path $\gamma$ in $\mathcal A_p$, let $n_3(\gamma)$ denote the
number of its edge occurrences parallel to $d_3$, counted with
multiplicity and with the two developed occurrences of a retained seam
edge distinguished.  The function
\[
   \widetilde\kappa_p([i,j]_r):=-i-rj+2N-a,
   \qquad r=1,2,3,
\]
is a single-valued integer function on $\mathcal V_p$, and every positive path
$\gamma:x\rightsquigarrow y$ satisfies
\[
   |\gamma|_\omega
      =\widetilde\kappa_p(y)-\widetilde\kappa_p(x)
       +3n_3(\gamma).
\]
Consequently, every nonempty closed positive path $\gamma$ satisfies
\[
                         |\gamma|_\omega\in 3\mathbb Z_{>0}.
\]
In particular, there are no negative- or zero-weight closed positive
paths.  The minimum in \eqref{eq:directed-distance} is finite and
attained, and for every
retained positive edge occurrence $e:x\to y$,
\begin{equation}
\label{eq:exact-adjacent-distances}
                  d_p(x,y)=\omega(e),
                  \qquad
                  d_p(y,x)=3-\omega(e).
\end{equation}
This includes exterior-boundary occurrences and either developed
occurrence of a retained seam edge.
\end{lemma}

\begin{proof}
At the seam between $P_r$ and $P_{r+1}$, the vertex at level $k$ has
the two coordinate names $[k,k]_r$ and $[0,k]_{r+1}$.  They give
\[
 \widetilde\kappa_p([k,k]_r)
     =-(r+1)k+2N-a
     =\widetilde\kappa_p([0,k]_{r+1}).
\]
Thus $\widetilde\kappa_p$ has no jump at either seam and descends to a
single-valued function on the vertices of the glued complex.  In
particular, its increments telescope even along a path that changes
panels at a seam vertex; changing panels itself contributes neither an
edge nor a cost.  Figure~\ref{fig:positive-cycle-local}~(a) depicts the
two oppositely oriented developed occurrences of a seam edge.

Let $e:x\to y$ be a positive edge occurrence in $P_r$.  Direct
substitution in the definition of $\widetilde\kappa_p$ gives
\[
\begin{aligned}
 e\parallel d_1&:
   &\widetilde\kappa_p(y)-\widetilde\kappa_p(x)&=1=\omega(e),\\
 e\parallel d_2&:
   &\widetilde\kappa_p(y)-\widetilde\kappa_p(x)&=r=\omega(e),\\
 e\parallel d_3&:
   &\widetilde\kappa_p(y)-\widetilde\kappa_p(x)&=-(r+1)
       =\omega(e)-3.
\end{aligned}
\]
Equivalently, all three cases are expressed by the single identity
\[
   \omega(e)=\widetilde\kappa_p(y)-\widetilde\kappa_p(x)
              +3\mathbf 1_{\{e\parallel d_3\}}.
\]
Summing this identity along $\gamma$ proves the asserted path formula.

If $\gamma$ is closed, the potential term telescopes to zero, so
\[
                         |\gamma|_\omega=3n_3(\gamma).
\]
A nonempty closed positive path cannot have $n_3(\gamma)=0$: along
every $d_1$- or $d_2$-edge the function
$\widetilde\kappa_p$ strictly increases, while changing panels at a
seam leaves it unchanged.  A nonempty path using only $d_1$- and
$d_2$-edges therefore cannot return to its initial vertex.  Hence
$n_3(\gamma)\geq1$, proving the positive-cycle gap.  Notice that this
argument permits arbitrarily many changes of panel and also covers the
edge of weight $-1$ in $P_3$: a negative individual edge is harmless
because every closed positive path still has weight at least $3$.

We next justify the assertion about the minimum in
\eqref{eq:directed-distance}.  The
underlying triangular complex is connected.  An unoriented edge can be
traversed in its positive direction by the edge itself and in the
opposite direction by the other two positive sides of an elementary
triangle incident to that edge occurrence.  Positive paths therefore
exist between every two vertices.  If a positive path repeats a
vertex, deleting the intervening closed subpath lowers its weight by at
least $3$.  A minimizing path can consequently be chosen
vertex-simple.  Since $\mathcal A_p$ is finite, there are only finitely many
such paths, and the minimum is attained.

Finally, let $e:x\to y$ be a retained positive edge occurrence.  By
the definition of the triangular cell complex $\mathcal A_p$, this occurrence
is a side of at least one retained elementary panel triangle.  The
other two sides give a positive path
$\overline e:y\rightsquigarrow x$.  The three positive sides of an
elementary triangle consist of one edge in each of the directions
$d_1,d_2,d_3$, so their total weight is
\[
                         1+r+(2-r)=3.
\]
It follows that $|\overline e|_\omega=3-\omega(e)$; see
Figure~\ref{fig:positive-cycle-local}~(b).

If $\eta:x\rightsquigarrow y$ is any positive path, then
$\eta\overline e$ is a nonempty closed positive path.  The cycle gap
therefore gives
\[
  3\leq |\eta|_\omega+3-\omega(e),
  \qquad\text{and hence}\qquad
  |\eta|_\omega\geq\omega(e).
\]
The one-edge path $e$ realizes equality.  Similarly, if
$\zeta:y\rightsquigarrow x$ is any positive path, the closed path
$e\zeta$ gives $|\zeta|_\omega\geq3-\omega(e)$, and
$\overline e$ realizes equality.  This proves
\eqref{eq:exact-adjacent-distances}.
\end{proof}

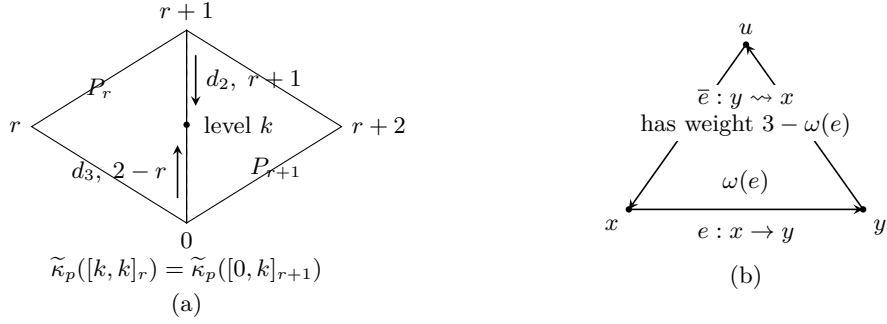
\begin{figure}[t]
\centering
\begin{tikzpicture}[>=stealth,font=\footnotesize]
  \begin{scope}[xshift=-3.25cm]
    \coordinate (z) at (0,0);
    \coordinate (s) at (0,2.55);
    \coordinate (l) at (-2.05,1.275);
    \coordinate (q) at ( 2.05,1.275);

    \draw (z)--(l)--(s)--cycle;
    \draw (z)--(q)--(s)--cycle;
    \draw[densely dashed] (z)--(s);

    \node[below] at (z) {$0$};
    \node[above] at (s) {$r+1$};
    \node[left]  at (l) {$r$};
    \node[right] at (q) {$r+2$};
    \node at (-1.15,1.78) {$P_r$};
    \node at ( 1.18,.72) {$P_{r+1}$};

    \fill (0,1.3) circle (1.2pt);
    \node[fill=white,inner sep=1.2pt,anchor=west] at (.18,1.3)
       {level $k$};

    \draw[->,semithick] (-.12,.33)--(-.12,1.03)
       node[midway,left] {$d_3,\;2-r$};
    \draw[->,semithick] (.12,2.22)--(.12,1.55)
       node[midway,right] {$d_2,\;r+1$};

    \node[align=center] at (0,-.62)
       {$\widetilde\kappa_p([k,k]_r)
         =\widetilde\kappa_p([0,k]_{r+1})$};
    \node at (0,-1.08) {(a)};
  \end{scope}

  \begin{scope}[xshift=4.15cm,yshift=.18cm]
    \coordinate (x) at (-1.55,0);
    \coordinate (y) at ( 1.55,0);
    \coordinate (u) at (0,2.18);

    \draw[->,semithick] (x)--(y)
       node[midway,below=3pt] {$e:x\to y$}
       node[midway,above=2pt] {$\omega(e)$};
    \draw[->,semithick] (y)--(u);
    \draw[->,semithick] (u)--(x);

    \fill (x) circle (1.2pt) node[below left] {$x$};
    \fill (y) circle (1.2pt) node[below right] {$y$};
    \fill (u) circle (1.2pt) node[above] {$u$};

    \node[align=center,fill=white,inner sep=1.4pt] at (0,1.28)
       {$\overline e:y\rightsquigarrow x$\\[-1pt]
        has weight $3-\omega(e)$};
    \node at (0,-.9) {(b)};
  \end{scope}
\end{tikzpicture}
\caption{The two local features used in
Lemma~\ref{lem:positive-cycle-gap}.  (a) The integer potential agrees
under the two coordinate names of a seam vertex.  The two developed
occurrences of a seam edge have opposite positive orientations and
weights $(2-r)$ and $(r+1)$, whose sum is $3$.  (b) The complementary
two-edge return path around an elementary panel triangle.}
\label{fig:positive-cycle-local}
\end{figure}

Reducing the integer lift modulo three, put
\[
       \kappa_p(z):=\widetilde\kappa_p(z)\pmod 3.
\]
The chosen constant normalizes this residue potential at $v_9$, since
\[
       v_9=[N-a,N]_1,
       \qquad
       \widetilde\kappa_p(v_9)=-(N-a)-N+2N-a=0.
\]
Concatenation of minimizing paths and the path identity in
Lemma~\ref{lem:positive-cycle-gap} now give
\begin{equation}
\label{eq:directed-triangle-and-residue}
 \begin{split}
 d_p(x,y)&\leq d_p(x,z)+d_p(z,y),\\
 d_p(x,y)&\equiv \kappa_p(y)-\kappa_p(x)\pmod 3.
 \end{split}
\end{equation}
Equivalently, on $P_r$ one has
\[
       \kappa_p([i,j]_r)\equiv-i-rj+2N-a\pmod 3.
\]
Thus the panel constants are automatically seam-compatible and equal
to $2N-a$ modulo three; no separate seam case is needed.

\medskip
\noindent\textbf{An explicit reference tiling.}
Recall the five vertices \(v_0,\ldots,v_4\) from
\eqref{eq:five-g-points}, and put
\[
\begin{aligned}
v_5&=(a+b,0,0,c,d),&
v_6&=(N-c,0,0,c,0),\\
v_7&=(a+d,0,b,c,0),&
v_8&=(N-b,0,b,0,0),\\
v_9&=(0,a,N-a,0,0).
\end{aligned}
\]
These are the five remaining corners of the support polygon, listed
clockwise from bottom to top in the developed picture.  Here and below
the ambient fan vertices are denoted \(u_0,\ldots,u_4\), to distinguish
them from the polygon vertices \(v_i\).  As declared above, figures
abbreviate \(u_i\) to \(i\).

Let \(M_{\rm ref}\) be the following simultaneous tiling of the three
developed panels; a degenerate region is omitted.
\begin{enumerate}
\item The quadrilateral \(v_0v_5v_6v_1\) is tiled by red lozenges.
The triangle \(v_0v_1u_3\) is tiled by blue lozenges, apart from the
purple singleton triangles adjacent to the seam segment \(v_1-u_3\).
\item The triangle \(v_1v_6v_2\) is tiled by blue lozenges, apart from
the purple singleton triangles adjacent to \(v_6-v_1\).  The
quadrilateral \(v_2v_3v_8v_7\) is tiled by red lozenges.  The pentagon
\(v_1v_2v_3u_2u_3\) is tiled by green lozenges, apart from the yellow
singleton triangles adjacent to \(v_3-u_2\).
\item The triangle \(v_3v_8v_4\) is tiled by green lozenges, apart from
the yellow singleton triangles adjacent to \(v_8-v_3\).  The
quadrilateral \(v_3v_4v_9u_2\) is tiled by blue lozenges.
\end{enumerate}
A labeled instance of this seven-region tiling is shown in
Figure~\ref{fig:step-v-bottom-matching}.
To make the color convention intrinsic, let \(d_1,d_2,d_3\) be the
positive directions in \eqref{eq:positive-panel-directions}; thus, on
\(P_r\),

\[
 d_1=e_r-e_{r+1},\qquad d_2=e_0-e_r,\qquad
 d_3=e_{r+1}-e_0.
\]
A color means that every lozenge in that region has its selected edge
parallel to the corresponding direction: on \(P_1\) and \(P_3\), red,
blue, and green mean \(d_1,d_2,d_3\), respectively; on the reflected
middle panel \(P_2\), they mean \(d_1,d_3,d_2\), respectively.
Thus the color and the polygon vertices determine every tile.

\medskip
\noindent\textbf{Legality of the reference tiling.}
This verification uses only the support inequalities, not an extremal
extension.  The inequalities in \eqref{eq:Sp-formula} determine the
outer sides of \(\mathcal A_p\).  The additional internal interfaces of
the seven-region decomposition are
\begin{align*}
 P_1=012:&\qquad v_3v_4\subset\{q\in P_1:q_0=c+d\},\\
 P_2=023:&\qquad v_1v_2\subset\{q\in P_2:q_0=d\},\quad
                    v_7v_2\subset\{q\in P_2:q_3=c\},\\
          &\qquad v_2v_3\subset\{q\in P_2:q_2=a+b\},\\
 P_3=034:&\qquad v_0v_1\subset\{q\in P_3:q_3=N-d\}.
\end{align*}
Together with the retained seams, these segments partition
\(\mathcal A_p\) into precisely the polygons listed above: their
interiors are disjoint and their union is \(\mathcal A_p\).  This
follows directly by subtracting the displayed endpoint coordinates.
The same subtraction shows that every non-seam side of each piece lies
in one of the two lattice directions complementary to its selected
lozenge direction.  Pairing consecutive elementary triangles in the
third direction therefore gives the asserted monochromatic lozenges.

On seam \(02\), the singleton intervals \([v_8,v_3]\) and
\([v_3,u_2]\) on the two developed occurrences are complementary; on
seam \(03\), the corresponding complementary intervals are
\([v_6,v_1]\) and \([v_1,u_3]\).
Hence every elementary triangle belongs to exactly one ordinary
lozenge or exactly one singleton seam triangle.  Thus \(M_{\rm ref}\)
is a legal tiling in the sense of
Definition~\ref{def:legal-lozenge-triangle-tiling}, and therefore a
perfect matching of \(G_p\).

Let
\[
                         H_{\rm ref}:=h_{M_{\rm ref}},
             \qquad      H_{\rm ref}(v_9)=0.
\]
No exterior boundary edge can be selected: it is not the common side
of two retained elementary triangles.  Consequently
\(H_{\rm ref}(y)-H_{\rm ref}(x)=\omega(e)\) on every positively
oriented boundary edge \(e:x\to y\).  Define
\[
                         f_p:=H_{\rm ref}|_{B_p}.
\]
Equivalently, \(f_p(v_9)=0\) and \(f_p\) is obtained by integrating the
forced upper increments around the boundary.  The reference potential
makes the consistency of this prescription automatic.  The same
exterior-edge observation shows that every legal tiling, after the
normalization \(h_M(v_9)=0\), has boundary trace \(f_p\).

\begin{lemma}
\label{lem:boundary-datum-compatible}
For all \(b,b'\in B_p\),
\begin{equation}
 f_p(b')-f_p(b)\le d_p(b,b'),
 \qquad
 f_p(b)\equiv\kappa_p(b)\pmod3.
\label{eq:boundary-datum-compatible}
\end{equation}
\end{lemma}
\begin{proof}
For a positive edge occurrence \(e:x\to y\), let
\(\chi_{\rm ref}(e)\) indicate whether \(e\) is selected by
\(M_{\rm ref}\).  The two-increment rule
\eqref{eq:two-edge-increments} gives
\begin{equation}
 H_{\rm ref}(y)-H_{\rm ref}(x)
   =\omega(e)-3\chi_{\rm ref}(e)\le\omega(e).
\label{eq:reference-potential-edge-bound}
\end{equation}
Let \(\eta:b\rightsquigarrow b'\) be any positive edge path in
\(\mathcal A_p\).  Only its endpoints are required to lie in \(B_p\);
its intermediate vertices may lie anywhere in \(\mathcal A_p\).  Then
\[
\begin{aligned}
 f_p(b')-f_p(b)
 &=H_{\rm ref}(b')-H_{\rm ref}(b)\\
 &=\sum_{e\in\eta}
      \bigl(\omega(e)-3\chi_{\rm ref}(e)\bigr)
 \le \sum_{e\in\eta}\omega(e)
 =|\eta|_\omega.
\end{aligned}
\]
The aligned height \(H_{\rm ref}\) is single-valued across the seams,
so the telescoping identity is unchanged when \(\eta\) changes panels.
Since \(\eta\) was arbitrary, taking the minimum over all positive paths
from \(b\) to \(b'\) and using the definition
\eqref{eq:directed-distance} gives
\[
 f_p(b')-f_p(b)
 \le \min_{\eta:b\rightsquigarrow b'}|\eta|_\omega
 =d_p(b,b').
\]
This proves the first assertion.
The identical summation with arbitrary endpoints \(u,v\) shows, more
generally, that
\[
                    H_{\rm ref}(v)-H_{\rm ref}(u)\le d_p(u,v);
\]
thus \(H_{\rm ref}\) is a directed Lipschitz extension of \(f_p\).

Modulo three, \eqref{eq:reference-potential-edge-bound} and
Lemma~\ref{lem:positive-cycle-gap} give
\[
 H_{\rm ref}(y)-H_{\rm ref}(x)
 \equiv\omega(e)
 \equiv\kappa_p(y)-\kappa_p(x)\pmod3
\]
on every positive edge occurrence.  Since the glued complex is
connected and
\(H_{\rm ref}(v_9)=0\equiv\kappa_p(v_9)\pmod3\), it follows that
\(H_{\rm ref}\equiv\kappa_p\pmod3\) everywhere.  Restriction to
\(B_p\) proves the second assertion.
\end{proof}

\begin{lemma}[Extremal extension]
\label{lem:directed-extremal-extension}
Define
\begin{equation}
 \begin{split}
 L_p(z)&:=\max_{b\in B_p}\{f_p(b)-d_p(z,b)\},\\
 U_p(z)&:=\min_{b\in B_p}\{f_p(b)+d_p(b,z)\}.
 \end{split}
\label{eq:directed-envelopes}
\end{equation}
Then \(L_p\) and \(U_p\) extend \(f_p\), have residue \(\kappa_p\)
modulo three, and satisfy
\begin{equation}
 H(y)-H(x)\le d_p(x,y)
 \qquad(H=L_p,U_p).
\label{eq:envelope-directed-lipschitz}
\end{equation}
Every function \(H:\mathcal V_p\to\mathbb Z\) with the same boundary
values, residue, and directed Lipschitz property satisfies
\begin{equation}
                         L_p\le H\le U_p.
\label{eq:directed-envelope-sandwich}
\end{equation}
Moreover, each such function is the aligned height of a unique legal
tiling of \(\mathcal A_p\), hence of a perfect matching of \(G_p\).
\end{lemma}

\begin{proof}
Compatibility~\eqref{eq:boundary-datum-compatible} implies that both
envelopes agree with \(f_p\) on \(B_p\).  The directed triangle
inequality proves~\eqref{eq:envelope-directed-lipschitz}.  Conversely,
if \(H\) is any directed Lipschitz extension, then for every \(b\in B_p\)
\[
 f_p(b)-d_p(z,b)\le H(z)\le f_p(b)+d_p(b,z);
\]
maximizing the left side and minimizing the right side proves
\eqref{eq:directed-envelope-sandwich}.  The congruences in
\eqref{eq:directed-triangle-and-residue} show that all candidates in each
envelope have residue \(\kappa_p(z)\).

For a positive edge occurrence \(e:x\to y\), apply
\eqref{eq:envelope-directed-lipschitz} to \(e\) and to the positive
return path of weight \(3-\omega(e)\).  Thus
\[
 \omega(e)-3\le H(y)-H(x)\le\omega(e).
\]
The residue condition leaves exactly the two values in
\eqref{eq:two-edge-increments}.  Select \(e\) precisely for the lower
one.  Around each elementary panel triangle, telescoping the three
increments shows that exactly one of its three sides is selected.
On a retained seam the single-valued aligned height gives the
complementary-singleton equation above.  Hence the selected ordinary
sides and singleton choices form a legal tiling.  The increment test
also recovers them uniquely from \(H\).
\end{proof}

Pointwise minimum and maximum preserve the two legal increments and the
common boundary normalization.  Consequently these height functions are
closed under \(\min\) and \(\max\) and form a finite distributive lattice.
This conclusion does not assert that every cover relation in that lattice is a
single-face rotation; no such assertion is used below.

\begin{definition}[Bottom and top heights]
\label{def:bottom-top-matchings}
Let \(M_{\min}\) and \(M_{\max}\) be the legal tilings supplied by
Lemma~\ref{lem:directed-extremal-extension} from \(L_p\) and \(U_p\),
and put
\[
                         h_{\rm bot}:=L_p,
             \qquad      h_{\rm top}:=U_p.
\]
For every matching \(M\), Lemma~\ref{lem:directed-extremal-extension}
gives
\begin{equation}
                 h_{\rm bot}\le h_M\le h_{\rm top}
                 \qquad(M\in\mathcal D(G_p)).
\label{eq:bottom-top-sandwich}
\end{equation}
Thus these are genuinely the bottom and top elements of the matching
lattice, established without flip-connectivity.
\end{definition}

\subsection{The two polygonal candidates}

The first candidate has already appeared as the explicit reference
tiling used in Lemma~\ref{lem:boundary-datum-compatible}.  We now
record the notation needed when the
target varies.  For
\(p'=(0,a',b',c',d')\), write
\[
        v_9(p')=(0,a',N'-a',0,0),
        \qquad N'=a'+b'+c'+d'.
\]
If \(p=(\alpha,\beta,\gamma,\delta)\) is a continuum target, then
\(v_9(np)\) abbreviates
\(v_9(0,n\alpha,n\beta,n\gamma,n\delta)\).
For a fixed target \(p\), the argument is omitted.

\begin{definition}[The polygonal bottom matching]
\label{def:polygonal-matching}
The explicit reference tiling \(M_{\rm ref}\) constructed before
Lemma~\ref{lem:boundary-datum-compatible} is denoted by
\(M_{\rm poly}\).  At this stage this is only a
legal polygonal candidate; its identification with \(M_{\min}\) is
proved later.
\end{definition}

\begin{lemma}[No negative rotation in the polygonal matching]
\label{lem:polygonal-bottom-no-negative}
The perfect matching \(M_{\rm poly}\) of
Definition~\ref{def:polygonal-matching} admits no negative face
rotation.
\end{lemma}

\begin{proof}
Contract the inserted seam zigzags and put
\(\mu=\overline M_{\rm poly}\).  In this proof, write \(1\) for an
occupied boundary edge of a contracted face and \(0\) for an
unoccupied one, reading cyclically.  Because every region in the
seven-region prescription defining \(M_{\rm ref}=M_{\rm poly}\) is filled by one fixed
lozenge orientation, the local matching pattern is translation
invariant in the relative interior of that region.  At a selected
six-sided face in such a monochromatic interior, the cyclic
occupied/unoccupied boundary pattern is \(100100\), up to cyclic
shift.  In the relative interior of a straight interface between two
regions, the two prescribed lozenge orientations likewise give exactly
two occupied sides, but these need not be opposite: the cyclic pattern
can be \(101000\), up to cyclic shift.  In either case the pattern is
nonalternating and therefore supports no face rotation.

At a straight selected face adjacent to a seam-triangle strip, the
contracted face is four-sided and its boundary pattern is
\[
                              1000.
\]
This is again nonalternating.  Thus only the three mutable corners
\(v_1,v_2,v_3\) of the internal polygonal interfaces require separate
inspection.  At \(v_1\) and \(v_3\) the pattern is \(1010\).  With the
white-on-the-right convention of the auxiliary quiver
\(Q_0^{\rm R}\) from Section~\ref{sec:cluster-dictionary}, the occupied opposite pair is
the positive alternating class; the negative class would occupy the
other pair.  Hence neither face admits a negative rotation.  At
\(v_2\) the six-sided pattern is \(100000\), which is not alternating.

There are no further selected corner types.  Indeed the five additional
polygon vertices satisfy
\[
\begin{array}{c|ccccc}
 &v_5&v_6&v_7&v_8&v_9\\ \hline
\text{failed strict boundary}&q_3=c&q_3=c&
 (q_2,q_3)=(b,c)&q_2=b&q_0=0,
\end{array}
\]
so~\eqref{eq:Sp-formula} places all of them outside \(S_p\).
The endpoint \(v_4\) is also unselected, \(v_0\) is frozen, and the
ambient fan vertices do not give rotatable bounded selected faces.
This exhausts the relative interiors, straight interfaces, seam
strips, and polygon corners.

Consequently \(M_{\rm poly}\) has no negative face rotation.
\end{proof}

For the top matching, introduce the three new interface vertices
\begin{equation}
\begin{aligned}
w_1&=(a,0,N-a,0,0),\\
w_2&=(a,0,b,c+d,0),\\
w_3&=(a+b,0,0,c+d,0).
\end{aligned}
\label{eq:polygonal-top-vertices}
\end{equation}
Thus \(w_1\) lies on the retained seam \(02\), \(w_2\) lies in the
middle panel \(023\), and \(w_3\) lies on the retained seam \(03\).
They are the top-matching analogues of \(v_1,v_2,v_3\).

\begin{definition}[The polygonal top matching]
\label{def:polygonal-top-matching}
The matching \(M_{\rm poly}^{+}\) is the following simultaneous tiling.
As above, degenerate regions are omitted.
\begin{enumerate}
\item On panel \(012\), the quadrilateral \(v_8w_1v_9v_4\) is tiled
by red lozenges.  The triangle \(w_1v_9u_2\) is tiled by green
lozenges, apart from the yellow singleton triangles adjacent to the
seam segment \(w_1-u_2\).
\item On panel \(023\), the quadrilateral \(v_6w_3w_2v_7\) is tiled
by red lozenges.  The quadrilateral \(v_7w_2w_1v_8\) is tiled by
green lozenges, apart from the yellow singleton triangles adjacent to
\(v_8-w_1\).  The pentagon \(w_3w_2w_1u_2u_3\) is tiled by blue
lozenges, apart from the purple singleton triangles adjacent to
\(w_3-u_3\).
\item On panel \(034\), the triangle \(v_6w_3v_5\) is tiled by blue
lozenges, apart from the purple singleton triangles adjacent to
\(v_6-w_3\).  The quadrilateral \(w_3u_3v_0v_5\) is tiled by green
lozenges.
\end{enumerate}
\end{definition}

These polygons partition \(\mathcal A_p\).
Their internal straight sides are \(q_2=N-a\) on panel \(012\),
\(q_2=b\), \(q_3=c+d\), and \(q_0=a\) on panel \(023\), and
\(q_0=a+b\) on panel \(034\).  The yellow strips on the two sides of
seam \(02\) are complementary, as are the purple strips on the two
sides of seam \(03\).  Hence \(M_{\rm poly}^{+}\) is a legal perfect
matching.

\begin{lemma}[No positive rotation in the polygonal top matching]
\label{lem:polygonal-top-no-positive}
The perfect matching \(M_{\rm poly}^{+}\) of
Definition~\ref{def:polygonal-top-matching} admits no positive face
rotation.
\end{lemma}

\begin{proof}
The local inspection is the top analogue of the proof of
Lemma~\ref{lem:polygonal-bottom-no-negative}.  At a selected six-sided
face in a monochromatic interior, the cyclic boundary pattern is
\(100100\), up to cyclic shift.  At a straight interface the two
prescribed orientations again give exactly two occupied sides, which
need not be opposite and can have cyclic pattern \(101000\), up to
cyclic shift.  Beside a seam-triangle strip the contracted four-sided pattern
is \(1000\).  These patterns are nonalternating.

Only \(w_1,w_2,w_3\) require a separate inspection.  At \(w_1\) and
\(w_3\) the pattern is \(1010\), but now the occupied opposite pair is
the negative alternating class, so neither face admits a positive
rotation.  At \(w_2\) the pattern is \(100000\), hence is nonalternating.
All remaining corners are unselected, frozen, or ambient boundary
vertices.  Thus \(M_{\rm poly}^{+}\) admits no positive face rotation.
\end{proof}

\subsection{The boundary layer of the bottom tiling}

\begin{figure}[htbp]
\centering
\begin{tikzpicture}[
  >=Stealth,
  every node/.style={font=\small},
  leaving/.style={->,very thick,blue!70!black},
  entering/.style={->,thick,dashed,orange!85!black},
  boundarypoint/.style={fill=white,draw=black,thick}]
  \begin{scope}[x={(1.022cm,0.59cm)},y={(1.022cm,-0.59cm)}]
    \clip (-3.15,-3.05) -- (3.15,-3.05) -- (3.15,3.05) --
          (-3.15,3.05) -- cycle;
    \foreach \j in {-3,...,3}
      \draw[gray!32] (-3.2,\j) -- (3.2,\j);
    \foreach \i in {-3,...,3}
      \draw[gray!32] (\i,-3.2) -- (\i,3.2);
    \foreach \s in {-6,...,6}
      \draw[gray!32] (-3.2,{\s+3.2}) -- (3.2,{\s-3.2});

    \coordinate (q)  at (0,0);
    \coordinate (As) at (0,-2);
    \coordinate (Bt) at (-2,0);
    \coordinate (z1) at (1,0);
    \coordinate (z2) at (-1,1);
    \coordinate (z3) at (1,-1);
    \coordinate (z4) at (0,1);

    \fill[red!10] (q) -- (-3,0) -- (-3,-3) -- (0,-3) -- cycle;
    \draw[red!55!black,dashed,thick]
         (q) -- (-3,0) -- (-3,-3) -- (0,-3) -- cycle;
    \draw[black,very thick] (-3,0) -- (q) -- (0,-3);

    \draw[entering] (z1) -- node[below] {\(+1\)} (q);
    \draw[entering] (z2) -- node[left] {\(+1\)} (q);
    \draw[leaving] (q) -- node[right] {\(-2\)} (z3);
    \draw[leaving] (q) -- node[right] {\(-2\)} (z4);

    \filldraw[boundarypoint] (q) circle (2.3pt);
    \filldraw[boundarypoint] (As) circle (2pt);
    \filldraw[boundarypoint] (Bt) circle (2pt);
    \filldraw[fill=orange!15,draw=orange!85!black,thick]
      (z1) circle (2.3pt);
    \filldraw[fill=orange!15,draw=orange!85!black,thick]
      (z2) circle (2.3pt);
    \filldraw[fill=blue!12,draw=blue!70!black,thick]
      (z3) circle (2.3pt);
    \filldraw[fill=blue!12,draw=blue!70!black,thick]
      (z4) circle (2.3pt);

    \node[left,xshift=-3pt] at (q) {\(v_7\)};
    \node[above left] at (As) {\(A_s\)};
    \node[below left] at (Bt) {\(B_t\)};
    \node[above right] at (z1) {\(z_1\)};
    \node[below] at (z2) {\(z_2\)};
    \node[above] at (z3) {\(z_3\)};
    \node[below right] at (z4) {\(z_4\)};

    \node[rotate=30,above] at (-1.25,0) {\(Y=0\)};
    \node[rotate=-30,below] at (0,-1.25) {\(X=0\)};
    \node[align=center,text=red!55!black] at (-1.90,-1.70)
         {excluded sector\\\(X<0,\ Y<0\)};
  \end{scope}

  \draw[leaving] (4.02,-1.10) -- +(0.72,0)
       node[right,black] {selected: lower increment};
  \draw[entering] (4.02,-1.58) -- +(0.72,0)
       node[right,black] {unselected: upper increment};
\end{tikzpicture}
\caption{The reentrant boundary configuration at \(q_*=v_7\), drawn
in exact triangular-lattice coordinates and in the same orientation as
the \(P_2=023\) frame in Figure~\ref{fig:three-panel-lozenge-height}:
increasing \(X=q_2-b\) points toward vertex \(u_2\), and increasing
\(Y=q_3-c\) points toward vertex \(u_3\).  Thus \(Y=0\) is parallel to
side \(02\), while \(X=0\) is parallel to side \(03\).  The red sector is outside
the support, and \(A_s,B_t\) lie on its two boundary arms.  The four marked
vertices \(z_1,\ldots,z_4\) are the adjacent interior neighbors.  The
positive blue arrows leave \(v_7\), while the positive dashed orange
arrows enter it.  The numbers beside the arrows are increments of the
ordinary lozenge-height function \(F\): \(+1\) on the two unselected
incoming edges and \(-2\) on the two selected outgoing edges.  The
corresponding aligned increments, obtained from
\(\Delta h=\Delta F-\Delta j\), are respectively
\(2,1,-2,-3\).  The colors display the local pattern of the explicit
reference tiling used in the proof of
Lemma~\ref{lem:bottom-boundary-layer}: it selects the blue outgoing
edges and not the orange incoming edges.}
\label{fig:reentrant-boundary-cone}
\end{figure}
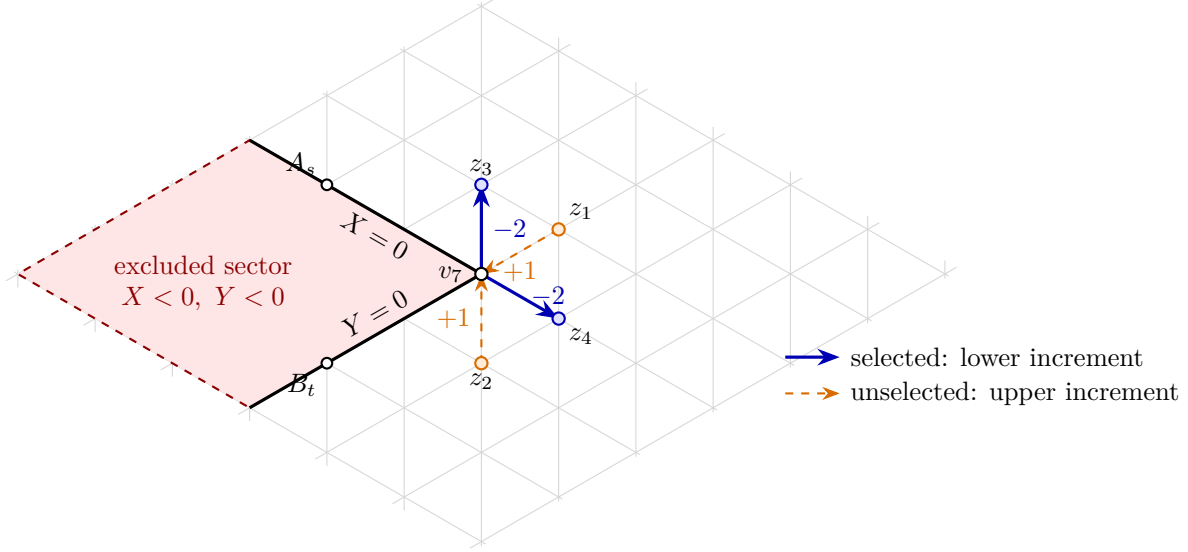

\begin{lemma}[Boundary layer of the bottom tiling]
\label{lem:bottom-boundary-layer}
Let \(e:x\to y\) be a positive edge occurrence, ordinary or on a
retained seam, with exactly one endpoint on \(B_p\).  Then:
\begin{enumerate}[label=\textup{(\roman*)}]
\item if \(x\in B_p\) and \(y\notin B_p\), then \(e\) is selected by
      the bottom tiling \(M_{\min}\);
\item if \(x\notin B_p\) and \(y\in B_p\), then \(e\) is not selected
      by \(M_{\min}\).
\end{enumerate}
\end{lemma}

\begin{proof}
Write \(H=H_{\rm ref}=h_{M_{\rm poly}}\), with the normalization already
chosen above, so that \(H|_{B_p}=f_p\).  Since \(H\) is a directed
Lipschitz extension of \(f_p\), the extremal sandwich
\eqref{eq:directed-envelope-sandwich} gives
\begin{equation}
                              L_p\le H.
\label{eq:reference-dominates-lower-envelope}
\end{equation}

We first read the inward boundary pattern directly from the explicit
reference tiling.  If a positive edge occurrence has one endpoint on
\(B_p\) and the other in the interior, then
\begin{equation}
 \begin{split}
 B_p\longrightarrow\mathcal A_p^\circ
     &\quad\Longrightarrow\quad \text{selected by }M_{\rm poly},\\
 \mathcal A_p^\circ\longrightarrow B_p
     &\quad\Longrightarrow\quad \text{unselected by }M_{\rm poly}.
 \end{split}
\label{eq:reference-inward-boundary-pattern}
\end{equation}
Indeed, order the straight boundary sides of \(P_1\) by
\(q_0=0,q_1=a,q_2=b\), those of \(P_2\) by
\(q_0=0,q_2=b,q_3=c\), and those of \(P_3\) by
\(q_0=0,q_3=c,q_4=d\).  In these orders, the positive directions from
the boundary into the interior are
\[
 P_1:(d_2,d_2,d_3),\qquad
 P_2:(d_2,d_1,d_3),\qquad
 P_3:(d_2,d_1,d_1),
\]
and the reference tiling selects them.  The positive directions from
the interior toward the boundary are, in the same orders,
\[
 P_1:(d_3,d_1,d_1),\qquad
 P_2:(d_3,d_2,d_1),\qquad
 P_3:(d_3,d_2,d_3),
\]
and the reference tiling leaves them unselected.  To obtain these
lists, use the active inequality in \eqref{eq:Sp-formula} to determine
which of the two incident directions points into the support, and then
apply the color--direction convention defining \(M_{\rm ref}\).  This
proves \eqref{eq:reference-inward-boundary-pattern} in the relative
interior of every straight side.  At a convex endpoint, the same
calculation applies simultaneously to the two active equalities.  At a
seam endpoint, the complementary singleton strips give the two required
statuses on the two developed occurrences.  These checks are unchanged
under the three cyclic panel orientations.  The only reentrant endpoint
is \(v_7\); its four incident occurrences are shown in
Figure~\ref{fig:reentrant-boundary-cone}, with the two positive outgoing
edges selected and the two positive incoming edges unselected.  The
polygonal decomposition exhausts the straight sides, convex endpoints,
seam endpoints, and this single reentrant endpoint, so
\eqref{eq:reference-inward-boundary-pattern} follows.

Now let \(e:x\to y\) be as in the statement.  Suppose first that
\(x\in B_p\) and \(y\notin B_p\).  By
\eqref{eq:reference-inward-boundary-pattern},
\[
                 H(y)-H(x)=\omega(e)-3.
\]
The exact adjacent distance in Lemma~\ref{lem:positive-cycle-gap} is
\(d_p(y,x)=3-\omega(e)\).  Taking \(x\) as a candidate in the
lower-envelope formula and then using
\eqref{eq:reference-dominates-lower-envelope} gives
\[
 L_p(y)\ge f_p(x)-d_p(y,x)
         =H(x)-(3-\omega(e))
         =H(y)\ge L_p(y).
\]
Thus equality holds throughout, and since \(L_p(x)=f_p(x)=H(x)\),
\[
                       L_p(y)-L_p(x)=\omega(e)-3.
\]
Hence the bottom tiling selects \(e\).

Suppose instead that \(x\notin B_p\) and \(y\in B_p\).  The reference
tiling leaves \(e\) unselected, so
\(H(y)-H(x)=\omega(e)\), while
\(d_p(x,y)=\omega(e)\).  Using \(y\) as the boundary candidate gives
\[
 L_p(x)\ge f_p(y)-d_p(x,y)
         =H(y)-\omega(e)
         =H(x)\ge L_p(x).
\]
Again equality holds throughout.  Since \(L_p(y)=H(y)\), we obtain
\[
                         L_p(y)-L_p(x)=\omega(e),
\]
and the bottom tiling does not select \(e\).
\end{proof}

\subsection{Strict width of the extremal interval}

The following is the geometric input needed for height forcing.  It is
proved directly from the two envelopes and the preceding boundary layer.

\begin{lemma}[One full interior height step]
\label{lem:strict-extremal-width}
\label{lem:one-layer-slack}
For every interior lattice vertex $z$,
\begin{equation}
             h_{\rm top}(z)-h_{\rm bot}(z)\ge3.
\label{eq:one-layer-slack}
\end{equation}
\end{lemma}

\begin{proof}
Fix \(z\notin B_p\).  For each \(b\in B_p\), choose a shortest
positive path \(\gamma_b:b\rightsquigarrow z\).  It has a first edge
leaving \(B_p\), and Lemma~\ref{lem:bottom-boundary-layer} says that
\(L_p\) uses the lower increment on that edge.  Telescoping along
\(\gamma_b\) therefore gives
\[
 L_p(z)-f_p(b)\le |\gamma_b|_\omega-3
                    =d_p(b,z)-3.
\]
Hence \(L_p(z)+3\le f_p(b)+d_p(b,z)\) for every \(b\in B_p\).
Taking the minimum over \(b\) and using
\eqref{eq:directed-envelopes} gives \(L_p(z)+3\le U_p(z)\), as
required.
\end{proof}

The following calculation is retained for later use.  At this point it
is only a comparison of the two explicit polygonal tilings; their
identification with the extremal envelopes will be made after
well-matchedness.

\begin{lemma}[Exact polygonal height gap]
\label{lem:exact-polygonal-height-gap}
Normalize the aligned heights of \(M_{\rm poly}\) and
\(M_{\rm poly}^{+}\) to agree on \(B_p\), and put
\[
 K(q):=\frac{h_{M_{\rm poly}^{+}}(q)-h_{M_{\rm poly}}(q)}{3}.
\]
Then \(K\) is given by
\eqref{eq:gap-panel-012}, \eqref{eq:gap-panel-034}, and
\eqref{eq:gap-panel-023} below.  In particular \(K(q)\ge1\) at every
interior vertex.
\end{lemma}

\begin{proof}
Put
\[
 K(q):=\frac{h_{M_{\rm poly}^{+}}(q)-h_{M_{\rm poly}}(q)}{3}.
\]
For a positive edge occurrence \(e:x\to y\), let
\(\chi_{\rm poly}(e)\) and \(\chi_{\rm poly}^{+}(e)\) indicate whether
\(e\) is selected by \(M_{\rm poly}\) and \(M_{\rm poly}^{+}\),
respectively.  The two
possible increments in~\eqref{eq:two-edge-increments} give
\begin{equation}
 K(y)-K(x)=\chi_{\rm poly}(e)-\chi_{\rm poly}^{+}(e).
\label{eq:gap-edge-increment}
\end{equation}

Overlay Definitions~\ref{def:polygonal-matching}
and~\ref{def:polygonal-top-matching}.  Every resulting polygon is
filled by a fixed lozenge orientation, so the right side of
\eqref{eq:gap-edge-increment} is read directly from the two polygonal
patterns.  Checking the three positive lattice-edge directions in
each panel gives the following formulas.  On panel \(012\),
\begin{equation}
 K(q)=\min\{q_0,\ a-q_1,\ q_2-b,\ c+d\}.
\label{eq:gap-panel-012}
\end{equation}
On panel \(034\),
\begin{equation}
 K(q)=\min\{q_0,\ q_3-c,\ d-q_4,\ a+b\}.
\label{eq:gap-panel-034}
\end{equation}
For panel \(023\), put
\[
       x=q_2-b,\qquad y=q_3-c,\qquad t_+=\max\{t,0\}.
\]
Then
\begin{equation}
 K(q)=\min\bigl\{q_0,\ x_++y_+,\
             a+(-x)_+,\ d+(-y)_+\bigr\}.
\label{eq:gap-panel-023}
\end{equation}
The exhaustive middle-panel check can equivalently be written
\begin{equation}
K(q)=
\begin{cases}
 \min\{q_0,y,a-x,d\},&x\le0,\ y\ge1,\\
 \min\{q_0,x,a,d-y\},&x\ge1,\ y\le0,\\
 \min\{q_0,x+y,a,d\},&x\ge1,\ y\ge1.
\end{cases}
\label{eq:gap-middle-cases}
\end{equation}
These are all the cases because the condition defining
\(S_p\cap\Delta_{023}\) excludes \(x\le0\) and \(y\le0\)
simultaneously.

For completeness, the verification of
\eqref{eq:gap-panel-012}--\eqref{eq:gap-middle-cases} is edgewise.
Start with value zero on \(B_p\).  In the relative interior of each
monochromatic region, subtract the selected-edge indicator of the top
orientation from that of the bottom orientation.  This gives the
directional differences of the displayed minimum.  The same
calculation on a straight color interface gives the common value of
the two adjacent linear expressions.  Beside a seam strip, the
complementary-singleton rule gives the same difference on the two
cut-open occurrences.  There are no other local types: the remaining
ones are the interface vertices \(v_1,v_2,v_3,w_1,w_2,w_3\), where
direct substitution gives~\eqref{eq:gap-edge-increment}.  Thus the
displayed formulas and \(K\) have the same edge increments and the
same boundary values, so connectedness of \(\mathcal A_p\) proves
them.  Notice in particular that the formulas agree on the seams:
on \(02\) they both reduce to
\[
       \min\{q_0,a,q_2-b,c+d\},
\]
and on \(03\) they both reduce to
\[
       \min\{q_0,q_3-c,d,a+b\}.
\]

It remains only to read positivity.  If
\(q\in S_p\cap\Delta_{012}\), the first three entries of
\eqref{eq:gap-panel-012} are positive integers and \(c+d\ge2\).
The same argument, with \(a+b\ge2\), applies to
\eqref{eq:gap-panel-034}.  In the middle panel, \(q_0\ge1\), at least
one of \(x,y\) is positive, and \(a,d\ge1\).  Hence every entry in
the minimum in~\eqref{eq:gap-panel-023} is at least one.  Therefore
\(K(q)\ge1\) at every interior vertex, as asserted.
\end{proof}

\subsection{Forcing one prescribed lozenge}

\begin{proposition}[Every ordinary edge is allowed]
\label{prop:ordinary-edge-allowed}
Every noninserted edge of $G_p$, equivalently every admissible
ordinary lozenge of $\mathcal A_p$, belongs to a perfect matching.
\end{proposition}

\begin{proof}
Let $\ell$ be an admissible lozenge, and let
\[
       e:x\longrightarrow y
\]
be the positive orientation of its shared lattice edge.

If the bottom tiling already contains $\ell$, there is nothing to prove.
Assume otherwise.  Lemma~\ref{lem:bottom-boundary-layer} then implies
that \(x\) is interior: if one endpoint is on \(B_p\), the only
alternative compatible with an unselected positive edge places \(y\),
not \(x\), on the boundary.

Since the bottom tiling does not contain \(\ell\),
\begin{equation}
                         L_p(y)-L_p(x)=\omega(e).
\label{eq:bottom-upper-increment-forcing}
\end{equation}
On the partial domain \(B_p\cup\{x,y\}\), prescribe
\begin{equation}
 g_e(x)=L_p(x)+3,\qquad g_e(y)=L_p(y),\qquad
 g_e(b)=f_p(b)\quad(b\in B_p).
\label{eq:forced-edge-partial-datum}
\end{equation}
This is well defined even when \(y\in B_p\), because then
\(L_p(y)=f_p(y)\).

We check the directed Lipschitz inequalities for this partial datum.
Lemma~\ref{lem:strict-extremal-width} and the upper-envelope formula give,
for every \(b\in B_p\),
\[
 L_p(x)+3\le U_p(x)\le f_p(b)+d_p(b,x).
\]
This is the only new boundary-to-\(x\) inequality.  The inequality from
\(x\) to \(b\) follows from the directed Lipschitz property of \(L_p\)
and is only strengthened by subtracting three.  All inequalities involving
\(y\) and a boundary point are those already satisfied by \(L_p\).
Finally, using~\eqref{eq:bottom-upper-increment-forcing} and the two
adjacent directed distances,
\[
 \begin{split}
 g_e(y)-g_e(x)&=\omega(e)-3\le d_p(x,y)=\omega(e),\\
 g_e(x)-g_e(y)&=3-\omega(e)=d_p(y,x).
 \end{split}
\]
Thus \(g_e\) is a compatible partial datum.  Its values have the required
residue, since the only change from \(L_p\) is by three.

Take the upper directed envelope of \(g_e\).  The proof of
Lemma~\ref{lem:directed-extremal-extension}, applied to this enlarged
partial domain, shows that the envelope agrees with \(g_e\), is a legal
aligned height function, and has the original boundary values.  At the
two prescribed vertices its increment is
\[
 g_e(y)-g_e(x)=L_p(y)-L_p(x)-3=\omega(e)-3.
\]
Therefore the resulting perfect matching selects \(e\) and contains
\(\ell\).
\end{proof}

\begin{theorem}[Well-matchedness]
\label{thm:well-matched}
Every edge of $\overline G_p$ lies in a perfect matching.  Moreover,
every edge of the cut-open graph $G_p$, including every inserted seam
edge, lies in a perfect matching.
\end{theorem}

\begin{proof}
Proposition~\ref{prop:ordinary-edge-allowed} proves the assertion for
every ordinary edge of $G_p$.  Contracting the seam zigzags under the
bijection of Theorem~\ref{thm:fixed-graph}, part~3, proves the same
assertion for every edge of $\overline G_p$.

It remains to force an inserted seam edge.  Every retained zigzag has
the form
\[
             x_{\rm L}-z-x_{\rm R}.
\]
To force $x_{\rm L}z$, choose an ordinary panel edge incident to the
sector vertex $x_{\rm R}$.  Such an edge is part of the two-sector seam datum:
zigzags are inserted only at retained seam levels having an ordinary
incident sector edge on both developed sides.  By
Proposition~\ref{prop:ordinary-edge-allowed}, there is a tiling
containing that ordinary edge, so $x_{\rm R}$ is the unique sector vertex
matched outside the zigzag.  Its alternating completion therefore
contains $x_{\rm L}z$.  Interchanging the two sectors gives the same
argument for $zx_{\rm R}$.  Hence every inserted edge is allowed.
\end{proof}

\subsection{Flip connectivity and the polygonal identifications}

Now, and only now, we use the dimer-lattice theorem.  Let
\(\mathcal R_p\) be the graph whose vertices are the perfect matchings
of \(\overline G_p\), with two matchings adjacent when they differ by a
rotation around one bounded face.  In the cut-open lozenge--triangle
picture this is an elementary lozenge rotation or the corresponding
shift of a seam triangle.

\begin{lemma}[Flip connectivity]
\label{lem:flip-connectivity}
The graph \(\mathcal R_p\) is connected.  With the positive-rotation
convention fixed above, it is the Hasse diagram of the finite
distributive lattice of normalized height functions.  Its bottom and
top elements have heights \(h_{\rm bot}=L_p\) and \(h_{\rm top}=U_p\).
\end{lemma}

\begin{proof}
Theorem~\ref{thm:well-matched} proves that every edge of
\(\overline G_p\) is allowed.  Propp's dimer-lattice theorem
\cite[Theorem~2]{ProppLattice}, in the formulation
\cite[Theorem~2.6]{MMSBV}, therefore applies.  It says that the
face-rotation graph is the Hasse diagram of the finite distributive
lattice of height functions; see also \cite{Desreux} in the lozenge
setting.  The least and greatest heights in that lattice are
\(L_p\) and \(U_p\) by
Lemma~\ref{lem:directed-extremal-extension} and
Definition~\ref{def:bottom-top-matchings}.
By Lemma~\ref{lem:seam-cut-invariance}, the seam-cut matching
bijection preserves this order and its face rotations; it is therefore
an isomorphism of the two distributive lattices.
\end{proof}

\begin{proposition}[The polygonal matching is the bottom matching]
\label{prop:polygonal-bottom-matching}
The perfect matching \(M_{\rm poly}\) of
Definition~\ref{def:polygonal-matching} is \(M_{\min}\).
\end{proposition}

\begin{proof}
By Lemma~\ref{lem:polygonal-bottom-no-negative},
\(M_{\rm poly}\) admits no negative face rotation.  By
Lemma~\ref{lem:flip-connectivity}, every nonbottom element has an
immediate predecessor, and the corresponding Hasse edge is a negative
face rotation.  Therefore \(M_{\rm poly}=M_{\min}\).
\end{proof}

\begin{proposition}[The polygonal top matching is the top matching]
\label{prop:polygonal-top-matching}
The perfect matching \(M_{\rm poly}^{+}\) of
Definition~\ref{def:polygonal-top-matching} is \(M_{\max}\).
\end{proposition}

\begin{proof}
Lemma~\ref{lem:polygonal-top-no-positive} rules out a positive face
rotation from \(M_{\rm poly}^{+}\).  Every nontop element has an
immediate successor, which is a positive face rotation by
Lemma~\ref{lem:flip-connectivity}.  Hence
\(M_{\rm poly}^{+}=M_{\max}\).
\end{proof}

Combining these two propositions with
Lemma~\ref{lem:exact-polygonal-height-gap} gives the exact formulas
\eqref{eq:gap-panel-012}--\eqref{eq:gap-middle-cases} for
\((h_{\rm top}-h_{\rm bot})/3\).  Thus the explicit calculation is
consistent with, but is no longer needed to prove, the strict-width
Lemma~\ref{lem:strict-extremal-width}.

\section{Step IV: The dimer \texorpdfstring{\(F\)-polynomial}{F-polynomial}
  and boundary incidences}
\label{sec:corrected-boundary-weights}
\label{sec:step4-fpoly}

\paragraph{Goal of Step~IV.}
Well-matchedness allows us to use the dimer face polynomial.  We identify
it with the cluster \(F\)-polynomial and track how each face rotation
changes the incidences at every initial fan vertex.  This produces the
matching-dependent exponents, including the frozen boundary terms.

We use the separation formula
\eqref{eq:cluster-separation-dictionary} and the exponent dictionary
\eqref{eq:matching-exponent-dictionary}.  Thus \(B=(b_{vu})\) has
rows indexed by all initial fan variables and mutable columns;
\(g(Z):=g_Z\) is the mutable \(g\)-vector, and \(\beta_Z\) is the
separate frozen normalization defined in
Section~\ref{sec:cluster-dictionary}.  We apply this notation to the
dimer variable \(Z_p\) constructed below, without presupposing the
identification \(Z_p=X_p\), which is the task of Step~V.

\paragraph{Canonical height functions.}
Flip connectivity and the identification of the polygonal bottom and
top matchings were established, after well-matchedness, in
Lemma~\ref{lem:flip-connectivity} and
Propositions~\ref{prop:polygonal-bottom-matching} and
\ref{prop:polygonal-top-matching}.

By Theorem~\ref{thm:well-matched}, \(G_p\) is well-matched.  Let
\(r_u(M)\) be the height of a perfect matching \(M\) above
\(M_{\min}\) at the face \(F_u\).  The dimer face polynomial is
\[
 D_{G_p}(y)=
   \sum_{M\in\D(G_p)}\prod_{u\in S_p}y_u^{r_u(M)}.
\]
Theorem~6.8 of \cite{MMSBV} says that this is an \(F\)-polynomial for
the cluster algebra with the white-on-the-right initial quiver
\(Q_0^{\rm R}[S_p]\), and determines
the corresponding \(g\)-vector on \(S_p\).

From this point onward all raw incidences are measured on the
\emph{uncut} graph \(\overline G_p\).  We transport a matching of
\(G_p\) to \(\overline G_p\) by
Theorem~\ref{thm:fixed-graph}.  For any lattice vertex
\(v\in\Sigma _0\), let \(E_v\) be the set of edges of
\(\overline G_p\) crossing elementary lattice edges incident to \(v\),
and put
\[
 m_v(M):=|M\cap E_v|.
\]
For a mutable \(v\), this is
\(|M\cap\partial F_v|\), with only the retained part of the boundary
present when \(v\notin S_p\).  For a frozen boundary vertex it is the
number of matching edges crossing lattice edges incident to \(v\).
Put
\[
 w_v:=\frac{|\partial_{\overline G_p}F_v|}{2}\quad(v\in S_p),
 \qquad
 m_v^{\max}:=\max_{M\in\D(\overline G_p)}m_v(M).
\]
Thus \(w_v\) and \(m_v\) are never mixed between the cut and uncut
graphs.  Lemma~\ref{lem:seam-cut-invariance} identifies only the
invariant combination \(w_v-m_v(M)-1\), the face heights, and the
matching lattice with their cut-open counterparts.

\begin{lemma}[Incidence change under flips]
\label{lem:incidence-change}
For every \(v\in\Sigma _0\) and every perfect matching \(M\),
\begin{equation}
\sum_{u\in S_p}b_{vu}r_u(M)
       =m_v(M_{\min})-m_v(M).
\label{eq:incidence-height}
\end{equation}
\end{lemma}

\begin{proof}
Both sides vanish for \(M=M_{\min}\).  Suppose that \(M'\) covers
\(M\) in the dimer lattice, so \(M'\) is obtained by a positive
rotation at \(F_u\) and
\(r(M')=r(M)+\mathbf e_u\).  With the white-on-the-right convention
defining \(Q_0^{\rm R}\),
the local rule is as follows.  If \(v\longrightarrow u\), then
\(b_{vu}=1\), and the positive rotation removes one crossing, so
\(m_v\) decreases by \(1\).  If \(u\longrightarrow v\), then
\(b_{vu}=-1\), and the rotation adds one crossing, so \(m_v\)
increases by \(1\).  Otherwise the dual incidence and the change in
\(m_v\) are both zero.  Therefore
\[
                 m_v(M')-m_v(M)=-b_{vu},
\]
including for a frozen row \(v\).
At a cut seam the inserted zigzag edges cross no lattice edge, while
the two opposite dual arrows cancel, so both quantities are again
zero.  Thus both sides of
\eqref{eq:incidence-height} change by \(b_{vu}\).  Flip connectivity
now proves the identity.
\end{proof}

\begin{proposition}[Lift to the full mutable fan quiver]
\label{prop:lifted-graph-g}
The mutation sequence producing \(D_{G_p}\) on the induced quiver
\(Q_0^{\rm R}[S_p]\) lifts to the full mutable fan quiver
\(Q_0^{\rm R}\).  The resulting
cluster variable \(Z_p\) has \(F\)-polynomial \(D_{G_p}\) and mutable
\(g\)-vector
\begin{equation}
g^{G_p}(v)=
\begin{cases}
w_v-m_v(M_{\min})-1,&v\in S_p,\\
m_v^{\max}-m_v(M_{\min}),
  &v\in I_{\rm mut}\setminus S_p,\ E_v\ne\varnothing,\\
0,&v\in I_{\rm mut}\setminus S_p,\ E_v=\varnothing.
\end{cases}
\label{eq:predicted-g}
\end{equation}
\end{proposition}

\begin{proof}
Theorem~6.8 of \cite{MMSBV} gives a mutation sequence on
\(Q_0^{\rm R}[S_p]\), the \(F\)-polynomial \(D_{G_p}\), and the first line of
\eqref{eq:predicted-g}.  Proposition~5.11 of \cite{MMSBV} applies the
same sequence in the full mutable quiver, preserves the
\(F\)-polynomial, and preserves these \(g\)-coordinates.

Let \(v\in I_{\rm mut}\setminus S_p\), and let \(\Sigma_m\) be the
terminal seed of the mutation sequence.  Since the sequence mutates
only in directions belonging to \(S_p\), it never mutates in direction
\(v\), so \(X_v\) remains in every intervening cluster.  In particular,
the terminal cluster contains both \(X_v\) and the distinct cluster
variable \(Z_p\).  Equivalently, the \(v\)-component of the
\(d\)-vector of \(Z_p\) relative to \(\Sigma_m\) is zero.  Changing the
reference seed by a mutation in a direction different from \(v\) does
not change this component \cite[Proposition~2.8]{CaoLiDenominators}.
Reversing the mutation sequence therefore gives
\[
 d_v^{\Sigma_0}(Z_p)=0.
\]
Thus the reduced numerator is not divisible by \(X_v\).  By
Lemma~\ref{lem:incidence-change}, the powers of \(X_v\) occurring in
\(D_{G_p}(\widehat y)\) are
\[
             m_v(M_{\min})-m_v(M).
\]
Their minimum is
\(m_v(M_{\min})-m_v^{\max}\).  In the reduced Laurent expression the
distinguished mutable monomial must cancel precisely this greatest
common monomial factor; otherwise either a negative denominator power
or a common numerator factor of \(X_v\) would remain.  Hence
\[
             g^{G_p}(v)=m_v^{\max}-m_v(M_{\min}).
\]
When \(E_v=\varnothing\), all the displayed powers are zero and so is
the \(g\)-coordinate.  This is the lowest-term argument used in the
proof of Theorem~7.2 of \cite{MMSBV}.
\end{proof}

After substituting \(\widehat y\), the mutable exponent attached to a
matching is
\begin{equation}
\epsilon^{G_p}_v(M)
   =g^{G_p}(v)+\sum_{u\in S_p}b_{vu}r_u(M)
 =
\begin{cases}
w_v-m_v(M)-1,&v\in S_p,\\
m_v^{\max}-m_v(M),
  &v\in I_{\rm mut}\setminus S_p,\ E_v\ne\varnothing,\\
0,&v\in I_{\rm mut}\setminus S_p,\ E_v=\varnothing.
\end{cases}
\label{eq:explicit-exponent}
\end{equation}
For a frozen boundary vertex there is no \(g\)-coordinate.  Its full
matching exponent is instead
\begin{equation}
\epsilon^{\partial}_v(M)
 =\beta_{Z_p}(v)+m_v(M_{\min})-m_v(M),
 \qquad v\in I_{\rm fr}.
\label{eq:boundary-exponent}
\end{equation}

\paragraph{Output of Step~IV.}
For every matching, the dimer height vector now determines the complete
mutable and frozen exponent field through the local incidence identity.
What remains is to identify which cluster variable this dimer polynomial
represents.  That is the purpose of the five-point extended-exponent calculation.

\section{Step V: Identifying the matching formula with the target variable}
\label{sec:step5-five-point}

\paragraph{What remains after Step~IV.}
Step~IV produced from \(G_p\) a cluster variable \(Z_p\) whose
\(F\)-polynomial is the perfect-matching generating polynomial.
We have not yet proved that \(Z_p\) is the target
octahedron variable \(X_p\).

There are two logically separate calculations.  On the recurrence
side, we compute the mutable \(g\)-vector of \(X_p\), together with
its separate frozen normalization.  On the graph side, we compute
the mutable \(g\)-vector of \(Z_p\) from the incidences of the
polygonal bottom matching constructed in Step~III.

Recall the five initial vertices \(v_0,\ldots,v_4\) from
\eqref{eq:five-g-points}.  We shall prove
\begin{equation}
 g(X_p)=g(Z_p)
   =-\mathbf e_{v_1}+\mathbf e_{v_2}
      -\mathbf e_{v_3}+\mathbf e_{v_4},
 \qquad
 \beta_p=\mathbf e_{v_0}.
\label{eq:step-v-two-fingerprints}
\end{equation}
The equality of the mutable \(g\)-vectors identifies \(Z_p\) with
\(X_p\) in the principal-coefficient algebra.  Only after that
identification do we specialize the coefficients; the frozen
normalization is therefore needed only on the recurrence side.

\subsection{Pictorial roadmap and the KTW face-counting convention}

Before carrying out the two calculations, we give the complete picture
in one worked example.  The picture is the polygonal matching used on
the graph side and also displays the five surviving exponents computed on the recurrence
side.  It also fixes a point of terminology that is easy to misread in
a lozenge picture.  A \emph{dual matching edge} is an elementary
triangular-lattice side separating the two triangles of a lozenge; the
corresponding \emph{matching edge} joins the centers of those triangles
and crosses that side.  Incidence counts below count the latter
center-to-center edges.

\paragraph{Cyclic incidence words.}
Recall that every mutable lattice vertex \(v\) labels a face \(F_v\)
of the full dual graph \(\Gamma_0\).  We call \(F_v\)
\emph{selected} when \(v\in S_p\), equivalently when \(v\) lies in
the positive denominator support of \(X_p\).  In that case the whole
boundary \(\partial F_v\) is retained in \(\overline G_p\), and
\(F_v\) remains a bounded face of the matching graph.  If
\(v\notin S_p\), we call the ambient face \(F_v\) \emph{unselected}:
it is not retained as a face of the matching graph, although some of
its sides may remain because they also bound neighboring selected
faces.  Those surviving sides form the incidence set \(E_v\), which
may be a path rather than a cycle.

Let \(M\in\D(G_p)\), let \(\overline M\) be its image after contracting
the inserted seam zigzags, and let \(v\in S_p\).  List the edges of
\(E_v\) in counterclockwise cyclic order around the selected face
\(F_v\):
\[
                 E_v=(e_1,e_2,\ldots,e_{2w_v}).
\]
The notation \(2w_v\) is justified by bipartiteness.  As one travels
around the closed boundary \(\partial F_v\), the graph vertices
alternate black, white, black, white.  One can return to the starting
vertex only after an even number of boundary edges, so
\(|E_v|=|\partial F_v|=2w_v\).  Before the seams are opened, the
selected faces are four-sided rhombi or six-sided hexagons; opening a
seam adds two boundary edges to the relevant face and therefore
preserves this parity.
The \emph{cyclic incidence word} of \(M\) at \(v\) is the binary word
\[
 \omega_v(M)=\omega_1\omega_2\cdots\omega_{2w_v},
 \qquad
 \omega_j=
 \begin{cases}
  1,&e_j\in\overline M,\\
  0,&e_j\notin\overline M.
 \end{cases}
\]
There is no distinguished first edge of \(F_v\), so this word is
regarded up to cyclic shift.  Its number of \(1\)'s is exactly the
incidence count:
\[
                       m_v(M)=\sum_{j=1}^{2w_v}\omega_j.
\]
Thus a digit records whether the center-to-center matching edge crossing
the corresponding retained lattice side is present; it does not record
the color or orientation of the lozenge.  If \(v\notin S_p\) and
\(E_v\) is a path rather than the boundary of a face, we use the
analogous \emph{linear} incidence word, ordered from the row-entry end,
and do not take it modulo cyclic shift.  Because this is a path rather
than a closed boundary, its length need not be even.  Inserted
seam-zigzag edges are excluded in both conventions.

The reference matching \(M_{\rm poly}=M_{\min}\) was defined and
identified as the bottom matching in
Definition~\ref{def:polygonal-matching} and
Proposition~\ref{prop:polygonal-bottom-matching}.  The following
figure records that polygonal matching together with the five vertices
used in the recurrence calculation.

\begin{figure}[htbp]
\centering
\includegraphics[height=0.78\textheight]
  {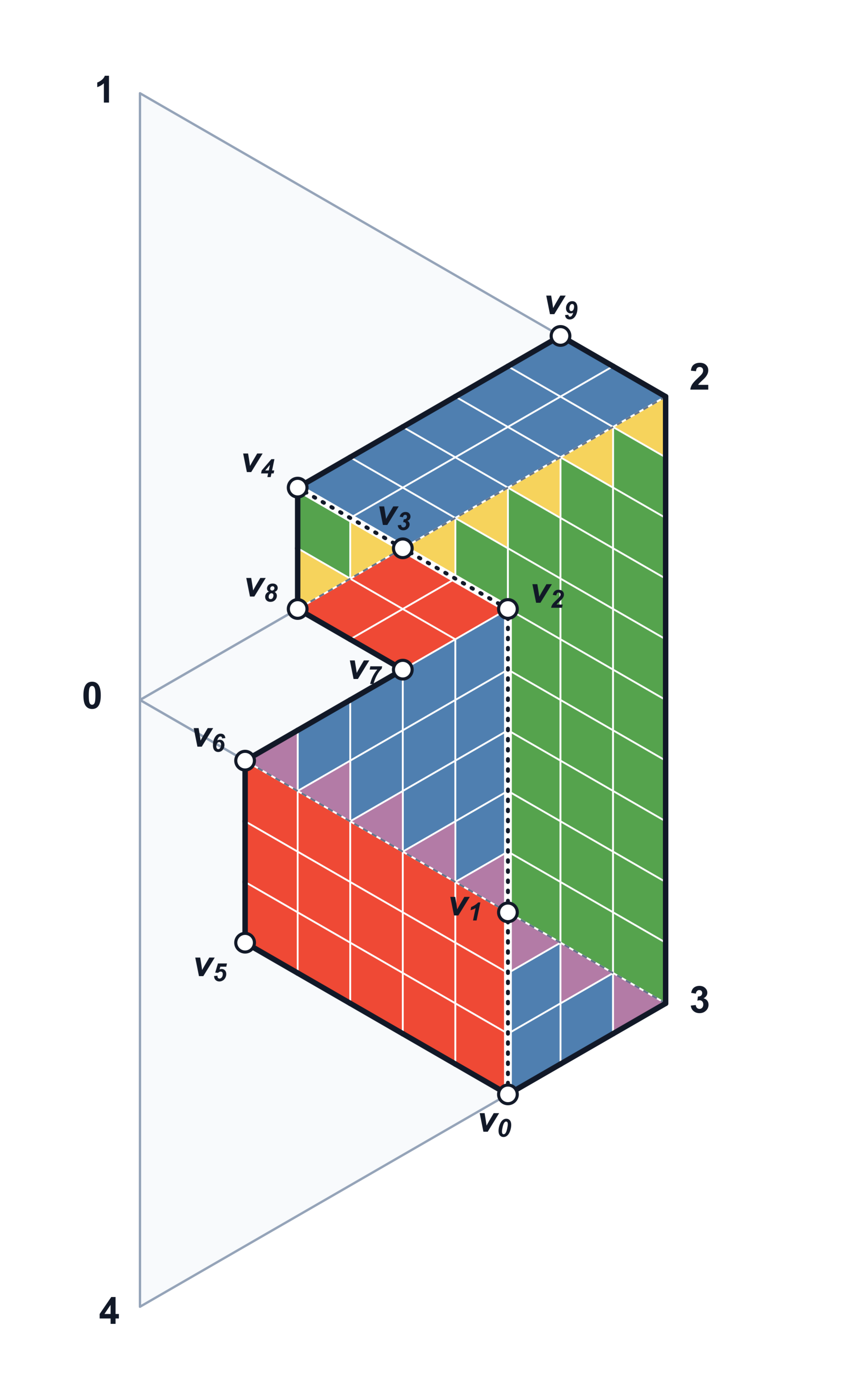}
\caption{The polygonal matching \(M_{\rm poly}\) for
\(p=(2,3,2,3)_{1234}\).  The ambient vertices
\(u_0,\ldots,u_4\) are abbreviated in the figure to
\(0,\ldots,4\), and all polygon vertices \(v_0,\ldots,v_9\) are shown.
The red, blue, and green regions are filled simultaneously by the
corresponding lozenge orientations; yellow and purple denote the
singleton seam triangles.  The dotted path
\(\mathcal P_p:v_0\to v_1\to v_2\to v_3\to v_4\) is not part of the
matching.  Proposition~\ref{prop:polygonal-bottom-matching} proves
that \(M_{\rm poly}=M_{\min}\).}
\label{fig:step-v-bottom-matching}
\end{figure}

The broken path drawn on the matching,
\[
 \mathcal P_p:\quad
 v_0\longrightarrow v_1\longrightarrow v_2
       \longrightarrow v_3\longrightarrow v_4,
\]
is the union of the four internal color interfaces.  It is not a
matching path.  Its endpoint \(v_0\) is frozen, whereas
\(v_1,v_2,v_3,v_4\) are mutable.  Thus the same picture separates the location of the frozen endpoint
\(v_0\) from the four mutable vertices that form the \(g\)-vector;
the coefficient at \(v_0\) is computed only on the recurrence side.

\begin{proposition}[Selected-face KTW form]
\label{prop:selected-face-ktw-form}
Let \(M\in\D(G_p)\), transport it to the uncut graph
\(\overline G_p\), and write \(\mu=\overline M\).  For
\(v\in S_p\), let \(a_\mu(F_v)\) be the number of matching edges of
\(\mu\) adjacent to \(F_v\).  Then the exponent of \(X_v\) in the
matching term indexed by \(M\) is
\begin{equation}
 \epsilon^{G_p}_v(M)=
 \begin{cases}
  1-a_\mu(F_v),&F_v\text{ is four-sided},\\
  2-a_\mu(F_v),&F_v\text{ is six-sided}.
 \end{cases}
\label{eq:selected-face-ktw-form}
\end{equation}
The same exponent is obtained after cutting open the seams and inserting
the alternating seam zigzags.
\end{proposition}

\begin{proof}
Equation~\eqref{eq:explicit-exponent}, proved in Step~IV from the
separation formula, Proposition~\ref{prop:lifted-graph-g}, and
Lemma~\ref{lem:incidence-change}, gives
\[
                 \epsilon^{G_p}_v(M)=w_v-m_v(M)-1
                 \qquad(v\in S_p).
\]
By Theorem~\ref{thm:fixed-graph}, the complete boundary of every
selected face is retained.  Hence
\[
                         E_v=\partial F_v,
                 \qquad m_v(M)=a_\mu(F_v).
\]
The faces of the initial fan dual graph are four-sided rhombi or
six-sided hexagons, and retaining a complete boundary introduces no
other selected-face type.  Thus a four-sided face has \(w_v=2\), while
a six-sided face has \(w_v=3\).  Substitution in the preceding formula
gives respectively
\[
             2-a_\mu(F_v)-1=1-a_\mu(F_v),
       \qquad
             3-a_\mu(F_v)-1=2-a_\mu(F_v),
\]
which proves \eqref{eq:selected-face-ktw-form} on \(\overline G_p\).

The cut-open assertion follows from
Lemma~\ref{lem:seam-cut-invariance}, which preserves
\(w_v-m_v(M)-1\) under the alternating seam-zigzag completion.
\end{proof}

We now compare this independently proved formula with the notation of
Knutson--Tao--Woodward \cite[\S7, pp.~15--16]{KTWOR}.  Contract the
inserted seam zigzags and write
\[
                         \mu:=\overline{M}_{\rm poly}
\]
for the resulting matching on the uncut graph.  In their graph with
open faces, let \(m_\mu\) denote the matching Laurent monomial and let
\(a_\mu(F)\) be the number of matching edges of \(\mu\) adjacent to a
face \(F\).  The exponent of the face variable in \(m_\mu\) is
\begin{equation}
 \operatorname{exp}_{F}(m_\mu)=
 \begin{cases}
  1-a_\mu(F),&F\text{ a rhombus or external hexagon},\\
  2-a_\mu(F),&F\text{ an interior hexagon}.
 \end{cases}
\label{eq:ktw-matching-monomial-exponent}
\end{equation}
Proposition~\ref{prop:selected-face-ktw-form} proves that the
selected-face exponents in our graph have exactly this KTW form; it is
not being assumed from the special graph treated in
\cite{KTWOR}.  The proposition makes no assertion about an unselected
mutable vertex or a frozen variable, whose exponents are instead given
by \eqref{eq:explicit-exponent} and \eqref{eq:boundary-exponent}.
Thus there is no separate ``neutral count'' to compare with an
``actual count'': the face-dependent offset \(1\) or \(2\) and the
adjacent-edge count are packaged into a single exponent.  For a
selected face \(F_v\) in our graph this is exactly
\[
 \operatorname{exp}_{F_v}(m_\mu)
       =w_v-m_v(M_{\rm poly})-1.
\]

\begin{remark}[The reentrant corner is not a sixth exceptional vertex]
The only point at which the outline of \(\mathcal A_p\) is not locally
convex is the middle-panel corner \(v_7\), namely
\[
             q_*=(a+d,0,b,c,0)\in\Delta_{023}(N)
\]
from the coordinate definition of the reference tiling preceding
Lemma~\ref{lem:boundary-datum-compatible}.  This point must be classified
before choosing between the selected-face formula
\eqref{eq:selected-face-ktw-form} and the unselected-vertex line of
\eqref{eq:explicit-exponent}.  By
\eqref{eq:Sp-formula}, a point \(q\in\Delta_{023}\) belongs to \(S_p\)
precisely when \(q_0\geq1\) and
\(q_2\geq b+1\) or \(q_3\geq c+1\).  At \(q_*\) one has
\(q_2=b\) and \(q_3=c\).  Hence
\[
                    q_*\in I_{\rm mut}\setminus S_p.
\]
In particular, \(q_*\) is not an interior selected hexagon.  Its
mutable exponent is governed by the unselected-vertex line of
\eqref{eq:explicit-exponent}:
\begin{equation}
 \epsilon^{G_p}_{q_*}(M)
       =m_{q_*}^{\max}-m_{q_*}(M).
\label{eq:reentrant-corner-matching-exponent}
\end{equation}

The local support and occupancy calculation in
Lemma~\ref{lem:reference-matching-incidence-calculation} gives
\(m_{q_*}(M_{\rm poly})=m_{q_*}^{\max}=2\).  Thus its bottom
exponent is zero, and the reentrant corner contributes no additional
basis vector to \eqref{eq:step-v-two-fingerprints}.

In the terminology of Knutson--Tao--Woodward, the corresponding
open/exterior face uses the \(1-a_\mu(F)\) rule rather than the
interior-hexagon \(2-a_\mu(F)\) rule.  That raw external-face factor
is not, by itself, the exponent of the unselected initial fan variable
in our lifted cluster expression: the latter is normalized by the
unselected line of \eqref{eq:explicit-exponent}, namely
\eqref{eq:reentrant-corner-matching-exponent}.

Finally, under the standing assumption \(a,b,c,d\geq1\), \(q_*\)
cannot coincide with any of \(v_0,\ldots,v_4\).  Indeed, respectively,
the separating coordinates are
\[
 q_4:0\ne d,\qquad q_2:b\ne0,\qquad
 q_0:a+d\ne d,\qquad q_3:c\ne0,\qquad q_1:0\ne a.
\]
Coincidences can occur only after leaving the strictly interior
setting; for example, if \(a=0\), then \(q_*=v_2\).  In such a
degeneration the local graph must be recounted after the faces merge,
rather than adding the two nominal contributions from the
nondegenerate formula.
\end{remark}

\subsection{The graph-side reading of the five marked vertices}

We now justify directly what was drawn in
Figure~\ref{fig:step-v-bottom-matching}.  Step~III already proved that
the displayed polygonal matching is \(M_{\min}\).  It remains here to
read its cyclic and linear incidence words.

\begin{lemma}[Incidences of the polygonal bottom matching]
\label{lem:reference-matching-incidence-calculation}
\label{lem:local-boundary-classification}
For selected vertices,
\begin{equation}
 w_v-m_v(M_{\rm poly})-1=
 \begin{cases}
  -1,&v=v_1\ \text{or}\ v=v_3,\\
  +1,&v=v_2,\\
  0,&v\in S_p\setminus\{v_1,v_2,v_3\}.
 \end{cases}
\label{eq:reference-selected-incidences}
\end{equation}
For unselected mutable vertices meeting the graph,
\begin{equation}
 m_v^{\max}-m_v(M_{\rm poly})=
 \begin{cases}
  +1,&v=v_4,\\
  0,&v\in I_{\rm mut}\setminus(S_p\cup\{v_4\}),
       \quad E_v\ne\varnothing.
 \end{cases}
\label{eq:reference-unselected-incidences}
\end{equation}
Moreover,
\begin{equation}
                         E_{v_6}=E_{v_8}=\varnothing.
\label{eq:seam-endpoint-empty-incidences}
\end{equation}
\end{lemma}

\begin{proof}
The local occupancy counts were established in the proof of
Lemma~\ref{lem:polygonal-bottom-no-negative}.  At a generic selected
six-sided face, whether in a monochromatic interior or on a straight
interface, \(w_v=3\) and exactly two sides are occupied, so \(m_v=2\).
At a generic selected four-sided seam face, \(w_v=2\) and the word
\(1000\) gives \(m_v=1\).  Both yield \(w_v-m_v-1=0\).

At \(v_1\) and \(v_3\), the positive word \(1010\) gives
\[
m_{v_1}(M_{\rm poly})=m_{v_3}(M_{\rm poly})=2,
\]
so \(w_v-m_v-1=2-2-1=-1\).  At \(v_3\), the red and the blue
lozenge each contribute an occupied incidence: color specifies orientation,
not whether the matching edge is counted.  At \(v_2\), the word
\(100000\) gives
\[
                  m_{v_2}(M_{\rm poly})=1,
\]
and \(w_{v_2}-m_{v_2}-1=3-1-1=+1\).  This proves
\eqref{eq:reference-selected-incidences}.

We make the unselected-boundary calculation explicit.  For an
unselected lattice vertex \(v\), let
\[
 D_v^{\rm ord}:=\{u-v:u\sim v,\ u\in S_p,\
                 [v,u]\not\subset
                 \Delta_{02}(N)\cup\Delta_{03}(N)\},
\]
where \(u\sim v\) means that \(u\) and \(v\) are joined by a unit
lattice edge.  If \([v,u]\) is not contained in a seam, then, since
\(v\notin S_p\), the dual edge crossing \([v,u]\) belongs to \(E_v\)
if and only if \(u\in S_p\).  Along a seam, however, the adjacent
same-colored dual vertices are identified in \(\Gamma_0\), rather than
joined by a dual edge.  Thus a seam-parallel direction contributes no
edge to \(E_v\), and \(D_v^{\rm ord}\) determines \(E_v\).

Suppose first that \(v\) lies in the relative interior of a straight
side of the support.  Substitution in \eqref{eq:Sp-formula} gives the
following complete list of inward neighbor directions; rows
corresponding to sides which are absent in a degenerate panel are
simply omitted:
\[
\begin{array}{c|c|c}
\text{panel}&\text{active equality at }v&D_v^{\rm ord}\\ \hline
012&q_0=0&\{e_0-e_1,e_0-e_2\}\\
012&q_1=a&\{e_0-e_1,e_2-e_1\}\\
012&q_2=b&\{e_2-e_0,e_2-e_1\}\\
023&q_0=0&\{e_0-e_2,e_0-e_3\}\\
023&q_2=b&\{e_2-e_0,e_2-e_3\}\\
023&q_3=c&\{e_3-e_0,e_3-e_2\}\\
034&q_0=0&\{e_0-e_3,e_0-e_4\}\\
034&q_3=c&\{e_3-e_0,e_3-e_4\}\\
034&q_4=d&\{e_0-e_4,e_3-e_4\}.
\end{array}
\]
Thus \(E_v\) consists of two consecutive edges.  Reading the fixed
lozenge orientation prescribed in the seven-region definition of
\(M_{\rm ref}=M_{\rm poly}\), exactly one of these two
edges is occupied by \(M_{\rm poly}\).  Equivalently, its linear
incidence word is \(10\), up to reversal.

At the five corners, the same substitution gives
\[
\begin{array}{c|c|c|c}
v&D_v^{\rm ord}&|E_v|&
 \text{linear incidence word of }M_{\rm poly}\\ \hline
v_5&\{e_3-e_4\}&1&1\\
v_6&\varnothing&0&\varnothing\\
v_7&
 \{e_2-e_3,e_2-e_0,e_3-e_0,e_3-e_2\}&4&
 1010\ \text{up to reversal}\\
v_8&\varnothing&0&\varnothing\\
v_9&\{e_0-e_1\}&1&1.
\end{array}
\]
The final column is read directly from the seven fixed-orientation
regions defining \(M_{\rm ref}=M_{\rm poly}\): at
\(v_5,v_9\) the unique retained ordinary side is crossed by the corner
lozenge, while at the reentrant corner \(v_7\) the red and blue
fillings cross two nonconsecutive retained sides.  At \(v_6\), the sole
direction toward \(S_p\) is \(e_3-e_0\), and at \(v_8\) it is
\(e_2-e_0\).  Both directions lie along an identified seam.  They
therefore give no edge of \(\overline G_p\), proving
\eqref{eq:seam-endpoint-empty-incidences}.  The inserted zigzag edges
in the cut-open graph cross no lattice side and do not alter this
conclusion.

For every straight-side type and for \(v_5,v_7,v_9\), it follows that
\[
 m_v(M_{\rm poly})=\left\lceil\frac{|E_v|}{2}\right\rceil .
\]
On the other hand, \(E_v\) is a path, and two consecutive edges of
that path share a graph vertex.  The intersection of any perfect
matching with \(E_v\) is therefore a matching of a path with
\(|E_v|\) edges, and hence has cardinality at most
\(\lceil |E_v|/2\rceil\).  Thus
\[
 m_v(M_{\rm poly})=m_v^{\max}.
\]
This proves the zero line of
\eqref{eq:reference-unselected-incidences}.
The vertices \(v_6,v_8\) instead fall under the empty-incidence line of
\eqref{eq:predicted-g} and \eqref{eq:explicit-exponent}, so their
mutable \(g\)-coordinates and matching exponents are zero.

Finally, \(v_4\notin S_p\).  Its sole retained incidence is absent from
\(M_{\rm poly}\).  By
Theorem~\ref{thm:well-matched}, that edge occurs in some perfect
matching.  Consequently
\[
 m_{v_4}(M_{\rm poly})=0,\qquad m_{v_4}^{\max}=1,
\]
which proves the remaining line.
\end{proof}

\begin{corollary}[The four mutable signs read from the bottom matching]
\label{cor:graph-five-point-g}
The cluster variable \(Z_p\) produced by the dimer theorem and lifted
to the full fan quiver satisfies
\[
 g(Z_p)
   =-\mathbf e_{v_1}+\mathbf e_{v_2}
      -\mathbf e_{v_3}+\mathbf e_{v_4}.
\]
\end{corollary}

\begin{proof}
Proposition~\ref{prop:lifted-graph-g} says that these local matching
exponents are the coordinates of \(g(Z_p)\).  The selected-face rule
reads \(1-2=-1\) at \(v_1\), \(2-1=+1\) at \(v_2\), and
\(1-2=-1\) at \(v_3\).  The unselected-vertex rule reads
\(1-0=+1\) at \(v_4\).  All other nonempty incidence words give zero by
Lemma~\ref{lem:reference-matching-incidence-calculation}; empty
incidence sets, including those at \(v_6,v_8\), give zero by
\eqref{eq:predicted-g}.  Reading the
four mutable signs from left to right along \(\mathcal P_p\) therefore
gives the displayed vector.
\end{proof}

\subsection{The recurrence-side reading of the same five vertices}

We now compute the distinguished exponent without propagating
\(c\)-vectors.  The geometric idea is to place an extremal
Gelfand--Tsetlin pattern on each initial panel, use the resulting
strict tropical branch choices, and then run those choices backwards.
The backward dependence telescopes to the five vertices marked in
Figure~\ref{fig:step-v-bottom-matching}.

\paragraph{An explicit all-right test.}
A point of the closed coefficient chamber at which all recurrence branches are strict suffices to read
the intrinsic \(g\)-exponent.  We now choose such a point explicitly,
including the large-gap offsets that turn the three
Gelfand--Tsetlin patterns into compatible hives.

\begin{proposition}[The all-right chamber]
\label{prop:all-right-chamber}
There is an explicit centered large-gap initial fan height \(h_R\),
made from three all-right Gelfand--Tsetlin patterns, for which
\[
                         B^{\mathsf T}h_R\leq0
\]
and for which the three complete tropical sweeps have strict branch
word
\[
                    \mathsf{NS},\mathsf{NS},\mathsf{EW}.
\]
At this specialization the Laurent monomial arising from the constant
term of \(F_{X_p}\) is the unique tropical maximizer at \(p\), and its
complete exponent is
\[
 \widetilde g_p
 =\mathbf e_{v_0}-\mathbf e_{v_1}
  +\mathbf e_{v_2}-\mathbf e_{v_3}+\mathbf e_{v_4}.
\]
Consequently
\[
 \beta_p=\mathbf e_{v_0},\qquad
 g(X_p)=-\mathbf e_{v_1}+\mathbf e_{v_2}
        -\mathbf e_{v_3}+\mathbf e_{v_4}.
\]
\end{proposition}

\begin{proof}
We first construct the promised numerical witness.  In
Definition~\ref{def:rf-centered-large-gap}, take \(n=N\) and
\[
 X_1=X_2=X_3=D_N:=\operatorname{diag}(1,2,\ldots,N),
 \qquad K_{\rm gap}=6.
\]
Indeed,
\[
 2\max_{\varnothing\ne I\subseteq\{1,2,3\}}
       \operatorname{width}(X_I)
 =6(N-1)<6N=NK_{\rm gap},
\]
so this choice is admissible.  The centered background and its prefix
sums are
\[
 \widehat\Lambda_q=3N(N+1-2q),
 \qquad
 C_j^{(0)}:=\sum_{q=1}^j\widehat\Lambda_q=3Nj(N-j).
\]
Put
\[
                         T_m:=\frac{m(m+1)}2,
                         \qquad T_0=0.
\]
The full decreasing top row is
\(\lambda_k^{(s)}=N-k+1\), while the leading \(j\)-minor of \(D_N\)
has row \((j,j-1,\ldots,1)\).  This is exactly the all-right pattern,
and its reversed top row is
\[
 \xi_m^{(s)}:=\lambda_{N-m+1}^{(s)}=m,
 \qquad 1\leq m\leq N.
\]
The successive incoming prefixes are
\[
                         C_j^{(s)}=C_j^{(0)}+sT_j,
                         \qquad s=1,2,3,
\]
and the all-right panel formula is
\begin{equation}
 h_R([i,j]_s)=G_N([i,j]_s)+sT_j-T_{j-i},
 \qquad
 G_N(q):=3Nq_0(N-q_0).
\label{eq:all-right-explicit-panel}
\end{equation}
Indeed, \(q_0=N-j\) at \([i,j]_s\), so
\(C_j^{(0)}=G_N([i,j]_s)\).  The formulas glue: the value at
\([j,j]_s\) and the value at \([0,j]_{s+1}\) are both
\(G_N+sT_j\).  The admissibility just checked and
Lemma~\ref{lem:rf-centered-large-gap-valid} therefore make \(h_R\) a
genuine compatible initial hive.

We first place this \(h_R\) in the closed initial coefficient chamber for
the white-on-the-right quiver \(Q_0^{\rm R}\).  At a mutable vertex in
the relative interior of one panel, its exchange-matrix column is the
alternating six-neighbor sum.  In the
coordinates above it is, up to simultaneously reversing the displayed
order,
\[
\begin{aligned}
 (B^{\mathsf T}h_R)_{[i,j]_r}
={}&h_R([i,j-1]_r)+h_R([i-1,j]_r)
   +h_R([i+1,j+1]_r)\\
 &-h_R([i,j+1]_r)-h_R([i+1,j]_r)
   -h_R([i-1,j-1]_r).
\end{aligned}
\]
Substitution of \eqref{eq:all-right-explicit-panel} cancels the six
large-gap gauge terms and the six triangular-number terms.  Hence
\begin{equation}
 (B^{\mathsf T}h_R)_u=0
 \qquad
 \text{for \(u\) in the relative interior of a panel.}
\label{eq:all-right-interior-chamber}
\end{equation}

The only remaining mutable vertices of the initial fan lie on the
two internal seams.  Write
\[
 z_k^{02}=(N-k,0,k,0,0)\in02,\qquad
 z_k^{03}=(N-k,0,0,k,0)\in03,
 \qquad 1\leq k\leq N-1.
\]
At the base octahedron over \(z_k^{02}\), the \(Q_0^{\rm R}\) convention says that
the column weight is the east--west product weight minus the
north--south product weight.  Formula
\eqref{eq:all-right-explicit-panel} gives
\[
 (B^{\mathsf T}h_R)_{z_k^{02}}
 =-\bigl(\xi^{(2)}_{k+1}-\xi^{(2)}_k\bigr)=-1.
\]
The identical calculation at \(z_k^{03}\) gives
\[
 (B^{\mathsf T}h_R)_{z_k^{03}}
 =-\bigl(\xi^{(3)}_{k+1}-\xi^{(3)}_k\bigr)=-1.
\]
Every other boundary vertex is frozen and does not index a column of
\(B\).  Together with
\eqref{eq:all-right-interior-chamber}, these formulas prove
\(B^{\mathsf T}h_R\leq0\).

Write the principal-coefficient \(F\)-polynomial of the target cluster
variable as
\[
 F_{X_p}(y)=\sum_{\rho\in\mathbb N^{I_{\rm mut}}}c_\rho y^\rho,
 \qquad c_0=1.
\]
Here \(c_0=1\) is supplied independently of the later identification
\(X_p=Z_p\) by the skew-symmetric constant-term theorem
\cite[Theorem~1.7]{DWZII}.
After the substitution \(y\mapsto\widehat y\), the term indexed by
\(\rho\) has exponent \(\widetilde g_p+B\rho\).  Its tropical weight
relative to the constant term is
\[
 \langle B\rho,h_R\rangle
 =\langle\rho,B^{\mathsf T}h_R\rangle\leq0.
\]
Thus the constant term is a tropical maximizer.

We next establish strictness of the actual recurrence choices, using
the five ambient coordinates throughout.  For a sweep through
\(\Delta_{ijk\ell}(N)\) from the input faces \(r_j=0\) and
\(r_\ell=0\), an interior output site \(r\) has
\(r_j,r_\ell\geq1\).  Put
\[
\begin{array}{lll}
 r_A=r+e_i-e_\ell,&
 r_B=r-e_j+e_k,&
 r^\vee=r+e_i-e_j+e_k-e_\ell,\\
 r_C=r+e_i-e_j,&
 r_D=r+e_k-e_\ell.&
\end{array}
\]
Then the solved recurrence is
\begin{equation}
 H(r)=\max\{H(r_A)+H(r_B),H(r_C)+H(r_D)\}-H(r^\vee),
\label{eq:all-right-ambient-recurrence}
\end{equation}
where the first and second sums are denoted by \(\mathsf{NS}\) and
\(\mathsf{EW}\), respectively.

The large-gap contribution is only a gauge.  In each of the first two
sweeps, the three opposite-pair sums in
\eqref{eq:all-right-ambient-recurrence} all contain one value of
\(G_N\) with zeroth coordinate \(r_0\) and one with zeroth coordinate
\(r_0+1\).  In the third sweep \(r_0=0\) throughout, so \(G_N=0\).
Consequently
\[
 G_N(r)+G_N(r^\vee)
 =G_N(r_A)+G_N(r_B)=G_N(r_C)+G_N(r_D),
\]
and subtracting \(G_N\) changes neither branch comparison nor
recurrence.  We may therefore use the reduced initial height
\(\bar h_R=h_R-G_N\).

Define the following candidate reduced arrays:
\begin{align}
 \Phi_1(r)
 &=2T_{N-r_0}-T_{r_1}-T_{r_1+r_2},
 &&r\in\Delta_{0123}(N),
\label{eq:all-right-first-array}\\
 \Phi_2(r)
 &=3T_{N-r_0}-2T_{r_1}-T_{r_1+r_3},
 &&r\in\Delta_{0134}(N),
\label{eq:all-right-second-array}\\
 \Phi_3(r)
 &=3T_N-T_{r_1}-T_{r_1+r_2}-T_{r_1+r_2+r_3},
 &&r\in\Delta_{1234}(N).
\label{eq:all-right-third-array}
\end{align}
For each sweep, the value formula and the branch sign are proved
simultaneously.  More precisely, if the two new coordinates are
\(r_j,r_\ell\), define the excavation depth by
\[
                              d(r):=r_j+r_\ell.
\]
The two input faces are known from the outset.  In the strong
induction step, assume the candidate formula at every interior point
of depth less than \(d\), together with the asserted branch sign at
every such point.  The four equatorial arguments in
\eqref{eq:all-right-ambient-recurrence} have depth \(d-1\), and
\(r^\vee\) has depth \(d-2\), so every required value is already
known.  We first compute the branch difference and only then use its
sign to select the update.

For the sweep through \(0123\), the depth is \(r_1+r_3\).  On the
input face \(r_3=0\), formula
\eqref{eq:all-right-first-array} reduces to
\(T_{r_1+r_2}-T_{r_1}=\bar h_R|_{012}\); on \(r_1=0\), it reduces to
\(2T_{r_2+r_3}-T_{r_2}=\bar h_R|_{023}\).  At an interior site put
\[
 u=N-r_0=r_1+r_2+r_3,
 \qquad v=r_1+r_2,
 \qquad b=r_1.
\]
Substitution of the induction hypothesis gives
\[
\begin{aligned}
 &\Phi_1(r_A)+\Phi_1(r_B)-\Phi_1(r^\vee)\\
 &\quad=(2T_{u-1}-T_v-T_b)+(2T_u-T_v-T_{b-1})
        -(2T_{u-1}-T_v-T_{b-1})
 =\Phi_1(r),
\end{aligned}
\]
while, before selecting a branch,
\begin{equation}
 \mathsf{NS}-\mathsf{EW}
 =T_{v-1}+T_{v+1}-2T_v=1.
\label{eq:all-right-first-gap}
\end{equation}
Thus the actual branch is \(\mathsf{NS}\), its output is \(\Phi_1(r)\),
and the simultaneous induction closes.

For the sweep through \(0134\), the depth is \(r_1+r_4\).  The face
\(r_4=0\) carries the already proved value
\[
 \Phi_1|_{013}=2T_{r_1+r_3}-2T_{r_1}
               =\Phi_2|_{r_4=0},
\]
and on \(r_1=0\),
\(\Phi_2=3T_{r_3+r_4}-T_{r_3}=\bar h_R|_{034}\).  At an interior
site put
\[
 u=N-r_0=r_1+r_3+r_4,
 \qquad v=r_1+r_3,
 \qquad b=r_1.
\]
Direct substitution gives
\[
 \Phi_2(r_A)+\Phi_2(r_B)-\Phi_2(r^\vee)=\Phi_2(r)
\]
and
\begin{equation}
 \mathsf{NS}-\mathsf{EW}
 =T_{v-1}+T_{v+1}-2T_v=1.
\label{eq:all-right-second-gap}
\end{equation}
Hence every update in the second sweep is strictly \(\mathsf{NS}\),
and the resulting reduced height is \(\Phi_2\).

For the sweep through \(1234\), the depth is \(r_2+r_4\).  Its input
faces carry the already proved values because
\[
 \Phi_3|_{r_4=0}=2T_N-T_{r_1}-T_{r_1+r_2}=\Phi_1|_{123},
 \qquad
 \Phi_3|_{r_2=0}=3T_N-2T_{r_1}-T_{r_1+r_3}=\Phi_2|_{134}.
\]
At an interior site, substitution gives the east--west recurrence
identity
\[
 \Phi_3(r_C)+\Phi_3(r_D)-\Phi_3(r^\vee)=\Phi_3(r),
\]
whereas, with \(v=r_1+r_2\), the competing difference is
\begin{equation}
 \mathsf{NS}-\mathsf{EW}
 =-\bigl(T_{v-1}+T_{v+1}-2T_v\bigr)=-1.
\label{eq:all-right-third-gap}
\end{equation}
Thus every third-sweep update is strictly \(\mathsf{EW}\), and the
resulting height is \(\Phi_3\).  At each interior site the relevant
index \(v\) lies between \(1\) and \(N-1\), so all three adjacent
second differences are defined, including in the boundary-adjacent
layers.  This proves the strict branch word
\(\mathsf{NS},\mathsf{NS},\mathsf{EW}\).

Initially each variable is a single Laurent monomial.  At every
subsequent exchange, the selected product is strictly heavier than
the other product, and its factors already have unique maximizing
monomials.  Induction through the sweeps therefore shows that the
tropical value at \(p\) has a unique maximizing Laurent monomial.
The constant term was already shown to be a maximizer, so it is this
unique monomial.

It remains to read the exponent.  A numerical equality at the chosen
witness would not by itself determine a Laurent exponent, so we now
propagate the selected branches formally.  Work in the free abelian
group on symbols \(z(q)\), one for each vertex \(q\) of the initial
fan.  The fixed north--south sweep through \(0123\) has the
three-boundary form
\begin{align}
 L_1(r)={}&z(r_0,0,r_1+r_2,r_3,0)
 -z(r_0+r_3,0,r_1+r_2,0,0)\notag\\
 &+z(r_0+r_3,r_1,r_2,0,0),
 \qquad r\in\Delta_{0123}(N).
\label{eq:all-right-formal-first}
\end{align}
The second fixed north--south sweep and the third fixed east--west
sweep have the corresponding forms
\begin{align}
 L_2(r)={}&z(r_0,0,0,r_1+r_3,r_4)
 -z(r_0+r_4,0,0,r_1+r_3,0)\notag\\
 &+L_1(r_0+r_4,r_1,0,r_3,0),
 \qquad r\in\Delta_{0134}(N),
\label{eq:all-right-formal-second}\\
 L_3(r)={}&L_1(0,r_1,r_2,r_3+r_4,0)
 -L_1(0,r_1+r_2,0,r_3+r_4,0)\notag\\
 &+L_2(0,r_1+r_2,0,r_3,r_4),
 \qquad r\in\Delta_{1234}(N).
\label{eq:all-right-formal-third}
\end{align}
Each identity follows by strong induction on the same depth as its
numerical counterpart: on either input face it collapses to the
given boundary symbol, and at an interior site substitution in the
selected linear update cancels the common three boundary terms.
Thus these are identities of coefficient vectors, not merely
equalities after evaluating at \(h_R\).

Set \(p=(0,a,b,c,d)\).  Expanding
\eqref{eq:all-right-formal-second} inside
\eqref{eq:all-right-formal-third} gives
\[
\begin{aligned}
 L_2(0,a+b,0,c,d)
 ={}&z(c+d,a+b,0,0,0)+z(v_2)-z(v_3)\\
    &+z(v_0)-z(v_1),
\end{aligned}
\]
whereas \eqref{eq:all-right-formal-first} gives
\[
\begin{aligned}
 &L_1(0,a,b,c+d,0)-L_1(0,a+b,0,c+d,0)\\
 &\qquad=z(v_4)-z(c+d,a+b,0,0,0).
\end{aligned}
\]
The auxiliary edge symbol cancels, leaving
\[
 L_3(p)=z(v_0)-z(v_1)+z(v_2)-z(v_3)+z(v_4).
\]
Taking \(z(q)=\mathbf e_q\) gives the complete exponent
\[
 \widetilde g_p
 =\mathbf e_{v_0}-\mathbf e_{v_1}
  +\mathbf e_{v_2}-\mathbf e_{v_3}+\mathbf e_{v_4}.
\]
Because the unique maximizing monomial is the constant-\(F\) term,
this is \(\widetilde g_p=\beta_p+g(X_p)\).  Finally \(v_0\) is frozen
and \(v_1,\ldots,v_4\) are mutable, so
\[
 \beta_p=\mathbf e_{v_0},\qquad
 g(X_p)=-\mathbf e_{v_1}+\mathbf e_{v_2}
        -\mathbf e_{v_3}+\mathbf e_{v_4}.
\]
\end{proof}

Only one numerical witness is required here because it isolates the
intrinsic constant-\(F\) exponent.  It does not identify the bottom
matching: independently, the dimer polynomial is normalized so that
its constant monomial is \(M_{\min}\), and
Proposition~\ref{prop:polygonal-bottom-matching} has already proved
\(M_{\rm poly}=M_{\min}\).

\subsection{Putting the two readings together}

The recurrence trace and the bottom-matching count have now produced
the same four mutable signs \(-1,+1,-1,+1\) at
\(v_1,v_2,v_3,v_4\).  Equivalently,
\[
 g(X_p)=g(Z_p)
   =-\mathbf e_{v_1}+\mathbf e_{v_2}
      -\mathbf e_{v_3}+\mathbf e_{v_4}.
\]
The mutable exchange matrix of the fan is skew-symmetric.  Hence its
principal-coefficient cluster algebra is a positive GLP algebra pointed
at the initial seed \cite[Remark~5.8(i)]{CaoLi}.  By
\cite[Corollary~5.7(i)]{CaoLi}, distinct cluster monomials in this
algebra have distinct mutable \(g\)-vectors.  Since \(X_p\) and \(Z_p\)
are cluster variables with the same mutable \(g\)-vector, it follows
that
\[
                              Z_p=X_p.
\]
The frozen endpoint \(v_0\) is deliberately absent from this
injectivity step.  It was carried separately by the recurrence trace.
After identifying the two principal-coefficient variables, specialize
to the frozen boundary pattern; the positive label at the left endpoint
of \(\mathcal P_p\) gives
\[
 \beta_{Z_p}=\beta_p=\mathbf e_{v_0}.
\]
Consequently
\[
 \widetilde g_p
   =\mathbf e_{v_0}-\mathbf e_{v_1}
      +\mathbf e_{v_2}-\mathbf e_{v_3}+\mathbf e_{v_4},
\]
which is \eqref{eq:five-point-g}.

\begin{targetbinarity}
The denominator vector of \(X_p\) relative to the initial fan is
\[
                              d_p=\1_{S_p}.
\]
\end{targetbinarity}

\begin{proof}
Theorem~\ref{thm:light-cone} already gives
\(D_p(v)=0\) for \(v\notin S_p\) and \(D_p(v)>0\) for
\(v\in S_p\).  The inequalities in \eqref{eq:Sp-formula} place every
selected vertex in \(I_{\rm mut}\), so no frozen coordinate is involved
in what follows.

After the identification \(Z_p=X_p\), the matching expansion and
\eqref{eq:explicit-exponent} give, for \(v\in S_p\), the exponent
\[
                         \epsilon_v(M)=w_v-m_v(M)-1
\]
in the term indexed by \(M\).  The face \(F_v\) has \(2w_v\) boundary
edges, and the boundary edges belonging to a matching are pairwise
vertex-disjoint.  Hence \(m_v(M)\leq w_v\) for every \(M\), and every
matching term has \(X_v\)-exponent at least \(-1\).  The Laurent sum is
positive, so its least exponent cannot be cancelled.  Therefore
\(D_p(v)\leq1\).  Since Theorem~\ref{thm:light-cone} says that this
integer is positive, it equals one.  This holds for every
\(v\in S_p\), while every coordinate outside \(S_p\), including every
frozen boundary coordinate, is already zero.
\end{proof}

\paragraph{Conclusion of Step~V.}
The bottom matching in Figure~\ref{fig:step-v-bottom-matching} carries
all the graph-side information needed in this step.  Its incidence
words give the four mutable signs \(-,+,-,+\), while the broken path
also marks the frozen endpoint \(+\mathbf e_{v_0}\).  The recurrence
trace produces the same five signed locations.  Equality of the four
mutable signs identifies the dimer variable with the octahedron
variable; the first sign supplies the separate frozen normalization.

\begin{proof}[Completion of the proof of
Theorem~\ref{thm:partI-main}]
Theorem~\ref{thm:light-cone} identifies the positive support, and the
binarity argument after the identification \(Z_p=X_p\) proves
item~\textup{(i)}.
Theorem~\ref{thm:fixed-graph}, the flip-connectivity lemma, and
Theorem~\ref{thm:well-matched} prove item~\textup{(ii)}.
The recurrence-side calculation, together with
Lemma~\ref{lem:reference-matching-incidence-calculation} and
Corollary~\ref{cor:graph-five-point-g}, proves item~\textup{(iii)}
and identifies \(Z_p=X_p\).

By Theorem~\ref{thm:well-matched} and Theorem~6.8 of \cite{MMSBV},
the dimer face polynomial is the \(F\)-polynomial of \(Z_p\).
Proposition~\ref{prop:lifted-graph-g} lifts this variable to the full
mutable fan quiver.  After the identification \(Z_p=X_p\),
substitution in the separation formula gives
\eqref{eq:laurent-matching-formula}, with mutable exponents
\eqref{eq:explicit-exponent} and frozen exponents
\eqref{eq:boundary-exponent}.  Tropicalizing the positive Laurent
sum gives \eqref{eq:tropical-matching-formula}.
\end{proof}

\section*{Conclusion and dependencies}

The reverse tropical calculation gives the positive support \(S_p\)
and hence the fixed graph.  After the dimer variable is identified with
\(X_p\), the matching exponents prove numerical binarity.  The directed extremal-extension criterion,
the boundary-layer lemma, and the height-forcing argument make that
graph well-matched.  The dimer theorem of \cite{MMSBV} supplies the
\(F\)-polynomial on the induced quiver, and
Proposition~\ref{prop:lifted-graph-g} lifts the same mutation sequence
to the full fan quiver.  Finally, the recurrence-side and
bottom-matching calculations give the same four-point mutable
\(g\)-vector, while the recurrence computes the frozen normalization
separately.  Principal-coefficient injectivity identifies the dimer
variable with \(X_p\), after which coefficient specialization yields
the Laurent and tropical matching formulas.

\part{The variational limit}

\section{Terminology and the main theorem}
\label{sec:partII-main-theorem}

\paragraph{Three-source GUE data and the discrete output.}
Fix \(\ell_1,\ell_2,\ell_3>0\).  For every \(n\), let
\(M_1^{(n)},M_2^{(n)},M_3^{(n)}\) be independent standard
\(n\times n\) GUE matrices and put
\[
             \widehat X_r^{(n)}=\ell_r\sqrt n\,M_r^{(n)},
             \qquad r=1,2,3.
\]
Use the centered large-gap construction of
Definition~\ref{def:rf-centered-large-gap}, written again in
\eqref{eq:rf-B-boundaries}--\eqref{eq:rf-gue-hive-map} below, with a
fixed gap parameter \(K_{\rm gap}\).  For
\(\varnothing\ne I\subseteq\{1,2,3\}\), put
\(\widehat X_I^{(n)}=\sum_{r\in I}\widehat X_r^{(n)}\), and define
\begin{equation}
 \mathcal E_n=
 \left\{\max_{\varnothing\ne I\subseteq\{1,2,3\}}
 \|\widehat X_I^{(n)}\|_{\rm op}
 \leq \frac18K_{\rm gap}n\right\}.
\label{eq:rf-upfront-large-gap-event}
\end{equation}
Choose \(K_{\rm gap}\) large enough that
\(\mathbb P(\mathcal E_n^c)\leq Ce^{-cn}\).
On \(\mathcal E_n\),
\[
 2\max_{\varnothing\ne I\subseteq\{1,2,3\}}
       \operatorname{width}(\widehat X_I^{(n)})
 \leq4\max_{\varnothing\ne I\subseteq\{1,2,3\}}
       \|\widehat X_I^{(n)}\|_{\rm op}
 \leq\frac12K_{\rm gap}n<K_{\rm gap}n.
\]
Thus Lemma~\ref{lem:rf-centered-large-gap-valid} shows that the three
initial arrays \(H_r^{(n)}\) on \(012,023,034\), and all four
boundary-face arrays subsequently used, are genuine hives.  On
\(\mathcal E_n^c\) we keep the same arrays and
define the output by the algebraic octahedron recurrence; no hive
inequality is needed for that definition.  The exceptional event and
the exact independence of the recurrence value from \(K_{\rm gap}\)
are treated explicitly in Section~\ref{sec:rf-gue-boundary}.

If \(v\) is a lattice vertex of the size-\(n\) tetrahedron \(1234\), write
\[
                         \widehat O_n(v)
\]
for the octahedron-recurrence value at \(v\) obtained from these three
initial arrays.  The flip-independence result of Part~I makes this
notation independent of the legal sequence of sweeps.

\begin{proof}[Proof of Theorem~\ref{thm:rf-top-bottom-double-hive}]\label{pf:thm3}
All coordinate Lebesgue measures and fiber normalizations in this proof
are those fixed in Definition~\ref{def:rf-maximum-entropy-field}.

\smallskip
\noindent\emph{The dimension-matched associator.}
For \(\lambda\in\mathsf W_n\), write
\[
 \Delta(\lambda)=\prod_{1\le i<j\le n}(\lambda_i-\lambda_j),
 \qquad
 V_n(\lambda)=\operatorname{Leb}(\operatorname{GT}_n(\lambda)).
\]
The Gelfand--Tsetlin volume formula gives
\begin{equation}
                         V_n(\lambda)=c_n\Delta(\lambda),
\label{eq:rf-GT-volume}
\end{equation}
for a positive constant \(c_n\) depending only on the coordinate
normalization.  A complete GUE minor process with top row \(\lambda\)
is conditionally uniform on \(\operatorname{GT}_n(\lambda)\).  Hence,
with respect to Lebesgue measure on the full Gelfand--Tsetlin cone, its
density is proportional to
\begin{equation}
       \exp\{-|\lambda|^2/(2ns^2)\}\Delta(\lambda).
\label{eq:rf-complete-GUE-minor-density}
\end{equation}
Indeed, the ordered GUE eigenvalue density contains
\(\Delta(\lambda)^2\), while normalization of the conditional uniform
law divides by \(V_n(\lambda)\).  These are the standard GUE-corners
and Gelfand--Tsetlin volume formulas; see
\cite{NSTGUE,NarDeformed}.

Apply the measure-preserving binary hive associator first to the
complete minor patterns with top rows \(\alpha,\beta\), and then to
the residual pattern from that operation together with the pattern
having top row \(\eta\).  This gives a Lebesgue-measure-preserving
piecewise-linear bijection
\begin{equation}
\begin{aligned}
 &\operatorname{GT}_n(\alpha)\times
  \operatorname{GT}_n(\beta)\times
  \operatorname{GT}_n(\eta)\\
 &\quad\longleftrightarrow
 \bigsqcup_{\nu,\phi}
 H_n(\alpha,\beta;\nu)
 \times H_n(\eta,\phi^\vee;\nu^\vee)
 \times\operatorname{GT}_n(\phi).
\end{aligned}
\label{eq:rf-iterated-associator-bijection}
\end{equation}
The second hive has merely been cyclically reoriented from the usual
form \(H_n(\nu,\eta;\phi)\).  Thus the three input patterns produce the
two pasted hives and exactly one residual Gelfand--Tsetlin pattern, on
the fourth exterior side.  They do not produce four simultaneously
retained exterior augmentations.  The binary bijection and its
measure-preserving property are given in
\cite[equation~(1.19) and Section~3]{NSTGUE}, where the fixed-diagonal
map is extended by Fubini--Tonelli to allow the residual
Gelfand--Tsetlin pattern to vary over its full polytope.  Iterating this
map gives \eqref{eq:rf-iterated-associator-bijection}.

By \eqref{eq:rf-complete-GUE-minor-density}, the product density of the
three input patterns is proportional to
\begin{equation}
 \exp\!\left\{-\frac1{2n}\left(
   \frac{|\alpha|^2}{\ell_1^2}
  +\frac{|\beta|^2}{\ell_2^2}
  +\frac{|\eta|^2}{\ell_3^2}\right)\right\}
 \Delta(\alpha)\Delta(\beta)\Delta(\eta).
\label{eq:rf-three-minor-density}
\end{equation}
The same is therefore the density of
\((h_{123},h_{134},\Gamma_\phi)\) on the right side of
\eqref{eq:rf-iterated-associator-bijection}.  Forgetting the residual
pattern \(\Gamma_\phi\) multiplies this density by
\(V_n(\phi)=c_n\Delta(\phi)\).  Consequently the top double hive has,
with respect to the natural Lebesgue measure on the unaugmented
double-hive cone, density proportional to
\begin{equation}
 \exp\!\left\{-\frac1{2n}\left(
   \frac{|\alpha|^2}{\ell_1^2}
  +\frac{|\beta|^2}{\ell_2^2}
  +\frac{|\eta|^2}{\ell_3^2}\right)\right\}
 \Delta(\alpha)\Delta(\beta)\Delta(\eta)\Delta(\phi).
\label{eq:rf-projected-double-hive-density}
\end{equation}

On the other hand, projecting
\(\mathbb Q_{n,13}^{\ell_1,\ell_2,\ell_3}\) forgets its four exterior
augmentations and therefore multiplies
\eqref{eq:rf-double-hive-density} by
\[
 V_n(\alpha)V_n(\beta)V_n(\eta)V_n(\phi).
\]
By \eqref{eq:rf-GT-volume}, its projected density is precisely
\eqref{eq:rf-projected-double-hive-density}, up to a constant depending
only on \(n\).  Normalization proves
\eqref{eq:rf-top-double-hive-law}, and hence proves that the complete
field produced by the GUE-minor construction has the intrinsic law in
Definition~\ref{def:rf-maximum-entropy-field}.

For later reference, if
\[
 \mathbf B_n=(\alpha_n,\beta_n,\eta_n,\phi_n)
\]
denotes the four exterior side-increment lists and
\(\widehat X_4^{(n)}=\widehat X_1^{(n)}+
\widehat X_2^{(n)}+\widehat X_3^{(n)}\), the GUE-hive associator also
gives the joint boundary identity
\begin{equation}
 \mathbf B_n\ \stackrel{\rm d}{=}\
 \bigl(
   \lambda(\widehat X_1^{(n)}),
   \lambda(\widehat X_2^{(n)}),
   \lambda(\widehat X_3^{(n)}),
   \lambda(\widehat X_4^{(n)})
 \bigr),
\label{eq:rf-four-boundary-joint-law}
\end{equation}
with reversal of the last list when it is read in the opposite boundary
orientation.  This is an equality in distribution of the four lists,
not a pathwise identification of all minor arrays.

\smallskip
\noindent\emph{Transport to the second triangulation.}
The octahedron recurrence changes the diagonal from \(13\) to \(24\),
preserves the four exterior sides, and is a
Lebesgue-measure-preserving piecewise-linear bijection on each
double-hive fiber.  Projection commutes with this recurrence.  Therefore
\[
 \mathcal L(\mathsf H_n^{\rm bot})
 =({\rm OR}_{13\to24})_\#
      (\pi_{13}^{\rm dh})_\#\mathbb Q_{n,13}^{\ell_1,\ell_2,\ell_3}
 =(\pi_{24}^{\rm dh})_\#\mathbb Q_{n,24}^{\ell_1,\ell_2,\ell_3},
\]
which proves \eqref{eq:rf-top-bottom-double-hive-laws}.  The same
measure-preserving property shows, independently, that the abstract
pushforward \(\mathbb Q_{n,24}^{\ell_1,\ell_2,\ell_3}\) has density
\eqref{eq:rf-double-hive-density} in the \(24\)-coordinates.

\smallskip
\noindent\emph{Conditional independence.}
It remains to record the conditional-independence consequence of the
projected law.  Once \(\nu_{13}\) is fixed, the cone
\eqref{eq:rf-double-hive-cone} is the Cartesian product of the union
of augmented-hive fibers on \(123\) and the corresponding union on
\(134\).  The density \eqref{eq:rf-double-hive-density} factors
between these two groups of variables.  The two augmented hives are
therefore conditionally independent under the abstract law
\(\mathbb Q_{n,13}^{\ell_1,\ell_2,\ell_3}\).  Projecting their exterior
augmentations preserves this conditional independence, and
\eqref{eq:rf-top-double-hive-law} gives the first assertion in
\eqref{eq:rf-top-bottom-conditional-independence}.  The identical
argument for the abstract law in the \(24\)-triangulation, followed by
\eqref{eq:rf-top-bottom-double-hive-laws}, gives the second.

The distinction between projected and augmented laws is essential.
The original three minor patterns are transformed by
\eqref{eq:rf-iterated-associator-bijection}; they are not four
independent exterior augmentations of its output.  If, after producing
the top double hive, one attaches four \emph{fresh}, conditionally
independent uniform Gelfand--Tsetlin patterns on its exterior sides,
then the resulting artificially augmented object does have law
\(\mathbb Q_{n,13}^{\ell_1,\ell_2,\ell_3}\).  No such additional randomization
is part of the field construction or is needed for the projected law.

The centered gap vector of
Definition~\ref{def:rf-centered-large-gap} is used only to realize the
three input Gelfand--Tsetlin patterns as hives.  Its deterministic
contribution at a size-\(n\) tetrahedral vertex \(q\) is
\begin{equation}
 G_n(q)=\frac{nK_{\rm gap}}2\,q_0(n-q_0).
\label{eq:rf-large-gap-gauge}
\end{equation}
This is a one-coordinate octahedron-recurrence gauge: in every
elementary octahedron, each of the three opposite vertex pairs has the
same multiset of \(q_0\)-coordinates, so all three pair sums acquire
the same total contribution.  Since \(q_0=0\) on the tetrahedron
\(1234\), \(G_n\)
vanishes there.  This proves the asserted exact independence of the
admissible gap size.
\end{proof}

By
Theorem~\ref{thm:rf-top-bottom-double-hive}, the joint field
\((\widehat O_n(v))_{v\in1234}\) is a GUE-minor realization of the
intrinsically defined maximum-entropy field
\(\mathcal H_n^{1234}(\ell_1,\ell_2,\ell_3)\) from
Definition~\ref{def:rf-maximum-entropy-field}.

\paragraph{The cut-open continuum region.}
Let
\[
 \mathsf T_{1234}
 =\{(\alpha,\beta,\gamma,\delta)\in\mathbb R_{\geq0}^4:
       \alpha+\beta+\gamma+\delta=1\}
\]
be the closed continuum tetrahedron.  For a strictly interior point
\(p=(\alpha,\beta,\gamma,\delta)\), put
\begin{align*}
 \mathsf R_{p,1}
 &=\{x\in\Delta_{012}(1):x_1\leq\alpha,\ x_2\geq\beta\},\\
 \mathsf R_{p,2}
 &=\{x\in\Delta_{023}(1):x_2\geq\beta
             \text{ or }x_3\geq\gamma\},\\
 \mathsf R_{p,3}
 &=\{x\in\Delta_{034}(1):x_3\geq\gamma,\ x_4\leq\delta\}.
\end{align*}
The retained seams are
\[
 \Gamma_{02}
  =\{x\in\Delta_{02}(1):x_2\in[\beta,1]\},
 \qquad
 \Gamma_{03}
  =\{x\in\Delta_{03}(1):x_3\in[\gamma,1]\}.
\]
Regard the three sets \(\mathsf R_{p,r}\) as separate sheets, keeping
two labelled copies of each retained seam.  Their cut-open union is
\(\mathsf R_p\).

\paragraph{Admissible continuum heights.}
We use geometric gradients in the equilateral development, as in
\eqref{eq:panel-slope-triangle} and \cite[\S3, Definition~9]{NarLimit}.
The triangular-lattice unit vectors
\(\widehat\imath,\widehat\jmath\) satisfy
\(\langle\widehat\imath,\widehat\jmath\rangle=1/2\), and
\[
 K=\operatorname{conv}\bigl\{
  2(\widehat\imath-\widehat\jmath),\ 2\widehat\jmath,\
  -2\widehat\imath\bigr\}.
\]
Throughout Part~II, \(\nabla f\), slopes \(a\in K\), and the affine
pairings \(\langle a,z-z_T\rangle\) refer to this Euclidean geometry.
Square array coordinates only parametrize it: up to translation their
embedding is
\(E(I,J)=-I\widehat\imath+J\widehat\jmath\).
Thus, if \(\nabla f=x\widehat\imath+y\widehat\jmath\), the array
derivatives of \(\widetilde f=f\circ E\) are
\[
 \partial_I\widetilde f=-x-\tfrac12y,
 \qquad
 \partial_J\widetilde f=\tfrac12x+y.
\]
The corresponding triangle of array derivatives is
\(\operatorname{conv}\{(2,-1),(-1,2),(-1,-1)\}\).
The space \(\operatorname{AHT}(\mathsf R_p)\) consists of triples
\(f=(f_1,f_2,f_3)\) such that each \(f_r\) is Lipschitz on
\(\mathsf R_{p,r}\), with \(\nabla f_r\in K\) almost everywhere;
the outer boundary traces are the exact deterministic traces
\(\gamma_{p,r}\) defined below, with \(f_1(v_9)=0\); and the two
labelled seam traces satisfy
\begin{align*}
 f_2([0,x]_2)&=f_1([x,x]_1)+x,
              &&x\in[\beta,1],\\
 f_3([0,x]_3)&=f_2([x,x]_2)+x,
              &&x\in[\gamma,1].
\end{align*}
Write
\[
 f_{02}(x)=f_2([0,x]_2),
 \qquad
 f_{03}(x)=f_3([0,x]_3)
\]
for the common aligned seam traces.  The precise discrete
approximation statement behind this definition is given in
Definition~\ref{def:rf-AHT} and
Lemma~\ref{lem:rf-AHT-compact}.

\paragraph{Global colors and local direction types.}
We retain throughout Part~II the global color convention fixed in
Part~I.  On \(P_r=\Delta_{0,r,r+1}(N)\), write
\[
 d_1^{(r)}=e_r-e_{r+1},\qquad
 d_2^{(r)}=e_0-e_r,\qquad
 d_3^{(r)}=e_{r+1}-e_0.
\]
A lozenge is of \emph{type \(k\)} in the local panel frame when its
selected edge is parallel to \(d_k^{(r)}\).  Type \(1\) is red on every
panel.  Types \(2\) and \(3\) are respectively blue and green on
\(P_1,P_3\), but green and blue on the reflected middle panel \(P_2\).
Thus the words red, green, and blue always refer to the global displayed
colors, whereas the neutral type labels record the local
Gelfand--Tsetlin coordinate directions.  Across seam \(02\), green is
preserved while red and blue exchange; across seam \(03\), blue is
preserved while red and green exchange.

\paragraph{The variational energy.}
For a bulk point \(x\) on the \(r\)-th panel, write the two
nonnegative interlacing gaps in the local type order as
\[
 g_2(i,j)=\lambda_i^{(j)}-\lambda_{i+1}^{(j+1)},\qquad
 g_3(i,j)=\lambda_{i+1}^{(j+1)}-\lambda_{i+1}^{(j)}.
\]
Here \(\lambda\) denotes the already source-scaled local minor array.
Thus the corresponding ordinary lozenge weights are
\(\omega_2=-g_2/3\) and \(\omega_3=-g_3/3\), including the fixed
source scale \(\ell_r\) and, on \(P_2\), the reflected-rank
convention.  We use the raw-gap parameters
\[
 \rho_r(x)=(\rho_{r,1}(x),\rho_{r,2}(x))
 :=\bigl(\mathbb E g_2,\mathbb E g_3\bigr)
\]
of the corresponding marked GUE-minor local limit, with the scaling
convention of \cite[Definition~17]{NarLimit}.  In particular, no
factor \(1/3\) is absorbed into \(\rho_r\).  If \(\Xi(a_m)\) is the lozenge
tiling corresponding to a finite-patch height \(a_m\), define the
flip-symmetric local pairing by
\[
 \langle\partial M_{m,x}(\rho),\partial a_m\rangle_{\rm sym}
 :=\frac32\sum_{\substack{\ell\in\Xi(a_m)\\\ell\ {\rm blue}}}
       \operatorname{wt}(\ell)
   +\frac32\sum_{\substack{\ell\in\Xi(a_m)\\\ell\ {\rm green}}}
       \operatorname{wt}(\ell).
\]
Let \(D_m\) be the underlying standard \(m\)-by-\(m\) patch.  With
the height neighborhood
\(\mathcal B_m(a,m\varepsilon)\) of
\cite[Definition~13]{NarLimit}, put
\[
 \sigma_{m,\edge}^{\rm sym}(\rho,a,\varepsilon)
 :=-\frac1{|D_m|}\,
   \mathbb E\max_{a_m\in\mathcal B_m(a,m\varepsilon)}
   \langle\partial M_{m,x}(\rho),\partial a_m\rangle_{\rm sym}.
\]
Define \(\sigma_\edge^{\rm sym}(\rho,a)\) from these finite-volume
quantities by the limiting and boundary-extension procedure of
\cite[Definitions~14--16]{NarLimit}.  Thus
\(\sigma_\edge^{\rm sym}\) is the thermodynamic limit of the negative
expected maximal \(0,3/2,3/2\) red--green--blue weight per unit area.
Here and below normalized lattice area assigns area \(1/2\) to an
elementary triangle and area \(1\) to a lozenge; hence a standard
\(m\)-by-\(m\) patch has area \(m^2\), while a triangular half-patch
has area \(m^2/2\).  Every finite-volume score is divided by the
actual area of its patch.  The continuum measure \(dx\) below is this
same normalized lattice measure: in array coordinates
\(dx=dI\,dJ=(2/\sqrt3)\,dA_{\rm Euclidean}\).
The function
\(\sigma_\edge^{\rm sym}(\rho,a)\) is continuous in
\((\rho,a)\in\mathbb R_{>0}^2\times K\), and, for each fixed
\(\rho\), it is convex in the slope \(a\).  The transfer of the cited
surface-tension results to this normalization is recorded explicitly
in Lemma~\ref{lem:rf-local-transfer}.
Define
\[
 S_{p,\mathrm{edge}}(f)
 =-\sum_{r=1}^3\int_{\mathsf R_{p,r}^{\circ}}
       \sigma_\edge^{\rm sym}(\rho_r(x),\nabla f_r(x))\,dx.
\]
For \(x\in(0,1]\) and \(\xi\in[-1,2]\), put
\begin{align*}
 \mathcal E_{12}(x,\xi)
 &=\frac{1+\xi}{3}\ell_2\sqrt x
   +\frac{2-\xi}{3}\ell_1\sqrt x,\\
 \mathcal E_{23}(x,\xi)
 &=\frac{1+\xi}{3}\ell_3\sqrt x
   +\frac{2-\xi}{3}\ell_2\sqrt x,
\end{align*}
and
\[
 S_{p,\Delta}^{\rm rf}(f)
 =\int_\beta^1\mathcal E_{12}(x,f_{02}'(x))\,dx
  +\int_\gamma^1\mathcal E_{23}(x,f_{03}'(x))\,dx.
\]
If \(\gamma_{\rm sc}\) is the decreasing semicircle quantile, set
\[
 S_{p,0}
 =\frac{\ell_3}{2}\int_0^\delta\gamma_{\rm sc}(u)\,du
  +\frac{\ell_1}{2}\int_0^{1-\alpha}
                  \gamma_{\rm sc}(u)\,du,
\qquad
 \mathcal S_p(f)
 =S_{p,\mathrm{edge}}(f)
  +S_{p,\Delta}^{\rm rf}(f)+S_{p,0}.
\]
For strictly interior rational \(p\), define
\begin{equation}
 \Phi_{\ell_1,\ell_2,\ell_3}(p)
 :=\sup_{f\in\operatorname{AHT}(\mathsf R_p)}
       \mathcal S_p(f).
\label{eq:partII-upfront-variational-value}
\end{equation}
The admissible space is compact and \(\mathcal S_p\) is upper
semicontinuous, so this supremum is attained.  No uniqueness of the
maximizing height is asserted.

\paragraph{Continuum interpolation.}
On each tetrahedral cell of the rescaled
tetrahedral--octahedral cellulation, interpolate affinely from the four
lattice-vertex values.  On an octahedral cell \(Q\), assign to its
center the average of its six vertex values, cone the center to the
eight triangular faces, and interpolate affinely on the resulting
eight tetrahedra.  Denote the resulting continuous random function by
\(\mathcal F_n\), normalized at lattice vertices by
\[
          \mathcal F_n(v/n)=\frac{\widehat O_n(v)}{n^2}.
\]

\begin{partiicorollary}
Let \(p\) be a strictly interior rational point and let
\(p_n/n\to p\) be any deterministic sequence of lattice
approximants.  Then
\[
 \frac{\widehat O_n(p_n)}{n^2}
 \xrightarrow{\mathbb P}\Phi_{\ell_1,\ell_2,\ell_3}(p).
\]
\end{partiicorollary}

\begin{partiimain}
The function \(\Phi_{\ell_1,\ell_2,\ell_3}\) has a unique Lipschitz extension
from the strictly interior rational points to
\(\mathsf T_{1234}\), and
\[
 \boxed{
 \sup_{p\in\mathsf T_{1234}}
 \left|\mathcal F_n(p)-\Phi_{\ell_1,\ell_2,\ell_3}(p)\right|
 \xrightarrow{\mathbb P}0.}
\]
Consequently, for every real
\(p\in\mathsf T_{1234}\) and every deterministic sequence of lattice
vertices \(p_n\) with \(p_n/n\to p\),
\[
 \frac{\widehat O_n(p_n)}{n^2}
 \xrightarrow{\mathbb P}\Phi_{\ell_1,\ell_2,\ell_3}(p).
\]
\end{partiimain}

The proofs of Corollary~\ref{cor:rf-rational-approximants}
and Theorem~\ref{thm:rf-uniform-continuum} appear in
Section~\ref{sec:rf-real-extension}.  Their principal input is the
rational variational and variance theorem,
Theorem~\ref{thm:rf-finite-panel-localization}.  The intervening
sections derive the exact red-free formula and identify the bulk,
seam, and fixed terms appearing in \(\mathcal S_p\).

\section{The red-lozenge-free tile gauge}
\label{sec:red-free-gauge}

\begin{remark}[Notation used in Part~II]
The discrete cut-open cell complex is \(\mathcal A_p\), its legal
lozenge--triangle tilings form \(\mathcal T_p\), and
\(F_{\Xi,r}\) is the lozenge height of a tiling \(\Xi\) on sheet
\(r\).  The normalized piecewise-affine height on the fixed continuum
region is \(\overline F_{\Xi,n}\).  The variable contribution is split
as
\[
 W_p^{\rm rf}=W_{p,\mathrm{edge}}^{\rm rf}
                 +W_{p,\Delta}^{\rm rf},
\]
where the first term is the panel bulk weight and the second is the
sum over the two cut seams.  Their continuum limits are
\(S_{p,\mathrm{edge}}\) and \(S_{p,\Delta}^{\rm rf}\); the remaining
matching-independent statistic has limit \(S_{p,0}\).  The final
variational value is denoted \(\Phi_{\ell_1,\ell_2,\ell_3}\), whereas the local
gauge potential below is denoted \(\mathcal G_p^{\rm rf}\).
\end{remark}

Part~I proves the corrected formula and identifies its terms with
lozenge--triangle tilings of one fixed cut-open region.  From this
point onward those tilings are the primary objects.  The
boundary-incidence correction from
Part~I must not be removed by hand.  Instead we absorb it by an explicit
tile gauge.  We use the flip-symmetric member of the red-lozenge-free
gauge family: red lozenges have weight zero, while blue and green
lozenges both have \(3/2\) times their ordinary rhombus weight.  At a
cut seam the same gauge fixes both singleton-triangle weights.  These
bulk, seam, and fixed-term changes must be made together; their sum is
the original corrected affine form.

Recall the legal-tiling convention and the notation
\(\mathcal T_p\) and \(\mathcal S_p^{\rm seam}\) from
Definition~\ref{def:legal-lozenge-triangle-tiling}.  In particular,
the two panel copies of a triangle facing a retained seam segment are
kept distinct, and the complementary-singleton rule is the complete
tile version of the cut-open seam condition.

\begin{remark}[Tile language used in Part~II]
\label{rem:rf-tile-language}
All maximizations and compatibility conditions below are stated in
terms of legal lozenge--triangle tilings.  A ``shared lattice edge''
always means the common side of the two triangles forming an ordinary
lozenge.  The subscript in
\(\sigma_\edge^{\rm sym}\) is retained to agree with the notation of
\cite{NarLimit}; the superscript records the flip-symmetric
\(0,3/2,3/2\) normalization used here.
\end{remark}

If an ordinary lozenge \(\ell\) is formed across the elementary lattice
edge \(\underline\ell=[x_\ell,y_\ell]\), define its raw weight, and the
raw weight of a seam triangle, by
\begin{equation}
 \rho_h^{\rm raw}(\tau)=
 \begin{cases}
  -h(x_\ell)-h(y_\ell),&\tau=\ell
       \text{ is an ordinary lozenge},\\
  0,&\tau\text{ is a seam triangle}.
 \end{cases}
 \label{eq:rf-raw-edge-weight}
\end{equation}
For a tiling \(\Xi\), let \(m_v(\Xi)\) be the number of its ordinary
lozenges whose shared lattice edge is incident to \(v\), and put
\[
 m_v^{\max}=\max_{\Xi\in\mathcal T_p}m_v(\Xi).
\]
Let \(\Xi_{M_{\min}}\) be the canonical bottom tiling constructed in
Part~I.  Thus all incidence quantities below have a direct tile
interpretation.

\begin{definition}[Raw tiling-independent term]
\label{def:rf-raw-fixed-term}
Put
\begin{align}
 C_p^{\rm raw}(h):={}&\mathsf b_p(h)
 +\sum_{v\in S_p}(w_v-1)h(v)
 \notag\\
 &+\sum_{\substack{v\in I_{\rm mut}\setminus S_p\\m_v^{\max}>0}}
       m_v^{\max}h(v)
 +\sum_{v\in I_{\rm fr}}m_v(\Xi_{M_{\min}})h(v).
\label{eq:rf-raw-fixed-term}
\end{align}
\end{definition}

\begin{proposition}[Raw tile form]
\label{prop:rf-raw-edge-form}
For every \(\Xi\in\mathcal T_p\), let
\(\Lambda_\Xi^{\rm corr}\) be its corrected affine form from Part~I.
Then
\begin{equation}
 \boxed{
 \Lambda_\Xi^{\rm corr}(h)
 =C_p^{\rm raw}(h)+\sum_{\tau\in\Xi}\rho_h^{\rm raw}(\tau).}
\label{eq:rf-raw-edge-form}
\end{equation}
\end{proposition}

\begin{proof}
Collect the three occurrences of the incidence term in
\eqref{eq:linear-form}, using the identification with tilings proved
in Part~I.  Every ordinary lozenge across
\([x_\ell,y_\ell]\) is counted once at \(x_\ell\), once at \(y_\ell\),
and nowhere else.  A single seam triangle has no shared lattice edge
and contributes zero raw weight.  Double counting gives
\eqref{eq:rf-raw-edge-form}.
\end{proof}

For \(r=1,2,3\), write \(P_r=\Delta_{0,r,r+1}(N)\), and use triangular
coordinates
\begin{equation}
 [i,j]_r=(N-j)e_0+(j-i)e_r+i e_{r+1},
 \qquad 0\le i\le j\le N.
\label{eq:rf-panel-coordinate}
\end{equation}
Thus \(H_r(i,j):=h([i,j]_r)\).  The two elementary triangle orientations
are
\begin{align}
 U^r_{i,j}&=\{[i,j]_r,[i,j+1]_r,[i+1,j+1]_r\},
       &&0\le i\le j<N,\notag\\
 D^r_{i,j}&=\{[i,j]_r,[i+1,j]_r,[i+1,j+1]_r\},
       &&0\le i<j<N.
\label{eq:rf-atomic-triangles}
\end{align}
These \(U/D\)-triangles are the cells of \(\mathcal A_p\).  In
particular, both facing panel triangles are retained as separate cells
at every cut seam.

\begin{definition}[Flip-symmetric red-free triangle potentials]
\label{def:rf-triangle-potential}
For an elementary triangle put
\begin{align}
 \phi_h(U^r_{i,j})
   &=\frac12\bigl(H_r(i,j+1)+H_r(i+1,j+1)\bigr),\notag\\
 \phi_h(D^r_{i,j})
   &=\frac12\bigl(H_r(i,j)+H_r(i+1,j)\bigr).
\label{eq:rf-triangle-potential}
\end{align}
\end{definition}

The seam \(0(r+1)\), between \(P_r\) and \(P_{r+1}\), has at level
\(k\) the segment with endpoints
\begin{equation}
 s_{r,k-1}=[k-1,k-1]_r=[0,k-1]_{r+1},
 \qquad
 s_{r,k}=[k,k]_r=[0,k]_{r+1}.
\label{eq:rf-seam-segment}
\end{equation}
Its two facing triangles are \(U^r_{k-1,k-1}\) and
\(U^{r+1}_{0,k-1}\).  When this is one of the seam segments retained in
the cut-open support, it belongs to
\(\mathcal S_p^{\rm seam}\).  Put
\begin{equation}
                         \psi_h(r,k)=-h(s_{r,k}).
\label{eq:rf-aux-potential}
\end{equation}
For a tile \(\tau\), define its gauged weight by
\begin{equation}
 \rho_h^{\rm rf}(\tau)=
 \begin{cases}
  \rho_h^{\rm raw}(\ell)+\phi_h(T_1)+\phi_h(T_2),
     &\tau=\ell=T_1\cup T_2\text{ is an ordinary lozenge},\\
  \phi_h(\Delta)+\psi_h(r,k),
     &\tau=\Delta\text{ is a seam triangle at }(r,k).
 \end{cases}
\label{eq:rf-gauged-edge-weight}
\end{equation}

The following calculation realizes the global color convention in the
local panel coordinates.  Whenever the indicated triangles exist, set
\begin{align*}
 R^r_{i,j}&:=U^r_{i,j}\cup D^r_{i,j+1},\\
 G^r_{i,j}&:=
 \begin{cases}
  U^r_{i,j}\cup D^r_{i,j},&r\in\{1,3\},\\
  U^r_{i,j}\cup D^r_{i-1,j},&r=2,
 \end{cases}\\
 B^r_{i,j}&:=
 \begin{cases}
  U^r_{i,j}\cup D^r_{i-1,j},&r\in\{1,3\},\\
  U^r_{i,j}\cup D^r_{i,j},&r=2.
 \end{cases}
\end{align*}
These are, respectively, the red, green, and blue lozenge orientations.
In the developed panel frames, seam \(02\) identifies the unoriented
geometric families
\[
 d_3^{(1)}\parallel d_2^{(2)},\qquad
 d_1^{(1)}\parallel d_3^{(2)},\qquad
 d_2^{(1)}\parallel d_1^{(2)},
\]
so green is preserved and red and blue exchange.  Seam \(03\) identifies
\[
 d_3^{(2)}\parallel d_2^{(3)},\qquad
 d_1^{(2)}\parallel d_3^{(3)},\qquad
 d_2^{(2)}\parallel d_1^{(3)},
\]
so blue is preserved and red and green exchange.

\begin{lemma}[Local flip-symmetric red-free tile weights]
\label{lem:rf-local-weights}
Let
\[
 \omega_h(\ell)=\frac13\bigl(h(a_1)+h(a_2)-h(o_1)-h(o_2)\bigr)
\]
be the ordinary rhombus weight, with \(a_1,a_2\) acute and \(o_1,o_2\)
obtuse.  Then an ordinary lozenge has gauged weight
\begin{equation}
 \rho_h^{\rm rf}(\ell)=
 \begin{cases}
  0,&\ell\text{ red},\\
  \frac32\omega_h(\ell),&\ell\text{ green},\\
  \frac32\omega_h(\ell),&\ell\text{ blue}.
 \end{cases}
\label{eq:rf-lozenge-weights}
\end{equation}
At a retained seam segment \eqref{eq:rf-seam-segment}, the two possible
singleton triangles have weights
\begin{align}
 \vartheta^-_{r,k}(h)
   &=\frac12\bigl(H_r(k-1,k)-H_r(k,k)\bigr),\notag\\
 \vartheta^+_{r,k}(h)
   &=\frac12\bigl(H_{r+1}(1,k)-H_{r+1}(0,k)\bigr).
\label{eq:rf-seam-weights}
\end{align}
Here \(-\) denotes the new-boundary side of \(P_r\), and \(+\) the
old-boundary side of \(P_{r+1}\).
\end{lemma}

\begin{proof}
For an ordinary lozenge, substitute
\eqref{eq:rf-triangle-potential} into
\eqref{eq:rf-gauged-edge-weight}.  The three possible adjacent pairs give
\begin{center}
\begin{tabular}{c|c|c}
color & atomic lozenge & gauged weight\\ \hline
red   & \(R^r_{i,j}\) & \(0\)\\
green & \(G^r_{i,j}\) & \(\frac32\omega_h(\ell)\)\\
blue  & \(B^r_{i,j}\) & \(\frac32\omega_h(\ell)\).
\end{tabular}
\end{center}
For example, \(R^r_{i,j}\) has shared vertices
\([i,j+1]_r,[i+1,j+1]_r\), and all four coefficients cancel.  The other
two rows follow by the same four-term substitution.  On \(P_2\) their
two atomic pairs are interchanged, but their common coefficient
\(3/2\) is unchanged.

For a seam triangle, use \eqref{eq:rf-triangle-potential} and
\eqref{eq:rf-aux-potential}.  On the \(-\) side,
\[
 \phi_h(U^r_{k-1,k-1})-h(s_{r,k})
 =\frac12\bigl(H_r(k-1,k)-H_r(k,k)\bigr).
\]
On the \(+\) side,
\[
 \phi_h(U^{r+1}_{0,k-1})-h(s_{r,k})
 =\frac12\bigl(H_{r+1}(1,k)-H_{r+1}(0,k)\bigr).
\]
This proves \eqref{eq:rf-seam-weights}.
\end{proof}

For a legal tiling \(\Xi\in\mathcal T_p\), define
\begin{equation}
 \begin{split}
 W_p^{\rm rf}(\Xi;h):={}&
 \frac32\sum_{\substack{\ell\in\Xi\\\ell\text{ blue}}}\omega_h(\ell)
 +\frac32\sum_{\substack{\ell\in\Xi\\\ell\text{ green}}}\omega_h(\ell)\\
 &+\sum_{\substack{\Delta\in\Xi\\\Delta\text{ on a cut seam}}}
       \vartheta_h(\Delta),
 \end{split}
\label{eq:rf-tiling-weight}
\end{equation}
where \(\vartheta_h\) is the appropriate one of
\eqref{eq:rf-seam-weights}.  Let
\begin{equation}
 \mathcal G_p^{\rm rf}(h)=
 \sum_{T\in\mathcal A_p}\phi_h(T)
 +\sum_{(r,k)\in\mathcal S_p^{\rm seam}}\psi_h(r,k),
 \qquad
 F_p^{\rm rf}(h)=C_p^{\rm raw}(h)-\mathcal G_p^{\rm rf}(h).
\label{eq:rf-fixed-definition}
\end{equation}

\begin{theorem}[Exact red-lozenge-free tiling formula]
\label{thm:rf-exact-tiling}
For every strictly interior \(p=(0,a,b,c,d)\) and every compatible
height function \(h\) on \(\Sigma_0\),
\begin{equation}
 \boxed{
 h(p)=F_p^{\rm rf}(h)
       +\max_{\Xi\in\mathcal T_p}W_p^{\rm rf}(\Xi;h).}
\label{eq:rf-exact-tropical}
\end{equation}
This is a tropical identity.  The ordinary positive Laurent formula is
the integral-exponent formula of Theorem~\ref{thm:partI-main}; the
tiling formulation above is the only one used below.
\end{theorem}

\begin{proof}
Every legal tiling uses every triangle of \(\mathcal A_p\) exactly
once and, at every retained seam segment, uses exactly one of the two
facing triangles as a singleton.  Consequently the tile gauge
\eqref{eq:rf-gauged-edge-weight} adds the same quantity
\(\mathcal G_p^{\rm rf}(h)\) to the raw weight of every tiling.
Proposition~\ref{prop:rf-raw-edge-form} and
Lemma~\ref{lem:rf-local-weights} therefore give
\[
 \Lambda_\Xi^{\rm corr}(h)
 =F_p^{\rm rf}(h)+W_p^{\rm rf}(\Xi;h).
\]
Part~I identifies these corrected affine forms with the terms in the
octahedron-recurrence maximum.  Taking the maximum over
\(\Xi\in\mathcal T_p\) proves the result.
\end{proof}

\begin{remark}[The seam weights are fixed by the bulk gauge]
\label{rem:rf-seam-gauge-coupling}
The two weights in \eqref{eq:rf-seam-weights} have the same
coefficient \(1/2\), and the bulk coefficient \(3/2\) is likewise
unchanged when the green and blue labels are interchanged.  These
weights are nevertheless not an additional convention:
they are forced by the triangle potential
\eqref{eq:rf-triangle-potential}.  Thus one cannot first set the red
lozenge weight to zero and then choose the seam-triangle normalization
independently.  The symmetric seam pair is the local tile-gauge image
of the boundary-incidence term in \eqref{eq:linear-form}.
\end{remark}

\section{The fixed term}
\label{sec:rf-fixed-term}

We now compute \(F_p^{\rm rf}\) exactly.  Let \(c_p(v)\) be three times
the coefficient of \(h(v)\) in the sum of the elementary-triangle potentials,
and let \(a_p(v)\) be the number of retained seam segments whose higher
endpoint is \(v\):
\begin{equation}
 \frac13c_p(v)=[h(v)]\sum_{T\in\mathcal A_p}\phi_h(T),
 \qquad
 a_p(v)=\#\{(r,k)\in\mathcal S_p^{\rm seam}:s_{r,k}=v\}.
\label{eq:rf-c-a-definition}
\end{equation}
Define
\begin{equation}
 \eta_p(v)=
 \begin{cases}
  w_v-1,&v\in S_p,\\
  m_v^{\max},&v\in I_{\rm mut}\setminus S_p,\ m_v^{\max}>0,\\
  \beta_p(v)+m_v(\Xi_{M_{\min}}),&v\in I_{\rm fr},\\
  0,&\text{otherwise}.
 \end{cases}
\label{eq:rf-eta}
\end{equation}
Then the coefficient of \(h(v)\) in the fixed term is
\begin{equation}
             \chi_p^{\rm rf}(v)=\eta_p(v)-\frac13c_p(v)+a_p(v).
\label{eq:rf-chi}
\end{equation}
The seam-potential sign is positive here because
\(\psi_h(r,k)=-h(s_{r,k})\) is subtracted in
\eqref{eq:rf-fixed-definition}.

Recall from the coordinate definition of the reference tiling
preceding Lemma~\ref{lem:boundary-datum-compatible} that
\[
 v_9=(0,a,N-a,0,0)\in\Delta_{012}(N).
\]

\begin{lemma}[Local-star exhaustion]
\label{lem:rf-local-star-exhaustion}
For a lattice vertex \(v\), let
\[
 \mathcal R_v:=\{T\in\mathcal A_p:v\text{ is a vertex of }T\},
 \qquad
 \mathcal G_v:=\{T\in\mathcal R_v:[h(v)]\phi_h(T)=\tfrac12\}.
\]
As \(v\) ranges over the initial fan, its retained local star belongs
to exactly one of the nine coefficient classes used below.  Their
pairs
\((\lvert\mathcal R_v\rvert,\lvert\mathcal G_v\rvert)\) are
\begin{align*}
 \textup{(a)}&:(0,0),&
 \textup{(b)}&:(3,2)\text{ or }(2,2),&
 \textup{(c)}&:(6,4),\\
 \textup{(d)}&:(6,4),&
 \textup{(e)}&:(2,0),&
 \textup{(f)}&:(2,2),\\
 \textup{(g)}&:(5,4),&
 \textup{(h)}&:(2,1),&
 \textup{(i)}&:(2,1).
\end{align*}
The second pair in class~\textup{(b)} occurs precisely at the convex
truncations \(v_4,v_5\); it has the same two gold cells and hence the
same coefficient.  The classes are, respectively: no retained incident
cell; an ordinary one-sector boundary; a selected off-seam vertex; an
internal retained-seam vertex; a lower seam foot \(v_8\) or \(v_6\);
an upper endpoint \(u_2\) or \(u_3\); the reentrant vertex \(v_7\);
and the two caps \(v_9,v_0\).
\end{lemma}

\begin{proof}
The lattice boundary of \(S_p\) in \eqref{eq:Sp-formula} lies on the
following walls.  On \(P_1=012\) they are
\[
 q_0=0,\qquad q_1=a,\qquad q_2=b,
 \qquad\text{together with the seam }q_1=0;
\]
on \(P_2=023\) they are \(q_0=0\), the two arms
\[
 q_2=b\ (q_3\le c),\qquad q_3=c\ (q_2\le b),
\]
and the seams \(q_3=0,q_2=0\); on \(P_3=034\) they are
\[
 q_0=0,\qquad q_3=c,\qquad q_4=d,
 \qquad\text{together with the seam }q_4=0.
\]
Because \(a,b,c,d>0\), the only intersections of these walls which are
adjacent to the support and differ from their relative-interior types
are
\[
 v_9,\ u_2,\ u_3,\ v_0,\ v_5,\ v_6,\ v_7,\ v_8,\ v_4,
\]
in this order around the boundary.  All other wall intersections fail
at least two of the unit threshold inequalities and have empty retained
star.

On seam \(02\), the support condition reduces to
\(q_0\ge1,q_2\ge b+1\).  Thus at \(s_{1,k}\), the ranges
\(k<b\), \(k=b\), \(b<k<N\), and \(k=N\) give, respectively,
classes \textup{(a)}, \textup{(e)}, \textup{(d)}, and \textup{(f)}.
The same argument on seam \(03\), with \(b\) replaced by \(c\), gives
the other occurrences of those classes.  Away from the seams, the open
strata and the relative interiors of their walls give
\textup{(a)}--\textup{(c)};
the remaining intersections give \(v_4,v_5\) in \textup{(b)},
\(v_7\) in \textup{(g)}, and \(v_9,v_0\) in
\textup{(h)},\textup{(i)}.  Substitution in the two lists of triangles
in \eqref{eq:rf-atomic-triangles} gives the displayed retained and gold
counts.  This exhausts the local coefficient stars.
\end{proof}

\begin{theorem}[Explicit fixed term]
\label{thm:rf-explicit-fixed-term}
For every strictly interior \(p=(0,a,b,c,d)\),
\[
 \boxed{
 F_p^{\rm rf}(h)
   =\frac12 h(v_0)+\frac12 h(v_9),}
 \qquad
 \begin{aligned}
 v_0&=(0,0,0,N-d,d),\\
 v_9&=(0,a,N-a,0,0).
 \end{aligned}
\]
Equivalently, if \(\chi_p^{\rm rf}(v)\) denotes the coefficient of
\(h(v)\) in \(F_p^{\rm rf}\), then
\[
 \chi_p^{\rm rf}(v)=
 \begin{cases}
  \frac12,&v=v_0,\\[1mm]
  \frac12,&v=v_9,\\[1mm]
  0,&\text{otherwise}.
 \end{cases}
\]
\end{theorem}

\begin{proof}
Put
\[
        b_p(v):=\eta_p(v)+a_p(v).
\]
By \eqref{eq:rf-chi},
\[
        \chi_p^{\rm rf}(v)
          =b_p(v)-\frac13c_p(v).
\]
We show that \(c_p(v)=3b_p(v)\) except at the two displayed
vertices.

At an ordinary interior lattice vertex, the coefficients contributed
by the six incident \(U/D\)-triangles are, in cyclic order and up to
reversal,
\[
                 0,\frac32,\frac32,0,\frac32,\frac32.
\]
Figure~\ref{fig:rf-six-incident-triangles} records the unscaled local
calculation behind this list.

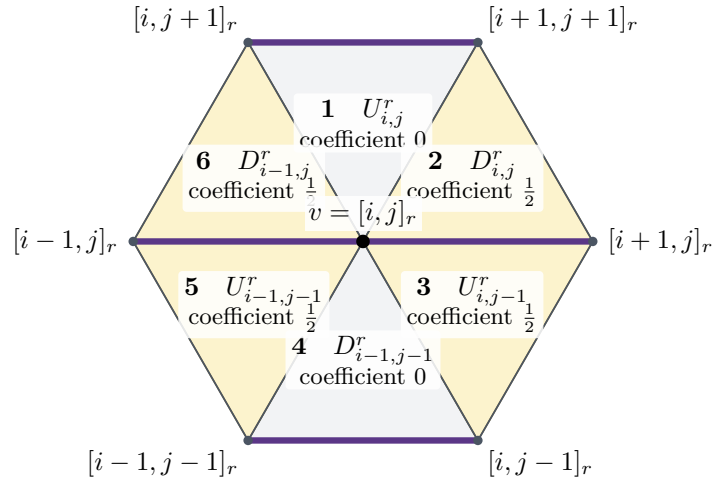
\begin{figure}[htbp]
\centering
\definecolor{rfactivefill}{RGB}{252,243,205}
\definecolor{rfinactivefill}{RGB}{242,243,245}
\definecolor{rfpotentialedge}{RGB}{91,55,135}
\definecolor{rflatticeedge}{RGB}{75,85,99}
\definecolor{rfactiveborder}{RGB}{181,132,0}
\definecolor{rfinactiveborder}{RGB}{145,151,161}
\begin{tikzpicture}[
  x=0.95cm,
  y=0.95cm,
  rfLattice/.style={draw=rflatticeedge,line width=0.7pt},
  rfPotential/.style={
    draw=rfpotentialedge,line width=2.4pt,line cap=round
  },
  rfVertex/.style={circle,fill=black,inner sep=1.8pt},
  rfOuterVertex/.style={circle,fill=rflatticeedge,inner sep=1.25pt},
  rfTriangleLabel/.style={
    align=center,
    fill=white,
    fill opacity=0.86,
    text opacity=1,
    rounded corners=1.5pt,
    inner sep=2.2pt,
    font=\small
  },
  rfCoordinateLabel/.style={
    fill=white,
    fill opacity=0.88,
    text opacity=1,
    inner sep=1.4pt,
    font=\small
  }
]
  \coordinate (rfv)  at (0,0);
  \coordinate (rfE)  at (3.20,0);
  \coordinate (rfNE) at (1.60,2.7713);
  \coordinate (rfNW) at (-1.60,2.7713);
  \coordinate (rfW)  at (-3.20,0);
  \coordinate (rfSW) at (-1.60,-2.7713);
  \coordinate (rfSE) at (1.60,-2.7713);

  \filldraw[
    fill=rfinactivefill,draw=rfinactiveborder,line width=0.55pt
  ] (rfv)--(rfNW)--(rfNE)--cycle;
  \filldraw[
    fill=rfinactivefill,draw=rfinactiveborder,line width=0.55pt
  ] (rfv)--(rfSW)--(rfSE)--cycle;

  \filldraw[
    fill=rfactivefill,draw=rfactiveborder,line width=0.55pt
  ] (rfv)--(rfNE)--(rfE)--cycle;
  \filldraw[
    fill=rfactivefill,draw=rfactiveborder,line width=0.55pt
  ] (rfv)--(rfE)--(rfSE)--cycle;
  \filldraw[
    fill=rfactivefill,draw=rfactiveborder,line width=0.55pt
  ] (rfv)--(rfSW)--(rfW)--cycle;
  \filldraw[
    fill=rfactivefill,draw=rfactiveborder,line width=0.55pt
  ] (rfv)--(rfW)--(rfNW)--cycle;

  \draw[rfLattice]
    (rfNW)--(rfNE)--(rfE)--(rfSE)--(rfSW)--(rfW)--cycle;
  \foreach \p in {rfNW,rfNE,rfE,rfSE,rfSW,rfW} {
    \draw[rfLattice] (rfv)--(\p);
  }

  \draw[rfPotential] (rfNW)--(rfNE);
  \draw[rfPotential] (rfv)--(rfE);
  \draw[rfPotential] (rfSW)--(rfSE);
  \draw[rfPotential] (rfW)--(rfv);

  \node[rfTriangleLabel] at (0,1.65)
    {\textbf{1}\quad \(U^r_{i,j}\)\\[-1pt]
     \footnotesize coefficient \(0\)};
  \node[rfTriangleLabel] at (1.52,0.88)
    {\textbf{2}\quad \(D^r_{i,j}\)\\[-1pt]
     \footnotesize coefficient \(\frac12\)};
  \node[rfTriangleLabel] at (1.52,-0.88)
    {\textbf{3}\quad \(U^r_{i,j-1}\)\\[-1pt]
     \footnotesize coefficient \(\frac12\)};
  \node[rfTriangleLabel] at (0,-1.65)
    {\textbf{4}\quad \(D^r_{i-1,j-1}\)\\[-1pt]
     \footnotesize coefficient \(0\)};
  \node[rfTriangleLabel] at (-1.52,-0.88)
    {\textbf{5}\quad \(U^r_{i-1,j-1}\)\\[-1pt]
     \footnotesize coefficient \(\frac12\)};
  \node[rfTriangleLabel] at (-1.52,0.88)
    {\textbf{6}\quad \(D^r_{i-1,j}\)\\[-1pt]
     \footnotesize coefficient \(\frac12\)};

  \foreach \p in {rfNW,rfNE,rfE,rfSE,rfSW,rfW} {
    \node[rfOuterVertex] at (\p) {};
  }
  \node[rfVertex] at (rfv) {};
  \node[rfCoordinateLabel,above=4pt] at (rfv) {\(v=[i,j]_r\)};

  \node[rfCoordinateLabel,above left=2pt]
    at (rfNW) {\([i,j+1]_r\)};
  \node[rfCoordinateLabel,above right=2pt]
    at (rfNE) {\([i+1,j+1]_r\)};
  \node[rfCoordinateLabel,right=4pt]
    at (rfE) {\([i+1,j]_r\)};
  \node[rfCoordinateLabel,below right=2pt]
    at (rfSE) {\([i,j-1]_r\)};
  \node[rfCoordinateLabel,below left=2pt]
    at (rfSW) {\([i-1,j-1]_r\)};
  \node[rfCoordinateLabel,left=4pt]
    at (rfW) {\([i-1,j]_r\)};
\end{tikzpicture}
\caption{Local triangle-potential coefficients at an ordinary interior
vertex \(v=[i,j]_r\).  Starting with \(U^r_{i,j}\), the labels
\(1,\ldots,6\) run clockwise.  A gold triangle contains \(h(v)\) in
\(\phi_h(T)\) with coefficient \(1/2\), whereas a gray triangle has
coefficient \(0\).  Within each triangle, the thick purple segment is
the potential edge whose endpoint heights enter \(\phi_h(T)\); it is
not a matching selection.  Whenever the adjacent elementary triangle
across this edge is available, the two triangles form a red lozenge in
the global color convention, and the purple edge is its short
diagonal.  Thus the unscaled cyclic list is
\(0,\frac12,\frac12,0,\frac12,\frac12\), and the factor \(3\) in the
definition of \(c_p(v)\) gives the displayed list in the proof.}
\label{fig:rf-six-incident-triangles}
\end{figure}

It is important here that not every cyclic interval in this list has
sum \(3\) or \(6\).  Lemma~\ref{lem:rf-local-star-exhaustion} gives the
stronger, more specific fact that, away from the two asymmetric caps,
the \emph{nonzero} retained coefficients occur in complete local pairs
\((\frac32,\frac32)\).  Its nine coefficient classes are shown in
Figure~\ref{fig:rf-local-type-atlas}.  Panels
\textup{(a)}--\textup{(g)} form the balanced block, while panels
\textup{(h)} and \textup{(i)} are the two asymmetric caps.  In the
table, a \emph{gold cell at \(v\)} means a retained incident elementary
triangle \(T\) with \([h(v)]\phi_h(T)=1/2\).  This is a coefficient
marker, not a fourth lozenge color.  Each gold cell contributes
\(3/2\) to \(c_p(v)\).
\begingroup
\renewcommand{\arraystretch}{1.08}
\[
\begin{array}{@{}c@{\quad}l@{\quad}c@{\quad}c@{\quad}c@{\quad}c@{}}
\text{panel}&\text{coefficient class}&\#\text{ gold}&c_p(v)
 &\eta_p(v)+a_p(v)=b_p(v)&\chi_p^{\rm rf}(v)\\ \hline
\textup{(a)}&\text{no retained incident cell}&0&0&0+0=0&0\\
\textup{(b)}&\text{ordinary one-sector boundary}&2&3&1+0=1&0\\
\textup{(c)}&\text{selected ordinary six-sided}&4&6&2+0=2&0\\
\textup{(d)}&\text{internal retained seam}&4&6&1+1=2&0\\
\textup{(e)}&\text{lower seam foot }(v_8,v_6)&0&0&0+0=0&0\\
\textup{(f)}&\text{upper seam endpoint }(u_2,u_3)&2&3&0+1=1&0\\
\textup{(g)}&\text{reentrant }v_7=q_*&4&6&2+0=2&0\\ \hline
\textup{(h)}&\text{asymmetric cap }v_9
  &1&\tfrac32&1+0=1&\tfrac12\\
\textup{(i)}&\text{asymmetric cap }v_0
  &1&\tfrac32&1+0=1&\tfrac12
\end{array}
\tag{\(*\)}
\]
\endgroup
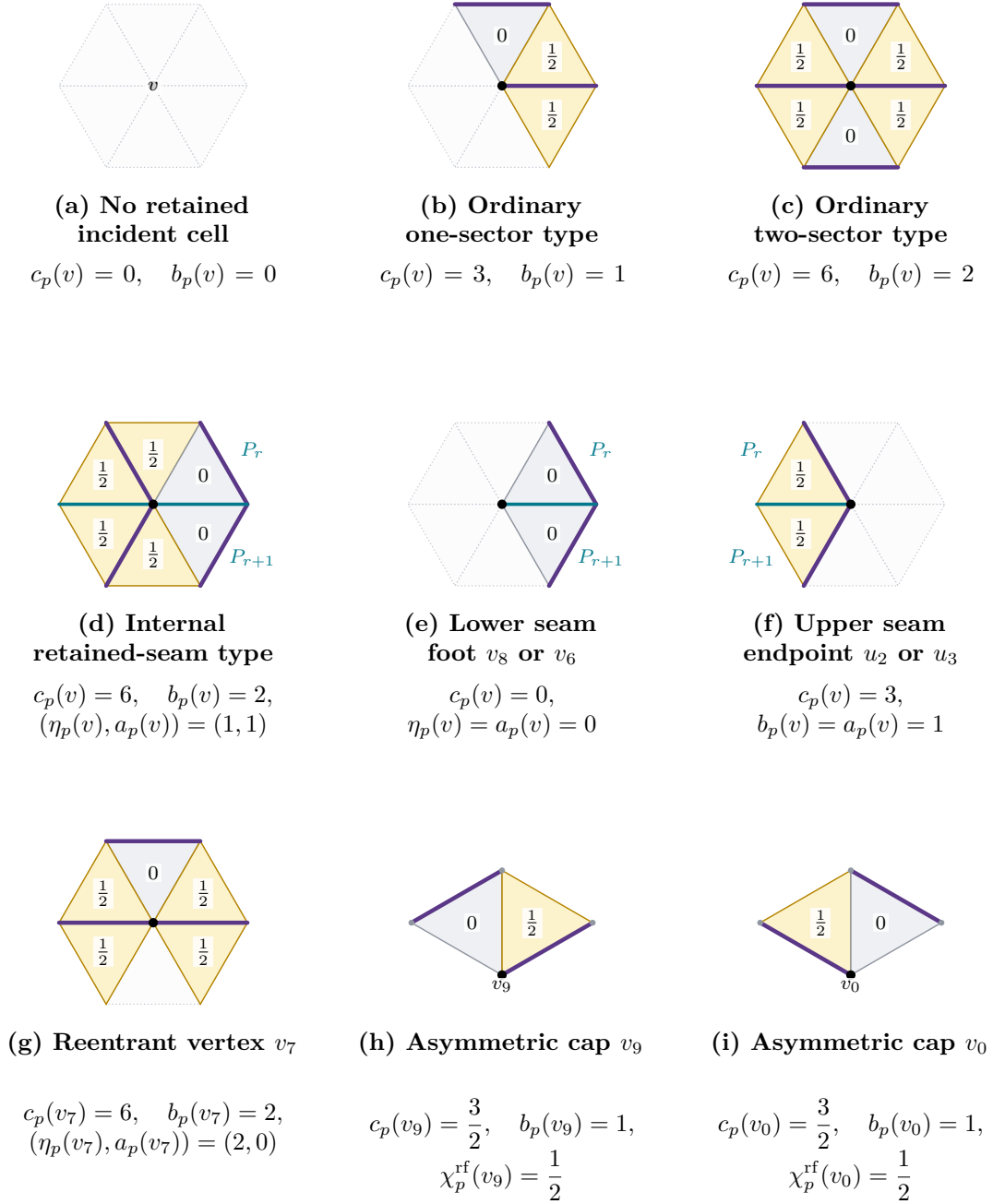
\begin{figure}[p]
\centering
\definecolor{rfatlaspositive}{RGB}{252,243,205}
\definecolor{rfatlaszero}{RGB}{238,240,243}
\definecolor{rfatlasghost}{RGB}{252,252,253}
\definecolor{rfatlasgoldedge}{RGB}{181,132,0}
\definecolor{rfatlasgrayedge}{RGB}{145,151,161}
\definecolor{rfatlasghostedge}{RGB}{190,195,202}
\definecolor{rfatlaspotential}{RGB}{91,55,135}
\definecolor{rfatlasseam}{RGB}{0,121,140}
\begin{tikzpicture}[
  rfGhost/.style={
    fill=rfatlasghost,draw=rfatlasghostedge,
    densely dotted,line width=0.45pt
  },
  rfPositive/.style={
    fill=rfatlaspositive,draw=rfatlasgoldedge,line width=0.55pt
  },
  rfZero/.style={
    fill=rfatlaszero,draw=rfatlasgrayedge,line width=0.55pt
  },
  rfPotential/.style={
    draw=rfatlaspotential,line width=1.7pt,line cap=round
  },
  rfSeam/.style={
    draw=rfatlasseam,line width=1.35pt,line cap=round
  },
  rfVertex/.style={circle,fill=black,inner sep=1.35pt},
  rfOuterVertex/.style={circle,fill=rfatlasgrayedge,inner sep=0.85pt},
  rfCoeff/.style={
    fill=white,fill opacity=0.82,text opacity=1,
    inner sep=1pt,font=\scriptsize
  },
  rfPanelTitle/.style={
    anchor=north,align=center,font=\small\bfseries,text width=4.1cm
  },
  rfPanelValue/.style={
    anchor=north,align=center,font=\small,text width=4.1cm
  },
  rfSheet/.style={font=\scriptsize,text=rfatlasseam}
]
  \begin{scope}[xshift=-5.0cm,yshift=6.0cm]
    \coordinate (v) at (0,0);
    \coordinate (E) at (1.35,0);
    \coordinate (NE) at (0.675,1.169);
    \coordinate (NW) at (-0.675,1.169);
    \coordinate (W) at (-1.35,0);
    \coordinate (SW) at (-0.675,-1.169);
    \coordinate (SE) at (0.675,-1.169);
    \path[rfGhost] (v)--(NW)--(NE)--cycle;
    \path[rfGhost] (v)--(NE)--(E)--cycle;
    \path[rfGhost] (v)--(E)--(SE)--cycle;
    \path[rfGhost] (v)--(SE)--(SW)--cycle;
    \path[rfGhost] (v)--(SW)--(W)--cycle;
    \path[rfGhost] (v)--(W)--(NW)--cycle;
    \node[rfVertex] at (v) {};
    \node[rfCoeff] at (0,0) {\(v\)};
    \node[rfPanelTitle] at (0,-1.42)
      {\textup{(a)} No retained incident cell};
    \node[rfPanelValue] at (0,-2.40)
      {\(c_p(v)=0,\quad b_p(v)=0\)};
  \end{scope}

  \begin{scope}[yshift=6.0cm]
    \coordinate (v) at (0,0);
    \coordinate (E) at (1.35,0);
    \coordinate (NE) at (0.675,1.169);
    \coordinate (NW) at (-0.675,1.169);
    \coordinate (W) at (-1.35,0);
    \coordinate (SW) at (-0.675,-1.169);
    \coordinate (SE) at (0.675,-1.169);
    \path[rfGhost] (v)--(NW)--(NE)--cycle;
    \path[rfGhost] (v)--(NE)--(E)--cycle;
    \path[rfGhost] (v)--(E)--(SE)--cycle;
    \path[rfGhost] (v)--(SE)--(SW)--cycle;
    \path[rfGhost] (v)--(SW)--(W)--cycle;
    \path[rfGhost] (v)--(W)--(NW)--cycle;
    \path[rfZero] (v)--(NW)--(NE)--cycle;
    \path[rfPositive] (v)--(NE)--(E)--cycle;
    \path[rfPositive] (v)--(E)--(SE)--cycle;
    \draw[rfPotential] (NW)--(NE);
    \draw[rfPotential] (v)--(E);
    \node[rfCoeff] at (0,0.72) {\(0\)};
    \node[rfCoeff] at (0.72,0.42) {\(\frac12\)};
    \node[rfCoeff] at (0.72,-0.42) {\(\frac12\)};
    \node[rfVertex] at (v) {};
    \node[rfPanelTitle] at (0,-1.42)
      {\textup{(b)} Ordinary one-sector type};
    \node[rfPanelValue] at (0,-2.40)
      {\(c_p(v)=3,\quad b_p(v)=1\)};
  \end{scope}

  \begin{scope}[xshift=5.0cm,yshift=6.0cm]
    \coordinate (v) at (0,0);
    \coordinate (E) at (1.35,0);
    \coordinate (NE) at (0.675,1.169);
    \coordinate (NW) at (-0.675,1.169);
    \coordinate (W) at (-1.35,0);
    \coordinate (SW) at (-0.675,-1.169);
    \coordinate (SE) at (0.675,-1.169);
    \path[rfGhost] (v)--(NW)--(NE)--cycle;
    \path[rfGhost] (v)--(NE)--(E)--cycle;
    \path[rfGhost] (v)--(E)--(SE)--cycle;
    \path[rfGhost] (v)--(SE)--(SW)--cycle;
    \path[rfGhost] (v)--(SW)--(W)--cycle;
    \path[rfGhost] (v)--(W)--(NW)--cycle;
    \path[rfZero] (v)--(NW)--(NE)--cycle;
    \path[rfZero] (v)--(SE)--(SW)--cycle;
    \path[rfPositive] (v)--(NE)--(E)--cycle;
    \path[rfPositive] (v)--(E)--(SE)--cycle;
    \path[rfPositive] (v)--(SW)--(W)--cycle;
    \path[rfPositive] (v)--(W)--(NW)--cycle;
    \draw[rfPotential] (v)--(E);
    \draw[rfPotential] (W)--(v);
    \draw[rfPotential] (NW)--(NE);
    \draw[rfPotential] (SW)--(SE);
    \node[rfCoeff] at (0,0.72) {\(0\)};
    \node[rfCoeff] at (0,-0.72) {\(0\)};
    \node[rfCoeff] at (0.72,0.42) {\(\frac12\)};
    \node[rfCoeff] at (0.72,-0.42) {\(\frac12\)};
    \node[rfCoeff] at (-0.72,-0.42) {\(\frac12\)};
    \node[rfCoeff] at (-0.72,0.42) {\(\frac12\)};
    \node[rfVertex] at (v) {};
    \node[rfPanelTitle] at (0,-1.42)
      {\textup{(c)} Ordinary two-sector type};
    \node[rfPanelValue] at (0,-2.40)
      {\(c_p(v)=6,\quad b_p(v)=2\)};
  \end{scope}

  \begin{scope}[xshift=-5.0cm]
    \coordinate (v) at (0,0);
    \coordinate (E) at (1.35,0);
    \coordinate (NE) at (0.675,1.169);
    \coordinate (NW) at (-0.675,1.169);
    \coordinate (W) at (-1.35,0);
    \coordinate (SW) at (-0.675,-1.169);
    \coordinate (SE) at (0.675,-1.169);
    \path[rfGhost] (v)--(W)--(NW)--cycle;
    \path[rfGhost] (v)--(NW)--(NE)--cycle;
    \path[rfGhost] (v)--(NE)--(E)--cycle;
    \path[rfGhost] (v)--(E)--(SE)--cycle;
    \path[rfGhost] (v)--(SE)--(SW)--cycle;
    \path[rfGhost] (v)--(SW)--(W)--cycle;
    \path[rfPositive] (v)--(W)--(NW)--cycle;
    \path[rfPositive] (v)--(NW)--(NE)--cycle;
    \path[rfZero] (v)--(NE)--(E)--cycle;
    \path[rfZero] (v)--(E)--(SE)--cycle;
    \path[rfPositive] (v)--(SE)--(SW)--cycle;
    \path[rfPositive] (v)--(SW)--(W)--cycle;
    \draw[rfPotential] (v)--(NW);
    \draw[rfPotential] (v)--(SW);
    \draw[rfPotential] (NE)--(E);
    \draw[rfPotential] (E)--(SE);
    \draw[rfSeam] (W)--(v)--(E);
    \node[rfCoeff] at (-0.72,0.42) {\(\frac12\)};
    \node[rfCoeff] at (0,0.72) {\(\frac12\)};
    \node[rfCoeff] at (0.72,0.42) {\(0\)};
    \node[rfCoeff] at (0.72,-0.42) {\(0\)};
    \node[rfCoeff] at (0,-0.72) {\(\frac12\)};
    \node[rfCoeff] at (-0.72,-0.42) {\(\frac12\)};
    \node[rfSheet] at (1.42,0.78) {\(P_r\)};
    \node[rfSheet] at (1.42,-0.78) {\(P_{r+1}\)};
    \node[rfVertex] at (v) {};
    \node[rfPanelTitle] at (0,-1.42)
      {\textup{(d)} Internal retained-seam type};
    \node[rfPanelValue] at (0,-2.40)
      {\(\begin{gathered}
        c_p(v)=6,\quad b_p(v)=2,\\[-1mm]
        (\eta_p(v),a_p(v))=(1,1)
      \end{gathered}\)};
  \end{scope}

  \begin{scope}
    \coordinate (v) at (0,0);
    \coordinate (E) at (1.35,0);
    \coordinate (NE) at (0.675,1.169);
    \coordinate (NW) at (-0.675,1.169);
    \coordinate (W) at (-1.35,0);
    \coordinate (SW) at (-0.675,-1.169);
    \coordinate (SE) at (0.675,-1.169);
    \path[rfGhost] (v)--(W)--(NW)--cycle;
    \path[rfGhost] (v)--(NW)--(NE)--cycle;
    \path[rfGhost] (v)--(NE)--(E)--cycle;
    \path[rfGhost] (v)--(E)--(SE)--cycle;
    \path[rfGhost] (v)--(SE)--(SW)--cycle;
    \path[rfGhost] (v)--(SW)--(W)--cycle;
    \path[rfZero] (v)--(NE)--(E)--cycle;
    \path[rfZero] (v)--(E)--(SE)--cycle;
    \draw[rfPotential] (NE)--(E);
    \draw[rfPotential] (E)--(SE);
    \draw[rfSeam] (v)--(E);
    \node[rfCoeff] at (0.72,0.42) {\(0\)};
    \node[rfCoeff] at (0.72,-0.42) {\(0\)};
    \node[rfSheet] at (1.42,0.78) {\(P_r\)};
    \node[rfSheet] at (1.42,-0.78) {\(P_{r+1}\)};
    \node[rfVertex] at (v) {};
    \node[rfPanelTitle] at (0,-1.42)
      {\textup{(e)} Lower seam foot \(v_8\) or \(v_6\)};
    \node[rfPanelValue] at (0,-2.40)
      {\(\begin{gathered}
        c_p(v)=0,\\[-1mm]
        \eta_p(v)=a_p(v)=0
      \end{gathered}\)};
  \end{scope}

  \begin{scope}[xshift=5.0cm]
    \coordinate (v) at (0,0);
    \coordinate (E) at (1.35,0);
    \coordinate (NE) at (0.675,1.169);
    \coordinate (NW) at (-0.675,1.169);
    \coordinate (W) at (-1.35,0);
    \coordinate (SW) at (-0.675,-1.169);
    \coordinate (SE) at (0.675,-1.169);
    \path[rfGhost] (v)--(W)--(NW)--cycle;
    \path[rfGhost] (v)--(NW)--(NE)--cycle;
    \path[rfGhost] (v)--(NE)--(E)--cycle;
    \path[rfGhost] (v)--(E)--(SE)--cycle;
    \path[rfGhost] (v)--(SE)--(SW)--cycle;
    \path[rfGhost] (v)--(SW)--(W)--cycle;
    \path[rfPositive] (v)--(W)--(NW)--cycle;
    \path[rfPositive] (v)--(SW)--(W)--cycle;
    \draw[rfPotential] (v)--(NW);
    \draw[rfPotential] (v)--(SW);
    \draw[rfSeam] (W)--(v);
    \node[rfCoeff] at (-0.72,0.42) {\(\frac12\)};
    \node[rfCoeff] at (-0.72,-0.42) {\(\frac12\)};
    \node[rfSheet] at (-1.42,0.78) {\(P_r\)};
    \node[rfSheet] at (-1.42,-0.78) {\(P_{r+1}\)};
    \node[rfVertex] at (v) {};
    \node[rfPanelTitle] at (0,-1.42)
      {\textup{(f)} Upper seam endpoint \(u_2\) or \(u_3\)};
    \node[rfPanelValue] at (0,-2.40)
      {\(\begin{gathered}
        c_p(v)=3,\\[-1mm]
        b_p(v)=a_p(v)=1
      \end{gathered}\)};
  \end{scope}

  \begin{scope}[xshift=-5.0cm,yshift=-6.0cm]
    \coordinate (v) at (0,0);
    \coordinate (E) at (1.35,0);
    \coordinate (NE) at (0.675,1.169);
    \coordinate (NW) at (-0.675,1.169);
    \coordinate (W) at (-1.35,0);
    \coordinate (SW) at (-0.675,-1.169);
    \coordinate (SE) at (0.675,-1.169);
    \path[rfGhost] (v)--(NW)--(NE)--cycle;
    \path[rfGhost] (v)--(NE)--(E)--cycle;
    \path[rfGhost] (v)--(E)--(SE)--cycle;
    \path[rfGhost] (v)--(SE)--(SW)--cycle;
    \path[rfGhost] (v)--(SW)--(W)--cycle;
    \path[rfGhost] (v)--(W)--(NW)--cycle;
    \path[rfZero] (v)--(NW)--(NE)--cycle;
    \path[rfPositive] (v)--(NE)--(E)--cycle;
    \path[rfPositive] (v)--(E)--(SE)--cycle;
    \path[rfPositive] (v)--(SW)--(W)--cycle;
    \path[rfPositive] (v)--(W)--(NW)--cycle;
    \draw[rfPotential] (NW)--(NE);
    \draw[rfPotential] (v)--(E);
    \draw[rfPotential] (W)--(v);
    \node[rfCoeff] at (0,0.72) {\(0\)};
    \node[rfCoeff] at (0.72,0.42) {\(\frac12\)};
    \node[rfCoeff] at (0.72,-0.42) {\(\frac12\)};
    \node[rfCoeff] at (-0.72,-0.42) {\(\frac12\)};
    \node[rfCoeff] at (-0.72,0.42) {\(\frac12\)};
    \node[rfVertex] at (v) {};
    \node[rfPanelTitle] at (0,-1.42)
      {\textup{(g)} Reentrant vertex \(v_7\)};
    \node[rfPanelValue] at (0,-2.40)
      {\(\begin{gathered}
        c_p(v_7)=6,\quad b_p(v_7)=2,\\[-1mm]
        (\eta_p(v_7),a_p(v_7))=(2,0)
      \end{gathered}\)};
  \end{scope}

  \begin{scope}[yshift=-6.0cm]
    \coordinate (v) at (0,-0.75);
    \coordinate (A) at (0,0.75);
    \coordinate (L) at (-1.299,0);
    \coordinate (R) at (1.299,0);
    \path[rfZero] (v)--(A)--(L)--cycle;
    \path[rfPositive] (v)--(R)--(A)--cycle;
    \draw[rfPotential] (A)--(L);
    \draw[rfPotential] (v)--(R);
    \node[rfCoeff] at (-0.43,0) {\(0\)};
    \node[rfCoeff] at (0.43,0) {\(\frac12\)};
    \foreach \p in {A,L,R} {\node[rfOuterVertex] at (\p) {};}
    \node[rfVertex] at (v) {};
    \node[rfCoeff,below=1pt] at (v) {\(v_9\)};
    \node[rfPanelTitle] at (0,-1.42)
      {\textup{(h)} Asymmetric cap \(v_9\)};
    \node[rfPanelValue] at (0,-2.40)
      {\(\begin{gathered}
        c_p(v_9)=\frac32,\quad b_p(v_9)=1,\\[-1mm]
        \chi_p^{\rm rf}(v_9)=\frac12
      \end{gathered}\)};
  \end{scope}

  \begin{scope}[xshift=5.0cm,yshift=-6.0cm]
    \coordinate (v) at (0,-0.75);
    \coordinate (A) at (0,0.75);
    \coordinate (L) at (-1.299,0);
    \coordinate (R) at (1.299,0);
    \path[rfPositive] (v)--(A)--(L)--cycle;
    \path[rfZero] (v)--(R)--(A)--cycle;
    \draw[rfPotential] (v)--(L);
    \draw[rfPotential] (A)--(R);
    \node[rfCoeff] at (-0.43,0) {\(\frac12\)};
    \node[rfCoeff] at (0.43,0) {\(0\)};
    \foreach \p in {A,L,R} {\node[rfOuterVertex] at (\p) {};}
    \node[rfVertex] at (v) {};
    \node[rfCoeff,below=1pt] at (v) {\(v_0\)};
    \node[rfPanelTitle] at (0,-1.42)
      {\textup{(i)} Asymmetric cap \(v_0\)};
    \node[rfPanelValue] at (0,-2.40)
      {\(\begin{gathered}
        c_p(v_0)=\frac32,\quad b_p(v_0)=1,\\[-1mm]
        \chi_p^{\rm rf}(v_0)=\frac12
      \end{gathered}\)};
  \end{scope}
\end{tikzpicture}
\caption{Atlas of the local coefficient types in the fixed-term
calculation.  Gold cells contribute \(1/2\) to the coefficient of
\(h(v)\) in the triangle-potential sum; after multiplication by \(3\),
each contributes \(3/2\) to \(c_p(v)\).  Gray cells are retained but
have coefficient \(0\), dotted cells are suppressed ambient
positions, purple segments are the potential edges described in
Figure~\ref{fig:rf-six-incident-triangles}, and teal segments are cut
seams.}
\label{fig:rf-local-type-atlas}
\end{figure}

In Figure~\ref{fig:rf-local-type-atlas}, panels \textup{(b)} and
\textup{(c)} stand for rotations and
reflections; panel \textup{(b)} also covers ordinary boundary
truncations such as \(v_4,v_5\).  Panel \textup{(g)} isolates
\(q_*=v_7\): in the
numbering of Figure~\ref{fig:rf-six-incident-triangles}, the dotted
fourth cell is absent, leaving the retained-incidence path
\(5\!-\!6\!-\!1\!-\!2\!-\!3\), of maximum-matching size \(2\).
Its total agrees with panel \textup{(c)}, although its geometry does
not.  With \(s_{r,k}\) as defined in
\eqref{eq:rf-seam-segment}, the lower feet \(v_8=s_{1,b}\) and
\(v_6=s_{2,c}\) begin a retained seam and therefore have \(a_p=0\);
the upper endpoints
\(u_2=s_{1,N}\) and \(u_3=s_{2,N}\) end one and have \(a_p=1\).
Panels \textup{(h)} and \textup{(i)} form the exceptional block of
\((*)\).  Each is drawn as the union of its
two equilateral elementary triangles.  Their respective \(b_p=1\)
contributions come from the bottom-tiling incidence and from
\(\beta_p(v_0)\).  Figure~\ref{fig:rf-local-type-atlas} depicts nine
coefficient classes, not literally nine cell-complex geometries,
because \(v_4,v_5\) are truncated versions of \textup{(b)}.

Here an ordinary sector is a pair of gold cells and therefore
contributes \(3\) to \(c_p(v)\).  Row \textup{(b)} contains the
straight-side configurations and their ordinary boundary
truncations, including \(v_4\) and \(v_5\); the local path bounds give
\(\eta_p(v)=1\), and in particular well-matchedness gives
\(m_{v_4}^{\max}=1\).  In row \textup{(c)}, the selected six-sided
vertex has \(w_v=3\), so \(\eta_p(v)=w_v-1=2\).  Row \textup{(g)} is
the geometrically distinct reentrant unselected vertex
\[
                     q_*=v_7=(a+d,0,b,c,0).
\]
Its four retained incidences form the open path
\(5\!-\!6\!-\!1\!-\!2\!-\!3\), so
\(\eta_p(v_7)=m_{v_7}^{\max}=2\) by
\eqref{eq:reentrant-corner-matching-exponent} and its accompanying
local path calculation.

For row \textup{(d)}, the selected four-sided seam face has \(w_v=2\),
so \(\eta_p(v)=1\), while \(a_p(v)=1\) comes from the incoming
retained seam segment.  The lower feet \(v_8=s_{1,b}\) and
\(v_6=s_{2,c}\) form the distinct zero-gold row \textup{(e)}: their
ordinary incidence sets are empty by
\eqref{eq:seam-endpoint-empty-incidences}, and neither is the higher
endpoint of a retained seam segment, so \((\eta_p,a_p)=(0,0)\).
At the upper fan endpoint---\(u_2=s_{1,N}\) on \(02\) or
\(u_3=s_{2,N}\) on \(03\)---row \textup{(f)} has
\((\eta_p,a_p)=(0,1)\).  Consequently rows
\textup{(a)}--\textup{(g)} of \((*)\) satisfy
\[
       c_p(v)=3b_p(v),\qquad \chi_p^{\rm rf}(v)=0.
\]

We verify the two exceptional rows directly.  For row \textup{(h)},
at \(v_9\), the only
incident triangles with possibly nonzero coefficient are
\[
 D^1_{N-a-1,N-1}
 \quad\text{and}\quad
 U^1_{N-a,N-1},
\]
and their contributions to \(c_p(v_9)\) are \(0\) and \(3/2\).  Here
\(\beta_p(v_9)=0\), while the bottom tiling uses the unique retained
ordinary-lozenge incidence, so
\[
 \eta_p(v_9)=0+1=1,\qquad a_p(v_9)=0,
 \qquad b_p(v_9)=1,\qquad c_p(v_9)=\frac32,
\]
and therefore
\[
        \chi_p^{\rm rf}(v_9)=1-\frac12=\frac12.
\]

For row \textup{(i)}, the corresponding two contributions to
\(c_p(v_0)\) are \(0\) and \(3/2\).  Moreover,
\[
        m_{v_0}(\Xi_{M_{\min}})=0,
        \qquad
        \beta_p(v_0)=1,
\]
so
\[
 \eta_p(v_0)=1+0=1,\qquad a_p(v_0)=0,
 \qquad b_p(v_0)=1,\qquad c_p(v_0)=\frac32.
\]
Consequently
\[
        \chi_p^{\rm rf}(v_0)=1-\frac12=\frac12.
\]
This verifies the two exceptional rows of \((*)\).

Lemma~\ref{lem:rf-local-star-exhaustion} proves that there are no
further coefficient classes; Lemma~\ref{lem:local-boundary-classification}
and the seam and cap calculations above supply the corresponding
\(b_p(v)\)-values.  Rows
\textup{(a)}--\textup{(g)} give \(\chi_p^{\rm rf}=0\), while rows
\textup{(h)} and \textup{(i)} give \(\chi_p^{\rm rf}=1/2\) at
\(v_9\) and \(v_0\), respectively.  This proves the theorem.
\end{proof}

\begin{remark}[Gauge covariance]
\label{rem:rf-symmetric-gauge-covariance}
For comparison, let \(W_p^{\rm asym}\) denote the coordinate-asymmetric
Gelfand--Tsetlin tiling weight: its coefficients on local types
\(1,2,3\) are \(0,2,1\), respectively.  Equivalently, its coefficients
on the global colors \((R,G,B)\) are
\[
 (0,1,2)\quad\text{on \(P_1,P_3\)},
 \qquad
 (0,2,1)\quad\text{on \(P_2\)}.
\]
Thus this is not a second meaning of the color names.  It is a
coordinate-direction normalization.  Its singleton seam weights, in
the \((-,+)\) order at either seam, are
\[
 -\frac13\lambda_{\min}
 \qquad\text{and}\qquad
 \frac23\lambda_{\max},
\]
and \(F_p^{\rm asym}\) denotes its matching-independent term.  The
same coefficient calculation, with type-\(3\) and type-\(2\)
triangle-potential coefficients \((1,2)\) in place of
\((3/2,3/2)\), gives, for every legal tiling \(\Xi\),
\[
 W_p^{\rm rf}(\Xi;h)
 =W_p^{\rm asym}(\Xi;h)
   +\frac16\bigl(h(v_9)-h(v_0)\bigr),
\qquad
 F_p^{\rm rf}(h)
 =F_p^{\rm asym}(h)
   -\frac16\bigl(h(v_9)-h(v_0)\bigr).
\]
Thus the total affine form, its maximum, and the resulting recurrence
value are unchanged.  What changes is only the division into bulk,
seam, and fixed contributions.
\end{remark}

\begin{corollary}[Bottom-tiling identity]
\label{cor:rf-bottom-identity}
If \(\Xi_{M_{\min}}\) is the canonical bottom tiling, then
\begin{equation}
 W_p^{\rm rf}(\Xi_{M_{\min}};h)
 =h(v_0)-h(v_1)+h(v_2)-h(v_3)+h(v_4)
  -\frac12h(v_0)-\frac12h(v_9).
\label{eq:rf-bottom-identity}
\end{equation}
\end{corollary}

\begin{proof}
Use the bottom affine term from Part~I in
Theorem~\ref{thm:rf-exact-tiling}, and then apply
\eqref{eq:five-point-g}.
\end{proof}

\section{GUE evaluation of the fixed and seam terms}
\label{sec:rf-gue-boundary}

We use the macroscopic normalization of \cite{NarLimit}.  Let
\(M_1^{(n)},M_2^{(n)},M_3^{(n)}\) be independent standard \(n\times n\)
GUE matrices, fix scales \(\ell_1,\ell_2,\ell_3\in\mathbb R_{>0}\), and put
\begin{equation}
             \widehat X_r^{(n)}=\ell_r\sqrt n\,M_r^{(n)}.
\label{eq:rf-macro-matrices}
\end{equation}
In the notation of Definition~\ref{def:rf-centered-large-gap}, choose
the centered large-gap boundary
\[
 \widehat\Lambda_q^{(n)}
 =nK_{\rm gap}\left(\frac{n+1}{2}-q\right),
 \qquad \sum_{q=1}^n\widehat\Lambda_q^{(n)}=0,
\]
with the fixed \(K_{\rm gap}\) chosen in
\eqref{eq:rf-upfront-large-gap-event}, and set
\begin{equation}
 B_0^{(n)}=\widehat\Lambda^{(n)},
 \qquad
 B_r^{(n)}=B_{r-1}^{(n)}+
       \operatorname{diag}(\widehat X_r^{(n)}).
\label{eq:rf-B-boundaries}
\end{equation}
For \(0\le i\le j\le n\), define
\begin{equation}
 H_r^{(n)}(i,j)
 =\sum_{q=1}^j(B_{r-1}^{(n)})_q
  +\sum_{q=1}^i\lambda_q\!\left(
       (\widehat X_r^{(n)})^{[j]}\right),
 \qquad r=1,2,3,
\label{eq:rf-gue-hive-map}
\end{equation}
with eigenvalues in decreasing order.  This is the Gelfand--Tsetlin to
hive map used in \cite{NarLimit}.

\paragraph{Large-gap event and exceptional realizations.}
For a Hermitian matrix \(Y\), write
\({\rm width}(Y)=\lambda_1(Y)-\lambda_n(Y)\).  On \(\mathcal E_n\),
\[
 2\max_{\varnothing\ne I\subseteq\{1,2,3\}}
       {\rm width}(\widehat X_I^{(n)})
 \leq4\max_{\varnothing\ne I\subseteq\{1,2,3\}}
       \|\widehat X_I^{(n)}\|_{\rm op}
 \leq\frac12K_{\rm gap}n<K_{\rm gap}n.
\]
Lemma~\ref{lem:rf-centered-large-gap-valid} therefore shows that the
three input arrays are hives on this event and that every face array
produced by the successive flips is a hive as well.

Each \(\widehat X_I^{(n)}\) is itself a scaled GUE matrix, with scale
\((\sum_{r\in I}\ell_r^2)^{1/2}\).  The Gaussian operator-norm tail,
followed by a union bound over the seven nonempty subsets, gives
\[
                  \mathbb P(\mathcal E_n^c)\leq Ce^{-cn}
\]
once \(K_{\rm gap}\) is fixed sufficiently large.  No conditioning is
used in this paper.  On \(\mathcal E_n^c\) the arrays
\eqref{eq:rf-gue-hive-map} and the tropical octahedron recurrence are
still defined, and the exact fixed-graph formula of Part~I remains an
identity of finite maxima of affine forms.

The gap is only an auxiliary embedding gauge: the recurrence value is
\emph{exactly independent} of \(K_{\rm gap}\).  To see this directly,
use the red-free identity of
Theorem~\ref{thm:rf-exact-tiling}.  Varying the centered tuple
\(\widehat\Lambda^{(n)}\) changes only the red rhombus slacks in the
ordinary-tile term, and red lozenges have coefficient zero.  The two
seam terms
\eqref{eq:rf-seam02-discrete}--\eqref{eq:rf-seam03-discrete} contain
only eigenvalues of the matrices \(\widehat X_r^{(n)}\).  Finally,
\eqref{eq:rf-fixed-gue-form} contains no gap term because
\(\sum_q\widehat\Lambda_q^{(n)}=0\).  Every term on the right-hand side
of the exact formula is therefore independent of \(K_{\rm gap}\).
Equivalently, the one-coordinate gauge
\eqref{eq:rf-large-gap-gauge} vanishes on any tetrahedron whose vertex
set does not contain \(u_0\).
Thus \(\mathcal E_n\) is used only to place the Gelfand--Tsetlin arrays
inside the category of hives, and later to invoke hive inequalities.
Since the output is defined without conditioning and the gap cancels
exactly, no truncation or conditioning correction is required when
computing the mean or variance of \(\widehat O_n\).

Let \(p_n=(0,a_n,b_n,c_n,d_n)\), with
\begin{equation}
 \frac1n(a_n,b_n,c_n,d_n)\longrightarrow
             (\alpha,\beta,\gamma,\delta),
 \qquad \alpha,\beta,\gamma,\delta>0,
 \qquad \alpha+\beta+\gamma+\delta=1.
\label{eq:rf-target-proportions}
\end{equation}
Write \(F_{p_n}^{\rm rf,(n)}\) for the fixed term in this normalization.

\begin{proposition}[GUE form of the fixed term]
\label{prop:rf-fixed-gue-form}
One has the exact identity
\begin{equation}
 \boxed{
 \begin{aligned}
 F_{p_n}^{\rm rf,(n)}={}&
 \frac12\operatorname{Tr}
       (\widehat X_1^{(n)}+\widehat X_2^{(n)})\\
 &+\frac12\sum_{i=1}^{d_n}
       \lambda_i(\widehat X_3^{(n)})
 +\frac12\sum_{i=1}^{n-a_n}
       \lambda_i(\widehat X_1^{(n)}).
 \end{aligned}}
\label{eq:rf-fixed-gue-form}
\end{equation}
In the trace-zero simulation convention, the first line vanishes
identically.
\end{proposition}

\begin{proof}
In panel coordinates, \(v_0=[d_n,n]_3\) and
\(v_9(p_n)=[n-a_n,n]_1\).  Hence \eqref{eq:rf-gue-hive-map} gives
\begin{align*}
 h(v_0)&=\sum_{q=1}^n(B_2^{(n)})_q
       +\sum_{i=1}^{d_n}\lambda_i(\widehat X_3^{(n)}),\\
 h(v_9(p_n))&=\sum_{q=1}^n(B_0^{(n)})_q
       +\sum_{i=1}^{n-a_n}\lambda_i(\widehat X_1^{(n)}).
\end{align*}
The second boundary sum is zero, while the first is
\(\operatorname{Tr}(\widehat X_1^{(n)}+\widehat X_2^{(n)})\).
Now use Theorem~\ref{thm:rf-explicit-fixed-term}.
\end{proof}

Let \(\gamma_{\rm sc}:[0,1]\to[-2,2]\) be the decreasing semicircle
quantile, characterized by
\begin{equation}
 \int_{\gamma_{\rm sc}(u)}^2
       \frac{\sqrt{4-x^2}}{2\pi}\,dx=u.
\label{eq:rf-semicircle-quantile}
\end{equation}

\paragraph{Tail terminology.}
Following \cite[\S2.7]{VershyninHDP}, a real random variable \(X\) is
\emph{subexponential} if its \(\psi_1\)-norm
\begin{equation}
 \|X\|_{\psi_1}
 :=\inf\bigl\{K>0:\mathbb E\exp(|X|/K)\leq2\bigr\}
\label{eq:rf-psi1-definition}
\end{equation}
is finite.  Up to changes in absolute constants, this is equivalent to
\[
 \mathbb P\{|X|\geq t\}\leq
 2\exp\bigl(-t/(C\|X\|_{\psi_1})\bigr),
 \qquad t\geq0.
\]
A centered log-concave real random variable is automatically
subexponential.  Indeed, the one-dimensional log-concave tail estimate
\cite[Lemma~5.7]{LovaszVempalaLogconcave}, applied after normalization,
gives
\begin{equation}
 \mathbb P\{|X|>t\|X\|_2\}\leq e^{1-t},
 \qquad t>1,
\label{eq:rf-logconcave-subexponential}
\end{equation}
and therefore
\(\|X\|_{\psi_1}\leq C\|X\|_2\).

We use \emph{stretched-exponential tail of exponent \(\theta>0\)} for
the more general assertion
\begin{equation}
 \mathbb P\{|X|\geq Kt\}\leq C\exp(-ct^\theta),
 \qquad t\geq1,
\label{eq:rf-stretched-tail-definition}
\end{equation}
where \(c,C,K>0\).  Thus \(\theta=1\) is subexponential in
Vershynin's sense (up to constants), \(\theta>1\) is stronger, and
\(0<\theta<1\) is weaker.  The Tao--Vu bounds used below have
\eqref{eq:rf-stretched-tail-definition} for some unspecified
\(\theta>0\), so that estimate alone need not imply a
\(\psi_1\)-bound when \(\theta<1\).  In the present GUE setting,
subexponentiality follows separately from log-concavity; the Tao--Vu
estimate is used to identify the much sharper rank-dependent scale.

\begin{theorem}[Limit and concentration of the fixed term]
\label{thm:rf-fixed-limit}
Under \eqref{eq:rf-target-proportions},
\begin{equation}
 \frac{F_{p_n}^{\rm rf,(n)}}{n^2}
 \xrightarrow{\mathbb P} S_{p,0},
\label{eq:rf-fixed-convergence}
\end{equation}
where
\begin{equation}
 \boxed{
 S_{p,0}=
 \frac{\ell_3}{2}\int_0^\delta\gamma_{\rm sc}(u)\,du
 +\frac{\ell_1}{2}\int_0^{1-\alpha}
                   \gamma_{\rm sc}(u)\,du.}
\label{eq:rf-S0}
\end{equation}
Moreover,
\begin{equation}
 F_{p_n}^{\rm rf,(n)}-\mathbb E F_{p_n}^{\rm rf,(n)}
 =O_{\mathbb P}\bigl(n\log^{O(1)}n\bigr)=o_{\mathbb P}(n^2),
\label{eq:rf-fixed-concentration}
\end{equation}
and, more quantitatively,
\begin{equation}
 \bigl\|F_{p_n}^{\rm rf,(n)}-\mathbb E F_{p_n}^{\rm rf,(n)}\bigr\|_2
 =O\bigl(n\log^{O(1)}n\bigr).
\label{eq:rf-fixed-L2-preliminary}
\end{equation}
In fact, the fixed statistic is subexponential on the same scale:
\begin{equation}
 \bigl\|F_{p_n}^{\rm rf,(n)}
       -\mathbb E F_{p_n}^{\rm rf,(n)}\bigr\|_{\psi_1}
 =O\bigl(n\log^{O(1)}n\bigr).
\label{eq:rf-fixed-psi1}
\end{equation}
The convergence in \eqref{eq:rf-fixed-convergence} also holds in
\(L^1\).
\end{theorem}

\begin{proof}
The empirical spectral measure of
\(n^{-1}\widehat X_r^{(n)}=\ell_r n^{-1/2}M_r^{(n)}\) converges to the
semicircle law of scale \(\ell_r\), and the GUE operator norm has uniformly
integrable tails.  Therefore normalized partial sums of the largest
\(\lfloor tn\rfloor\) eigenvalues converge in expectation to
\(\ell_r\int_0^t\gamma_{\rm sc}(u)\,du\).  The centered trace terms divided
by \(n^2\) converge to zero.  Substitution in
\eqref{eq:rf-fixed-gue-form} therefore gives
\begin{equation}
 \frac{\mathbb E F_{p_n}^{\rm rf,(n)}}{n^2}
 \longrightarrow S_{p,0}.
\label{eq:rf-fixed-mean-convergence}
\end{equation}

For concentration, apply the eigenvalue estimate of
\cite{TaoVuConcentration}, in the normalization used in
\cite{NarLimit}.  Uniformly in the rank,
\[
 |\lambda_i-\mathbb E\lambda_i|
 \le n^{1/3}\min(i,n-i+1)^{-1/3}\log^{O(1)}n
\]
with overwhelming probability.  Summing over any interval of ranks gives
\(O(n\log^{O(1)}n)\).  The trace is a centered Gaussian of order \(n\).
Substitute these bounds in \eqref{eq:rf-fixed-gue-form}.  More
explicitly, put
\[
 Y_{r,i}:=\lambda_i(\widehat X_r^{(n)})
          -\mathbb E\lambda_i(\widehat X_r^{(n)}),
 \qquad
 T:=\operatorname{Tr}(\widehat X_1^{(n)}
             +\widehat X_2^{(n)})
       -\mathbb E\operatorname{Tr}(\widehat X_1^{(n)}
             +\widehat X_2^{(n)}).
\]
Here Minkowski's inequality means the triangle inequality in
\(L^2(\Omega)\),
\[
 \left\|\sum_j Z_j\right\|_2\leq\sum_j\|Z_j\|_2,
 \qquad Z_j\in L^2(\Omega).
\]
Applied to \eqref{eq:rf-fixed-gue-form}, it gives the concrete bound
\begin{align*}
 \|F_{p_n}^{\rm rf,(n)}
       -\mathbb EF_{p_n}^{\rm rf,(n)}\|_2
 &\leq \frac12\|T\|_2
   +\frac12\sum_{i=1}^{d_n}\|Y_{3,i}\|_2
   +\frac12\sum_{i=1}^{n-a_n}\|Y_{1,i}\|_2.
\end{align*}
After the polynomial prefactor in the Tao--Vu estimate is absorbed
into \(\log^{O(1)}n\), integration of its
stretched-exponential tail yields
\[
 \|Y_{r,i}\|_2
 \leq C n^{1/3}\min(i,n-i+1)^{-1/3}\log^{O(1)}n.
\]
Since
\(\sum_{i=1}^n\min(i,n-i+1)^{-1/3}=O(n^{2/3})\)
and \(\|T\|_2=O(n)\), the preceding Minkowski bound proves
\eqref{eq:rf-fixed-L2-preliminary}.

For completeness, on the decreasing Weyl chamber the density of the
ordered eigenvalue vector of a GUE matrix is proportional to
\[
 \exp\left(-c\sum_i\lambda_i^2\right)
 \prod_{i<j}(\lambda_i-\lambda_j)^2.
\]
Its logarithm is concave.  Products of log-concave densities are
log-concave, and linear marginals remain log-concave
\cite[Theorem~5.1]{LovaszVempalaLogconcave}.  Formula
\eqref{eq:rf-fixed-gue-form} shows that
\(F_{p_n}^{\rm rf,(n)}\) is a linear functional of the three ordered
eigenvalue vectors.  Its law is therefore log-concave, and
\eqref{eq:rf-logconcave-subexponential} together with
\eqref{eq:rf-fixed-L2-preliminary} proves
\eqref{eq:rf-fixed-psi1}.

Finally, the claimed \(L^1\) convergence follows directly.  By Cauchy--Schwarz,
\begin{align*}
 \mathbb E\left|
   \frac{F_{p_n}^{\rm rf,(n)}}{n^2}-S_{p,0}\right|
 &\leq
 \frac{\|F_{p_n}^{\rm rf,(n)}
       -\mathbb EF_{p_n}^{\rm rf,(n)}\|_2}{n^2}
 +\left|
   \frac{\mathbb EF_{p_n}^{\rm rf,(n)}}{n^2}-S_{p,0}\right|.
\end{align*}
The first term is
\(O(n^{-1}\log^{O(1)}n)\) by
\eqref{eq:rf-fixed-L2-preliminary}, and the second tends to zero by
\eqref{eq:rf-fixed-mean-convergence}.  Hence the left-hand side tends
to zero, which is precisely convergence in \(L^1\).
\end{proof}

\subsection{Continuum region and admissible height functions}
\label{subsec:rf-height-spaces}

We first make the continuum domain and the admissible height space
explicit.  Write
\[
 p=(\alpha,\beta,\gamma,\delta),\qquad
 \alpha,\beta,\gamma,\delta>0,
 \qquad \alpha+\beta+\gamma+\delta=1,
\]
for the limiting barycentric coordinates of the target.  In the three
initial panels put
\begin{align}
 \mathsf R_{p,1}
 &=\{x\in\Delta_{012}(1):x_1\leq\alpha,\ x_2\geq\beta\},\notag\\
 \mathsf R_{p,2}
 &=\{x\in\Delta_{023}(1):x_2\geq\beta
             \text{ or }x_3\geq\gamma\},\notag\\
 \mathsf R_{p,3}
 &=\{x\in\Delta_{034}(1):x_3\geq\gamma,\ x_4\leq\delta\}.
\label{eq:rf-continuum-panel-regions}
\end{align}
The two copies of the retained seams are
\begin{align}
 \Gamma_{02}&=\{x\in\Delta_{02}(1):x_2\in[\beta,1]\},\notag\\
 \Gamma_{03}&=\{x\in\Delta_{03}(1):x_3\in[\gamma,1]\}.
\label{eq:rf-continuum-seams}
\end{align}
We regard the three sets in \eqref{eq:rf-continuum-panel-regions} as
separate sheets, retaining two labelled copies of every seam.  Their
cut-open union is denoted by \(\mathsf R_p\).  For an exact dilation
\(p_n=np\), this is already the domain tiled by the rescaled lozenges
and singleton seam triangles; there is no boundary collar.  Indeed,
passing from the selected vertices in \eqref{eq:Sp-formula} to their
incident elementary triangles changes
\[
\begin{array}{c|c}
\text{selected-vertex wall}&\text{closed atomic-region wall}\\ \hline
q_0\ge1&q_0\ge0\\
q_1\le a_n-1&q_1\le a_n\\
q_2\ge b_n+1&q_2\ge b_n\\
q_3\ge c_n+1&q_3\ge c_n\\
q_4\le d_n-1&q_4\le d_n.
\end{array}
\]
After division by \(n\), these are exactly the inequalities in
\eqref{eq:rf-continuum-panel-regions}.  The same calculation on the
two common sides gives exactly the seam intervals
\eqref{eq:rf-continuum-seams}.
Write \(\mathsf R_{np,r}^{\mathbb Z}\) for the lattice vertices of the
\(r\)-th sheet before dilation.

For a tiling \(\Xi\in\mathcal T_{np}\), let
\(F_{\Xi,r}\) be its usual lozenge height function on the \(r\)-th
sheet.  Along a positively oriented lattice edge its increment is
\(+1\) unless that edge is the selected tile edge, in which case the
increment is \(-2\).  The two panel developments orient a common seam
in opposite lattice directions, so the unmodified panel traces are
\emph{not} equal.

For the seam \(0(r+1)\), \(r=1,2\), and its level \(k\), set
\begin{equation}
 F^{\Xi}_{r,-}(k)=F_{\Xi,r}([k,k]_r),\qquad
 F^{\Xi}_{r,+}(k)=F_{\Xi,r+1}([0,k]_{r+1}).
\label{eq:rf-two-panel-seam-traces}
\end{equation}
Here \(-\) is the new-boundary side of \(P_r\), while \(+\) is the
old-boundary side of \(P_{r+1}\).  The additive constants are chosen so
that
\begin{equation}
 F^{\Xi}_{r,+}(k)=F^{\Xi}_{r,-}(k)+k
\label{eq:rf-aligned-discrete-traces}
\end{equation}
on every retained seam and \(F_{\Xi,1}(v_9(p_n))=0\).  The consistency of
this normalization is the content of
Lemma~\ref{lem:rf-seam-compatibility}.  The common \emph{aligned} seam
trace is
\begin{equation}
 B^{\Xi}_r(k):=F^{\Xi}_{r,+}(k)
             =F^{\Xi}_{r,-}(k)+k.
\label{eq:rf-common-aligned-trace}
\end{equation}
Let \(\overline F_{\Xi,n}\) be the natural continuum height of the
tiling: extend \(n^{-1}F_{\Xi,r}\) affinely over the two elementary
triangles of every lozenge and over every singleton seam triangle, and
then dilate the three sheets by \(1/n\).  Thus
\(\overline F_{\Xi,n}\) is already a piecewise-affine function with
gradient in \(K\) on the fixed three-sheet domain \(\mathsf R_p\).
This is the required canonical map from a discrete tiling to a
continuum height function.

Assume for the remainder of this part that \(p\) is rational, and let
\(q\) be a common denominator.

For every multiple \(n\) of \(q\), let
\(F^{\partial}_{np,r}\) denote the lozenge-height trace on the closed
outer boundary of the \(r\)-th sheet, with
\(F^{\partial}_{np,1}(v_9(np))=0\) and the other sheet constants fixed
by the aligned seam normalization
\eqref{eq:rf-aligned-discrete-traces}.  This trace is independent of
the tiling: no lozenge crosses a closed outer boundary edge, so its
oriented height increment is forced.  Define
\(\gamma_{p,r}\) on the closed outer boundary of
\(\mathsf R_{p,r}\) by
\[
 \gamma_{p,r}(z/q)=q^{-1}F^{\partial}_{qp,r}(z)
\]
at the boundary vertices of the basic rational dilation and extend it
affinely along every straight boundary segment.  Forced boundary
increments are homogeneous under integral dilation.  Hence, for every
multiple \(n\) of \(q\) and every closed outer-boundary lattice vertex
\(z\),
\begin{equation}
 F^{\partial}_{np,r}(z)
       =n\gamma_{p,r}(z/n).
\label{eq:rf-exact-outer-boundary-trace}
\end{equation}
This exact identity, rather than only convergence of boundary traces,
is what is used below.

\begin{definition}[Continuum admissible height space]
\label{def:rf-AHT}
The space \(\operatorname{AHT}(\mathsf R_p)\) consists of tuples
\(f=(f_1,f_2,f_3)\) with the following properties:
\begin{enumerate}
\item each \(f_r\) is Lipschitz on \(\mathsf R_{p,r}\) and
      \(\nabla f_r\in K\) almost everywhere, where
      \(K=\operatorname{conv}\{2(\widehat\imath-\widehat\jmath),
      2\widehat\jmath,-2\widehat\imath\}\) is the geometric slope
      triangle \eqref{eq:panel-slope-triangle};
\item on every closed outer boundary segment,
      \(f_r=\gamma_{p,r}\), with the normalization
      \(f_1(v_9)=0\);
\item using the panel coordinates \eqref{eq:rf-panel-coordinate} with
      \(N=1\), the two labelled traces satisfy the aligned gluing
      relations
      \begin{align}
       f_2([0,x]_2)&=f_1([x,x]_1)+x,
              &&x\in[\beta,1],\notag\\
       f_3([0,x]_3)&=f_2([x,x]_2)+x,
              &&x\in[\gamma,1].
      \label{eq:rf-continuum-aligned-traces}
      \end{align}
\end{enumerate}
We write
\begin{equation}
 f_{02}(x):=f_2([0,x]_2),\qquad
 f_{03}(x):=f_3([0,x]_3)
\label{eq:rf-continuum-seam-traces}
\end{equation}
for the two common aligned seam traces.  For every sufficiently large
multiple \(n\) of \(q\), for
\(f\in\operatorname{AHT}(\mathsf R_p)\), and for \(\eta>0\), put
\begin{equation}
 \mathcal T_{np}(f,\eta n)
 =\left\{\Xi\in\mathcal T_{np}:
       \|\overline F_{\Xi,n}-f
          \|_{L^\infty(\mathsf R_p)}<\eta\right\}.
\label{eq:rf-height-neighborhood}
\end{equation}
\end{definition}

\paragraph{Alignment of the two seam traces.}
The two copies of a cut carry oppositely oriented lozenge-height
traces, so their raw values should not be identified.  The correct
comparison differs by the deterministic seam level \(k\).  The next
identity shows that this affine alignment is exactly equivalent to the
complementary-singleton rule.

\begin{lemma}[Seam alignment]
\label{lem:rf-seam-compatibility}
Consider a retained seam, and index its lattice vertices by increasing
level \(k\).  Let \(F_-\) and \(F_+\) be the height traces on its two
labelled sides, both read in the direction of increasing \(k\).  For
the seam segment from level \(k-1\) to level \(k\), let
\[
   \epsilon_-(k),\epsilon_+(k)\in\{0,1\}
\]
indicate whether the facing triangle on the corresponding side is
used as the singleton seam triangle.  Write
\[
   \Delta G(k):=G(k)-G(k-1).
\]
Then
\begin{equation}
 \boxed{
 \Delta\bigl(F_+-F_-\bigr)(k)
   =3\bigl(\epsilon_-(k)+\epsilon_+(k)\bigr)-2.}
\label{eq:rf-seam-difference-identity}
\end{equation}

Consequently, the tiling is legal at this seam segment---that is,
exactly one of the two facing triangles is a singleton---if and only if
\begin{equation}
       \Delta\bigl(F_+-F_-\bigr)(k)=1.
\label{eq:rf-local-seam-alignment}
\end{equation}
It is legal along the entire connected seam if and only if, after
choosing the two additive height constants,
\begin{equation}
              F_+(k)=F_-(k)+k
\label{eq:rf-seam-alignment}
\end{equation}
at every retained seam vertex.

For a legal seam, define its common aligned height by
\[
             B(k):=F_+(k)=F_-(k)+k.
\]
Its increment
\begin{equation}
 b_{r,k}:=\Delta B(k)
          =-1+3\epsilon_+(k)
          =2-3\epsilon_-(k)
          \in\{-1,2\}
\label{eq:rf-aligned-seam-increment}
\end{equation}
records which side contains the singleton.  Equivalently,
\begin{equation}
 \epsilon_+(k)=\frac{1+b_{r,k}}{3},
 \qquad
 \epsilon_-(k)=\frac{2-b_{r,k}}{3}.
\label{eq:rf-seam-indicators-from-height}
\end{equation}
Thus \(b_{r,k}=2\) when the \(+\)-side triangle is the singleton, and
\(b_{r,k}=-1\) when the \(-\)-side triangle is the singleton.
\end{lemma}

\begin{proof}
On the \(-\)-side, increasing \(k\) follows the positive lattice
orientation.  The height increment is therefore \(+1\) when the
facing triangle is paired into an ordinary lozenge and \(-2\) when it
is the singleton.  Hence
\[
                 \Delta F_-(k)=1-3\epsilon_-(k).
\]
On the \(+\)-side, increasing \(k\) runs opposite to the positive
boundary orientation, and therefore
\[
                 \Delta F_+(k)=-1+3\epsilon_+(k).
\]
Subtracting these two identities gives
\eqref{eq:rf-seam-difference-identity}.

The complementary-singleton rule is
\(\epsilon_-(k)+\epsilon_+(k)=1\), so
\eqref{eq:rf-seam-difference-identity} proves its equivalence with
\eqref{eq:rf-local-seam-alignment}.  Summing the latter identity along
the connected seam gives
\[
                 F_+(k)-F_-(k)=k+C.
\]
The two panel heights are defined only up to additive constants, so
they may be normalized to make \(C=0\), proving
\eqref{eq:rf-seam-alignment}.  Finally,
\[
       \Delta B(k)=\Delta F_+(k)=-1+3\epsilon_+(k),
\]
and the complementary relation
\(\epsilon_-=1-\epsilon_+\) gives the second expression in
\eqref{eq:rf-aligned-seam-increment}.  Solving for the two indicators
proves \eqref{eq:rf-seam-indicators-from-height}.
\end{proof}

\begin{lemma}[Compactness and piecewise-affine discrete approximation]
\label{lem:rf-AHT-compact}
The space \(\operatorname{AHT}(\mathsf R_p)\) is nonempty and compact
in the uniform topology.  The common aligned seam traces
\(f_{02},f_{03}\) have tangential derivative in \([-1,2]\) almost
everywhere.  The natural piecewise-affine heights constructed above
satisfy
\begin{equation}
 \sup_{\Xi\in\mathcal T_{np}}
 {\rm dist}_{L^\infty(\mathsf R_p)}
 \bigl(\overline F_{\Xi,n},
       \operatorname{AHT}(\mathsf R_p)\bigr)
 \longrightarrow0.
\label{eq:rf-discrete-to-continuum-compactness}
\end{equation}
Conversely, there is an absolute constant \(C\), depending only on the
finite list of panel orientations, such that for every
\(f\in\operatorname{AHT}(\mathsf R_p)\) and every sufficiently large
multiple \(n\) of \(q\), there is a tiling
\(\Xi_n\in\mathcal T_{np}\) satisfying
\begin{equation}
 \|\overline F_{\Xi_n,n}-f
       \|_{L^\infty(\mathsf R_p)}\le \frac Cn.
\label{eq:rf-discrete-approximation}
\end{equation}
The traces obey the aligned seam relations exactly at the discrete seam
vertices and affinely between them, and their outer traces agree
exactly with the prescribed continuum traces.  In particular, the
analytic definition agrees with the compact
closure of normalized discrete cut-compatible height functions, and
the sets \(\mathcal T_{np}(f,\eta n)\) are nonempty for all sufficiently
large \(n\).
\end{lemma}

\begin{proof}
The natural piecewise-affine heights
\(\overline F_{\Xi,n}\) have a common Lipschitz constant \(L_*\), and
the normalization at \(v_9\) bounds their additive constants uniformly.
Consider arbitrary \(n_j\to\infty\) and
\(\Xi_j\in\mathcal T_{n_jp}\).  Arzel\`a--Ascoli gives a uniformly
convergent subsequence on the fixed three-sheet space
\(\mathsf R_p\), with limit \(f\).

On every elementary triangle the gradient of the piecewise-affine
height lies in \(K\).
The directional Lipschitz inequalities defining the closed convex
slope constraint pass to the uniform limit, so
\(\nabla f_r\in K\) almost everywhere.  By
\eqref{eq:rf-exact-outer-boundary-trace}, the normalized discrete
outer traces equal \(\gamma_{p,r}\) exactly.  On the common portion
of a retained seam, \eqref{eq:rf-aligned-discrete-traces} is exact at
the vertices and hence, by affine interpolation, on the entire seam.
Consequently the limiting traces satisfy
\eqref{eq:rf-continuum-aligned-traces}.  Thus every subsequential limit
belongs to \(\operatorname{AHT}(\mathsf R_p)\).

If \eqref{eq:rf-discrete-to-continuum-compactness} failed, one could
choose a sequence staying a fixed positive distance from the admissible
space; the preceding compactness argument would give an admissible
uniform limit, a contradiction.  It also proves nonemptiness.  The
slope constraint gives a common Lipschitz constant on the admissible
space, while the outer normalization bounds the additive constant.
Arzel\`a--Ascoli and closedness of the slope, boundary, and aligned
trace conditions now give compactness of
\(\operatorname{AHT}(\mathsf R_p)\).  By
\eqref{eq:rf-continuum-seam-traces}, each common seam trace is the trace
on an old boundary of a panel; its derivative therefore belongs to the
projection \([-1,2]\) of \(K\) onto the seam direction.

We prove the converse discrete approximation by applying the
sheetwise height-function construction of
\cite[\S7.2]{NarLimit} on the three sheets
\emph{separately}.  This is important: the two copies of a cut are not
identified as atomic triangles, so the construction cannot create a
spurious lozenge crossing a seam.

Choose the three-colorings \(c_r\) on the developed panels so that a
height function has residue \(c_r(x)+c_0\pmod3\).  Along the seam
\(0(r+1)\), increasing \(k\) follows the positive direction on the
\(-\)-side and the negative direction on the \(+\)-side.  Consequently
we may, after fixing the sheet constants, arrange
\begin{equation}
 c_{r+1}([0,k]_{r+1})-c_r([k,k]_r)\equiv k\pmod3.
\label{eq:rf-seam-residue-alignment}
\end{equation}
For a lattice vertex away from a cut, let \(g_{r,n}(x)\) be the least
integer not smaller than \(nf_r(x/n)\) and congruent to its prescribed
class.  On a closed outer boundary use the exact discrete boundary
height.

On a retained seam choose one common aligned integer
\(\widehat g_{r,n}(k)\): it is the least integer not smaller than
\(n f_{r+1}([0,k/n]_{r+1})\) and congruent to the \(+\)-side class.
Prescribe
\begin{equation}
 g_{r+1,n}([0,k]_{r+1})=\widehat g_{r,n}(k),\qquad
 g_{r,n}([k,k]_r)=\widehat g_{r,n}(k)-k.
\label{eq:rf-rounded-aligned-seam}
\end{equation}
The continuum relation \eqref{eq:rf-continuum-aligned-traces} and the
residue relation \eqref{eq:rf-seam-residue-alignment} show that both
values are in the correct residue class and lie in the corresponding
half-open interval
\[
                         nf\le g<nf+3.
\]
At seam
endpoints these prescriptions agree with the exact closed-boundary
values: item~2 of Definition~\ref{def:rf-AHT} and
\eqref{eq:rf-exact-outer-boundary-trace} give precisely the same
integer there.  Thus endpoint compatibility is an exact consequence
of the prescribed trace, rather than an inference from a scaling
limit.

Let \(d_{r,n}(x,z)\) be the directed positive-path distance inside the
\(r\)-th sheet and define
\begin{equation}
 F_{r,n}(z)=\min_{x\in \mathsf R_{np,r}^{\mathbb Z}}
              \{g_{r,n}(x)+d_{r,n}(x,z)\}.
\label{eq:rf-envelope-approximation}
\end{equation}
The condition \(\nabla f_r\in K\) is equivalent to
\(nf_r(z/n)-nf_r(x/n)\le d_{r,n}(x,z)\).  Therefore
\[
 nf_r(z/n)\le F_{r,n}(z)\le g_{r,n}(z) < nf_r(z/n)+3.
\]
The half-open interval $[nf,nf+3)$ contains at most one integer in a fixed residue class.
All terms in the minimum have the same residue modulo three.  Thus,
\(F_{r,n}=g_{r,n}\) on every closed boundary and on every labelled cut.
In particular, \eqref{eq:rf-rounded-aligned-seam} gives
\[
 F_{r+1,n}([0,k]_{r+1})
 =F_{r,n}([k,k]_r)+k.
\]

Along a positive unit edge, the envelope increment is at most \(+1\).
The other two sides of an incident elementary triangle give a positive
path of length two in the reverse direction, so the increment is at
least \(-2\).  The residue condition forces it to be either \(+1\) or
\(-2\).  Around each elementary triangle there is consequently exactly
one \(-2\) edge, and hence the sheetwise envelope is a lozenge height
function with the prescribed boundary singleton choices.  By
Lemma~\ref{lem:rf-seam-compatibility}, the exact aligned relation on the
two labelled sides of every cut gives complementary singleton
seam-triangle choices.  Thus the three sheetwise tilings, together with
those singletons, form a legal tiling in \(\mathcal T_{np}\).

The preceding bounds give
\(\|F_{r,n}-nf_r\|_\infty < 3\) on the lattice vertices.  After
the harmless piecewise-affine interpolation constant is included, this
proves \eqref{eq:rf-discrete-approximation} on all of
\(\mathsf R_p\).  The exact relations on the retained discrete seam
vertices extend affinely along their elementary segments.  The outer
boundary trace remains exact; the global \(C/n\) estimate comes from
lattice rounding and piecewise-affine interpolation.
\end{proof}

\subsection{Exact seam-triangle formula}

The support formula \eqref{eq:Sp-formula} shows that seam \(02\) is
retained at levels \(k=b_n+1,\ldots,n\), and seam \(03\) at levels
\(k=c_n+1,\ldots,n\).  For a tiling \(\Xi\), let
\(b_{r,k}(\Xi)\) be the upward increment of the common aligned seam
height, as in \eqref{eq:rf-seam-alignment}--\eqref{eq:rf-aligned-seam-increment}.
The two singleton indicators are given by
\eqref{eq:rf-seam-indicators-from-height}.

Write
\(\lambda_i^{(r,k)}=
 \lambda_i((\widehat X_r^{(n)})^{[k]})\).  Lemma~\ref{lem:rf-local-weights}
and \eqref{eq:rf-gue-hive-map} give
\begin{align}
 W_{02,\Delta}^{\rm rf}(\Xi)
 &=\sum_{k=b_n+1}^{n}
 \left[
 \frac{1+b_{1,k}(\Xi)}3\frac{\lambda_1^{(2,k)}}2
 +\frac{2-b_{1,k}(\Xi)}3
       \left(-\frac{\lambda_k^{(1,k)}}2\right)
 \right],
\label{eq:rf-seam02-discrete}\\
 W_{03,\Delta}^{\rm rf}(\Xi)
 &=\sum_{k=c_n+1}^{n}
 \left[
 \frac{1+b_{2,k}(\Xi)}3\frac{\lambda_1^{(3,k)}}2
 +\frac{2-b_{2,k}(\Xi)}3
       \left(-\frac{\lambda_k^{(2,k)}}2\right)
 \right].
\label{eq:rf-seam03-discrete}
\end{align}
These are the flip-symmetric analogues of the equator expression in
\cite{NarLimit}; the two facing singleton weights now have the same
coefficient \(1/2\).

For \(x\in(0,1]\) and \(\xi\in[-1,2]\), define
\begin{align}
 \mathcal E_{12}(x,\xi)
 &=\frac{1+\xi}{3}\ell_2\sqrt x
   +\frac{2-\xi}{3}\ell_1\sqrt x,\notag\\
 \mathcal E_{23}(x,\xi)
 &=\frac{1+\xi}{3}\ell_3\sqrt x
   +\frac{2-\xi}{3}\ell_2\sqrt x.
\label{eq:rf-seam-density}
\end{align}
For the continuum aligned seam traces defined in
\eqref{eq:rf-continuum-seam-traces}, put
\begin{equation}
 S_{p,\Delta}^{\rm rf}(f)=
 \int_\beta^1\mathcal E_{12}(x,f_{02}'(x))\,dx
 +\int_\gamma^1\mathcal E_{23}(x,f_{03}'(x))\,dx.
\label{eq:rf-seam-functional}
\end{equation}

\begin{lemma}[Seam concentration and \(L^\infty\)-stability]
\label{lem:rf-seam-limit}
Fix a continuum admissible height function \(f\), and let
\(\mathcal T_{np}(f,\varepsilon n)\) be the tilings whose normalized
height functions are within \(\varepsilon\) of \(f\) in \(L^\infty\).
Then
\begin{align}
 \limsup_{n\to\infty}n^{-2}\mathbb E
 \max_{\Xi\in\mathcal T_{np}(f,\varepsilon n)}
  \bigl(W_{02,\Delta}^{\rm rf}(\Xi)+W_{03,\Delta}^{\rm rf}(\Xi)\bigr)
 &\le S_{p,\Delta}^{\rm rf}(f)+o_\varepsilon(1),\notag\\
 \liminf_{n\to\infty}n^{-2}\mathbb E
 \min_{\Xi\in\mathcal T_{np}(f,\varepsilon n)}
  \bigl(W_{02,\Delta}^{\rm rf}(\Xi)+W_{03,\Delta}^{\rm rf}(\Xi)\bigr)
 &\ge S_{p,\Delta}^{\rm rf}(f)-o_\varepsilon(1).
\label{eq:rf-seam-neighborhood}
\end{align}
Uniformly over all tilings,
\begin{equation}
 \sup_{\Xi}\bigl|W_{p,\Delta}^{\rm rf}(\Xi)
       -\mathbb EW_{p,\Delta}^{\rm rf}(\Xi)\bigr|
 =O_{\mathbb P}(n^{4/3}\log^{O(1)}n)=o_{\mathbb P}(n^2),
\label{eq:rf-seam-uniform-concentration}
\end{equation}
and the corresponding \(L^2\) estimate is
\begin{equation}
 \left\|\sup_{\Xi}\bigl|W_{p,\Delta}^{\rm rf}(\Xi)
       -\mathbb EW_{p,\Delta}^{\rm rf}(\Xi)\bigr|\right\|_2
 =O(n^{4/3}\log^{O(1)}n).
\label{eq:rf-seam-uniform-L2}
\end{equation}
\end{lemma}

\begin{proof}
For \(k\ge\min(\beta,\gamma)n+O(1)\), the Tao--Vu estimate applied to
each \(k\times k\) minor gives
\[
 |\lambda_1^{(r,k)}-\mathbb E\lambda_1^{(r,k)}|
 +|\lambda_k^{(r,k)}-\mathbb E\lambda_k^{(r,k)}|
 =O(n^{1/3}\log^{O(1)}n)
\]
simultaneously in \(r,k\), with overwhelming probability.  Every
singleton-indicator coefficient is bounded, so the supremum over tilings is bounded
by the sum of the absolute deviations of all extreme eigenvalues
appearing in \eqref{eq:rf-seam02-discrete}--
\eqref{eq:rf-seam03-discrete}.  This proves
\eqref{eq:rf-seam-uniform-concentration}.  Integrating the Tao--Vu
tails and applying Minkowski's inequality gives
\eqref{eq:rf-seam-uniform-L2}.

After replacing the eigenvalues by their expectations, each seam sum is
of the form \(\sum_k(A_k+D_k b_k)\).  If
\(F_k=\sum_{j\le k}b_j\) is the cumulative seam height, discrete
summation by parts gives
\[
 \sum_kD_kb_k
 =\text{endpoint terms}-\sum_k(D_{k+1}-D_k)F_k.
\]
The expected extreme-eigenvalue profiles converge uniformly on compact
subintervals of \((0,1]\) to \(\pm2\ell_r\sqrt x\).  By interlacing, the
top profile is monotone increasing in the minor level and the bottom
profile is monotone decreasing; their discrete total variation is therefore
\(O(n)\).  Replacing \(F_k\) by the trace of a height
function within \(\varepsilon n\) of \(f\) therefore costs
\(O(\varepsilon n^2)\).  The resulting Riemann sums converge to
\eqref{eq:rf-seam-functional}.  This is verbatim the summation-by-parts
argument for the equator term in \cite{NarLimit}, applied to
\eqref{eq:rf-seam02-discrete} and \eqref{eq:rf-seam03-discrete}.
\end{proof}

\section{The finite-panel variational limit}
\label{sec:rf-variational-limit}

On the interior of the \(r\)-th sheet, let
\(\rho_r(x)=(\mathbb E g_2,\mathbb E g_3)\) be the raw expected
interlacing-gap parameters of the corresponding GUE minor source,
ordered by local types \(2,3\) as above; the source scale is included
and no factor \(1/3\) is absorbed into \(\rho_r\).  Define
\begin{equation}
 S_{p,\mathrm{edge}}(f)
 =-\sum_{r=1}^3\int_{\mathsf R_{p,r}^{\circ}}
       \sigma_\edge^{\rm sym}(\rho_r(x),\nabla f_r(x))\,dx.
\label{eq:rf-bulk-functional}
\end{equation}
Notice that the domain of integration is the portion
\(\mathsf R_{p,r}\) of the canonical region, rather than the whole
initial triangular panel.  The variable bulk weight in
\eqref{eq:rf-tiling-weight} is \(3/2\) times the sum of the global blue
and green weights, equivalently the local type-\(2\) and type-\(3\)
weights.  Accordingly \eqref{eq:rf-bulk-functional} uses the
flip-symmetric surface tension \(\sigma_\edge^{\rm sym}\), not the
numerically asymmetric \(\sigma_\edge\) of
\cite{NSTGUE,NarLimit}.  The local-limit and surface-tension proofs
transfer because the two color coefficients remain fixed, positive,
and bounded; the precise reduction is given in
Lemma~\ref{lem:rf-local-transfer}.

For later use, write
\begin{equation}
 W_{p,\mathrm{edge}}^{\rm rf}(\Xi)
 =\frac32\sum_{\substack{\ell\in\Xi\\\ell\text{ blue}}}\omega_h(\ell)
  +\frac32\sum_{\substack{\ell\in\Xi\\\ell\text{ green}}}\omega_h(\ell).
\label{eq:rf-discrete-bulk-weight}
\end{equation}

\subsection{Quantitative differentiation on the panel cores}
\label{subsec:rf-fefferman}

The upper variational bound involves a maximum over many discrete height
functions.  The scales on which a maximizing height is approximately affine
may depend on that height and on the random weights.  The required uniform
multiscale decomposition is supplied by the following theorem from
Fefferman's unpublished notes \cite{FeffermanQD}; a proof is reproduced in
\cite[Theorem~7.6 and Appendix~B]{NarLimit}.

\begin{theorem}[Fefferman's quantitative-differentiation theorem]
\label{thm:rf-fefferman}
There is an absolute constant \(C^\sharp>0\) with the following property.
Let \(Q=(0,1)^2\), and let \(F:Q\to\mathbb R\) have Lipschitz constant at
most one.  Suppose \(\varepsilon,\eta>0\) and \(k\ge1\) satisfy
\begin{equation}
             \varepsilon^4\eta k\ge C^\sharp.
\label{eq:rf-fefferman-parameters}
\end{equation}
Then \(Q\) has a finite dyadic decomposition into GOOD and BAD squares such
that:
\begin{enumerate}
\item every GOOD square \(Q'\) has sidelength at least \(2^{-k}\), and
      there is an affine function \(L_{Q'}\) satisfying
      \[
          \|F-L_{Q'}\|_{L^\infty(Q')}
          \le \varepsilon\,\operatorname{side}(Q');
      \]
\item every BAD square has sidelength exactly \(2^{-k}\), and the total
      area of the BAD squares is at most \(\eta\).
\end{enumerate}
The affine function may be taken to be the best affine approximation in
\(L^2(Q')\).
\end{theorem}

The triangular-lattice panels are described in the same square coordinates
used to parametrize the patches in the definition of
\(\sigma_\edge^{\rm sym}\).  These array coordinates organize the dyadic
decomposition; the slopes of affine approximants remain geometric
gradients.  The fixed affine change of coordinates affects only
absolute Lipschitz constants.  Applying
Theorem~\ref{thm:rf-fefferman} block by block, and splitting each dyadic
square into two triangles when needed, gives the following panelwise form.

\begin{corollary}[Function-dependent GOOD/BAD decomposition of a panel core]
\label{cor:rf-panel-fefferman}
Fix \(\varepsilon,\eta>0\), choose \(k\) satisfying
\eqref{eq:rf-fefferman-parameters}.
Fix a dyadic \(M\) large enough (depending on \(\varepsilon,\eta,k\)), and let \(n\to\infty\).
Later send \(M\to\infty\).
Let \(C_{r,n}\) be one of the three panel cores obtained after deleting a
fixed-width boundary and cut strip, and let \(F_\Xi\) be the
piecewise-linear height of any tiling of the cut-open region.  There are
collections \(\mathcal G_{r,\Xi}\) and \(\mathcal B_{r,\Xi}\) of dyadic
triangular cells in \(C_{r,n}\) such that:
\begin{enumerate}
\item the cells cover the core apart from \(O(Mn)\) lattice area along the
      boundaries of the \(M\)-blocks;
\item every \(T\in\mathcal G_{r,\Xi}\) has sidelength
      \(m_T\in[2^{-k}M,M]\), and there is an affine \(L_T\) with
      \begin{equation}
         \|F_\Xi-L_T\|_{L^\infty(T)}
          \le C\varepsilon m_T;
      \label{eq:rf-panel-good-affine}
      \end{equation}
\item the total lattice area of the cells in \(\mathcal B_{r,\Xi}\) is at
      most \(C\eta n^2+O(Mn)\).
\end{enumerate}
Although the chosen GOOD/BAD subcollection depends on \(\Xi\), every cell
belongs to a deterministic dyadic family depending only on
\(C_{r,n},M,k\).  The constants are uniform over the three panel shapes.
\end{corollary}

\begin{proof}
Tile the panel core, up to a collar of lattice area \(O(Mn)\), by rhombic
\(M\)-blocks in panel square coordinates.  A lozenge height has a universal
Lipschitz constant; divide by that constant and rescale each block to the
unit square.  Apply Theorem~\ref{thm:rf-fefferman} on every block, rescale
back, and split the resulting dyadic squares into a bounded number of
triangular cells.  Summing the BAD areas over the blocks proves the third
assertion.  The full dyadic trees inside all \(M\)-blocks form the stated
deterministic candidate family.
\end{proof}

\begin{lemma}[Fixed-target two-scale affine interpolation]
\label{lem:rf-fixed-target-interpolation}
Fix \(f=(f_1,f_2,f_3)\in\operatorname{AHT}(\mathsf R_p)\), a
boundary-strip width \(\zeta>0\), and \(\theta>0\).  On each developed
lattice panel put
\[
                 f_{r,n}(z)=n f_r(z/n).
\]
Fix an integer \(C_0\geq1\), and triangulate the lattice version of
\(C_r^\zeta\), apart from a set of lattice area \(O(C_0n)\), by a
deterministic triangular mesh of sidelength \(C_0+O(1)\).  Let
\(\bar f_{r,n}\) agree with \(f_{r,n}\) at the mesh vertices and be affine
on each mesh triangle.  For all sufficiently large \(n\), the mesh triangles
have a subcollection \(\mathcal Q_{r,n}^{\rm good}\) with the following
properties:
\begin{enumerate}
\item the triangles outside \(\mathcal Q_{r,n}^{\rm good}\), together
      with the uncovered set, have total lattice area at most
      \(C\theta n^2+O(C_0n)\);
\item for every \(Q\in\mathcal Q_{r,n}^{\rm good}\),
      \begin{equation}
       \|f_{r,n}-\bar f_{r,n}\|_{L^\infty(Q)}
          \leq C\theta C_0;
      \label{eq:rf-fixed-target-relative-affinity}
      \end{equation}
      and \(\nabla(\bar f_{r,n}|_Q)\in K\);
\item on a subset of \(Q\) of relative area at least \(1-C\theta\),
      \(f_r\) is differentiable and
      \begin{equation}
       \left\|\nabla f_r(z/n)-
                    \nabla(\bar f_{r,n}|_Q)\right\|
          \leq C\theta.
      \label{eq:rf-fixed-target-gradient-affinity}
      \end{equation}
\end{enumerate}
The constants are uniform over the three panels.  If, in addition,
\[
 \operatorname{dist}(\nabla f_r(x),\partial K)\geq\mu
 \quad\text{for a.e. }x\in C_r^\zeta,
\]
then, after replacing \(\theta\) by a sufficiently small constant multiple
of \(\mu\), every retained affine slope satisfies
\begin{equation}
 \operatorname{dist}\bigl(
       \nabla(\bar f_{r,n}|_Q),\partial K\bigr)\geq\mu/2.
\label{eq:rf-fixed-target-interior-slope}
\end{equation}
\end{lemma}

\begin{proof}
After rescaling by \(n^{-1}\), the mesh size is
\(\ell_n=(C_0+O(1))/n\), which tends to zero with \(C_0\) fixed.  Apply
the lozenge-height formulation \cite[Lemma~8.3]{NarLimit}, quoted there
from \cite[Lemma~2.2]{CKP}, to the fixed Lipschitz function \(f_r\) on
each panel core and sum over the three panels.  That lemma gives
\eqref{eq:rf-fixed-target-relative-affinity} and
\eqref{eq:rf-fixed-target-gradient-affinity} on all but a
\(C\theta\)-fraction of the mesh triangles.  Only \(O(n/C_0)\) triangles
meet the polygonal boundary of a core, and their total lattice area is
\(O(C_0n)\).

Every side of a mesh triangle is parallel to a triangular-lattice
direction.  Since \(\nabla f_r\in K\) almost everywhere, integration along
such a side places its difference quotient in the corresponding allowed
interval.  The three side inequalities are exactly the three inequalities
defining the slope triangle \(K\).  Hence the affine interpolant on every
mesh triangle has slope in \(K\).

For the final assertion, choose in each retained triangle a
differentiability point at which
\eqref{eq:rf-fixed-target-gradient-affinity} holds.  The distance to the
closed set \(\partial K\) is a one-Lipschitz function of the slope, so
\[
 \operatorname{dist}\bigl(
       \nabla(\bar f_{r,n}|_Q),\partial K\bigr)
 \geq \mu-C\theta.
\]
Taking \(C\theta\leq\mu/2\) proves
\eqref{eq:rf-fixed-target-interior-slope}.
\end{proof}

\begin{remark}
\label{rem:rf-role-fefferman}
Theorem~\ref{thm:rf-fefferman} is used directly in the upper bound below,
because the maximizing discrete height is not fixed in advance.  It is also
used indirectly through the continuity and scale-comparison properties of
\(\sigma_\edge^{\rm sym}\), obtained by the coefficient-stable
adaptation of \cite{NarLimit}.  The lower bound uses no decomposition that
must be uniform over the maximizing random height.  After fixing the target
continuum height, Lemma~\ref{lem:rf-fixed-target-interpolation} selects the
good cells of a deterministic \(C_0\)-mesh, and the global patching lemma
inserts local optimizers there.  Thus relative affine approximation is
required only on a set whose complement has vanishing area, not on every
coarse triangle.
\end{remark}

\begin{lemma}[Boundary-strip estimate]
\label{lem:rf-boundary-strip}
Let \(\mathsf B_{np}(\eta)\) be the set of atomic triangles at lattice
distance at most \(\eta n\) from the polygonal boundary of one of the
three sheets, including both sides of the two cuts.  Let
\(W_{p,\mathrm{edge}}^{\rm strip}(\Xi;\eta)\) be the contribution of
nonred lozenges meeting this set, equivalently lozenges of local type
\(2\) or \(3\).  Then
\begin{equation}
 \lim_{\eta\downarrow0}\limsup_{n\to\infty}
 n^{-2}\mathbb E\sup_{\Xi\in\mathcal T_{np}}
 \left|W_{p,\mathrm{edge}}^{\rm strip}(\Xi;\eta)\right|=0.
\label{eq:rf-boundary-strip-estimate}
\end{equation}
\end{lemma}

\begin{proof}
For the \(r\)-th developed panel, let
\(\mathcal C_{r,n}(\eta)\) be the set of all candidate lozenges of
local type \(2\) or \(3\) meeting \(\mathsf B_{np}(\eta)\).  The
selected lozenges of any tiling form a subset of this deterministic
set.  Hence, by the triangle
inequality and the symmetric coefficient \(3/2\),
\begin{equation}
 \sup_{\Xi\in\mathcal T_{np}}
 \left|W_{p,\mathrm{edge}}^{\rm strip}(\Xi;\eta)\right|
 \leq \frac32\sum_{r=1}^3
       \sum_{\ell\in\mathcal C_{r,n}(\eta)}
          |\omega_{H_r^{(n)}}(\ell)|.
\label{eq:rf-strip-candidate-majorant}
\end{equation}

Fix an auxiliary spectral cutoff \(\varepsilon>0\).  Split the
candidate set on each developed panel according as its location in the
underlying full Gelfand--Tsetlin panel lies inside or outside the standard
spectral \(\varepsilon\)-core.  In the core, the finitely many polygonal
boundary and cut segments have total length \(O(n)\), so their
\(\eta n\)-neighborhood meets \(O(\eta n^2)\) candidate lozenges.  The
expected absolute interlacing gaps are uniformly bounded on this fixed
spectral core.  Therefore
\[
 \mathbb E\sum_{r=1}^3
   \sum_{\substack{\ell\in\mathcal C_{r,n}(\eta)\\
                    \ell\ \mathrm{spectrally}\ \varepsilon\text{-bulk}}}
       |\omega_{H_r^{(n)}}(\ell)|
 \leq C_\varepsilon\eta n^2+o(n^2).
\]
This estimate includes the interiors of both artificial cuts; no
spectral-edge argument is needed there.

For the complementary candidates, enlarge the set to all type-\(2\)
and type-\(3\) lozenges outside the standard \(\varepsilon\)-core of
the corresponding full Gelfand--Tsetlin panel.  The telescoping calculation of
\cite[\S4, pp.~1149--1150]{NSTGUE}, with the three classes displayed on
p.~1150, bounds their total absolute weight by the sum of
\(O(\varepsilon n)\) full spectral widths and \(O(n)\) edge spectral
increments whose ranks are \(O(\varepsilon n)\).  Eigenvalue rigidity
therefore gives
\[
 \sum_{r=1}^3
   \sum_{\substack{\ell\text{ of type }2\text{ or }3\\
                    \ell\ \mathrm{not\ spectrally}\
                    \ \varepsilon\text{-bulk}}}
       |\omega_{H_r^{(n)}}(\ell)|
 \leq C\varepsilon^{1/3}n^2+o(n^2)
\]
with overwhelming probability.  Restricting a full panel to one of our
polygonal panel regions can only decrease this sum.  Re-expressing
\(P_2\) in its local \([i,j]_2\) frame exchanges the global blue and
green names and applies the reflected-rank convention; the
type-\(2\)/type-\(3\) absolute-weight sum is unchanged.  The three
fixed source scales and the coefficient \(3/2\) only change \(C\).

The rigidity estimates used in the cited calculation have overwhelming
tails and bounded moments.  Since the displayed telescoping majorant is
itself bounded by \(O(n)\) full spectral widths, Cauchy--Schwarz on the
exceptional event upgrades the same estimate to \(L^1\).  Taking
expectations in \eqref{eq:rf-strip-candidate-majorant} consequently gives
\[
 \limsup_{n\to\infty}n^{-2}\mathbb E
 \sup_{\Xi\in\mathcal T_{np}}
 \left|W_{p,\mathrm{edge}}^{\rm strip}(\Xi;\eta)\right|
 \leq C_\varepsilon\eta+C\varepsilon^{1/3}.
\]
First let \(\eta\downarrow0\), with \(\varepsilon\) fixed, and then let
\(\varepsilon\downarrow0\).  This proves
\eqref{eq:rf-boundary-strip-estimate}.
\end{proof}

\paragraph{Convention for local height states and their scores.}
\phantomsection\label{conv:rf-local-height-states}
Let \(D\) be a finite connected union of elementary triangular-lattice
cells.  A \emph{local height state on \(D\)} is represented, modulo an
additive integer constant, by an integer-valued function
\(F:V(D)\to\mathbb Z\) whose increment on every
positively oriented lattice edge \(e=x\to y\) of \(D\) belongs to
\(\{+1,-2\}\).  We call \(e\) \emph{selected} when
\(F(y)-F(x)=-2\).  Around every elementary triangle the three positively
oriented side increments sum to zero; consequently exactly one of its
three sides is selected.  If a selected edge is internal to \(D\), its
two incident triangles form the corresponding lozenge.  A selected
boundary edge is also allowed: it records that the incident triangle of
\(D\) is paired across \(\partial D\) with the triangle in an adjoining
domain.  Thus a local height state need not tile \(D\) using only
lozenges contained in \(D\), and the restriction of every global panel
tiling is a local height state.

The local score of such a state is the relevant weighted sum over
selected internal edges, equivalently over lozenges both of whose
elementary triangles lie in \(D\).  Lozenges crossing \(\partial D\) are
omitted from the local score and are assigned to the boundary-layer
remainder when local scores are summed.  If local states on two adjoining
domains agree, after an additive integer shift, at every vertex of their
common interface, then
their increments on each interface edge agree.  Hence their selected
interface edges coincide: a selected one pairs the two incident
triangles across the interface, while at an unselected one both
triangles are paired on their own sides.  In this sense the two local
states glue.  This convention applies to all local maxima and patching
lemmas below.

\paragraph{Panelwise transfer: precise imported estimates.}
We first state the uniform transfer from the standard Gelfand--Tsetlin
patches as a separate proposition.  This keeps the imported
probabilistic input, the deterministic change of patch shape, and the
red-free gauge calculation logically distinct.  Throughout, fix a
compact macroscopic panel core
\(\mathcal C_{\rm core}\subset
\{E(I,J):0<I<J<1\}\), and let all cells lie in that core.

For a patch \(D\), define the asymmetric local score by
\[
 \mathcal S_D^{\rm asym}(H,\Xi)
 =2\sum_{\substack{\ell\in\Xi\\\ell\ {\rm of\ type}\ 2}}\omega_H(\ell)
   +\sum_{\substack{\ell\in\Xi\\\ell\ {\rm of\ type}\ 3}}\omega_H(\ell);
\]
the red coefficient is zero.  We define
\(\sigma_\edge^{\rm asym}\) as the thermodynamic limit of the
negative expected maximum of this score, over the same height
neighborhoods used above, divided by the actual normalized lattice
area \(|D|\).  This is the convention-translated asymmetric surface
tension associated with \(\sigma_\edge\) of \cite{NarLimit}; the
deterministic weight and area conversion is included whenever a cited
estimate is used below.
Let \(W^{\rm asym}_{r,T}\) denote the corresponding score, using the
\hyperref[conv:rf-local-height-states]{local-state convention} above.
We write \(|T|\) for normalized lattice area, assigning area \(1/2\)
to each elementary triangle and area \(1\) to each lozenge.  Thus the
standard \(m\)-by-\(m\) parallelogram has area \(m^2\), and either
triangular half has area \(m^2/2\).  When we invoke
\cite[Lemma~8.7]{NarLimit}, its standard box is first parametrized by
its actual lattice side length and its score divided by this area.
Thus all local maxima below are divided by the actual area of their
patch, and no fixed geometric factor is hidden in \(\rho\).

\begin{proposition}[Uniform transfer from standard patches]
\label{prop:rf-uniform-standard-patch-transfer}
Fix \(\tau\in(0,1)\), \(\delta>0\), and
\(\bar\varepsilon>0\).  There exist
\(0<\varepsilon_{\rm loc}\leq\bar\varepsilon\) and an integer
\(m_0\geq1\), depending only on
\(\mathcal C_{\rm core},\tau,\delta,\bar\varepsilon\), such that
for every integer \(m\geq m_0\) there is \(n_0(m)\) for which the
following assertions hold whenever \(n\geq n_0(m)\).  They hold
uniformly in
\begin{itemize}
\item the source \(r\in\{1,2,3\}\);
\item every translate of either lattice orientation of a triangular
      cell of sidelength \(m\), whose rescaled closure lies in
      \(\mathcal C_{\rm core}\); and
\item the slope \(a\in K\).
\end{itemize}
If \(z_T\) is a fixed lattice anchor at distance \(O(1)\) from the
barycenter of \(T\), \(x_T=z_T/n\), and
\begin{equation}
 Z^{\rm asym}_{r,T}(a,\varepsilon_{\rm loc})
 :=\max_{\Xi}\,W^{\rm asym}_{r,T}(\Xi),
 \qquad
 \inf_{c\in\mathbb R}
 \|F_\Xi-[c+\langle a,\,\cdot-z_T\rangle]\|_{L^\infty(T)}
 \leq\varepsilon_{\rm loc}m,
 \label{eq:rf-asymmetric-local-maximum}
\end{equation}
where the maximum is over local height states on \(T\), then
\begin{align}
 \frac1{|T|}\mathbb E
 Z^{\rm asym}_{r,T}(a,\varepsilon_{\rm loc})
 &\leq-\sigma_\edge^{\rm asym}(\rho_r(x_T),a)+\delta,
 \label{eq:rf-uniform-asymmetric-upper}\\
 \left|\frac1{|T|}\mathbb E
 Z^{\rm asym}_{r,T}(a,\varepsilon_{\rm loc})
 +\sigma_\edge^{\rm asym}(\rho_r(x_T),a)\right|
 &\leq\delta,
 \qquad a\in(1-\tau)K.
 \label{eq:rf-uniform-asymmetric-two-sided}
\end{align}
The choice of \(m_0\) is made large enough that every height
neighborhood appearing here is nonempty.
The same assertions hold for the reflected middle-panel rank
convention described below.
\end{proposition}

\begin{proof}
We first work in the stationary local limit.  Let \(\mathcal C\) be a
compact subset of \(\mathbb R_{>0}^2\) containing all parameters
\(\rho_r(x)\) on the three panel cores.  For either triangular
orientation \(s\in\{+,-\}\), let \(T_m^s\) have integer side \(m\),
and abbreviate
\[
 \sigma_m^s(\rho,a,\varepsilon)
 :=-\frac1{|T_m^s|}\mathbb E
   \max_{h\in\mathcal H_{T_m^s}(a,\varepsilon)}
          \mathcal S_{T_m^s}^{\rm asym}(h),
 \qquad
 \sigma(\rho,a):=\sigma_\edge^{\rm asym}(\rho,a).
\]
Here \(\mathcal H\) is the class in
\eqref{eq:rf-asymmetric-local-maximum}.

\smallskip
\noindent\emph{Fixed radius in the interior.}
For each fixed sufficiently small \(\varepsilon>0\), we claim
\begin{equation}
 \sigma(\rho,a)-C_{\mathcal C,\tau}\varepsilon-o_m(1)
 \leq \sigma_m^s(\rho,a,\varepsilon)
 \leq \sigma(\rho,a)+o_m(1),
 \qquad (\rho,a)\in\mathcal C\times(1-\tau)K.
 \label{eq:rf-fixed-radius-interior}
\end{equation}
The remainders are uniform on this set.  In particular, this asserts
an error tending to zero as \(\varepsilon\downarrow0\); it does not
identify the limit at every fixed positive radius with \(\sigma\).
The standard-patch inputs are the packing estimate
\cite[Lemma~11.1(B)]{NarLimit}, the fixed-radius limit in
\cite[Corollary~11.3]{NarLimit}, and the independence of the
\emph{vanishing} profile in \cite[Lemma~11.4]{NarLimit}.
We give the triangular comparison explicitly.

Round down an affine height of slope \(a\) to the appropriate color
residue at each lattice vertex.  This gives a legal background height
within \(3\) of the affine function.  Indeed, on a positive edge the
affine increment lies in \([-2,1]\), and residue rounding leaves only
the two allowed increments \(+1,-2\).  A patch whose boundary is
within \(\varepsilon m\) of that affine height can be inserted into
this background through a collar of width
\(C_\tau\varepsilon m+O(1)\).
This follows from the directed-distance proof of
Lemma~\ref{lem:rf-annulus-patching}, which is deterministic and uses
no surface-tension estimate: the interior slope supplies slack
\(\tau d\), exceeding the two boundary errors and a bounded residue
correction.  The same proof applies to nested lattice
parallelograms.  In these constructions all collars and their
interfaces are deterministic; their expected absolute weight is
bounded by a constant times their area plus boundary length.
Thus boundary-crossing lozenges omitted from local scores are
included in the remainder.

For the upper inequality in \eqref{eq:rf-fixed-radius-interior}, fill
all but an \(O(\chi+L/m)\)-fraction of \(T_m^s\) with standard
side-\(L\) patches having collars of width \(\chi L\).
Choose \(\chi\) small and then \(L\) large, so that maximizers at
the standard vanishing tolerance fit through the collars and their
expected scores per unit area are close to \(-\sigma(\rho,a)\).
Their inserted heights and the filled collars deviate from the
common affine background by \(O(L)\).  With \(L\) fixed this is
less than \(\varepsilon m\) for all sufficiently large \(m\).
The first-moment bound on the deterministic remainder proves the
upper inequality, separately for both orientations.

For the reverse inequality, place a side-\(m\) triangle of each
orientation, each with such a collar, in a parallelogram of side
\(m+O_\tau(\varepsilon m+1)\).  Insert a maximizing local state in
each triangle and use the same affine background on both outer
collar boundaries.  Repeat these paired cells inside a standard
patch of side \(M\), filling the leftover strip with the background.
For fixed \(m,\varepsilon\), the resulting global affine error is
\(O_\tau(\varepsilon m+1)\).  It therefore fits the standard
vanishing neighborhood \(M\psi(M)\) as \(M\to\infty\), since
\(M\psi(M)\to\infty\).  Stationarity, linearity of expectation,
and the collar bound give
\[
 \sigma(\rho,a)
 \leq\frac{\sigma_m^+(\rho,a,\varepsilon)
              +\sigma_m^-(\rho,a,\varepsilon)}2
       +C_{\mathcal C,\tau}(\varepsilon+m^{-1}).
\]
All these tensions are nonnegative, since ordinary nonred weights
are nonpositive; hence the enlarged-area denominator can be
discarded in this direction.  Subtract the upper inequality already
proved for the opposite orientation to obtain the lower inequality
in \eqref{eq:rf-fixed-radius-interior}.
The standard-patch comparisons and continuity of the limiting tension
in \cite[Propositions~7.4--7.5]{NarLimit}, together with finite-patch
Wasserstein continuity in \(\rho\) from the local law used in
\cite[Theorem~5.1 and Appendix~A]{NarLimit}, allow finite parameter
nets in \((\rho,a)\).  These make the packing choices uniform on the
stated compact sets.  The constructions
use arbitrary integer sizes, as do Definition~13 and Lemma~11.1
of \cite{NarLimit}.

\smallskip
\noindent\emph{All slopes: pointwise path modification.}
We record the height control behind the path modification of
\cite[Lemma~11.8, Claim~8]{NarLimit}.
Let \(\ell_p\) be a pure-phase height with the same color residue as
a legal height \(h\), and let \(0<t<1\).  Put, for \(0\leq b<1\),
\[
 q=\frac{\ell_p-h}{3}\in\mathbb Z,\qquad
 h_b'=\ell_p-3\lfloor(1-t)q+b\rfloor.
\]
On an edge in the selected direction of \(\ell_p\), the increment
of \(q\) is \(0\) or \(-1\); in either other positive direction it
is \(0\) or \(1\).  The rounded contraction preserves these signs
and changes each nonzero increment only to itself or to zero.
Consequently \(h_b'\) is legal and
\begin{equation}
 h_b'=(1-t)h+t\ell_p+e_b,\qquad
 \|e_b\|_\infty<3.
 \label{eq:rf-pointwise-path-modification}
\end{equation}
Averaging over \(b\in[0,1)\), each nonzero increment is erased with
probability \(t\).  Hence some \(b\) changes at most \(Ctm^2\)
candidate incidences on \(T_m^s\).  This is a deterministic
existence statement for each input height.

Fix \(0<\kappa<1/2\).  The three pure slopes sum to zero.
Normalize all three pure heights to \(h(v_0)\) at a common vertex
\(v_0\).  Each floor map preserves that value and the color residue.
Apply
\eqref{eq:rf-pointwise-path-modification} successively with
\[
 t_1=\frac{\kappa}{3-2\kappa},\qquad
 t_2=\frac{\kappa}{3-\kappa},\qquad
 t_3=\frac{\kappa}{3}.
\]
The resulting height differs on at most \(C\kappa m^2\) incidences
and, for the new affine slope \(a'=(1-\kappa)a\in(1-\kappa)K\),
satisfies, modulo an additive constant,
\[
 \|h'-\ell_{a'}\|_\infty
 \leq(1-\kappa)\varepsilon m+9
 \leq\varepsilon m
 \qquad\bigl(m\geq9/(\kappa\varepsilon)\bigr)
\]
whenever \(h\in\mathcal H_{T_m^s}(a,\varepsilon)\).
Thus the whole height neighborhood is controlled, not merely its
average slope.  The changed incidences may depend on the maximizing
height and on the weights.  The uniform small-subset estimate
\cite[Lemma~4.3]{NarLimit} still bounds their expected total absolute
weight by \(C_{\mathcal C}\kappa^{1/4}m^2\).
Apply \eqref{eq:rf-fixed-radius-interior} to \(a'\).  If
\(\omega_{\mathcal C}\) is a uniform continuity modulus of
\(\sigma\) on \(\mathcal C\times K\), it follows that
\begin{equation}
 \sigma_m^s(\rho,a,\varepsilon)
 \geq \sigma(\rho,a)
       -C_{\mathcal C}\kappa^{1/4}
       -\omega_{\mathcal C}(C\kappa)
       -C_{\mathcal C,\kappa}\varepsilon-o_m(1),
 \qquad a\in K.
 \label{eq:rf-fixed-radius-all-slopes}
\end{equation}

Choose \(\kappa\) first, then
\(0<\varepsilon_{\rm loc}\leq\bar\varepsilon\), and finally \(m_0\),
so that \eqref{eq:rf-fixed-radius-all-slopes} and
\eqref{eq:rf-fixed-radius-interior} have errors within the prescribed
\(\delta\)-budget for every integer \(m\geq m_0\).
The tolerance can be chosen with spare margin: the same estimates
hold for radii between \(\varepsilon_{\rm loc}/2\) and
\(2\varepsilon_{\rm loc}\).
If two slopes are within \(\theta\), their anchored affine heights
on \(T_m^s\) differ by \(O(\theta m)\).  A finite slope net with
\(C\theta<\varepsilon_{\rm loc}/4\), these neighborhood inclusions,
and uniform continuity therefore give the asserted slope uniformity.
No dyadic depth enters these choices.

\smallskip
\noindent\emph{Transfer back to the GUE panels.}
Now fix an integer \(m\geq m_0\).
The Wasserstein bead-model convergence used in the proof of
\cite[Theorem~5.1]{NarLimit}, together with
\cite[Theorem~A.4]{NarLimit}, applies to the fixed-size local gap
vector.  It is uniform over centers in the compact panel core:
otherwise a sequence of counterexample centers has a subsequence
converging to a bulk point; the fixed-index local limit and uniform
subexponential tails give Wasserstein convergence and uniform
integrability along that subsequence, a contradiction.
A local maximum is the maximum of finitely many linear forms with
bounded coefficients.  The same convergence thus transfers its
expectation, uniformly over the deterministic height neighborhoods.

Finally, on \(P_2\) use its local \([i,j]_2\) frame.  At minor level
\(j\), reflection sends rank \(q\) to \(j+1-q\), together with
\(X_2\mapsto-X_2\).  Since \(-X_2\) has the same GUE law, this is an
exact reindexing in distribution; types \(2,3\) retain their local
direction labels.  Taking the maximum of the thresholds for the
finitely many orientations and sources gives \(n_0(m)\).
Thus the order is: choose the accuracy parameters and \(m_0\);
fix any integer \(m\geq m_0\); then let \(n\to\infty\).
This proves \eqref{eq:rf-uniform-asymmetric-upper} and
\eqref{eq:rf-uniform-asymmetric-two-sided}.
\end{proof}

\medskip

\begin{lemma}[Local surface tension and concentration]
\label{lem:rf-local-transfer}
Fix \(\tau\in(0,1)\), \(A>0\), and \(\delta>0\).
There exist \(\varepsilon_{\mathrm{loc}}>0\) and an integer
\(m_0\geq1\) such that, for every integer \(m\geq m_0\), there is
\(n_0(m)\) for which the following holds whenever \(n\geq n_0(m)\),
uniformly over the choices below.  The radius may be required to lie below any prescribed
positive threshold, and is chosen before \(m_0\).

\begin{itemize}
\item both lattice orientations of triangular cells \(T\) of
      sidelength \(m\) with
      \(T\subset n\mathcal C_{\rm core}\),
\item all slopes \(a\in K\),
\item all sources \(r\in\{1,2,3\}\).
\end{itemize}

Let \(z_T\) be a fixed lattice anchor at distance \(O(1)\) from the
barycenter of \(T\), use the normalized area \(|T|\) fixed above, put
\(x_T=z_T/n\), and define
\[
Z_{r,T}(a,\varepsilon_{\mathrm{loc}})
:=
\max_{\Xi:\ \inf_{c\in\mathbb R}
 \|F_\Xi-[c+\langle a,\,\cdot-z_T\rangle]\|_{L^\infty(T)}
 \le \varepsilon_{\mathrm{loc}}m}
W^{\rm rf}_{r,T}(\Xi).
\]
The states \(\Xi\) and scores \(W^{\rm rf}_{r,T}\) follow the
\hyperref[conv:rf-local-height-states]{local-state convention};
\(Z_{r,T}\) depends only on the slope, not on a height normalization.
Then:

\begin{enumerate}
\item[(i)] (Upper bound, all slopes)
\[
\frac{1}{|T|}\,\mathbb E Z_{r,T}(a,\varepsilon_{\mathrm{loc}})
\le -\sigma_\edge^{\rm sym}(\rho_r(x_T),a)+\delta.
\]

\item[(ii)] (Two-sided bound in the interior)
if \(a\in (1-\tau)K\), then
\[
\left|
\frac{1}{|T|}\,\mathbb E Z_{r,T}(a,\varepsilon_{\mathrm{loc}})
+\sigma_\edge^{\rm sym}(\rho_r(x_T),a)
\right|\le \delta.
\]

\item[(iii)] (Local fluctuation bound)
\[
\mathrm{Var}\,Z_{r,T}(a,\varepsilon_{\mathrm{loc}})
\;\le\;
C_{A,\mathcal C_{\rm core}}\,\frac{m^4}{(\log m)^A}.
\]
\end{enumerate}

Moreover, for each \(C_{\rm net}<\infty\), if
\(\mathcal A_T\subset K\) is any deterministic family of
slopes with \(|\mathcal A_T|\le C_{\rm net}m^2\), then
\[
\mathbb E \max_{a\in\mathcal A_T}
\bigl(Z_{r,T}(a,\varepsilon_{\mathrm{loc}})
      -\mathbb E Z_{r,T}(a,\varepsilon_{\mathrm{loc}})\bigr)
\;\le\;
C_{A,\mathcal C_{\rm core},C_{\rm net}}\,
\frac{m^2}{(\log m)^A}.
\]
In particular, for any dyadic depth \(k\geq1\), choose a dyadic
\(M_0(k)\geq2^k m_0\).  For dyadic \(M\geq M_0(k)\), the assertions
hold simultaneously for \(m\in\{2^{-j}M:0\leq j\leq k\}\) once
\[
 n\geq n_0(M,k):=\max_{0\leq j\leq k}n_0(2^{-j}M).
\]
The radius \(\varepsilon_{\rm loc}\) is independent of \(k\).
\end{lemma}

\begin{proof}[Justification]
Use the asymmetric normalization fixed before
Proposition~\ref{prop:rf-uniform-standard-patch-transfer}.  Its
type-\(2\) and type-\(3\) coefficients are respectively \(2\) and
\(1\).  These are local direction types, not a second color
convention.  Write the slope in
barycentric form
\begin{equation}
 a=\vartheta_R(a)a_R+\vartheta_2(a)a_2+\vartheta_3(a)a_3,
 \qquad
 \vartheta_R(a)+\vartheta_2(a)+\vartheta_3(a)=1,
 \label{eq:rf-slope-barycentric-types}
\end{equation}
where
\[
 a_R=-2\widehat\imath,\qquad
 a_2=2\widehat\jmath,\qquad
 a_3=2(\widehat\imath-\widehat\jmath)
\]
are the geometric pure red, type-\(2\), and type-\(3\) slopes of
\eqref{eq:panel-slope-triangle}.  Thus the
\(\vartheta\)'s are the three lozenge-orientation densities in the
local type order.  Put
\begin{equation}
 \mathcal L_\rho(a)
 :=\frac16\rho_2\vartheta_3(a)
   -\frac16\rho_1\vartheta_2(a).
 \label{eq:rf-surface-tension-affine-correction}
\end{equation}
We claim that
\begin{equation}
 \boxed{\quad
 \sigma_\edge^{\rm sym}(\rho,a)
 =\sigma_\edge^{\rm asym}(\rho,a)+\mathcal L_\rho(a).
 \quad}
 \label{eq:rf-symmetric-asymmetric-surface-tension}
\end{equation}

Here is the exact finite-volume identity behind this claim.  With our
raw-gap convention, the two score normalizations and their pure-phase
surface tensions are
\[
\begin{array}{c|c|c|c}
\text{local type}&\text{ordinary weight}
 &\sigma_\edge^{\rm asym}&\sigma_\edge^{\rm sym}\\ \hline
R&0&0&0\\
2&-g_2/3&2\rho_1/3&\rho_1/2\\
3&-g_3/3&\rho_2/3&\rho_2/2 .
\end{array}
\]
Equivalently,
\[
 \mathcal S_D^{\rm asym}(H,\Xi)
 =2\sum_{\substack{\ell\in\Xi\\\ell\ {\rm type}\ 2}}\omega_H(\ell)
   +\sum_{\substack{\ell\in\Xi\\\ell\ {\rm type}\ 3}}\omega_H(\ell),
 \qquad
 \mathcal S_D^{\rm sym}(H,\Xi)
 =\frac32\sum_{\substack{\ell\in\Xi\\
                         \ell\ {\rm type}\ 2\ {\rm or}\ 3}}\omega_H(\ell).
\]
The red coefficient is zero in both gauges.

Let \(D\) be a standard triangular patch in one developed panel, let
\(H\) be the source height, and let \(a_m\) be a tiling height on \(D\).
Orient each internal lattice side \(e\) positively, and let
\(\ell_e\) be the candidate lozenge across \(e\).  Put
\[
 \epsilon(e)=
 \begin{cases}
  +1,&\ell_e\text{ is of type }3,\\
  -1,&\ell_e\text{ is of type }2,\\
   0,&\ell_e\text{ is red (type }1\text{)},
 \end{cases}
 \qquad
 J_H(e):=\epsilon(e)\omega_H(\ell_e).
\]
Because a positive edge has height increment \(+1\) when it is
unselected and \(-2\) when it is selected,
\begin{equation}
 {\bf 1}_{\{\ell_e\text{ is selected}\}}
   =\frac{1-\nabla_ea_m}{3},
 \qquad
 \nabla_ea_m:=a_m(e^+)-a_m(e^-).
\label{eq:rf-lozenge-indicator-from-height}
\end{equation}
Writing \(\mathcal S_D^{\rm sym}(H,a_m)\) and
\(\mathcal S_D^{\rm asym}(H,a_m)\) for the two patch scores,
their difference is
\begin{align}
 \mathcal S_D^{\rm sym}(H,a_m)-\mathcal S_D^{\rm asym}(H,a_m)
 &=
 \frac12\sum_{e\in E(D)}
       J_H(e){\bf 1}_{\{\ell_e\text{ is selected}\}}\notag\\
 &=\frac16\sum_{e\in E(D)}J_H(e)
   -\frac16\sum_{e\in E(D)}J_H(e)\nabla_ea_m.
\label{eq:rf-finite-pairing-gauge-difference}
\end{align}

The current \(J_H\) is divergence-free at every interior vertex.
This can be checked without an appeal to a continuum limit.  In the
local \([i,j]_r\) frame, write \(H(i,j)=H([i,j]_r)\).  The two
rhombus-weight families entering \(J_H\) are
\begin{align*}
 \omega_3(i,j)
 &=\frac{H(i,j+1)+H(i+1,j)-H(i,j)-H(i+1,j+1)}3,\\
 \omega_2(i,j)
 &=\frac{H(i+1,j+1)+H(i-1,j)-H(i,j)-H(i,j+1)}3.
\end{align*}
The type-\(3\) positive edge is
\([i,j]_r\to[i+1,j+1]_r\), whereas the type-\(2\) positive edge is
\([i,j+1]_r\to[i,j]_r\).  Therefore, with
\[
 (\operatorname{div}_D J_H)(v)
 :=\sum_{\substack{e\in E(D)\\e^+=v}}J_H(e)
   -\sum_{\substack{e\in E(D)\\e^-=v}}J_H(e),
\]
the divergence at \(v=[i,j]_r\) is
\[
 \omega_3(i-1,j-1)-\omega_2(i,j)
 -\omega_3(i,j)+\omega_2(i,j-1)=0.
\]
The four displayed second differences cancel term by term.  This
calculation is written in the local \([i,j]_r\) frame.  In the globally
reflected drawing of \(P_2\), global green is type \(2\) and global blue
is type \(3\); the neutral type labels remain tied to the
\(d_2^{(2)}\)- and \(d_3^{(2)}\)-families.  The two
GUE gap coordinates are ordered in this local type order, with the
reflected-rank convention.  Hence the same divergence identity holds
on \(P_2\).  Discrete summation by parts in
\eqref{eq:rf-finite-pairing-gauge-difference} now gives
\begin{equation}
 \mathcal S_D^{\rm sym}(H,a_m)-\mathcal S_D^{\rm asym}(H,a_m)
 =
 \frac16\sum_{e\in E(D)}J_H(e)
 -\frac16\sum_{v\in\partial D}
       a_m(v)(\operatorname{div}_D J_H)(v).
\label{eq:rf-finite-gauge-boundary-formula}
\end{equation}
Thus, for fixed \(H\), the difference depends on the tiling only
through its boundary height trace.  The total boundary divergence is
zero, so the formula is also unchanged by adding a constant to
\(a_m\), as required by the definition of the height neighborhood.

On compact subsets of the positive density parameters, the first
moments of the interlacing gaps are uniformly bounded.  Since
\(\partial D\) has \(O(m)\) vertices,
\begin{equation}
 \mathbb E\sum_{v\in\partial D}
       |(\operatorname{div}_D J_H)(v)|=O(m).
\label{eq:rf-boundary-current-first-moment}
\end{equation}
It follows that replacing an affine boundary trace of slope \(a\) by
one within \(m\varepsilon\) of it changes the expected right-hand side
of \eqref{eq:rf-finite-gauge-boundary-formula} by
\(O(\varepsilon m^2)\), uniformly in the tiling.

After division by the patch area, the expectation of
\eqref{eq:rf-finite-gauge-boundary-formula} for an affine trace is an
affine function of \(a\).  Its value is fixed by the three pure
phases.  Recall that
\(\rho_1=\mathbb E g_2\) and \(\rho_2=\mathbb E g_3\).
Since a pure phase has one lozenge per unit normalized area, direct
evaluation from \(\omega_t=-g_t/3\) gives the asymmetric
surface-tension values at the pure red, type-\(2\), and type-\(3\)
slopes as
\[
 0,\qquad \frac23\rho_1,\qquad \frac13\rho_2,
\]
whereas direct evaluation of the symmetric pairing gives
\[
 0,\qquad \frac12\rho_1,\qquad \frac12\rho_2.
\]
The difference is consequently \(0,-\rho_1/6,\rho_2/6\), and the
unique affine function on \(K\) with these vertex values is precisely
\(\mathcal L_\rho\).  This proves
\eqref{eq:rf-symmetric-asymmetric-surface-tension}.  Explicitly, if
\(a=x\widehat\imath+y\widehat\jmath\), where \(x,y\) are its
coefficients in the oblique geometric unit basis, then
\[
 \vartheta_2(a)=\frac{x+2y+2}{6},\qquad
 \vartheta_3(a)=\frac{x-y+2}{6},
\]
and hence
\[
 \mathcal L_\rho(a)
 =\frac{\rho_2(x-y+2)-\rho_1(x+2y+2)}{36}.
\]

The same calculation gives the finite-patch comparison, uniformly
for \(\rho\) in a compact subset of \(\mathbb R_{>0}^2\), for
\(a\in K\), and for the height neighborhood of radius \(m\varepsilon\):
if \(Z_D^{\rm sym}\) and \(Z_D^{\rm asym}\) denote the maxima of the
two displayed scores over that same neighborhood, then
\begin{equation}
 \frac1{|D|}\mathbb E Z_D^{\rm sym}(\rho,a,\varepsilon)
 =\frac1{|D|}\mathbb E Z_D^{\rm asym}(\rho,a,\varepsilon)
  -\mathcal L_\rho(a)+O_\rho(\varepsilon)+o_m(1).
 \label{eq:rf-finite-patch-gauge-comparison}
\end{equation}
Here \eqref{eq:rf-boundary-current-first-moment} gives the
\(O_\rho(\varepsilon)\) error, while bead-model convergence and uniform
integrability give the \(o_m(1)\) error for the affine trace.  The
signs in
\eqref{eq:rf-finite-patch-gauge-comparison} follow from the convention
that surface tension is the negative expected maximal weight.
Combining this comparison with
Proposition~\ref{prop:rf-uniform-standard-patch-transfer}, first choose
\(\varepsilon_{\rm loc}\) below any previously prescribed threshold
and then choose \(m_0\) so that the transfer error, the
\(O_\rho(\varepsilon_{\rm loc})\) term, and the \(o_m(1)\) term are
absorbed by \(\delta\).  This gives parts~(i) and~(ii).  It also
gives existence, joint continuity in \((\rho,a)\), convexity in \(a\)
for each fixed \(\rho\), and the required scale comparison for
\(\sigma_\edge^{\rm sym}\) directly from the
corresponding assertions for \(\sigma_\edge^{\rm asym}\), since
\(\mathcal L_\rho\) is affine in \(a\) and continuous in \(\rho\).

For the fluctuation assertion one needs a finite-volume statement, not
merely the limiting identity
\eqref{eq:rf-symmetric-asymmetric-surface-tension}.  We therefore
record the precise coefficient pattern.  Fix a tiling \(\Xi\).  For
each type \(t\in\{2,3\}\) and minor level \(k\), let
\(S_{t,k}(\Xi)\) be the set of ranks at which \(\Xi\) selects a
lozenge of type \(t\).  Let \(g_{t,j,k}\) denote the raw gap with the
fixed source scale and reflected-rank convention already included,
and set the signed ordinary-weight variable
\(G_{t,j,k}=\omega_t=-g_{t,j,k}/3\).  The centered
symmetric score then has the exact form
\begin{equation}
 X_\Xi^{\rm sym}-\mathbb E X_\Xi^{\rm sym}
 =
 \frac32\sum_{t\in\{2,3\}}\sum_k
   \sum_{j\in S_{t,k}(\Xi)}
       \bigl(G_{t,j,k}-\mathbb E G_{t,j,k}\bigr).
\label{eq:rf-local-level-subset-decomposition}
\end{equation}
Thus the forms that occur are sums over actual rank subsets at each
minor level; no assertion about arbitrary bounded full-panel
coefficient vectors is needed.

The subset-sum reduction in
\cite[\S4, equations~(4.1)--(4.2)]{NSTGUE} applies to each type and
minor level.  The local logarithmic saving comes from
\cite[Theorem~A.5]{NarLimit}; see also its covariance formulation in
\cite[Lemma~9.1 and the proof of Proposition~7.1]{NarLimit}.
For a rank subset in an interval of length \(O(m)\), that estimate
bounds the variance of its gap sum by
\(C_A m^2/(\log m)^A+o_n(1)\), with \(m\) fixed before
\(n\to\infty\).  The source scale and the factors \(1/3\) and
\(3/2\) change only the constant on the compact panel core.
Splitting \eqref{eq:rf-local-level-subset-decomposition} by type and
panel orientation, applying Minkowski's inequality over the
\(O(m)\) minor levels, and enlarging \(n_0(m)\) to absorb the
finite-size remainder therefore gives, uniformly over every actual
tiling \(\Xi\),
\[
 \operatorname{Var}X_\Xi^{\rm sym}
 \le C_A\frac{m^4}{(\log m)^A}.
\]
We pass from this fixed-tiling estimate to the local maximum, and then
to a finite maximum over slopes, in one whitening argument at the end
of the proof.

Proposition~\ref{prop:rf-uniform-standard-patch-transfer} supplies the
compact-core, moving-center, slope, and orientation uniformity used
here: the symmetric coefficients in
\eqref{eq:rf-local-level-subset-decomposition} differ only by fixed
type-dependent factors, so the same subsequence argument applies.
Enlarge \(n_0(m)\) for the variance estimate, with \(n\to\infty\) at
fixed \(m\) before \(m\to\infty\), and take the maximum of these
thresholds for any finite family of integer sizes.  This gives the
dyadic specialization while keeping \(\varepsilon_{\rm loc}\)
independent of \(k\).

Finally, let \(\mathcal A_T\subset K\) be deterministic with
\(\lvert\mathcal A_T\rvert\leq C_{\rm net}m^2\).  We spell out the concentration
argument needed to take the maximum over \(\mathcal A_T\).  In
particular, we do not assert that the one-dimensional law of a local
maximum is itself log-concave.

Let \(G_T\) be the vector of local signed ordinary-weight coordinates
entering the tiling scores on \(T\), and let \(E_T\) be the linear
space parallel to its
affine support.  On \(E_T\), put
\[
 \Sigma_T:=\operatorname{Cov}(G_T),\qquad
 Y_T:=\Sigma_T^{-1/2}(G_T-\mathbb E G_T).
\]
Thus the inverse is taken only on \(E_T\).  The law of \(G_T\) is a
scaled and reflected linear image of a log-concave marginal of the
product Gelfand--Tsetlin law, so \(Y_T\)
is isotropic and log-concave in a dimension
\(d_T\leq C_{\mathcal C_{\rm core}}m^2\).

For \(a\in\mathcal A_T\), let \(\mathcal H_T(a)\) denote the
deterministic family of tilings in the definition of
\(Z_{r,T}(a,\varepsilon_{\mathrm{loc}})\).  Write the score of
\(\Xi\in\mathcal H_T(a)\) as
\[
 W^{\rm rf}_{r,T}(\Xi)=\langle c_\Xi,G_T\rangle+b_\Xi.
\]
After absorbing \(\mathbb E G_T\) into the constant term,
\(Z_{r,T}(a,\varepsilon_{\mathrm{loc}})\), as a function of \(Y_T\),
is the maximum of affine functions with linear parts
\(\Sigma_T^{1/2}c_\Xi\).  Hence
\[
 \operatorname{Lip}_{Y_T}Z_{r,T}(a,\varepsilon_{\mathrm{loc}})^2
 \leq \sup_{\Xi\in\mathcal H_T(a)}
       \|\Sigma_T^{1/2}c_\Xi\|_2^2
 =\sup_{\Xi\in\mathcal H_T(a)}
       \operatorname{Var}W^{\rm rf}_{r,T}(\Xi).
\]
Applying the preceding fixed-tiling estimate with an arbitrary
logarithmic exponent \(B>0\) therefore gives, uniformly in \(a,T,r\),
\[
 \operatorname{Lip}_{Y_T}Z_{r,T}(a,\varepsilon_{\mathrm{loc}})
 \leq C_{B,\mathcal C_{\rm core}}
       \frac{m^2}{(\log m)^{B/2}}.
\]

We now use the Cheeger estimate appearing in the proof of
\cite[Lemma~3]{NSTGUE}, which is based on Klartag's logarithmic KLS
bound \cite{KlartagKLS}.  By the quantitative equivalence between
Cheeger isoperimetry and exponential concentration for log-concave
measures \cite[Theorem~1.5]{MilmanConvexity}, if \(Y\) is isotropic
and log-concave in dimension \(d\), and \(F\) is \(L\)-Lipschitz, then
\[
   \|F(Y)-\mathbb EF(Y)\|_{\psi_1}
      \leq C L\sqrt{\log(2+d)}.
\]
Here the \(\psi_1\)-norm is normalized as in~\eqref{eq:rf-psi1-definition}.

 Consequently,
\[
 \left\|Z_{r,T}(a,\varepsilon_{\mathrm{loc}})
       -\mathbb E Z_{r,T}(a,\varepsilon_{\mathrm{loc}})\right\|_{\psi_1}
 \leq C_{B,\mathcal C_{\rm core}}
       \frac{m^2}{(\log m)^{B/2-1/2}}.
\]
Since a \(\psi_1\)-bound controls the second moment, choosing
\(B\geq A+1\) already proves part~\textup{(iii)}.
No independence among the variables indexed by \(a\) is needed.  The
standard maximal inequality for subexponential random variables and
\(\lvert\mathcal A_T\rvert\leq C_{\rm net}m^2\) now yield
\[
 \mathbb E\max_{a\in\mathcal A_T}
 \bigl(Z_{r,T}(a,\varepsilon_{\mathrm{loc}})
      -\mathbb E Z_{r,T}(a,\varepsilon_{\mathrm{loc}})\bigr)
 \leq C_{B,\mathcal C_{\rm core},C_{\rm net}}
       \frac{m^2}{(\log m)^{B/2-3/2}}.
\]
Since the fixed-tiling estimate is available for every \(B>0\), take
\(B=2A+4\).  The last exponent is then \(A+\tfrac12\), which, after
enlarging the constant to cover bounded \(m\), proves the asserted
bound with exponent \(A\) and completes the justification.
\end{proof}

\paragraph{Uniform domination of the continuum bulk energy.}
Nonpositivity of every local score, convexity in the slope, and the
three pure-phase values computed above give
\begin{equation}
 0\leq\sigma_\edge^{\rm sym}(\rho,a)
 \leq\frac12\bigl(\vartheta_2(a)\rho_1+
                         \vartheta_3(a)\rho_2\bigr)
 \leq\frac12(\rho_1+\rho_2),\qquad a\in K.
\label{eq:rf-continuum-bulk-domination}
\end{equation}
The majorant is integrable on each full panel.  Indeed, the raw-gap
identities give
\[
 \sum_{i=1}^{j-1}\bigl(g_2(i,j)+g_3(i,j)\bigr)
       =\lambda_1^{(j)}-\lambda_j^{(j)},\qquad 1\leq j<n.
\]
The expected sum of these spectral widths over the minor levels is
\(O(n^2)\), by the spectral-width estimates used in
Lemma~\ref{lem:rf-boundary-strip}.  Local Wasserstein convergence of
the gap vectors, followed by exhaustion by compact panel cores,
therefore gives
\(\int_{0<I<J<1}(\rho_{r,1}+\rho_{r,2})\,dI\,dJ<\infty\).
Restricting to \(\mathsf R_{p,r}\) preserves this bound.  In
particular, continuum bulk integrals over the deleted boundary and
cut strips tend to zero uniformly over all admissible slopes.

\medskip

\begin{lemma}[Small-support bound]
\label{lem:rf-small-support}
Let \(\mathfrak U_n(\theta)\) be the family of all sets \(U\) of
candidate bulk rhombi contained in the panel core with
\(|U|\leq\theta n^2\).  Then
\[
\mathbb E\sup_{U\in\mathfrak U_n(\theta)}\sup_\Xi
\bigl|W^{\rm rf}_{r,U}(\Xi)\bigr|
\;\le\;
C\,\theta^{1/4}\,n^2,
\]
for all sufficiently small \(\theta\).  In particular,
\[
\mathbb E\sup_{U\in\mathfrak U_n(\theta)}\sup_\Xi
\bigl|W^{\rm rf}_{r,U}(\Xi)\bigr|
=o_\theta(n^2).
\]
\end{lemma}

\begin{proof}[Justification]
Let \(X=(x_i)\) be the vector of absolute type-\(2\)/type-\(3\)
ordinary weights, equivalently the corresponding raw gaps divided by
three with the source scale included.
By \cite[Theorem~4.1]{NarLimit}, this vector is log-concave, and its
coordinates have uniformly bounded means on the core.

Applying the small-subset inequality
\cite[Lemma~4.3]{NarLimit} to subsets of size \(\le \theta n^2\)
gives the stated bound (see also equation~(8.4) in \cite{NarLimit}).
\end{proof}

\medskip

\begin{lemma}[Bulk variance for red-free tiling patterns]
\label{lem:rf-family-variance}
Fix \(\varepsilon>0\).  After developing panel \(r\) in the
orientation of \eqref{eq:rf-gue-hive-map}, every ordinary type-\(2\) or
type-\(3\) lozenge weight---equivalently, every nonred lozenge
weight---is minus one third of one of the two nonnegative
interlacing-gap variables of the \(r\)-th GUE minor process, with the
fixed panel scale included.  Call such a candidate lozenge \emph{spectrally
\(\varepsilon\)-bulk} if the associated gap has indices \((j,k)\)
satisfying
\begin{equation}
             j\geq\varepsilon n,\qquad
             k-j\geq\varepsilon n,\qquad
             n-k\geq\varepsilon n.
\label{eq:rf-spectral-bulk-indices}
\end{equation}
On \(P_2\), global green is type \(2\) and global blue is type \(3\);
the type association is retained and the reflected rank index is used.
Let
\[
 G_{\varepsilon}^{(n)}
  =\bigl(G_{1,\varepsilon}^{(n)},
         G_{2,\varepsilon}^{(n)},
         G_{3,\varepsilon}^{(n)}\bigr)
\]
be the concatenated vector of these spectrally bulk type-\(2\)/type-\(3\)
signed ordinary weights \(\omega_t=-g_t/3\), including the source
scale and reflected-rank convention, from the three independent
sources.  For a red-free tiling
\(\Xi\), let \(c_{\Xi,\varepsilon}\) have coefficient \(3/2\) on a
selected bulk lozenge of type \(2\) or \(3\), and coefficient \(0\) otherwise,
and put
\begin{equation}
 B_{\Xi,\varepsilon}
  :=\langle c_{\Xi,\varepsilon},G_{\varepsilon}^{(n)}\rangle.
\label{eq:rf-spectral-bulk-weight}
\end{equation}

If \(\mathcal A_n\subseteq\mathcal T_{np}\) is any nonempty
deterministic subfamily and \((a_\Xi)_{\Xi\in\mathcal A_n}\) are
arbitrary deterministic offsets, then, for some \(\kappa>0\),
\begin{equation}
 \operatorname{Var}\max_{\Xi\in\mathcal A_n}
       \bigl(B_{\Xi,\varepsilon}+a_\Xi\bigr)
 \leq C_\varepsilon n^{4-\kappa+o(1)}\log n
 =o(n^4),
\label{eq:rf-bulk-tiling-family-variance}
\end{equation}
uniformly over \(\mathcal A_n\) and the offsets.

Moreover, if
\begin{equation}
 E_{\varepsilon,n}:=
 \frac32\sum_{r=1}^3
   \sum_{\substack{\ell\text{ of type }2\\
          \ell\ \text{not spectrally }\varepsilon\text{-bulk}}}
      |\omega_{H_r^{(n)}}(\ell)|
 +\frac32\sum_{r=1}^3
   \sum_{\substack{\ell\text{ of type }3\\
          \ell\ \text{not spectrally }\varepsilon\text{-bulk}}}
      |\omega_{H_r^{(n)}}(\ell)|,
\label{eq:rf-spectral-edge-remainder}
\end{equation}
where the sums range over candidate ordinary lozenges in the three
developed panel regions, then
\begin{equation}
 \sup_{\Xi\in\mathcal T_{np}}
       |W_{p,\mathrm{edge}}^{\rm rf}(\Xi)
                    -B_{\Xi,\varepsilon}|
 \leq E_{\varepsilon,n},
 \qquad
 \lim_{\varepsilon\downarrow0}\limsup_{n\to\infty}
       n^{-2}\|E_{\varepsilon,n}\|_2=0.
\label{eq:rf-spectral-edge-L2-removal}
\end{equation}

\end{lemma}

\begin{proof}[Justification]
Fix \(\varepsilon\).  Split \(B_{\Xi,\varepsilon}\) by source,
lozenge type, panel orientation, and minor level.  For each such
piece, the selected ranks form a subset of the bulk indices in
\eqref{eq:rf-spectral-bulk-indices}.  The calculation in
\cite[\S4, equations~(4.1)--(4.2)]{NSTGUE} reduces its variance to that
of an arbitrary subset sum of consecutive-level interlacing gaps.
The estimate following equation~(4.2) there is uniform in that subset.
Minkowski's inequality over the \(O(n)\) minor levels and the bounded
number of types and orientations therefore gives, for every fixed
tiling,
\begin{equation}
       \operatorname{Var}B_{\Xi,\varepsilon}
       \leq C_\varepsilon n^{4-\kappa+o(1)}.
\label{eq:rf-fixed-tiling-bulk-variance}
\end{equation}
The factors \(3/2\), \(1/3\), and the fixed source scales only change
the constant.  The three source contributions are independent, so
their variances add.

The product law of the three GUE minor gap vectors is log-concave by
\cite[Theorem~4.1]{NarLimit}; so is the scaled and reflected marginal
\(G_{\varepsilon}^{(n)}\).  The maximal-affine-family estimate used
next, \cite[Lemma~3]{NSTGUE}, is based on Klartag's work on the
Kannan--Lov\'asz--Simonovits (KLS) conjecture
\cite{KlartagKLS}.  Apply
\cite[Lemma~3]{NSTGUE} to the affine functions
\[
 g\longmapsto\langle c_{\Xi,\varepsilon},g\rangle+a_\Xi,
             \qquad \Xi\in\mathcal A_n.
\]
Its ambient dimension is \(O(n^2)\), and
\eqref{eq:rf-fixed-tiling-bulk-variance} proves
\eqref{eq:rf-bulk-tiling-family-variance}.  Notice that the argument
uses only the actual \(3/2,3/2,0\) incidence patterns of red-free tilings;
it makes no assertion for arbitrary coefficient vectors.

It remains to remove the spectral edge.  The first inequality in
\eqref{eq:rf-spectral-edge-L2-removal} follows immediately from
\eqref{eq:rf-discrete-bulk-weight}.  For the second, use the
rigidity-and-telescoping argument in \cite[\S4]{NSTGUE}: the sum of all
absolute type-\(2\)/type-\(3\) ordinary weights failing
\eqref{eq:rf-spectral-bulk-indices} reduces to
\(O(\varepsilon n)\) full spectral widths and \(O(n)\) spectral-edge
increments of rank \(O(\varepsilon n)\).  Eigenvalue rigidity makes
its \(L^2\)-norm \(o_\varepsilon(n^2)\).  More explicitly, the
stretched-exponential rigidity tails used there may be integrated,
while the complementary exceptional event has negligible second
moment, exactly as at the beginning of that section.  Restricting from
a full Gelfand--Tsetlin panel to one of our polygonal panel regions can
only decrease this sum of absolute weights.  Re-expressing \(P_2\) in
its local \([i,j]_2\) frame preserves this type-\(2\)/type-\(3\)
absolute-weight sum and only reverses the rank convention; the three
fixed source scales again change only constants.  Summing the three
panel estimates proves the double-limit assertion in
\eqref{eq:rf-spectral-edge-L2-removal}.
\end{proof}

\begin{lemma}[Fixed-width boundary and cut collars in \(L^2\)]
\label{lem:rf-fixed-collar-L2}
Fix \(L<\infty\).  For \(j=1,2,3\), let
\(\mathcal C_{j,n}^{(L)}\) be any set of candidate type-\(2\) or
type-\(3\) lozenges contained in the lattice \(L\)-neighborhood of a
fixed finite collection of polygonal segments in triangular-lattice
directions, of total length at most \(C_{\rm len}n\), in the \(j\)-th
developed panel.  Then
\begin{equation}
 n^{-2}\left\|
 \frac32\sum_{j=1}^3\sum_{\ell\in\mathcal C_{j,n}^{(L)}}
       |\omega_{H_j^{(n)}}(\ell)|
 \right\|_2\longrightarrow0.
\label{eq:rf-fixed-collar-L2}
\end{equation}
The conclusion is uniform when \(L\), \(C_{\rm len}\), and the number
of segments are bounded.
\end{lemma}

\begin{proof}
Fix a spectral cutoff \(\varepsilon>0\).  The part of
\(\mathcal C_{j,n}^{(L)}\) satisfying
\eqref{eq:rf-spectral-bulk-indices} contains \(O_L(n)\) candidates.
On this fixed spectral core, the absolute value of each type-\(2\) or
type-\(3\) ordinary weight is one third of the corresponding
nonnegative interlacing gap, with the fixed positive source scale
included.  Its one-dimensional marginal is
log-concave and has uniformly bounded mean.  The standard
one-dimensional log-concave moment comparison
\cite[Lemma~5.7]{LovaszVempalaLogconcave} therefore gives
\[
 \sup_{j,n}\sup_{\ell\ \mathrm{spectrally}\
                  \varepsilon\text{-bulk}}
 \|\omega_{H_j^{(n)}}(\ell)\|_2\leq C_\varepsilon.
\]
Minkowski's inequality consequently bounds the \(L^2\)-norm of the
spectrally bulk part by \(C_{\varepsilon,L}n=o(n^2)\).

The complementary part is bounded by the full-panel spectral-edge
majorant used in the proof of
\eqref{eq:rf-spectral-edge-L2-removal}.  That proof gives
\[
 \lim_{\varepsilon\downarrow0}\limsup_{n\to\infty}n^{-2}
 \left\|\frac32\sum_{j=1}^3
   \sum_{\substack{\ell\text{ of type }2\text{ or }3\\
                    \ell\text{ not spectrally }
                    \varepsilon\text{-bulk}}}
       |\omega_{H_j^{(n)}}(\ell)|\right\|_2=0,
\]
where enlarging a polygonal panel region to the full developed panel
only increases the nonnegative sum.  First let \(n\to\infty\), with
\(\varepsilon\) fixed, and then let \(\varepsilon\downarrow0\).
This proves \eqref{eq:rf-fixed-collar-L2}.
\end{proof}

\begin{lemma}[Extension criterion for a domain with cuts]
\label{lem:rf-multiply-connected-extension}
This is the directed-envelope form of Thurston's tileability criterion;
compare \cite[Theorem~1.4]{GorinLozengeTilings}.  The statement here
allows holes and records the cut-compatible gluing conclusion needed
below, so we include the proof.

Let \(H\) be a finite connected union of elementary triangular-lattice
cells; holes are allowed.  Fix the standard three-coloring
\(c:V(H)\to\mathbb Z/3\mathbb Z\), increasing by one along every
positive lattice edge.  Let \(g:\partial H\to\mathbb Z\) satisfy
\begin{enumerate}
\item \(g(y)-g(x)\le d_H(x,y)\) for every pair of boundary vertices
      for which the positive directed distance is finite;
\item \(g(x)-c(x)\) is independent of \(x\in\partial H\) modulo three.
\end{enumerate}
Then \(g\) extends to an integer height function on \(H\) whose
increment on every positive edge is \(+1\) or \(-2\).  Consequently,
when the boundary increments are those induced by local height states on
adjoining domains, the extensions glue to a lozenge tiling of the union.
\end{lemma}

\begin{proof}
Define the upper envelope
\[
 U(z)=\min_{x\in\partial H}\{g(x)+d_H(x,z)\}.
\]
The directed-distance inequalities imply \(U=g\) on \(\partial H\).
For a positive unit edge \(z\to z'\), the triangle inequality gives
\(U(z')\le U(z)+1\).  The other two sides of either incident elementary
triangle give a positive path of length two from \(z'\) to \(z\), so
\(U(z')\ge U(z)-2\).

Every positive path from \(x\) to \(z\) has length congruent to
\(c(z)-c(x)\pmod3\).  The second hypothesis therefore makes every term
in the definition of \(U(z)\) congruent to the same residue.  Hence the
increment of \(U\) on a positive unit edge is congruent to one modulo
three and lies in \([-2,1]\); it is therefore \(+1\) or \(-2\).
Around an elementary triangle the three directed increments sum to
zero, so exactly one edge has increment \(-2\).  Pair the two triangles
adjacent to every such internal edge.  At the boundary, a \(-2\) edge
means that the adjacent triangle is already paired with a triangle in
the adjoining local state.  This gives the asserted height function
and the gluing statement.
\end{proof}

\begin{lemma}[Triangular annulus patching]
\label{lem:rf-annulus-patching}
Let \(T^-\subset T^+\) be two nested lattice triangles in one
developed panel, and put
\(A=T^+\setminus\operatorname{int}(T^-)\).  Suppose their boundaries are
separated by at least \(\widetilde\varepsilon m\), and let \(g^-\) and
\(g^+\) be integer boundary functions which extend to local height
states on \(T^-\) and \(T^+\), respectively.  Assume that for an affine
height function \(L\) with
\(\nabla L\in(1-\widehat\varepsilon)K\),
\[
 \|g^\pm-L\|_{L^\infty(\partial T^\pm)}
       <\varepsilon' m,
 \qquad 0<2\varepsilon'<
       \widehat\varepsilon\widetilde\varepsilon,
\]
and that
\((\widehat\varepsilon\widetilde\varepsilon-2\varepsilon')m>1\).
Then there is \(\chi\in\{-1,0,1\}\) such that the boundary data
\(g^+\) on \(\partial T^+\) and \(g^-+\chi\) on \(\partial T^-\)
extend to a local height state on the annulus \(A\).  This state glues
to adjoining local states with those two traces in the
\hyperref[conv:rf-local-height-states]{interface sense fixed above}.
\end{lemma}

\begin{proof}
Let \(F^+\) and \(F^-\) be local-height-state extensions of \(g^+\) and
\(g^-\) to \(T^+\) and \(T^-\).  For any local height state \(F\), the
quantity \(F-c\) is constant modulo three on each connected domain,
because both allowed increments are congruent to \(1\) modulo three.
Thus there are residues \(r^+,r^-\in\mathbb Z/3\mathbb Z\) such that
\[
 F^+-c\equiv r^+\pmod3\quad\hbox{on }T^+,
 \qquad
 F^--c\equiv r^-\pmod3\quad\hbox{on }T^-.
\]
Choose \(\chi\in\{-1,0,1\}\) so that
\(r^-+\chi=r^+\) modulo three.  If \(g\) denotes the boundary datum
equal to \(g^+\) on \(\partial T^+\) and to \(g^-+\chi\) on
\(\partial T^-\), then
\[
                    g-c\equiv r^+\pmod3
\]
on both boundary components.  This is precisely the residue
compatibility required by
Lemma~\ref{lem:rf-multiply-connected-extension}.

It remains to verify the directed-distance inequalities in
Lemma~\ref{lem:rf-multiply-connected-extension}, with \(H=A\).
Pairs of points on the same boundary component satisfy them because
the local extensions give the required inequalities in \(T^+\) or
\(T^-\).  A lattice triangle is convex in the three lattice
coordinates, so its directed distance between two of its vertices is
the ambient directed distance.  Every directed path in \(A\) is also an
ambient directed path.  Hence, on either boundary component,
\(d_{T^\pm}(x,y)\leq d_A(x,y)\), as required.

Now let \(u,v\) lie on different boundary components, in either order,
and suppose \(d_A(u,v)<\infty\).  Every positive path from \(u\) to \(v\)
crosses the collar, so
\[
 d_A(u,v)\geq\widetilde\varepsilon m.
\]
Since \(\nabla L\in(1-\widehat\varepsilon)K\), summing the directional
slope inequality along such a path and then minimizing gives
\[
 L(v)-L(u)\leq
 (1-\widehat\varepsilon)d_A(u,v).
\]
Let \(g\) denote the two-component boundary datum defined above.  The
two approximation errors and the possible residue shift give, in both
cross orders,
\begin{align*}
 g(v)-g(u)
 &<L(v)-L(u)+2\varepsilon'm+|\chi|\\
 &\leq d_A(u,v)
   -(\widehat\varepsilon\widetilde\varepsilon-2\varepsilon')m+1\\
 &<d_A(u,v),
\end{align*}
where the last inequality is exactly the final numerical hypothesis.
Thus all directed-distance inequalities hold.
Lemma~\ref{lem:rf-multiply-connected-extension} supplies a local
height state on \(A\).  Its trace agrees exactly with the adjoining
traces, so the interface-selected-edge convention gives the asserted
gluing.
\end{proof}

\begin{lemma}[Global seam-compatible scaffold patching]
\label{lem:rf-global-patching}
Let \(\Xi^0\in\mathcal T_{np}\) be a global reference tiling.  In the
interiors of the three developed sheets choose a finite collection of
nested lattice triangles
\[
 T^-_\alpha\subset T^+_\alpha,
 \qquad \alpha\in\mathcal I,
\]
whose closed outer triangles are pairwise disjoint and do not meet a cut
or a closed outer boundary.  Write \(m_\alpha\) for the side length of
\(T^+_\alpha\), and assume that
\[
 \operatorname{dist}(\partial T^-_\alpha,
                     \partial T^+_\alpha)
 \ge \widetilde\varepsilon m_\alpha.
\]
For every \(\alpha\), let \(\Xi_\alpha\) be a local height state on
\(T^-_\alpha\), in the convention above.  Suppose that there is an
affine height function \(L_\alpha\), with
\(\nabla L_\alpha\in(1-\widehat\varepsilon)K\), such that the trace of
\(\Xi^0\) on \(\partial T^+_\alpha\) and the trace of
\(\Xi_\alpha\) on \(\partial T^-_\alpha\) are both within
\(\varepsilon' m_\alpha\) of \(L_\alpha\), and suppose in addition that
\begin{equation}
 \|F_{\Xi_\alpha}-L_\alpha\|_{L^\infty(T^-_\alpha)}
 \le \zeta_\alpha m_\alpha.
\label{eq:rf-inner-patch-closeness}
\end{equation}
Here
\[
 0<2\varepsilon'<\widehat\varepsilon\widetilde\varepsilon,
 \qquad
 (\widehat\varepsilon\widetilde\varepsilon-2\varepsilon')m_\alpha>1
 \quad(\alpha\in\mathcal I).
\]
Then there is a global tiling \(\Xi\in\mathcal T_{np}\) such that
\begin{enumerate}
\item \(\Xi=\Xi^0\) outside \(\bigcup_\alpha T^+_\alpha\);
\item on every \(T^-_\alpha\), the restriction of \(\Xi\) is the same
      local height state as \(\Xi_\alpha\), up to its additive constant;
\item at every retained seam segment, the same facing triangle is the
      singleton in \(\Xi\) as in \(\Xi^0\).
\end{enumerate}
Moreover, if \(w_\alpha\) is the maximum triangular-lattice distance of a point of
\(T^+_\alpha\setminus T^-_\alpha\) from its two boundary components,
then for an absolute constant \(C_K\),
\begin{equation}
 \|F_\Xi-L_\alpha\|_{L^\infty(T^+_\alpha)}
 \le \max\{\zeta_\alpha m_\alpha+1,
            \varepsilon'm_\alpha+C_Kw_\alpha+2\}.
\label{eq:rf-global-patching-closeness}
\end{equation}
The additive constants of the inserted inner height functions may change
by elements of \(\{-1,0,1\}\), but their selected edges, local scores,
and all resulting tile weights are unchanged.
\end{lemma}

\begin{proof}
Apply Lemma~\ref{lem:rf-annulus-patching} to each annulus
\(T^+_\alpha\setminus T^-_\alpha\).  The restriction of \(\Xi^0\) to
\(T^+_\alpha\) and the state \(\Xi_\alpha\) provide the two local-state
extensions required there.  The lemma supplies a shift
\(\chi_\alpha\in\{-1,0,1\}\) of the inner height trace and a local height
state on the annulus joining that shifted trace to the trace of \(\Xi^0\)
on the outer boundary.  Adding a constant to a height function changes no
\(+1/-2\) increment, hence it changes neither the selected edges nor the
local score of the inner state.  The outer triangles are disjoint, so all
annular fillings can be inserted simultaneously.

On an inner interface, the annular trace is exactly the shifted trace of
\(\Xi_\alpha\); on an outer interface, it is exactly the trace of
\(\Xi^0\).  Hence the increments, and therefore the selected interface
edges, agree on the two sides.  Each selected interface edge pairs the
two adjacent triangles across that interface, while at each unselected
interface edge both adjacent triangles retain their pairings on their
own sides.  The local states therefore glue to a global tiling.  The
construction leaves \(\Xi^0\) unchanged in a neighborhood of every cut,
so its complementary singleton choices and affine-aligned seam traces
are unchanged.  This proves the first three assertions.

For the quantitative estimate, both a lozenge height function and an
affine function with slope in \(K\) have a universal Lipschitz constant
in the triangular-lattice graph metric.  Given a point of an annulus,
join it to the nearer boundary component by a path of length at most
\(w_\alpha\).  The boundary discrepancy from \(L_\alpha\) is at most
\(\varepsilon'm_\alpha+1\), the extra one accounting for
\(\chi_\alpha\).  Comparing both functions along that path gives the
second term on the right of
\eqref{eq:rf-global-patching-closeness}.  On \(T^-_\alpha\), use the
separate hypothesis \eqref{eq:rf-inner-patch-closeness}, together with
the possible unit shift.  This proves the displayed maximum bound.
\end{proof}

\begin{lemma}[A boundary-compatible strict reference height]
\label{lem:rf-strict-reference-height}
There is a function
\(f_p^\ddagger\in\operatorname{AHT}(\mathsf R_p)\) such that
\begin{equation}
 \nabla f_{p,r}^\ddagger(x)\in K^\circ
 \quad\text{for a.e. }x\in\operatorname{int}(\mathsf R_{p,r}).
\label{eq:rf-strict-reference-slope}
\end{equation}
For every \(\zeta>0\), after deleting a \(\zeta\)-neighborhood of the
panel boundaries, there is
\(\kappa=\kappa(p,\zeta)>0\) such that
\begin{equation}
 \operatorname{dist}\bigl(\nabla f_{p,r}^\ddagger(x),\partial K\bigr)
 \geq\kappa
\label{eq:rf-strict-reference-margin}
\end{equation}
at every differentiability point of the remaining panel cores.
In particular, \(f_p^\ddagger\) has the exact outer traces
\(\gamma_{p,r}\) and satisfies the two aligned seam relations
\eqref{eq:rf-continuum-aligned-traces}.
\end{lemma}

\begin{proof}
Let \(q\) be a common denominator of \(p\), and choose increasing
multiples \(n_j=2^jq\).  Give the finite nonempty set
\(\mathcal T_{n_jp}\) the uniform probability measure \(\mu_j\), and
define the global barycentric height
\begin{equation}
 g^{(j)}:=\int_{\mathcal T_{n_jp}}
              \overline F_{\Xi,n_j}\,d\mu_j(\Xi).
\label{eq:rf-global-height-barycenter}
\end{equation}
It is important here that the average is over global legal tilings,
not over independent choices on the three panels.  Every summand has
gradient in \(K\), the exact outer traces \(\gamma_{p,r}\), and the
same aligned normalization.  Hence \(g^{(j)}\) has these properties as
well.

We first record strictness away from the panel boundaries.  If \(e\) is
an ordinary internal positive lattice edge, let
\(I_e(\Xi)\) indicate that \(e\) is the selected edge of its lozenge.
The directional increment of the discrete height is
\[
                   1-3I_e(\Xi).
\]
By Proposition~\ref{prop:ordinary-edge-allowed}, some global tiling
selects \(e\).  Since \(\mu_j\) has full support,
\begin{equation}
 \mathbb E_{\mu_j}I_e>0,
 \qquad
 \Delta_e g^{(j)}=1-3\mathbb E_{\mu_j}I_e<1.
\label{eq:rf-barycenter-ordinary-edge-slack}
\end{equation}
For every elementary triangle not meeting a panel boundary, each of
its three possible selected edges is an ordinary internal edge.
Thus all three supporting inequalities defining \(K\) are strict on
that triangle, and \(\nabla g^{(j)}\in K^\circ\) there.

The cut seams require the global averaging used above.  On a retained
seam segment let \(\epsilon_-(\Xi),\epsilon_+(\Xi)\) be the two facing
singleton indicators.  Lemma~\ref{lem:rf-seam-compatibility} gives,
for every \(\Xi\),
\begin{equation}
 \epsilon_-(\Xi)+\epsilon_+(\Xi)=1,
 \qquad
 F_+(k)=F_-(k)+k.
\label{eq:rf-barycenter-seam-consistency}
\end{equation}
Both identities are affine and therefore survive averaging.  In
normalized coordinates they give exactly
\eqref{eq:rf-continuum-aligned-traces}.  Moreover, for the common
aligned trace \(B\),
\begin{equation}
 \Delta\mathbb E_{\mu_j}B
 =-1+3\mathbb E_{\mu_j}\epsilon_+
 = 2-3\mathbb E_{\mu_j}\epsilon_-.
\label{eq:rf-barycenter-aligned-seam-slope}
\end{equation}
The two inserted edges of the corresponding seam zigzag encode the two
complementary singleton choices.  Theorem~\ref{thm:well-matched} puts
each inserted edge in a perfect matching.  Full support consequently
gives
\[
 0<\mathbb E_{\mu_j}\epsilon_-<1,
 \qquad
 0<\mathbb E_{\mu_j}\epsilon_+<1.
\]
Thus the common aligned seam slope in
\eqref{eq:rf-barycenter-aligned-seam-slope} lies strictly between
\(-1\) and \(2\).  In particular the barycenter respects both the
affine trace relation and the complementary-singleton consistency; no
independent panel extension is being made.

Set
\begin{equation}
             f_p^\ddagger:=\sum_{j=1}^{\infty}2^{-j}g^{(j)}.
\label{eq:rf-multiscale-strict-reference}
\end{equation}
The summands have a common Lipschitz bound and a common normalization,
so the series converges uniformly and in \(W^{1,\infty}\).  Convexity
of \(K\), together with the exact affine boundary and seam identities,
shows that \(f_p^\ddagger\in\operatorname{AHT}(\mathsf R_p)\).

Fix \(\zeta>0\), and let \(C_{r,\zeta}\) be the part of the \(r\)-th
panel at distance at least \(\zeta\) from its boundary.  Choose \(j\)
so large that every elementary triangle of the \(n_j\)-mesh meeting
\(C_{r,\zeta}\) has all three edges internal.  There are only finitely
many such edges.  Hence
\[
 \theta_{j,\zeta}
 :=3\min_e\mathbb E_{\mu_j}I_e>0,
\]
where the minimum is over those edges.  Equation
\eqref{eq:rf-barycenter-ordinary-edge-slack} says that each of the
three directional inequalities for \(g^{(j)}\) has slack at least
\(\theta_{j,\zeta}\) on this core.  Write
\[
 f_p^\ddagger=2^{-j}g^{(j)}+(1-2^{-j})h^{(j)},
 \qquad \nabla h^{(j)}\in K.
\]
The same inequalities for \(f_p^\ddagger\) therefore have slack at
least \(2^{-j}\theta_{j,\zeta}\).  The three affine inequalities
defining the fixed triangle \(K\) are quantitatively equivalent to
distance from its boundary, which proves
\eqref{eq:rf-strict-reference-margin}.  The assertion at every
differentiability point follows by integrating the three directional
inequalities across the finitely many mesh triangles and then taking
difference quotients.  Finally, every point in a panel interior lies
in some such core, so \eqref{eq:rf-strict-reference-slope} follows.
\end{proof}

\begin{proposition}[Panelwise bulk neighborhood bounds]
\label{prop:rf-bulk-neighborhood}
For every \(f\in\operatorname{AHT}(\mathsf R_p)\),
\begin{align}
 \limsup_{n\to\infty}n^{-2}\mathbb E
 \max_{\Xi\in\mathcal T_{np}(f,\eta n)}
 W_{p,\mathrm{edge}}^{\rm rf}(\Xi)
 &\le S_{p,\mathrm{edge}}(f)+o_\eta(1),
\label{eq:rf-bulk-upper-neighborhood}\\
 \liminf_{n\to\infty}n^{-2}\mathbb E
 \max_{\Xi\in\mathcal T_{np}(f,\eta n)}
 W_{p,\mathrm{edge}}^{\rm rf}(\Xi)
 &\ge S_{p,\mathrm{edge}}(f)-o_\eta(1).
\label{eq:rf-bulk-lower-neighborhood}
\end{align}
\end{proposition}

\begin{proof}
We treat the two inequalities differently.

\smallskip
\noindent\emph{Upper bound.}
We use the two-scale argument from the upper-bound proof in
\cite[\S8, Lemma~8.7]{NarLimit}, separately on the three developed
panel cores.  The mesoscopic scale is used only for quantitative
differentiation and the local random estimate; a second, macroscopic
scale is what identifies the barycentric slope with the slope of the
target height.

Put
\[
 \phi_r(x,a):=-\sigma_\edge^{\rm sym}(\rho_r(x),a).
\]
Fix \(\zeta>0\) and remove the boundary-and-cut strip
\(\mathsf B_{np}(\zeta)\).  On the resulting compact continuum cores,
write \(C_r^\zeta\) for the corresponding rescaled panel cores.
The finitely many functions \(\phi_r\) are uniformly bounded and
uniformly continuous on \(C_r^\zeta\times K\), and are concave in
\(a\).  First fix a dyadic macroscopic parameter
\(0<\varepsilon_1<\zeta/10\), independently of the height-neighborhood
radius \(\eta\), and put
\[
 L_{\rm mac}=\varepsilon_1 n.
\]
In the square coordinates of each panel, cover the core, apart from a
set of lattice area \(O(\varepsilon_1n^2)\), by triangular cells
\(T\) of sidelength \(L_{\rm mac}+O(1)\).  We call these the macroscopic
cells.

Here is the dependency order for the remaining parameters.  Fix a
large exponent \(A\) and a local error \(\delta>0\), and invoke
Lemma~\ref{lem:rf-local-transfer} with, say, \(\tau=1/2\).  Choose its
radius \(\varepsilon_{\rm loc}\) independently of the dyadic depth
\(k\), as the lemma permits.  Next choose
\(\varepsilon_{\rm qd}>0\) so small that
\(C\varepsilon_{\rm qd}\leq\varepsilon_{\rm loc}\) and that the
resulting slope-continuity error is within the assigned budget, and choose
\(\eta_{\rm bad}>0\).  Only then choose an integer \(k\) satisfying
the explicit quantitative-differentiation condition
\[
 \varepsilon_{\rm qd}^{4}\eta_{\rm bad}k\geq C^\sharp.
\]
Finally choose a \emph{fixed}
dyadic \(M\), large enough both for the resulting local-transfer
threshold \(M_0(k)\) and for
Corollary~\ref{cor:rf-panel-fefferman}.  Keep all these parameters fixed
for the moment.  At the end of the selection, choose \(\eta_0>0\) so
small that \(\omega_\zeta(C\eta_0/\varepsilon_1)\) is within the
remaining error budget, and fix \(0<\eta\leq\eta_0\) before letting
\(n\to\infty\).  This order is noncircular: the macroscopic scale and
the local-transfer radius are fixed before \(\eta\), whereas the
dependence on \(k\) is absorbed by the finite-size scale \(M\).  Inside each macroscopic cell take the
maximal union of complete \(M\)-blocks and apply the blockwise
construction in the proof of
Corollary~\ref{cor:rf-panel-fefferman}.  The additional collars along
the macroscopic-cell boundaries have total lattice area
\[
 O_{\varepsilon_1}(Mn)=o(n^2)
\]
for fixed \(\varepsilon_1\) and \(M\).  Thus, for every candidate
\(\Xi\in\mathcal T_{np}(f,\eta n)\), each retained \(M\)-block is
partitioned into GOOD cells \(U\), of sidelength
\(m_U\in[2^{-k}M,M]\), and BAD cells.  The total area of the BAD cells
is at most \(C\eta_{\rm bad}n^2\).

Let \(L_U^0\) be the affine approximant on a GOOD cell, and write
\(a_U^0=\nabla L_U^0\).  Since \(\nabla F_{\Xi,r}\in K\) almost
everywhere,
\[
 \overline a_U:=\frac1{|U|}\int_U\nabla F_{\Xi,r}\in K.
\]
The divergence theorem and
\eqref{eq:rf-panel-good-affine} give
\begin{equation}
 \left|a_U^0-\overline a_U\right|
 \leq \frac{C}{|U|}\int_{\partial U}|L_U^0-F_{\Xi,r}|\,ds
 \leq C\varepsilon_{\rm qd}.
\label{eq:rf-good-slope-average}
\end{equation}
Project \(a_U^0\) onto \(K\), then replace the projected slope by
the nearest point of a fixed deterministic
\(\varepsilon_{\rm qd}\)-net \(\mathcal N_{\rm qd}\subset K\).
Alter the additive constant so that the resulting affine function
\(L_U\) agrees with \(L_U^0\) at the center of \(U\).  After
increasing the absolute constant in
\eqref{eq:rf-panel-good-affine}, we still have
\[
 \|F_{\Xi,r}-L_U\|_{L^\infty(U)}
 \leq C\varepsilon_{\rm qd}m_U,
 \qquad a_U:=\nabla L_U\in K,
\]
and \(|a_U-\overline a_U|\leq C\varepsilon_{\rm qd}\).  The ordered
choice above ensures that
\(C\varepsilon_{\rm qd}\leq\varepsilon_{\rm loc}\).

We next remove the dependence of the GOOD decomposition on \(\Xi\).
All possible cells belong to the deterministic dyadic family from
Corollary~\ref{cor:rf-panel-fefferman}; write this family on panel
\(r\) as \(\mathcal D_{r,n}\).  For a cell \(U\), let
\(\mathcal A_U\subseteq\mathcal N_{\rm qd}\) be the slopes for which
the local height neighborhood in Lemma~\ref{lem:rf-local-transfer} is
nonempty.  Cells with \(\mathcal A_U=\varnothing\) can be discarded,
since they cannot occur as GOOD cells.  For \(M\) large,
\(|\mathcal A_U|\leq Cm_U^2\), as required there.  The full dyadic tree
in one
\(M\)-block has total
\(\sum_U m_U^2\leq C_kM^2\).  Hence the local finite-family and
variance bounds in Lemma~\ref{lem:rf-local-transfer}, summed over all
blocks and all three panels, give
\begin{equation}
 \mathbb E\!\sum_{r=1}^3\sum_{U\in\mathcal D_{r,n}}
 \max_{a\in\mathcal A_U}
 \bigl(Z_{r,U}(a,\varepsilon_{\rm loc})
       -\mathbb EZ_{r,U}(a,\varepsilon_{\rm loc})\bigr)_+
 \leq
 \frac{C_{A,k,\varepsilon_{\rm qd}}n^2}
      {(\log(2^{-k}M))^A}
 =o_M(n^2).
\label{eq:rf-good-cell-total-fluctuation}
\end{equation}
Here one may pass from the one-sided maximum in
Lemma~\ref{lem:rf-local-transfer} to the positive part by fixing one
candidate \(a_0\) and using
\(\max(0,\max_a X_a)\leq\max_aX_a+|X_{a_0}|\), together with the
variance bound for \(X_{a_0}\).

The BAD cells, the macroscopic-mesh remainder, all block collars, and
one lattice layer along the boundaries of the dyadic cells are
controlled by Lemma~\ref{lem:rf-small-support}.  The last family has
total area \(O_k(n^2/M)\), because it is contained in the union of the
boundary layers of the full deterministic dyadic trees.  Consequently,
using part~(i) of Lemma~\ref{lem:rf-local-transfer} on every GOOD cell
(and writing \(x_U\) for the center of \(U/n\)), we obtain the
following estimate, uniformly over the neighborhood
\(\mathcal T_{np}(f,\eta n)\).  If
\(W_{p,\mathrm{edge}}^{\rm core}(\Xi;\zeta)\) denotes the
nonred contribution from lozenges not meeting the deleted strip,
equivalently the type-\(2\)/type-\(3\) contribution, then
\begin{align}
 &\mathbb E\max_{\Xi\in\mathcal T_{np}(f,\eta n)}
 W_{p,\mathrm{edge}}^{\rm core}(\Xi;\zeta)
 \notag\\
 &\qquad\leq
 \max_{\Xi\in\mathcal T_{np}(f,\eta n)}
 \sum_{r=1}^3\sum_{U\in\mathcal G_{r,\Xi}}
       |U|\phi_r(x_U,a_U)
 +\delta n^2+o_M(n^2)
 +o_{\eta_{\rm bad}}(n^2)+o_{\varepsilon_1}(n^2)+o_n(n^2).
\label{eq:rf-good-cell-surface-sum}
\end{align}
The first error on the second line is the local surface-tension error;
\(\delta>0\) is arbitrary.

It remains to aggregate the mesoscopic slopes inside one macroscopic
cell.  Fix such a cell \(T\) in panel \(r\), let
\(\mathcal G_T\) be its GOOD cells, and put
\[
 D_T=T\setminus\bigcup_{U\in\mathcal G_T}U.
\]
If \(|D_T|>0\), set
\[
 b_T=\frac1{|D_T|}\int_{D_T}\nabla F_{\Xi,r}\in K;
\]
if \(|D_T|=0\), choose any \(b_T\in K\).  Define
\[
 \widehat a_T
 =\frac1{|T|}\left(
      \sum_{U\in\mathcal G_T}|U|a_U+|D_T|b_T\right)\in K.
\]
Summing \eqref{eq:rf-good-slope-average} over the GOOD cells gives
\begin{equation}
 \left|
 \widehat a_T-\frac1{|T|}\int_T\nabla F_{\Xi,r}
 \right|
 \leq C\varepsilon_{\rm qd}.
\label{eq:rf-macro-barycenter-discrete}
\end{equation}
Let \(x_T\) be the center of the rescaled cell \(T/n\).  Uniform
continuity in the first argument, followed by concavity in the second,
gives
\begin{align}
 \sum_{U\in\mathcal G_T}|U|\phi_r(x_U,a_U)
 &\leq
 \sum_{U\in\mathcal G_T}|U|\phi_r(x_T,a_U)
       +|T|\omega_\zeta(\varepsilon_1)\notag\\
 &\leq
 |T|\phi_r(x_T,\widehat a_T)
       +C_\zeta|D_T|+|T|\omega_\zeta(\varepsilon_1),
\label{eq:rf-macro-concavity}
\end{align}
where \(\omega_\zeta(t)\downarrow0\) as \(t\downarrow0\).  The term
\(C_\zeta|D_T|\) is harmless because \(\phi_r\) is bounded on the
compact core.

Write
\[
 g_{\Xi,r}(x)=n^{-1}F_{\Xi,r}(nx),
 \qquad x\in\mathsf R_{p,r},
\]
and let \(\overline T=T/n\).  Define the target average slope
\[
 a_T(f)=\frac1{|\overline T|}
        \int_{\overline T}\nabla f_r(x)\,dx\in K.
\]
Since \(\|g_{\Xi,r}-f_r\|_{L^\infty}\leq\eta\), another application
of the divergence theorem, now on the macroscopic cell, yields
\begin{equation}
 \left|
  \frac1{|\overline T|}\int_{\overline T}\nabla g_{\Xi,r}
  -a_T(f)
 \right|
 \leq C\frac{\eta}{\varepsilon_1}.
\label{eq:rf-macro-target-average}
\end{equation}
Combining \eqref{eq:rf-macro-barycenter-discrete} and
\eqref{eq:rf-macro-target-average}, and using uniform continuity in the
slope variable in \eqref{eq:rf-macro-concavity}, gives
\begin{align}
 \sum_{U\in\mathcal G_T}|U|\phi_r(x_U,a_U)
 &\leq
 |T|\phi_r(x_T,a_T(f))
\notag\\
 &\quad+|T|\omega_\zeta\!\left(
      \varepsilon_1+\varepsilon_{\rm qd}
      +\frac{\eta}{\varepsilon_1}\right)
 +C_\zeta|D_T|.
\label{eq:rf-macro-target-bound}
\end{align}
This is the macroscopic convex-aggregation step: the GOOD-cell slopes
are not compared individually with \(\nabla f_r\); only their
area-weighted barycenter is compared with the average slope of the
target on \(\overline T\).

Sum \eqref{eq:rf-macro-target-bound} over all macroscopic cells and
all three panels.  For a sequence of dyadic macroscopic meshes with
\(\varepsilon_1\downarrow0\), the Lebesgue differentiation theorem and
bounded continuity of \(\phi_r\) imply
\begin{equation}
 \sum_{r=1}^3\sum_T |\overline T|
       \phi_r(x_T,a_T(f))
 \longrightarrow
 -\sum_{r=1}^3\int_{C_r^\zeta}
       \sigma_\edge^{\rm sym}(\rho_r(x),\nabla f_r(x))\,dx.
\label{eq:rf-macro-riemann-limit}
\end{equation}
Indeed, \(a_T(f)\to\nabla f_r(x)\) at almost every point as the mesh
shrinks.

For fixed \(\zeta\), choose the parameters in the dependency order
specified above, taking \(\varepsilon_1\) small enough also for
\eqref{eq:rf-macro-riemann-limit}, and then let \(n\to\infty\).
Repeat this ordered selection as the desired error tolerance tends
to zero.
Equations
\eqref{eq:rf-good-cell-surface-sum}--\eqref{eq:rf-macro-riemann-limit}
therefore give the required upper bound on the three cores, up to
\(o_\eta(1)+o_\zeta(1)\).  Finally let \(\zeta\downarrow0\), or choose
\(\zeta=\zeta(\eta)\downarrow0\) sufficiently slowly, and restore the
deleted boundary-and-cut strip using
Lemma~\ref{lem:rf-boundary-strip}.  This proves
\eqref{eq:rf-bulk-upper-neighborhood}.

\smallskip
\noindent\emph{Lower bound.}
A strict slope cannot be
chosen independently on every coarse triangle, since that would in
general change the prescribed boundary heights.  Instead take the
global boundary-compatible height \(f_p^\ddagger\) from
Lemma~\ref{lem:rf-strict-reference-height} and put
\begin{equation}
 f^\delta=(1-\delta)f+\delta f_p^\ddagger,
 \qquad 0<\delta<1.
\label{eq:rf-global-interiorization}
\end{equation}
Both summands lie in \(\operatorname{AHT}(\mathsf R_p)\).  Thus
\(f^\delta\) has the exact prescribed outer heights and aligned seam
traces, and \(\|f^\delta-f\|_\infty=O_p(\delta)\).

Fix a small boundary-strip width \(\zeta>0\).  Convexity of \(K\) and
\eqref{eq:rf-strict-reference-margin} give
\begin{equation}
 \operatorname{dist}\bigl(\nabla f_r^\delta(x),\partial K\bigr)
 \geq\delta\kappa(p,\zeta)
\label{eq:rf-global-interiorized-margin}
\end{equation}
at almost every point of the remaining cores.  By
Lemma~\ref{lem:rf-AHT-compact}, choose a global reference tiling
\(\Xi_n^0\) whose normalized height is within \(O(n^{-1})\) of
\(f^\delta\), and leave it unchanged in
\(\mathsf B_{np}(\zeta)\).  In particular,
all singleton choices at the seams, and both unaligned seam traces, are
fixed once and for all.

We now use two separate scales.  The first is the macroscopic scale
\(n\), which determines the neighborhood
\(\mathcal T_{np}(f,\eta n)\).  The second is a large lattice scale
\(C_0\), chosen after all local accuracy parameters, kept fixed while
\(n\to\infty\), and sent to infinity only along the final diagonal.  Put
\[
                 \mu=\delta\kappa(p,\zeta).
\]
Fix a small collar fraction \(\tau_{\rm col}>0\).  Choose an affine
accuracy \(\theta>0\) so small that \(C\theta<\mu/4\), and apply
Lemma~\ref{lem:rf-fixed-target-interpolation} to \(u=f^\delta\).
For a fixed dyadic \(C_0\), write
\(\bar f_{r,n}^\delta\) for the resulting affine interpolant and
\(\mathcal Q_{r,n}^{\rm good}\) for its retained triangles.  If
\(Q\in\mathcal Q_{r,n}^{\rm good}\), set
\[
        L_Q=\bar f_{r,n}^\delta|_Q,
        \qquad a_Q=\nabla L_Q.
\]
Then
\begin{equation}
 \|n f_r^\delta(\,\cdot\,/n)-L_Q\|_{L^\infty(Q)}
       \le C\theta C_0,
 \qquad
 \operatorname{dist}(a_Q,\partial K)\ge\mu/2.
\label{eq:rf-lower-good-cell-properties}
\end{equation}
Notice that no such relative affine estimate is asserted on the discarded
triangles.

The compact set
\(\{a\in K:\operatorname{dist}(a,\partial K)\ge\mu/2\}\)
is contained in \((1-\widehat\varepsilon)K\) for some
\(\widehat\varepsilon=\widehat\varepsilon(\mu)>0\).  Inside every good
mesh triangle \(Q\), choose concentric lattice triangles
\[
             T_Q^-\subset T_Q^+\subset Q
\]
so that \(T_Q^+\) is separated from \(\partial Q\), and
\(T_Q^-\) from \(\partial T_Q^+\), by lattice distance comparable to
\(\tau_{\rm col}C_0\).  Their closed outer triangles are pairwise
disjoint, their sidelengths are comparable to \(C_0\), and the annuli
and the gaps between the \(T_Q^+\)'s occupy an
\(O(\tau_{\rm col})\)-fraction of the good-cell area.  In the notation
of Lemma~\ref{lem:rf-global-patching}, the separation parameter may be
taken to satisfy
\(\widetilde\varepsilon\ge c\tau_{\rm col}\).

The inner sidelengths satisfy
\[
 m_Q^-:=\operatorname{side}(T_Q^-)
       =(1-O(\tau_{\rm col}))C_0+O(1)\in\mathbb N.
\]
They need not be dyadic; only the coarse mesh has that restriction.
We use the integer-size form of Lemma~\ref{lem:rf-local-transfer}.

Fix a local surface-tension error \(\varepsilon_{\rm st}>0\), and invoke
Lemma~\ref{lem:rf-local-transfer} with interior parameter
\(\widehat\varepsilon\) and with its parameter
\(\delta=\varepsilon_{\rm st}\), requiring the resulting
\(\varepsilon_{\rm loc}\) to be smaller than
\(c\widehat\varepsilon\tau_{\rm col}\).  First make \(\theta\) and
\(\varepsilon_{\rm loc}\) small enough, and then make \(C_0\) large
enough, that one common number \(\varepsilon'>0\) bounds both
\[
 C\theta+O(C_0^{-1})
 \quad\hbox{and}\quad
 C\varepsilon_{\rm loc},
 \qquad
 0<2\varepsilon'<
       \widehat\varepsilon\widetilde\varepsilon,
\label{eq:rf-lower-patching-parameters}
\]
and that
\((\widehat\varepsilon\widetilde\varepsilon-2\varepsilon')
\operatorname{side}(T_Q^+)>1\).  On each \(T_Q^-\), for every
realization of the weights choose a maximizing local height state in
the \(\varepsilon_{\rm loc}\)-neighborhood of \(L_Q\).  Fixing an
ordering of the finite family of local states modulo additive constants
and taking its first maximizer makes this choice measurable.  The
local-transfer estimate places its expected score within the
prescribed local error of
\(-|T_Q^-|\sigma_\edge^{\rm sym}(\rho_r(x_Q),a_Q)\).
These states and scores use the
\hyperref[conv:rf-local-height-states]{local-state convention}.

Here \(C_0\) is also chosen so that every \(m_Q^-\geq m_0\).
For this fixed \(C_0\), the inner sizes lie in a finite integer range;
take the maximum of its thresholds \(n_0(m)\) before letting
\(n\to\infty\).  Thus the local estimate applies to every inner
triangle, including the bounded adjustments needed to make its
vertices lattice points.

The trace of this local height state on \(\partial T_Q^-\) is within
\(\varepsilon'\operatorname{side}(T_Q^+)\) of \(L_Q\).  The same is
true of the trace of \(\Xi_n^0\) on \(\partial T_Q^+\): indeed,
\(F_{\Xi_n^0}\) is \(O(1)\)-close to
\(n f_r^\delta(\,\cdot\,/n)\), and
\eqref{eq:rf-lower-good-cell-properties} supplies the remaining
\(C\theta C_0\) error.  Thus
Lemma~\ref{lem:rf-global-patching} inserts all the local height states
simultaneously, without changing a singleton choice or a trace near either
seam.

The gluing creates possible crossing lozenges along both interfaces
\(\partial T_Q^-\) and \(\partial T_Q^+\); those on the inner interface
are, in particular, omitted from the local scores.  There are
\(O(n^2/C_0^2)\) good mesh triangles, and the two interface boundaries
for each have altogether \(O(C_0)\) edges, so all of these crossing
lozenges belong to a deterministic family of
\(O(n^2/C_0)\) candidate lozenges.  Lemma~\ref{lem:rf-small-support}
therefore places their total absolute contribution in an
\(o_{C_0}(n^2)\) remainder.  Thus passing from the sum of the local
scores to the weight of the glued global tiling does not alter the
normalized lower bound after \(C_0\to\infty\).

It remains to compare the local sum with the continuum energy.  On each
good \(Q\), equation
\eqref{eq:rf-fixed-target-gradient-affinity} says that
\(a_Q\) is within \(C\theta\) of \(\nabla f_r^\delta\) on all but a
\(C\theta\)-fraction of \(Q\).  Uniform continuity and boundedness of
\(\sigma_\edge^{\rm sym}\) on the panel cores therefore give
\begin{align}
 &n^{-2}\sum_{r=1}^3
   \sum_{Q\in\mathcal Q_{r,n}^{\rm good}}
   -|T_Q^-|\,
   \sigma_\edge^{\rm sym}(\rho_r(x_Q),a_Q)
\notag\\
 &\hspace{35mm}\ge
 S_{p,\mathrm{edge}}(f^\delta)
 -O(\theta+\tau_{\rm col})
 -C\varepsilon_{\rm st}
 -\omega_\zeta(C\theta)-o_\zeta(1)-o_n(1),
\label{eq:rf-lower-two-scale-sum}
\end{align}
where \(\omega_\zeta\) is a modulus of continuity on the three compact
cores.  The discarded triangles have total area
\(O(\theta n^2)+O(C_0n)\), and the collars have area
\(O(\tau_{\rm col}n^2)\).  Their contribution, and that of the fixed
scaffold, is controlled by the small-subset estimate; the
boundary-crossing lozenges just identified contribute the additional
\(o_{C_0}(n^2)\) remainder.  The outer and cut strips are controlled by
Lemma~\ref{lem:rf-boundary-strip}.

For fixed \(C_0\), the quantitative estimate in
Lemma~\ref{lem:rf-global-patching}, together with
\eqref{eq:rf-lower-good-cell-properties}, shows that the normalized
patched height differs uniformly from \(f^\delta\) by
\(O(C_0/n)\).  Hence it belongs to
\(\mathcal T_{np}(f,\eta n)\) once \(\delta\) is first chosen so that
\(\|f^\delta-f\|_\infty<\eta/4\) and then \(n\) is large.
Let \(n\to\infty\), then let the local surface-tension error,
the discarded-area parameter, and the collar fraction tend to zero along
a diagonal on which
\(\theta+\varepsilon_{\rm loc}
 =o(\widehat\varepsilon\tau_{\rm col})\), taking
\(C_0\to\infty\) after these parameters have been fixed.  Next send
\(\zeta\downarrow0\), and finally \(\delta\downarrow0\), repeating the
inner diagonal for each fixed \((\delta,\zeta)\).  Continuity of
\(\sigma_\edge^{\rm sym}\), the integrable bound
\eqref{eq:rf-continuum-bulk-domination}, and dominated convergence give
\(S_{p,\mathrm{edge}}(f^\delta)\to S_{p,\mathrm{edge}}(f)\), proving
\eqref{eq:rf-bulk-lower-neighborhood}.
\end{proof}

The seam functional \eqref{eq:rf-seam-functional} is well defined on
\(\operatorname{AHT}(\mathsf R_p)\).  Indeed, each integrand in
\eqref{eq:rf-seam-density} is affine in the tangential derivative.  A
one-dimensional integration by parts writes it as endpoint terms plus a
Stieltjes integral of the seam trace against a bounded-variation
coefficient.  Consequently:

\begin{lemma}[Continuity of the seam functional]
\label{lem:rf-seam-continuity}
The map \(f\mapsto S_{p,\Delta}^{\rm rf}(f)\) is continuous on
\(\operatorname{AHT}(\mathsf R_p)\) in the uniform topology.
\end{lemma}

\begin{proof}
Write \(\mathcal E_{12}(x,\xi)=A_{12}(x)+D_{12}(x)\xi\), and similarly
for \(23\).  Since \(\beta,\gamma>0\), the functions \(D_{12}\) and
\(D_{23}\) are continuously differentiable on the relevant closed
intervals.  Integration by parts expresses
\(\int D(x)f'(x)\,dx\) using only the endpoint values of \(f\) and
\(\int f(x)D'(x)\,dx\).  Uniform convergence of the traces therefore
implies convergence of the seam functional.
\end{proof}

Set
\begin{equation}
 \mathcal S_p(f)=S_{p,\mathrm{edge}}(f)
               +S_{p,\Delta}^{\rm rf}(f)+S_{p,0}.
\label{eq:rf-full-functional}
\end{equation}
Apply the upper semicontinuity theorem for
\(S_{\mathrm{edge}}\) in \cite{NarLimit} on each compact panel core,
and restore the deleted strips using
\eqref{eq:rf-continuum-bulk-domination}.
Together with Lemma~\ref{lem:rf-seam-continuity} and compactness, this
shows that \(\mathcal S_p\) is upper semicontinuous and attains its
supremum.

\begin{lemma}[Uniform variance for red-free tiling families]
\label{lem:rf-uniform-family-variance}
Let \(\mathcal A_n\) be any nonempty deterministic subfamily of
\(\mathcal T_{np}\), and put
\begin{equation}
 Z_{\mathcal A_n}:=F_{p_n}^{\rm rf,(n)}+
   \max_{\Xi\in\mathcal A_n}W_{p,\mathrm{var}}^{\rm rf}(\Xi).
\label{eq:rf-restricted-family-variable}
\end{equation}
Then, uniformly in the choice of \(\mathcal A_n\),
\begin{equation}
 \operatorname{Var}(Z_{\mathcal A_n})=o(n^4),
 \qquad
 \|Z_{\mathcal A_n}-\mathbb EZ_{\mathcal A_n}\|_2=o(n^2).
\label{eq:rf-family-supremum-variance}
\end{equation}
\end{lemma}

\begin{proof}
Write \(B_\Xi=W_{p,\mathrm{edge}}^{\rm rf}(\Xi)\) and
\(R_\Xi=W_{p,\Delta}^{\rm rf}(\Xi)
          -\mathbb EW_{p,\Delta}^{\rm rf}(\Xi)\), and set
\[
 Y_{\mathcal A_n}
 =\max_{\Xi\in\mathcal A_n}
   \{B_\Xi+\mathbb EW_{p,\Delta}^{\rm rf}(\Xi)\}.
\]
For fixed \(\varepsilon>0\), also set
\[
 Y_{\mathcal A_n,\varepsilon}
 =\max_{\Xi\in\mathcal A_n}
   \{B_{\Xi,\varepsilon}
                    +\mathbb EW_{p,\Delta}^{\rm rf}(\Xi)\}.
\]
The seam expectations are deterministic offsets.  Hence
\eqref{eq:rf-bulk-tiling-family-variance} gives, for every fixed
\(\varepsilon\),
\[
 \|Y_{\mathcal A_n,\varepsilon}
       -\mathbb EY_{\mathcal A_n,\varepsilon}\|_2=o(n^2)
\]
uniformly over \(\mathcal A_n\).  On the other hand,
\eqref{eq:rf-spectral-edge-L2-removal} gives
\[
 |Y_{\mathcal A_n}-Y_{\mathcal A_n,\varepsilon}|
 \leq E_{\varepsilon,n}.
\]
Using
\(\|(U-\mathbb EU)-(V-\mathbb EV)\|_2
  \leq2\|U-V\|_2\), first let \(n\to\infty\) and then
\(\varepsilon\downarrow0\).  We obtain
\begin{equation}
 \|Y_{\mathcal A_n}-\mathbb EY_{\mathcal A_n}\|_2=o(n^2)
\label{eq:rf-bulk-offset-supremum}
\end{equation}
uniformly over \(\mathcal A_n\).  The log-concave supremum lemma used
above permits arbitrary deterministic offsets and depends on the
ambient \(O(n^2)\) dimension, not on the cardinality of
\(\mathcal A_n\).

The exact gauge identity shows that the remaining random coordinates
are precisely the explicit seam and fixed statistics; no corrected
seam incidence has been discarded.  They satisfy
\[
 \left\|\sup_{\Xi}
 |R_\Xi|\right\|_2
 =O(n^{4/3}\log^{O(1)}n)
\]
and
\[
 \|F_{p_n}^{\rm rf,(n)}-\mathbb EF_{p_n}^{\rm rf,(n)}\|_2
 =O(n\log^{O(1)}n)
\]
by Lemma~\ref{lem:rf-seam-limit} and
Theorem~\ref{thm:rf-fixed-limit}.  Since
\[
 \left|Z_{\mathcal A_n}
   -\bigl(\mathbb EF_{p_n}^{\rm rf,(n)}+Y_{\mathcal A_n}\bigr)\right|
 \le
 |F_{p_n}^{\rm rf,(n)}-\mathbb EF_{p_n}^{\rm rf,(n)}|
 +\sup_\Xi|R_\Xi|,
\]
Minkowski's inequality and
\(\| (U-\mathbb EU)-(V-\mathbb EV)\|_2
   \le2\|U-V\|_2\), together with
\eqref{eq:rf-bulk-offset-supremum}, give
\(\|Z_{\mathcal A_n}-\mathbb EZ_{\mathcal A_n}\|_2=o(n^2)\).
Squaring proves the variance assertion, uniformly in the subfamily.
\end{proof}

\begin{proposition}[Finite-panel neighborhood localization]
\label{prop:rf-total-neighborhood}
For every \(f\in\operatorname{AHT}(\mathsf R_p)\), let
\[
 W_{p,\mathrm{var}}^{\rm rf}(\Xi)
 :=W_{p,\mathrm{edge}}^{\rm rf}(\Xi)
   +W_{p,\Delta}^{\rm rf}(\Xi).
\]
Then
\begin{align}
 \limsup_{n\to\infty}n^{-2}\mathbb E
 \max_{\Xi\in\mathcal T_{np}(f,\eta n)}
 W_{p,\mathrm{var}}^{\rm rf}(\Xi)
 &\le S_{p,\mathrm{edge}}(f)+S_{p,\Delta}^{\rm rf}(f)
       +o_\eta(1),
\label{eq:rf-total-neighborhood-upper}\\
 \liminf_{n\to\infty}n^{-2}\mathbb E
 \max_{\Xi\in\mathcal T_{np}(f,\eta n)}
 W_{p,\mathrm{var}}^{\rm rf}(\Xi)
 &\ge S_{p,\mathrm{edge}}(f)+S_{p,\Delta}^{\rm rf}(f)
       -o_\eta(1).
\label{eq:rf-total-neighborhood-lower}
\end{align}
\end{proposition}

\begin{proof}
For the upper bound use
\[
 \max_{\mathcal T(f,\eta n)}(W_{\rm edge}+W_\Delta)
 \le \max_{\mathcal T(f,\eta n)}W_{\rm edge}
    +\max_{\mathcal T(f,\eta n)}W_\Delta,
\]
and apply Proposition~\ref{prop:rf-bulk-neighborhood} and the first
inequality of Lemma~\ref{lem:rf-seam-limit}.  For the lower bound use
\[
 \max_{\mathcal T(f,\eta n)}(W_{\rm edge}+W_\Delta)
 \ge \max_{\mathcal T(f,\eta n)}W_{\rm edge}
    +\min_{\mathcal T(f,\eta n)}W_\Delta,
\]
and apply the corresponding lower estimates.
\end{proof}

\begin{theorem}[Finite-panel red-free localization]
\label{thm:rf-finite-panel-localization}
Let \(p\) be a strictly interior rational continuum target, let \(q\) be a common denominator, and take exact dilations \(p_n=np\) along multiples \(n\) of \(q\).  For the three-source GUE input above,
\begin{equation}
 \frac1{n^2}\mathbb E\widehat O_n(p_n)
 \longrightarrow
 \Phi_{\ell_1,\ell_2,\ell_3}(p):=
 \sup_{f\in\operatorname{AHT}(\mathsf R_p)}\mathcal S_p(f).
\label{eq:rf-mean-limit}
\end{equation}
Moreover,
\begin{equation}
 n^{-4}\operatorname{Var}(\widehat O_n(p_n))\longrightarrow0.
\label{eq:rf-main-variance}
\end{equation}
\end{theorem}

\begin{proof}
The exact formula of Theorem~\ref{thm:rf-exact-tiling} gives
\begin{equation}
 \widehat O_n(p_n)=F_{p_n}^{\rm rf,(n)}
 +\max_{\Xi\in\mathcal T_{np}}W_{p,\mathrm{var}}^{\rm rf}(\Xi).
\label{eq:rf-finite-panel-decomposition}
\end{equation}
The normalized fixed term converges in \(L^1\) to \(S_{p,0}\) by
Theorem~\ref{thm:rf-fixed-limit}.

For the lower bound, fix
\(f\in\operatorname{AHT}(\mathsf R_p)\).  Restrict the maximum in
\eqref{eq:rf-finite-panel-decomposition} to
\(\mathcal T_{np}(f,\eta n)\), apply
Proposition~\ref{prop:rf-total-neighborhood}, and then let
\(\eta\downarrow0\).  This gives
\[
 \liminf_{n\to\infty}n^{-2}\mathbb E\widehat O_n(p_n)
 \ge \mathcal S_p(f).
\]
Taking the supremum over \(f\) proves the required lower bound.

For the upper bound, fix \(\varepsilon>0\) and abbreviate
\(\mathcal A=\operatorname{AHT}(\mathsf R_p)\).  For \(f\in\mathcal A\)
and \(r>0\), put
\[
 Z_n(f,r)=F_{p_n}^{\rm rf,(n)}+
 \max_{\Xi\in\mathcal T_{np}(f,rn)}
      W_{p,\mathrm{var}}^{\rm rf}(\Xi).
\]
The pointwise upper estimate in
Proposition~\ref{prop:rf-total-neighborhood}, together with the fixed-term
limit, allows us, for every \(f\in\mathcal A\), to choose \(r_f>0\) so
that
\begin{equation}
 \limsup_{n\to\infty}n^{-2}\mathbb E Z_n(f,2r_f)
 \leq \mathcal S_p(f)+\varepsilon
 \leq \sup_{g\in\mathcal A}\mathcal S_p(g)+\varepsilon.
\label{eq:rf-pointwise-upper-radius}
\end{equation}

The balls \(B_\infty(f,r_f)\), \(f\in\mathcal A\), cover the compact
space \(\mathcal A\).  Extract a finite subcover
\[
 B_\infty(f^{(1)},r_1),\ldots,B_\infty(f^{(J)},r_J),
 \qquad r_j=r_{f^{(j)}},
\]
and put \(r_*=\min_{j\leq J}r_j>0\).  By
Lemma~\ref{lem:rf-AHT-compact}, for all sufficiently large \(n\), every
normalized discrete height is within \(r_*\) of \(\mathcal A\), and
each of the finitely many families
\(\mathcal T_{np}(f^{(j)},2r_jn)\) is nonempty.  If a discrete height is
within \(r_*\) of \(g\in\mathcal A\), and
\(g\in B_\infty(f^{(j)},r_j)\), then it lies in
\(B_\infty(f^{(j)},2r_j)\).  Consequently these doubled neighborhoods
cover all tilings.  With
\[
 Z_{j,n}:=Z_n(f^{(j)},2r_j),
\]
we therefore have, for all sufficiently large \(n\),
\(\widehat O_n(p_n)=\max_{j\leq J}Z_{j,n}\).

This is a fixed finite family, independent of \(n\).  Hence
Lemma~\ref{lem:rf-uniform-family-variance} gives
\(\|Z_{j,n}-\mathbb EZ_{j,n}\|_2=o(n^2)\) for every \(j\), and
\[
 \mathbb E\max_{j\le J}Z_{j,n}
 \le \max_{j\le J}\mathbb EZ_{j,n}
    +\sum_{j=1}^J\|Z_{j,n}-\mathbb EZ_{j,n}\|_2
 =\max_{j\le J}\mathbb EZ_{j,n}+o(n^2).
\]
Because \(J\) is fixed, \eqref{eq:rf-pointwise-upper-radius} now gives
\[
 \limsup_{n\to\infty}n^{-2}\mathbb E\widehat O_n(p_n)
 \leq \sup_{f\in\mathcal A}\mathcal S_p(f)+\varepsilon.
\]
Finally let \(\varepsilon\downarrow0\).  This proves the upper bound
and therefore \eqref{eq:rf-mean-limit}.

Finally apply Lemma~\ref{lem:rf-uniform-family-variance} to
the full family \(\mathcal A_n=\mathcal T_{np}\).  In view of
\eqref{eq:rf-finite-panel-decomposition}, this gives
\(\operatorname{Var}(\widehat O_n(p_n))=o(n^4)\), which is
\eqref{eq:rf-main-variance}.
\end{proof}

\begin{corollary}[Rational-point law of large numbers]
\label{thm:rf-rational-limit}
Under the hypotheses of Theorem~\ref{thm:rf-finite-panel-localization},
\begin{equation}
 \boxed{
 \frac{\widehat O_n(p_n)}{n^2}
 \xrightarrow{\mathbb P}
 \Phi_{\ell_1,\ell_2,\ell_3}(p).}
\label{eq:rf-main-limit}
\end{equation}
The normalized expectation has the same limit.
\end{corollary}

\begin{proof}
Combine \eqref{eq:rf-mean-limit} and \eqref{eq:rf-main-variance}.
\end{proof}

\section{Extension to real tetrahedral points}
\label{sec:rf-real-extension}

We now remove the rationality restriction.  The required spatial
regularity comes from the hive inequalities on constant-coordinate
slices of the tetrahedral labeling.

\begin{lemma}[Constant-coordinate slices and a lattice Lipschitz bound]
\label{lem:rf-target-lipschitz}
On an event of probability tending to one, there is a constant
\(C=C(\ell_1,\ell_2,\ell_3)\) such that for every pair of adjacent lattice
vertices \(x,y\) in the tetrahedron \(1234\),
\begin{equation}
 \left|\widehat O_n(x)-\widehat O_n(y)\right|\le Cn.
\label{eq:rf-target-lattice-lipschitz}
\end{equation}
The same bound holds simultaneously on all boundary faces.
\end{lemma}

\begin{proof}
First note that the tetrahedron considered here has vertex set
\(\{u_1,u_2,u_3,u_4\}\).  Since \(u_0\) is not one of its vertices, the
one-coordinate large-gap gauge \eqref{eq:rf-large-gap-gauge} vanishes
identically on this tetrahedron: it is needed only to embed the initial
panels containing \(u_0\) into the category of hives.  In particular,
neither the face laws below nor the constant in the desired estimate
depends on \(K_{\rm gap}\).

The boundary-face dictionary suppressed in the associator notation is
a dictionary of \emph{marginal laws}, not a pathwise identification.
Put
\(\widehat X_{12}=\widehat X_1+\widehat X_2\),
\(\widehat X_{23}=\widehat X_2+\widehat X_3\), and
\(\widehat X_{123}=\widehat X_1+\widehat X_2+\widehat X_3\).
Taking the one-hive marginals of
Theorem~\ref{thm:rf-top-bottom-double-hive}, equivalently
using the GUE-hive invariance under the associator
\cite{NSTGUE}, gives
\begin{equation}
\begin{array}{c|c|c}
\text{face}&\text{marginal GUE-hive law}&
             \text{law of the three boundary-increment lists}\\ \hline
123&(\widehat X_1,\widehat X_2;\widehat X_{12})&
 \lambda(\widehat X_1),\lambda(\widehat X_2),
 \lambda(\widehat X_{12})\\
134&(\widehat X_{12},\widehat X_3;\widehat X_{123})&
 \lambda(\widehat X_{12}),\lambda(\widehat X_3),
 \lambda(\widehat X_{123})\\
234&(\widehat X_2,\widehat X_3;\widehat X_{23})&
 \lambda(\widehat X_2),\lambda(\widehat X_3),
 \lambda(\widehat X_{23})\\
124&(\widehat X_1,\widehat X_{23};\widehat X_{123})&
 \lambda(\widehat X_1),\lambda(\widehat X_{23}),
 \lambda(\widehat X_{123}).
\end{array}
\label{eq:rf-boundary-face-dictionary}
\end{equation}
Here each row means equality in distribution for that face alone.
The table does not identify all four face triples simultaneously with
spectra formed from one matrix realization; the joint assertion used
above is only the four \emph{exterior}-side law
\eqref{eq:rf-four-boundary-joint-law}.  There is no unconditional
independence between the faces; the conditional independence supplied
by Theorem~\ref{thm:rf-top-bottom-double-hive} holds only within the top
or bottom pair after its common diagonal is fixed.

For a face \(F\in\{123,124,134,234\}\), let \(D_F^{(n)}\) be the
largest absolute increment on its three outer sides.  By the marginal
identity \eqref{eq:rf-boundary-face-dictionary} and the Gaussian
operator-norm tail, one may choose
\(C_0=C_0(\ell_1,\ell_2,\ell_3)\) so that
\[
 \mathbb P\!\left\{D_F^{(n)}>C_0n\right\}\le Ce^{-cn}
 \qquad\text{for every }F.
\]
Consequently the event
\[
 \mathcal B_n:=
 \left\{\max_{F\in\{123,124,134,234\}}D_F^{(n)}
                         \le C_0n\right\}
\]
has probability tending to one, by a union bound over the four
faces.  No joint distributional assertion is used here.

Let \([x,y]\) be a unit lattice edge, in direction \(e_i-e_j\), and
choose one of the two remaining barycentric coordinates, say \(q_k\),
which is constant on this edge.  The excavation construction labels all
lattice points of the tetrahedron.  In the proof of the hive associator,
Knutson--Tao--Woodward prove the stronger local assertion that
\emph{every} unit rhombus in this tetrahedral labeling satisfies its hive
inequality; see \cite[\S5]{KTWOR}.  It follows immediately that the
constant-coordinate triangular slice
\[
                         q_k=x_k=y_k
\]
is a hive and contains \([x,y]\).  Flip-path independence identifies
its labels at \(x,y\) with the values
\(\widehat O_n(x),\widehat O_n(y)\) obtained in the chosen global sweep.

On the large-gap event \(\mathcal E_n\), the three initial arrays are
hives, so the associator supplies hives on all four boundary faces and
on every constant-coordinate slice.  In a hive, the unit-edge
increments in each parallel direction are monotone between the two
corresponding outer sides.  Hence, on
\(\mathcal E_n\cap\mathcal B_n\), every unit-edge increment anywhere
on any boundary face is bounded in absolute value by \(C_0n\).
The boundary of the constant-coordinate slice lies in three such
faces, so all its boundary increments obey the same bound.

Applying the same monotonicity once more inside the slice, an increment
in its interior lies
between the corresponding increments on its two boundary sides.  In
particular,
\[
 |\widehat O_n(x)-\widehat O_n(y)|\le C_0n.
\]
Both \(\mathcal E_n\) and \(\mathcal B_n\) have probability tending to
one, and neither event depends on the choice of edge.  The estimate is
therefore simultaneous over all lattice edges and all boundary faces.
\end{proof}

\begin{lemma}[Stability of the aligned directed envelope]
\label{lem:rf-aligned-envelope-stability}
Let \(p'\) be a strictly interior lattice target and use the aligned
height, boundary datum, residue, and directed distance
\[
 h([i,j]_r)=F_r([i,j]_r)-(r-1)j,\qquad
 f_{p'},\quad\kappa_{p'},\quad d_{p'}
\]
from \eqref{eq:aligned-global-height} and
Lemma~\ref{lem:directed-extremal-extension}.  Suppose that
\(u:\mathcal V_{p'}\to\mathbb R\) satisfies, for some \(A\geq0\),
\begin{align}
 |u(b)-f_{p'}(b)|&\leq A,
       &&b\in B_{p'},\label{eq:rf-aligned-approx-boundary}\\
 u(y)-u(x)&\leq d_{p'}(x,y)+A,
       &&d_{p'}(x,y)<\infty.
\label{eq:rf-aligned-approx-directed}
\end{align}
Then there is a legal tiling of \(\mathcal A_{p'}\) whose aligned
integer height \(H\) has the exact boundary datum \(f_{p'}\), residue
\(\kappa_{p'}\) modulo three, and
\begin{equation}
                         \|H-u\|_\infty\leq3A+3.
\label{eq:rf-aligned-envelope-stability}
\end{equation}
\end{lemma}

\begin{proof}
Omit terms with infinite directed distance and define
\[
 u_0(z):=\min_{x\in\mathcal V_{p'}}
              \{u(x)+d_{p'}(x,z)\}.
\]
The directed triangle inequality makes \(u_0\) directed Lipschitz,
while \eqref{eq:rf-aligned-approx-directed} gives
\[
                         u-A\leq u_0\leq u.
\]
Let \(L_{p'}\) and \(U_{p'}\) be the two exact boundary envelopes in
\eqref{eq:directed-envelopes}.  Their extremal property and
\eqref{eq:rf-aligned-approx-boundary} imply
\[
                     L_{p'}-2A\leq u_0\leq U_{p'}+A.
\]
Pointwise minimum and maximum preserve the directed-Lipschitz
inequalities.  Hence
\[
 h:=\min\{U_{p'},\max\{L_{p'},u_0\}\}
\]
is directed Lipschitz, equals \(f_{p'}\) on \(B_{p'}\), and satisfies
\(\|h-u\|_\infty\leq3A\).

Let \(\widehat h(z)\) be the least integer not smaller than \(h(z)\)
and congruent to \(\kappa_{p'}(z)\) modulo three, and put
\[
 H(z):=\min_{x\in\mathcal V_{p'}}
       \{\widehat h(x)+d_{p'}(x,z)\}.
\]
Since \(h\) is directed Lipschitz,
\[
                         h(z)\leq H(z)<h(z)+3.
\]
All terms in the minimum have residue \(\kappa_{p'}(z)\).  At a
boundary vertex, the half-open interval
\([f_{p'}(z),f_{p'}(z)+3)\) contains only the prescribed integer of
that residue, so \(H=f_{p'}\) on \(B_{p'}\).  Thus
\eqref{eq:rf-aligned-envelope-stability} holds, and
Lemma~\ref{lem:directed-extremal-extension} identifies \(H\) with the
aligned height of a unique legal tiling.  Since this is one
single-valued aligned height, both seam relations, including the
cut-to-cut constraints on the middle panel, are automatic.
\end{proof}

\begin{lemma}[Stability under bounded motion of the panel walls]
\label{lem:rf-moving-domain-stability}
Let \(p=(\alpha,\beta,\gamma,\delta)\in\mathsf T_{1234}\) be
strictly interior and rational, and write
\(\bar p=(0,\alpha,\beta,\gamma,\delta)\) for its lift to the face
\(1234\).  Fix a common denominator \(q\), a residue
\(\varrho\in\{0,\ldots,q-1\}\), and an integer vector
\(e^{(\varrho)}=(0,e_a,e_b,e_c,e_d)\) of coordinate sum
\(\varrho\).  Put
\[
 N=qm,\qquad n=N+\varrho,\qquad
 p_n=N\bar p+e^{(\varrho)},
\]
and assume that \(p_n\) is strictly interior for all sufficiently
large \(m\).  Let \(\mathsf D_{n,j}\) be the closed atomic region in
the \(j\)-th size-\(n\) developed panel determined by \(p_n\).  In the
same size-\(n\) panel, let \(\mathsf D^0_{n,j}\) be the auxiliary
polygon obtained by using the exact wall positions
\(N\alpha,N\beta,N\gamma,N\delta\) in the closed-wall description
preceding \eqref{eq:rf-exact-outer-boundary-trace}.  Then, for a
constant depending only on \(p,q,\varrho,e^{(\varrho)}\),
\begin{equation}
 \sum_{j=1}^3
 \#_\triangle\bigl(\mathsf D_{n,j}
             \mathbin\triangle\mathsf D^0_{n,j}\bigr)
 \leq Cn,
 \qquad
 \operatorname{dist}_{\rm H}
       (n^{-1}\mathsf D_{n,j},\mathsf R_{p,j})\leq \frac Cn,
\label{eq:rf-bounded-wall-geometry}
\end{equation}
where \(\#_\triangle\) counts elementary triangles.  The symmetric
difference is contained in a \(C\)-layer neighborhood of the finitely
many closed sides and labelled cuts.

There are compatible piecewise-affine homeomorphisms
\[
 \psi_{n,j}:\mathsf R_{p,j}\longrightarrow
             n^{-1}\mathsf D_{n,j}
\]
which send corresponding polygonal vertices and the two labelled
copies of each seam to one another and satisfy, on every affine piece,
\[
 \|\psi_{n,j}-{\rm id}\|_\infty
 +\|D\psi_{n,j}-I\|_\infty\leq \frac Cn.
\]
Write \(\psi_n=\bigsqcup_{j=1}^3\psi_{n,j}\) for their sheetwise
union on the cut-open complex.
Write \(\partial_{\rm cl}\mathsf R_{p,j}\) for the union of the
closed outer sides of the \(j\)-th cut-open continuum sheet.
If \(F^\partial_{p_n,j}\) is the forced closed-boundary height with the
normalization and aligned sheet constants used in
\eqref{eq:rf-aligned-discrete-traces}, then
\begin{equation}
 \left\|n^{-1}F^\partial_{p_n,j}
       (n\psi_{n,j}(\,\cdot\,))-\gamma_{p,j}\right\|_
       {L^\infty(\partial_{\rm cl}\mathsf R_{p,j})}
 \leq \frac Cn.
\label{eq:rf-perturbed-boundary-trace}
\end{equation}
Here and below discrete heights are extended affinely between lattice
vertices.

For \(\Xi\in\mathcal T_{p_n}\), put
\[
 \widetilde F_{\Xi,n,j}(x)
 :=n^{-1}F_{\Xi,j}(n\psi_{n,j}(x)),
 \qquad x\in\mathsf R_{p,j}.
\]
For \(f\in\operatorname{AHT}(\mathsf R_p)\) and \(\eta>0\), set
\[
 \mathcal T_{p_n}^{\psi}(f,\eta n)
 :=\{\Xi\in\mathcal T_{p_n}:
       \|\widetilde F_{\Xi,n}-f\|_\infty<\eta\}.
\]
All perturbed height neighborhoods below are understood in this sense.
Then
\begin{equation}
 \sup_{\Xi\in\mathcal T_{p_n}}
 \operatorname{dist}_{L^\infty(\mathsf R_p)}
 \bigl(\widetilde F_{\Xi,n},
       \operatorname{AHT}(\mathsf R_p)\bigr)\longrightarrow0,
\label{eq:rf-moving-height-upper}
\end{equation}
and, for every \(f\in\operatorname{AHT}(\mathsf R_p)\),
\begin{equation}
 \inf_{\Xi\in\mathcal T_{p_n}}
 \|\widetilde F_{\Xi,n}-f\|_{L^\infty(\mathsf R_p)}
 \longrightarrow0.
\label{eq:rf-moving-height-recovery}
\end{equation}
Moreover, fix \(0<\tau<1\) and \(0<\zeta<\zeta_0(p)\), with
\(\zeta_0(p)\) small enough that the \(2\zeta\)-cores below are
nonempty.  Put
\(f^\tau=(1-\tau)f+\tau f_p^\ddagger\), and let
\(\mathsf R_{p,j}^{\,2\zeta}\) be the complement of the
\(2\zeta\)-neighborhood of the closed sides and cuts.  The strictified
recovery may be chosen so that
\begin{equation}
 \sup_j\sup_{x\in\mathsf R_{p,j}^{\,2\zeta}}
 \left|F_{\Xi_n,j}(n\psi_{n,j}(x))-nf_j^\tau(x)\right|
 \leq C_{\tau,\zeta}.
\label{eq:rf-moving-core-recovery}
\end{equation}
The constant \(C_{\tau,\zeta}\) is uniform in
\(f\in\operatorname{AHT}(\mathsf R_p)\).
\end{lemma}

\begin{proof}
Since \(n=N+\varrho\),
\[
 a_n-n\alpha=e_a-\varrho\alpha,
 \quad b_n-n\beta=e_b-\varrho\beta,
 \quad c_n-n\gamma=e_c-\varrho\gamma,
 \quad d_n-n\delta=e_d-\varrho\delta.
\]
Thus every wall in the closed-wall table moves by a bounded number of
lattice steps.  Each wall segment has length \(O(n)\), proving the
count in \eqref{eq:rf-bounded-wall-geometry}.  Strict interiority keeps
the combinatorial type of all three polygons fixed for large \(n\).
Triangulate each limiting polygon using its polygonal vertices, use
the same compatible triangulation on the two copies of a cut, and map
corresponding vertices.  This gives the maps \(\psi_{n,j}\) and their
stated estimates.

Use the forced boundary-increment rule defining \(f_p\) immediately
before Lemma~\ref{lem:boundary-datum-compatible}, and undo the affine
alignment \eqref{eq:aligned-global-height}.  By the closed-wall
description preceding \eqref{eq:rf-exact-outer-boundary-trace}, each
closed outer side is a straight triangular-lattice segment.
Corresponding side lengths and corner coordinates differ by
\(O(1)\).  Integrating those forced increments
from the normalization vertex therefore changes the cumulative
heights at corresponding corners by \(O(1)\), and the same bound
holds along each side under the affine vertex correspondence supplied
by \(\psi_n\).  The aligned sheet constants obey the same bound.
This proves \eqref{eq:rf-perturbed-boundary-trace}.  The
retained seam intervals
are
\[
 [b_n/n,1]=[\beta+O(n^{-1}),1],
 \qquad [c_n/n,1]=[\gamma+O(n^{-1}),1].
\]
The exact discrete relation is still \(F_+(k)=F_-(k)+k\); after
pullback it differs from the relation in
\eqref{eq:rf-continuum-aligned-traces} by \(O(n^{-1})\).

The pulled-back heights have a common Lipschitz constant.  Given any
sequence of them, Arzel\`a--Ascoli gives a uniformly convergent
subsequence.  Since \(D\psi_{n,j}\to I\), the directional slope
inequalities pass to the limit and give \(\nabla f_j\in K\) almost
everywhere.  Equation \eqref{eq:rf-perturbed-boundary-trace} gives the
closed outer trace, and the preceding exact seam identity gives the
two aligned limiting traces.  Every subsequential limit therefore
belongs to \(\operatorname{AHT}(\mathsf R_p)\).  The usual
contradiction argument proves \eqref{eq:rf-moving-height-upper}.

For recovery, fix \(f\in\operatorname{AHT}(\mathsf R_p)\) and use
the aligned height already defined in
\eqref{eq:aligned-global-height}.  For a lattice vertex represented
as \(z=[i,j]_r\) on the \(r\)-th perturbed sheet, set
\[
 u_n(z):=
 n f_r\!\left(\psi_{n,r}^{-1}(z/n)\right)-(r-1)j.
\]
At a seam the two panel representatives differ by \(O(1)\), by
\eqref{eq:rf-continuum-aligned-traces} and
\(\|\psi_n-{\rm id}\|_\infty=O(n^{-1})\); choose either value there.
Thus \(u_n\) is a single-valued real function on
\(\mathcal V_{p_n}\), up to a uniform \(O(1)\) choice.

We now verify the approximate directed-Lipschitz inequality required
by Lemma~\ref{lem:rf-aligned-envelope-stability}.  This can be done
edgewise using the maps $\psi_{n,r}$, without comparing geodesics in
two different polygonal domains.

For the $r$-th panel, let $d_1,d_2,d_3$ denote its three positive
unit directions and put
\[
        w_{r,1}=1,\qquad w_{r,2}=r,\qquad w_{r,3}=2-r.
\]
Thus $w_{r,k}$ is precisely the aligned weight $\omega(e)$ of a
positive edge occurrence of type $k$ in the $r$-th sheet.  Define the
polyhedral directed cost
\[
 \ell_r(v):=
 \min\left\{
       \sum_{k=1}^3 a_k w_{r,k}:
       a_k\geq0,\quad
       v=\sum_{k=1}^3a_kd_k
      \right\}.
\]
Although $\ell_r$ need not be nonnegative, it is finite everywhere.
Indeed,
\[
             d_1+d_2+d_3=0,
             \qquad
             w_{r,1}+w_{r,2}+w_{r,3}=3>0.
\]
Consequently $\ell_r$ is positively homogeneous, subadditive and
globally Lipschitz, and
\[
                         \ell_r(d_k)=w_{r,k}.
\]

Let $J_r$ denote the second square coordinate on the $r$-th reference
panel and define the aligned continuum height
\[
                         g_r=f_r-(r-1)J_r.
\]
The condition $\nabla f_r\in K$ is equivalent to
\[
                         \partial_{d_k}g_r\leq w_{r,k},
                         \qquad k=1,2,3.
\]
It follows, first for polygonal curves and then by approximation, that
every rectifiable curve $\eta$ contained in the $r$-th sheet satisfies
\begin{equation}
 g_r(\eta(1))-g_r(\eta(0))
 \leq
 \int_0^1\ell_r(\dot\eta(t))\,dt.
\label{eq:rf-continuum-directed-cost}
\end{equation}

Use the compatible choices of $\psi_{n,r}$ on the two labelled copies
of every cut and put
\[
 \overline u_n(z):=
 n g_r\!\left(\psi_{n,r}^{-1}(z/n)\right)
\]
when $z$ is represented in the $r$-th sheet.  The aligned continuum
trace identities and the compatibility of the two seam
parametrizations make $\overline u_n$ single-valued at every seam
vertex.  Moreover,
\begin{equation}
                         \|u_n-\overline u_n\|_\infty\leq C,
\label{eq:rf-u-bar-u-comparison}
\end{equation}
because
\[
 \|\psi_n^{-1}-\mathrm{id}\|_\infty=O(n^{-1})
\]
and hence the difference between the lattice coordinate $j$ and
$nJ_r(\psi_{n,r}^{-1}(z/n))$ is uniformly $O(1)$.

Let $e:z\to z'$ be a positive edge occurrence of type $k$ in the
$r$-th sheet.  Pull the normalized straight edge from $z/n$ to
$z'/n$ back by $\psi_{n,r}^{-1}$.  On each affine piece its velocity
is
\[
                         n^{-1}d_k+O(n^{-2}).
\]
Applying \eqref{eq:rf-continuum-directed-cost} and the Lipschitz
property of $\ell_r$ gives, uniformly over all positive edge
occurrences,
\begin{equation}
 \overline u_n(z')-\overline u_n(z)
       \leq \omega(e)+\frac{C}{n}.
\label{eq:rf-edgewise-directed-error}
\end{equation}

It remains to note that a minimizing positive path has $O(n)$ edge
occurrences.  Strict interiority and the fixed polygonal
combinatorial type imply that the undirected one-skeleton of
$\mathcal A_{p_n}$ has diameter $O(n)$.  Replacing an edge traversed
oppositely to its positive orientation by the other two positive
sides of an incident elementary triangle gives a positive competitor
with $O(n)$ edge occurrences.  Since the aligned edge weights belong
to the fixed set $\{-1,0,1,2,3\}$, this implies
\[
                         d_{p_n}(x,y)\leq Cn.
\]

Let $\gamma$ be a minimizing positive path from $x$ to $y$.  Write
$N_{k,r}$ for the number of its type-$k$ edge occurrences in the
$r$-th sheet, and put
\[
             N_1=\sum_rN_{1,r},
             \qquad
             N_3=\sum_rN_{3,r}.
\]
The path identity in
Lemma~\ref{lem:positive-cycle-gap} gives
\begin{equation}
 d_{p_n}(x,y)
 =
 \widetilde\kappa_{p_n}(y)-\widetilde\kappa_{p_n}(x)+3N_3.
\label{eq:rf-minimizing-path-count}
\end{equation}
Since
\[
 \left|
 \widetilde\kappa_{p_n}(y)-\widetilde\kappa_{p_n}(x)
 \right|\leq Cn,
\]
the preceding upper bound for $d_{p_n}(x,y)$ implies $N_3\leq Cn$.
The potential part of the same path identity is
\[
 \widetilde\kappa_{p_n}(y)-\widetilde\kappa_{p_n}(x)
 =
 N_1+\sum_r rN_{2,r}-\sum_r(r+1)N_{3,r}.
\]
It follows that
\[
                         N_1+\sum_rN_{2,r}\leq Cn.
\]
Consequently $\gamma$ has at most $Cn$ edge occurrences.

Summing \eqref{eq:rf-edgewise-directed-error} along this minimizing
path now gives
\[
 \overline u_n(y)-\overline u_n(x)
 \leq d_{p_n}(x,y)+C.
\]
Equation \eqref{eq:rf-u-bar-u-comparison} therefore yields
\[
                         u_n(y)-u_n(x)
                         \leq d_{p_n}(x,y)+A
\]
for one constant $A=O(1)$, uniformly over all reachable ordered pairs.
Panel changes at a seam carry no edge cost, while the two developed
seam-edge occurrences remain distinguished throughout the argument.
Thus the estimate also covers paths between the two labelled cuts of
the middle panel.

Equation
\eqref{eq:rf-perturbed-boundary-trace}, after subtracting the same
aligned affine term, gives
\[
                   |u_n(b)-f_{p_n}(b)|\leq A,
                   \qquad b\in B_{p_n}.
\]

Lemma~\ref{lem:rf-aligned-envelope-stability} now gives one legal
tiling \(\Xi_n\in\mathcal T_{p_n}\), with exact outer values and both
exact aligned seam relations, whose aligned height is \(O(1)\) from
\(u_n\).  Undoing the aligned affine term therefore gives
\[
 \|\widetilde F_{\Xi_n,n}-f\|_\infty\leq \frac Cn,
\]
which proves \eqref{eq:rf-moving-height-recovery}.  Applying the same
construction to \(f^\tau=(1-\tau)f+\tau f_p^\ddagger\) proves
\eqref{eq:rf-moving-core-recovery} (indeed, it gives the bound on the
whole moving domain).
\end{proof}

\begin{lemma}[Bounded perturbations of exact rational dilations]
\label{lem:rf-bounded-perturbation}
Let \(p=(\alpha,\beta,\gamma,\delta)\in\mathsf T_{1234}\) be
strictly interior and rational, write
\(\bar p=(0,\alpha,\beta,\gamma,\delta)\), and let \(q\) be a common
denominator.  Fix a residue class \(\varrho\pmod q\) and an integer
vector \(e^{(\varrho)}=(0,e_a,e_b,e_c,e_d)\) of coordinate sum
\(\varrho\).  For
\[
 n=qm+\varrho,\qquad p_n=m(q\bar p)+e^{(\varrho)},
\]
assume that \(p_n\) is strictly interior for all sufficiently large
\(m\).  Then
\[
 n^{-2}\mathbb E\widehat O_n(p_n)
 \longrightarrow\Phi_{\ell_1,\ell_2,\ell_3}(p),
 \qquad
 \operatorname{Var}\bigl(\widehat O_n(p_n)\bigr)=o(n^4).
\]
Consequently
\[
 \frac{\widehat O_n(p_n)}{n^2}
 \xrightarrow{\mathbb P}\Phi_{\ell_1,\ell_2,\ell_3}(p).
\]
\end{lemma}

\begin{proof}
Put \(N=qm\), so \(n=N+\varrho\).  All comparisons below are made inside
the size-\(n\) developed panels and use the size-\(n\) random arrays
\(H_j^{(n)}\).  The auxiliary polygons \(\mathsf D^0_{n,j}\) of
Lemma~\ref{lem:rf-moving-domain-stability} are used only to identify
the limiting geometry; no coupling of GUE matrices of two different
sizes is required.

By \eqref{eq:rf-bounded-wall-geometry}, the rescaled atomic sheets and
cuts converge to \(\mathsf R_p\), and their symmetric differences are
contained in a fixed-width collar with \(O(n)\) atomic cells.  By
\eqref{eq:rf-perturbed-boundary-trace} and
\eqref{eq:rf-moving-height-upper}--
\eqref{eq:rf-moving-height-recovery}, the normalized discrete height
spaces converge to the same compact space
\(\operatorname{AHT}(\mathsf R_p)\), with the same closed outer traces
and the same two aligned seam relations.

We first check the bulk localization estimates.  Fix a macroscopic
strip width \(\zeta>0\).  For all sufficiently large \(n\), the three
\(\zeta\)-cores lie in the perturbed sheets, and their panel
coordinates differ from those in the exact-wall polygons by
\(O(n^{-1})\).  The local-transfer estimates, the Fefferman
decomposition, and the small-support estimate are uniform on these
fixed compact cores.  Uniform continuity of the local parameters
\(\rho_j\) shows that the resulting Riemann sums converge to
\(S_{p,\mathrm{edge}}\).  The candidate ordinary lozenges meeting the
fixed-width symmetric difference are covered by
Lemma~\ref{lem:rf-fixed-collar-L2}.  Their contribution is therefore
\(o(n^2)\) in \(L^2\), and hence also in expectation.

For the lower bound, first replace a target height \(f\) by
\(f^\tau=(1-\tau)f+\tau f_p^\ddagger\), as in
\eqref{eq:rf-global-interiorization}.  The recovery construction in
Lemma~\ref{lem:rf-moving-domain-stability} supplies a legal reference
tiling satisfying \eqref{eq:rf-moving-core-recovery}.  Choose
\(2\varepsilon'<\widehat\varepsilon\widetilde\varepsilon\)
with fixed spare margin, and let \(C_{\rm wall}\) bound the
wall-motion, endpoint, and residue corrections.  After the local
accuracy parameters have been fixed, choose one dyadic scaffold scale
\(C_0\) so large that all inner integer sidelengths exceed the
threshold \(m_0\) in Lemma~\ref{lem:rf-local-transfer} and that
\[
 \frac{C_{\tau,\zeta}+C_{\rm wall}}{C_0}<\varepsilon',
 \qquad
 (\widehat\varepsilon\widetilde\varepsilon-2\varepsilon')
 C_0>1.
\]
Keep \(C_0\) fixed while \(n\to\infty\), exactly as in the proof of
Proposition~\ref{prop:rf-bulk-neighborhood}.
Lemma~\ref{lem:rf-global-patching} therefore inserts the local
optimizers without altering either cut trace or any singleton choice.
This proves the two bulk-neighborhood bounds on the perturbed sheets,
with the same functional.  Only afterward send \(C_0\to\infty\), if
needed, and then send \(\zeta,\tau\downarrow0\), in the same iterated
order as in the exact-wall proof.

The retained seam ranges are
\[
 k=b_n+1,\ldots,n,
 \qquad k=c_n+1,\ldots,n,
 \qquad
 b_n=n\beta+O(1),\quad c_n=n\gamma+O(1).
\]
Thus their endpoints move by only \(O(1)\).  If
\(\Xi\in\mathcal T_{p_n}^{\psi}(f,\eta n)\), its cumulative aligned
seam height differs from the sampled trace of \(f\) by
\(O(\eta n+1)\).  The deterministic coefficient profiles in the
summation-by-parts proof of Lemma~\ref{lem:rf-seam-limit} have total
variation \(O(n)\).  Hence replacing the discrete trace by the sampled
continuum trace changes either seam sum by
\[
                         O(\eta n^2+n),
\]
uniformly in \(\Xi\).  The changed endpoints contribute only \(O(1)\)
extreme eigenvalues.  Since
\[
 \|\lambda_1((\widehat X_j^{(n)})^{[k]})\|_2
 +\|\lambda_k((\widehat X_j^{(n)})^{[k]})\|_2=O(n),
\]
these endpoint terms are \(O(n)=o(n^2)\) in \(L^2\) and in
expectation.  The intervals in the limiting Riemann sums move by
\(O(n^{-1})\), so the limiting seam functional is still
\(S_{p,\Delta}^{\rm rf}\).

For the fixed term,
\[
 d_n=n\delta+O(1),
 \qquad n-a_n=n(1-\alpha)+O(1).
\]
Theorem~\ref{thm:rf-fixed-limit} therefore applies directly and gives
the same fixed limit \(S_{p,0}\).

We have thus obtained the pointwise perturbed analogues of
Proposition~\ref{prop:rf-total-neighborhood}.  For the lower bound, fix
one \(f\in\operatorname{AHT}(\mathsf R_p)\), use
\eqref{eq:rf-moving-height-recovery} to ensure that its restricted
neighborhood is nonempty, apply the perturbed lower-neighborhood
estimate, and then let its fixed radius tend to zero.  Taking the
supremum over the fixed choice of \(f\) gives the variational lower
bound.

For the upper bound, fix \(\varepsilon>0\).  At each fixed
\(f\in\operatorname{AHT}(\mathsf R_p)\), use the pointwise perturbed
upper estimate to choose a radius \(r_f>0\) for which the doubled
neighborhood has normalized expected value at most
\(\mathcal S_p(f)+\varepsilon\).  Extract a finite subcover from the
balls \(B_\infty(f,r_f)\).  This finite list of centers and radii is
chosen once and is independent of \(n\).  Equation
\eqref{eq:rf-moving-height-upper} puts every normalized moving-domain
height in one of the doubled balls for all sufficiently large \(n\),
while \eqref{eq:rf-moving-height-recovery} makes each of the associated
restricted families nonempty.  The same uniform-family \(L^2\)
argument as in Lemma~\ref{lem:rf-uniform-family-variance}, applied to
this fixed finite list, shows that the expectation of their maximum is
their largest expectation plus \(o(n^2)\).  Thus the fixed-subcover
argument from Theorem~\ref{thm:rf-finite-panel-localization}, with all
neighborhoods interpreted through
\(\mathcal T_{p_n}^{\psi}(f,rn)\), applies pointwise and gives
\[
 n^{-2}\mathbb E\widehat O_n(p_n)
 \longrightarrow
 \sup_{f\in\operatorname{AHT}(\mathsf R_p)}\mathcal S_p(f)
 =\Phi_{\ell_1,\ell_2,\ell_3}(p).
\]

For variance, rerun the proof of
Lemma~\ref{lem:rf-uniform-family-variance} for the actual family
\(\mathcal T_{p_n}\).  Pad every actual \(0,3/2\)-coefficient vector
by zeros in the full size-\(n\) three-panel gap vector.  The
log-concave maximal-affine estimate is uniform over arbitrary
deterministic subfamilies, and the spectral-edge remainder is a
full-panel bound.  The full-panel spectral-edge removal, together with
the uniform seam and fixed-term \(L^2\) estimates just recorded,
therefore gives the conclusion of
Lemma~\ref{lem:rf-uniform-family-variance} for
\(\mathcal T_{p_n}\).  Hence
\[
 \operatorname{Var}(\widehat O_n(p_n))=o(n^4).
\]
Combining this with the mean limit proves the asserted convergence in
probability.
\end{proof}

\begin{corollary}[Rational points and arbitrary lattice approximants]
\label{cor:rf-rational-approximants}
Let \(p\) be a strictly interior rational point and let \(p_n/n\to p\)
be any deterministic lattice approximants.  Then
\begin{equation}
 \frac{\widehat O_n(p_n)}{n^2}
 \xrightarrow{\mathbb P}\Phi_{\ell_1,\ell_2,\ell_3}(p).
\label{eq:rf-rational-arbitrary-approximants}
\end{equation}
\end{corollary}

\begin{proof}
Write \(p=(\alpha,\beta,\gamma,\delta)\) and
\(\bar p=(0,\alpha,\beta,\gamma,\delta)\).  Let \(q\) be a common
denominator of \(p\).  For each residue
\(\varrho\in\{0,\ldots,q-1\}\), choose once and for all an integer
vector
\(e^{(\varrho)}=(0,e_a^{(\varrho)},e_b^{(\varrho)},
e_c^{(\varrho)},e_d^{(\varrho)})\) of coordinate sum \(\varrho\)
such that
\[
 p_{qm+\varrho}^{\circ}=m(q\bar p)+e^{(\varrho)}
\]
is strictly interior for all large \(m\).  Lemma~\ref{lem:rf-bounded-perturbation}
gives \eqref{eq:rf-rational-arbitrary-approximants} for this canonical
bounded-error sequence \(p_n^{\circ}\).

An arbitrary sequence \(p_n/n\to p\) is at lattice distance \(o(n)\)
from \(p_n^{\circ}\) in the same size-\(n\) tetrahedron.  Summing
\eqref{eq:rf-target-lattice-lipschitz} along a shortest lattice path
changes the normalized value by \(o(1)\) on the high-probability event,
which proves the claim.
\end{proof}

\begin{lemma}[Boundary-point law of large numbers]
\label{lem:rf-boundary-point-limit}
The function \(\Phi_{\ell_1,\ell_2,\ell_3}\), defined on the strictly interior
rational points of \(\mathsf T_{1234}\), has a unique Lipschitz
extension to the closed tetrahedron.  If
\(p\in\partial\mathsf T_{1234}\) and \(p_n/n\to p\) is any
deterministic sequence of lattice vertices, then
\begin{equation}
 \frac{\widehat O_n(p_n)}{n^2}
 \xrightarrow{\mathbb P}\Phi_{\ell_1,\ell_2,\ell_3}(p).
\label{eq:rf-boundary-point-limit}
\end{equation}
\end{lemma}

\begin{proof}
Take strictly interior rational points \(r,s\), and choose deterministic
lattice approximants \(r_n/n\to r\) and \(s_n/n\to s\) in the same
size-\(n\) tetrahedron.  On the event in
Lemma~\ref{lem:rf-target-lipschitz}, summing
\eqref{eq:rf-target-lattice-lipschitz} along a shortest lattice path gives,
after changing the deterministic constant to account for the chosen norm,
\[
 \frac{|\widehat O_n(r_n)-\widehat O_n(s_n)|}{n^2}
 \leq C\|r-s\|+o(1).
\]
Corollary~\ref{cor:rf-rational-approximants} gives convergence in
probability at both endpoints.  Since the displayed event has probability
tending to one, it follows that
\[
 |\Phi_{\ell_1,\ell_2,\ell_3}(r)-\Phi_{\ell_1,\ell_2,\ell_3}(s)|
 \leq C\|r-s\|.
\]
The strictly interior rational points are dense in the closed tetrahedron,
so \(\Phi_{\ell_1,\ell_2,\ell_3}\) has a unique Lipschitz extension to
\(\mathsf T_{1234}\).

Now fix \(p\in\partial\mathsf T_{1234}\) and choose a strictly interior
rational point \(r\).  For deterministic lattice approximants
\(r_n/n\to r\), the same lattice-path estimate gives, on the same
high-probability event,
\[
 \frac{|\widehat O_n(p_n)-\widehat O_n(r_n)|}{n^2}
 \leq C\|p-r\|+o(1).
\]
Consequently,
\begin{align*}
 \left|\frac{\widehat O_n(p_n)}{n^2}
              -\Phi_{\ell_1,\ell_2,\ell_3}(p)\right|
 &\leq
 \frac{|\widehat O_n(p_n)-\widehat O_n(r_n)|}{n^2}\\
 &\quad+
 \left|\frac{\widehat O_n(r_n)}{n^2}
              -\Phi_{\ell_1,\ell_2,\ell_3}(r)\right|
 +|\Phi_{\ell_1,\ell_2,\ell_3}(r)-\Phi_{\ell_1,\ell_2,\ell_3}(p)|.
\end{align*}
For fixed \(r\), the middle term tends to zero in probability by
Corollary~\ref{cor:rf-rational-approximants}; the first and third terms
are bounded by \(C\|p-r\|+o(1)\) and \(C\|p-r\|\), respectively.
Letting first \(n\to\infty\) and then the strictly interior rational point
\(r\to p\) proves \eqref{eq:rf-boundary-point-limit}.
\end{proof}

Recall the continuum interpolation introduced in
Section~\ref{sec:partII-main-theorem}.  It is normalized at lattice
vertices by
\begin{equation}
 \mathcal F_n(v/n)=\frac{\widehat O_n(v)}{n^2}.
\label{eq:rf-continuum-interpolation}
\end{equation}

\begin{theorem}[Uniform continuum law of large numbers]
\label{thm:rf-uniform-continuum}
The function \(\Phi_{\ell_1,\ell_2,\ell_3}\), defined at strictly interior
rational points by
\eqref{eq:partII-upfront-variational-value} and identified as the
finite-panel limit by
Theorem~\ref{thm:rf-finite-panel-localization}, has a unique Lipschitz
extension to \(\mathsf T_{1234}\).  Moreover,
\begin{equation}
 \boxed{
 \sup_{p\in\mathsf T_{1234}}
 \left|\mathcal F_n(p)-\Phi_{\ell_1,\ell_2,\ell_3}(p)\right|
 \xrightarrow{\mathbb P}0.}
\label{eq:rf-uniform-continuum-limit}
\end{equation}
In particular, for every real \(p\in\mathsf T_{1234}\) and every
deterministic sequence of lattice points \(p_n/n\to p\),
\begin{equation}
 \frac{\widehat O_n(p_n)}{n^2}
 \xrightarrow{\mathbb P}\Phi_{\ell_1,\ell_2,\ell_3}(p).
\label{eq:rf-real-pointwise-limit}
\end{equation}
\end{theorem}

\begin{proof}
Let \(\Phi_{\ell_1,\ell_2,\ell_3}\) denote the unique Lipschitz extension
supplied by Lemma~\ref{lem:rf-boundary-point-limit}.

The simplices in the interpolation above have finitely many shapes.
By definition, the center value of an octahedron is the average of
its six vertex values and therefore lies between their minimum and
maximum.  Hence Lemma~\ref{lem:rf-target-lipschitz} gives, with probability
tending to one, a common Lipschitz constant for all \(\mathcal F_n\).
Fix \(\varepsilon>0\) and choose a finite rational \(\delta\)-net
\(\mathcal N_\delta\) in the tetrahedron, taking its points strictly
interior.  Joint convergence in probability holds at this finite set.
For any \(p\), choose \(r\in\mathcal N_\delta\) with
\(\|p-r\|\le\delta\).  On the common Lipschitz event,
\[
 |\mathcal F_n(p)-\Phi(p)|
 \le |\mathcal F_n(r)-\Phi(r)|+2C\delta.
\]
Let first \(n\to\infty\) and then \(\delta\downarrow0\) to obtain
\eqref{eq:rf-uniform-continuum-limit}.  The pointwise assertion follows
by applying the same estimate to the interpolation simplex containing
\(p_n/n\).
\end{proof}

\begin{remark}[Relation to the two-source theory]
\label{rem:rf-proof-status}
The bulk localization and local concentration inputs are exactly those
of \cite{NSTGUE,NarLimit}.  In particular, Fefferman's theorem is used
directly, panel by panel, in the upper neighborhood bound through
Theorem~\ref{thm:rf-fefferman} and
Corollary~\ref{cor:rf-panel-fefferman}; it is also present indirectly in
the scale-comparison properties of \(\sigma_\edge^{\rm sym}\) obtained
from the corresponding arguments in \cite{NarLimit}.  The lower bound
instead uses the fixed-target two-scale interpolant construction of
Lemma~\ref{lem:rf-fixed-target-interpolation}, modeled on
\cite[\S8.5]{NarLimit}, together with the new global scaffold patching
lemma.  Only the good interpolation cells are modified; their complement
and the seam collars remain in the reference scaffold.

The genuinely new points are: the exact red-free seam weights in
Lemma~\ref{lem:rf-local-weights}; the two-vertex fixed term in
Theorem~\ref{thm:rf-explicit-fixed-term}; the definition of the admissible
space as the compact set of full-sequence global cut-compatible heights;
and the seam-compatible scaffold patching in the proof of
Proposition~\ref{prop:rf-bulk-neighborhood}.  No estimate for red-lozenge
weights is used.
\end{remark}

\section*{Acknowledgements}

This work was supported by the Department of Atomic Energy, Government of
India [project number RTI4014]; by the Infosys--Chandrasekharan virtual center
for Random Geometry at the Tata Institute of Fundamental Research (TIFR).
I am grateful to Charles Fefferman for sharing his notes on quantitative
differentiation, Scott Sheffield for explaining the surface tension used here and Terence Tao for answering my queries on the fixed index properties of the GUE minor process.

\section*{Disclosure of AI-assisted work}

OpenAI's ChatGPT and Codex systems were used as research-assistance and
writing tools.  The specific AI-assisted ideas that remain in the
manuscript are: the organization of the denominator calculation into
two seam-compatible light-cone pieces and their Boolean union in
Section~3; the
polygonal descriptions of the extremal matchings and the directed-envelope
forcing argument in Section~5; the separation of the
four-point mutable
\(g\)-vector from the frozen normalization; and the comparison of the upper
bound in Proposition~\ref{prop:rf-bulk-neighborhood} with the two-scale
argument of \cite{NarLimit}, including the fixed-target good-cell
interpolation used in its lower bound.  The systems were also used for index checks, small examples,
figure preparation, language editing, and compilation.

\end{document}